\documentclass[11pt,letterpaper,twoside,reqno,nosumlimits]{amsart}
\usepackage{amsmath,amsfonts,amsbsy,amsgen,amscd,mathrsfs,amssymb,amsthm,mathtools,tensor,bbm,stmaryrd}
\usepackage{siunitx}
\usepackage{algorithm,algorithmic}
\usepackage{enumitem}
\setlist[enumerate,1]{label={\roman*)}}

\usepackage{authblk}
\usepackage[foot]{amsaddr}

\usepackage[toc, page]{appendix}

\usepackage[margin=1in]{geometry}
\usepackage[dvipsnames]{xcolor}
\definecolor{MyColor}{HTML}{0047AB}
\usepackage{fancyhdr}
\usepackage[colorlinks=true, bookmarks=false, citecolor=blue, urlcolor=blue]{hyperref}

\usepackage{etoolbox}
\makeatletter
\renewcommand{\@secnumfont}{\bfseries}
\makeatother
\patchcmd{\section}{\scshape}{\bfseries}{}{}
\patchcmd{\section}{\normalfont}{\normalfont\color{MyColor}}{}{}
\patchcmd{\subsection}{\normalfont}{\normalfont\color{MyColor}}{}{}
\makeatletter
\def\subsubsection{\@startsection{subsubsection}{3}%
\z@{.5\linespacing\@plus.7\linespacing}{-.5em}%
{\normalfont\bfseries}}
\makeatother

\usepackage{graphicx}
\usepackage{float}
\usepackage{caption}
\usepackage{subcaption}
\usepackage{tikz}

\usepackage{booktabs} 
\usepackage{multirow}
\usepackage{tabu} 

\usepackage{soul}
\usepackage[color=yellow]{todonotes}
\newtheorem{theorem}{Theorem}[section]
\newtheorem{lemma}[theorem]{Lemma}
\newtheorem{definition}[theorem]{Definition}
\newtheorem{proposition}[theorem]{Proposition}
\newtheorem{corollary}[theorem]{Corollary}
\newtheorem{assumption}[theorem]{Assumption}

\newtheorem{remark}[theorem]{Remark}

\makeatletter
\@tfor\next:=abcdefghijklmnopqrstuvwxyzABCDEFGHIJKLMNOPQRSTUVWXYZ\do{
\def\command@factory#1{%
\expandafter\def\csname b#1\endcsname{\mathbf{#1}}
\expandafter\def\csname fk#1\endcsname{\mathfrak{#1}}
\expandafter\def\csname bb#1\endcsname{\mathbb{#1}}
\expandafter\def\csname cl#1\endcsname{\mathcal{#1}}
\expandafter\def\csname scr#1\endcsname{\mathscr{#1}}
\expandafter\def\csname bcl#1\endcsname{\mathbfcal{#1}}
}
\expandafter\command@factory\next
}

\newcommand{\rmd}{\mathrm{d}}
\newcommand{\ds}{\, \rmd s}
\newcommand{\dt}{\, \rmd t}

\newcommand{\dZt}{\, \rmd \bZ_t}

\newcommand{\dd}{\mathrm{d}}
\newcommand{\var}{\textnormal{var}}
\newcommand{\loc}{\textnormal{loc}}
\newcommand{\cB}{\mathcal{B}}
\newcommand{\cD}{\mathcal{D}}
\newcommand{\cE}{\mathcal{E}}
\newcommand{\cF}{\mathcal{F}}
\newcommand{\cL}{\mathcal{L}}
\newcommand{\cX}{\mathcal{X}}

\newcommand{\N}{\mathbb{N}}
\newcommand{\R}{\mathbb{R}}
\newcommand{\T}{\mathbb{T}}

\newcommand{\Z}{\mathbb{Z}}
\newcommand{\eps}{\varepsilon}

\renewcommand{\leq}{\leqslant}
\renewcommand{\geq}{\geqslant}

\newcommand{\lp}{\llparenthesis}
\newcommand{\rp}{\rrparenthesis}

\DeclareMathOperator{\supp}{supp}

\begin{document}

\title[On the well-posedness of (nonlinear) rough continuity equations]{On the well-posedness of (nonlinear) rough continuity equations}

\author{Lucio Galeati$^1$}
\address{$^1$Dipartimento di Ingegneria e Scienze dell'Informazione e Matematica, Università degli Studi dell'Aquila, Italy}
\email{$^1$lucio.galeati@univaq.it}
\author{James-Michael Leahy$^2$}
\address{$^2$ Department of Mathematics, Imperial College London, United Kingdom and PhysicsX, United Kingdom}
\email{$^2$j.leahy@imperial.ac.uk}
\author{Torstein Nilssen$^3$}
\address{$^3$Department of Mathematics, University of Adger, Norway}
\email{$^3$torstein.nilssen@uia.no}

\maketitle
\begin{abstract}
Motivated by applications to fluid dynamics, we study rough differential equations (RDEs) and rough partial differential equations (RPDEs) with non-Lipschitz drifts.
We prove well-posedness and existence of a flow for RDEs with Osgood drifts, as well as well-posedness of weak $L^p$-valued solutions to linear rough continuity and transport equations on $\R^d$ under DiPerna--Lions regularity conditions; a combination of the two then yields flow representation formula for linear RPDEs.
We apply these results to obtain existence, uniqueness and continuous dependence for $L^1\cap L^\infty$-valued solutions to a general class of nonlinear continuity equations.
In particular, our framework covers the $2$D Euler equations in vorticity form with rough transport noise, providing a rough analogue of Yudovich's theorem.
As a consequence, we construct an associated continuous random dynamical system, when the driving noise is a fractional Brownian motion with Hurst parameter $H \in (1/3,1)$.
We further prove weak existence of solutions for initial vorticities in $L^1\cap L^p$, for any $p\in [1,\infty)$.\\


\noindent \textbf{MSC Classification (2020):} 60L20, 60L50, 60H15, 35R60, 35Q31.

\end{abstract}

{\hypersetup{linkcolor=MyColor}
\setcounter{tocdepth}{3}
\tableofcontents}

\section{Introduction}

Over the last thirty years, several authors have advocated for the use of stochastic fluid dynamics equations with transport-type noise, in order to account for the turbulent small scales of realistic fluids, see for instance \cite{BrCaFl1991,BrCaFl1992,MikRoz2004,memin2014fluid,Holm2015,crisan2019solution}.
The starting point of the theoretical derivation in most of these works is to prescribe an evolution for the Lagrangian particles composing the fluid, formally of the form
\begin{equation} \label{eq:intro ODE}
\dot{y}_t = u_t(y_t) +  \sum_{k=1}^m \xi_k(y_t) \dot{Z}^k_t, \qquad y_0 = x \in \R^d.
\end{equation}
Here $u$ is the coarse-grained macroscopic velocity field of the fluid, while the random terms $\dot Z^k$ represent the unresolved, turbulent small scales fluctuations around it.
Starting from this Lagrangian description one can then move to an Eulerian one, by means of the (stochastic) material derivative.
For instance, the evolution of a density $\rho$ carried along the fluid flow \eqref{eq:intro ODE}, by the conservation of mass principle, is (formally) given by the stochastic continuity equation
\begin{equation}\label{eq:intro fluid CE}
    \partial_t \rho_t + \nabla\cdot (u_t \rho_t) + \sum_{k=1}^m \nabla\cdot (\xi_k \dot Z^k_t) = 0.
\end{equation}
One can similarly derive more complex nonlinear equations for the fluid flow $u$ itself, resulting in stochastic Euler and Navier--Stokes-type systems.

Separation of scales and homogenization arguments suggest to take the fluctuations Gaussian in nature. Due to mathematical convenience, most authors (cf. \cite{BrCaFl1991,MikRoz2004,memin2014fluid}) then assume $Z^k$ to be Brownian, in order to have access to It\^o calculus and consequently solve \eqref{eq:intro ODE}-\eqref{eq:intro fluid CE}.
However, in terms of modelling turbulent fluids, it is important to allow for non-Markovian noise with memory dependence, see for instance \cite{lilly2017fractional, faranda2014modelling,franzke2015stochastic}.

Rough path theory, introduced by Terry Lyons in \cite{Lyons98}, allows to overcome this obstacle, providing a comprehensive solution theory for differential equations of the form \eqref{eq:intro ODE} whenever $Z$ is not differentiable in time, but can be lifted to a \emph{rough path}. As a consequence, it allows for a large class of non-Markovian signals, like fractional Brownian motion of Hurst parameter $H\in (1/4,1)$.
However, the implementation of this theory typically requires generous regularity assumptions on $u$ and $\xi_k$, which can be become challenging to verify when $u$ is itself a solution to a nonlinear equation.

The seminal paper \cite{Holm2015} by Holm provides a systematic way of deriving stochastic fluid dynamics equations from geometric constraints, by enforcing the validity of an underlying variational principle; at the Eulerian level, the resulting noise is usually of \emph{Lie transport} type. As a consequence, the geometrical structure of geophysical fluids can become as a guiding principle for designing robust \emph{stochastic parametrization schemes} and has relevant applications in uncertainty quantification, data assimilation and filtering (see e.g. \cite{cotter2020modelling,cotter2020filtering}).

Recently, the framework from \cite{Holm2015} has been combined with rough path theory in \cite{CrisanHolmLeahyNilssen22a}, where the authors allow the noise $Z$ to be any geometric rough path.
The resulting incompressible Euler equations with rough Lie transport take the form 
\begin{equation} \label{eq:Euler intro}
    \begin{cases}
        \partial_t u +  u_t \cdot \nabla u_t + \sum_{k=1}^m \left( \xi_k \cdot \nabla u_t + (D \xi_k)^T u_t \right) \dot{Z}_t^k + \nabla p_t = 0,\\
        \nabla\cdot u_t =0.
    \end{cases}
\end{equation}
On the $d$-dimensional torus $\T^d$, the mathematical solvability of \eqref{eq:Euler intro} is studied in \cite{CrisanHolmLeahyNilssen22b}, proving local well-posedness of maximal solutions in $W^{n,2}$ for $n \geq \lfloor \frac{d}{2} \rfloor + 2$, as well as a Beale-Kato-Majda blow-up criterion in terms of the $L^1_t L^\infty_x$-norm of the vorticity; the results hold for sufficiently regular, divergence-free $\xi_k$.
In dimension $d=2$, the equation for the vorticity $\omega := \nabla^\perp\cdot u$ reads as
\begin{equation} \label{eq:vorticity intro}
    \partial_t \omega_t +  u_t \cdot \nabla \omega_t + \sum_{k=1}^m  \xi_k \cdot \nabla \omega_t\,  \dot{Z}_t^k  = 0,
\end{equation}
where the velocity $u$ can be recovered from $\omega$ by the Biot-Savart law
$u = K\ast \omega$; the active scalar transport nature of \eqref{eq:vorticity intro} then allows to deduce preservation of the $L^\infty_x$-norm of $\omega$.
Combined with the aforementioned blow-up criterion, this ultimately allows the authors in \cite{CrisanHolmLeahyNilssen22b} to deduce global well-posedness for the rough $2$D Euler equations \eqref{eq:Euler intro} for velocities $u$ in $W^{3,2}$.
On the other hand, in light of the structure \eqref{eq:vorticity intro} and the classical results by Yudovich \cite{yudocivh1963,yudovich1995} in the deterministic case, it is reasonable to expect \eqref{eq:vorticity intro} to be wellposed as soon as (for instance) $\omega_0\in L^1_x\cap L^\infty_x$, without the need for any further Sobolev regularity.

The goal of the present paper is to address the current gaps in the existing literature concerning the solvability of \eqref{eq:intro ODE}-\eqref{eq:intro fluid CE}-\eqref{eq:vorticity intro}, in the presence of rough path noise $Z$, in order to provide a robust theoretical framework to support the aforementioned applications.
In short, our main contributions the following:
\begin{itemize}
    \item Wellposedness of the RDE \eqref{eq:intro ODE} and existence of an underlying continuous flow whenever $u$ satisfies an Osgood regularity condition, but is not necessarily Lipschitz.
    \item Solvability of rough transport and continuity equations of the form \eqref{eq:intro fluid CE} when $u$ satisfies DiPerna--Lions regularity conditions \cite{diperna1989ordinary}.
    \item Solvability the rough $2$D Euler equations in the class of vorticities $\omega\in L^1_x\cap L^\infty_x$, as part of a more general class of nonlinear rough continuity equations (see \eqref{eq:intro_nonlinear_RCE} below), and construction of an underlying random dynamical system.
\end{itemize}

Let us finally remark that throughout this paper we take the underlying state space to be the (technically slightly more challenging) full space $\R^d$, rather than the torus $\T^d$ as in previous works. We believe this choice to be important in order to pave the way for future generalizations to arbitrary smooth domains, by using extension theorems and flow representation methods; it is clear however that, mutatis mutandis, our results readily readapt to the torus $\T^d$ as well.

\subsection{Main results and discussion} %\label{subsec:main results}

The first main goal of the present work is to 
%In the present work we will show how to 
prove well-posedness of rough differential equations and construct the corresponding flow of homeomorphisms, when the drift term enjoys Osgood-type regularity.
In the following, in order to specify what we mean by the \eqref{eq:intro ODE} when $Z^k$ are not differentiable, we will assume that $Z$ admits a rough lift $\bZ=(Z,\bbZ)$ of finite $\fkp$-variation for some $\fkp\in [2,3)$, $\bZ\in \mathcal{C}^p$ for short, so that we can apply rough path theory (cf. Section \ref{sec:RDEs}).
Within this framework, \eqref{eq:intro ODE} admits a rigorous interpretation as the rough differential equation (RDE) \eqref{eq:intro_RDE} below.

Our main result for RDEs can be summarized as follows:
%(see Theorem \ref{thm:wellposedness-RDE} for the precise statement). 
\begin{theorem}\label{thm:Osgood intro}%[Informal statement] 
Assume that $Z$ lifts to a rough path $\bZ \in \clC^{\mathfrak{p}}$ for $\mathfrak{p} \in [2,3)$, $\xi \in C_b^3$ and $b$ is a bounded Osgood vector field. Then the rough differential equation
\begin{equation}\label{eq:intro_RDE}
    \dd y_t = b_t(y_t) \dd t + \sum_{k=1}^m \xi_k(y_t) \dd \mathbf{Z}^k_t, \quad y_0=x\in\R^d
\end{equation}
is well-posed and induces a flow $(t,x) \mapsto \Phi_t(x)$ of homeomorphisms on $\R^d$, which satisfies 
$$
\sup_{t \in [0,T]} |\Phi_t(x) - \Phi_t(\tilde{x})|  \leq F(|x- \tilde{x}|)
$$
for some modulus of continuity $F$. The same estimate holds with $\Phi$ replaced by its inverse $\Phi^{-1}$. 
\end{theorem}

See Theorem \ref{thm:wellposedness-RDE} in Section \ref{subsec:RDE-flow} for the more precise version of the above statement, with an explicit formula for $F$, which is defined in terms of the Osgood modulus of continuity of $b$.
At a technical level, Theorem \ref{thm:Osgood intro} is non-trivial as it truly requires to work with finite $\fkp$-variation spaces, and the same proof would not work when manipulating simpler H\"older spaces $C^\gamma$ (with $\gamma\sim 1/\fkp$) as often done in rough paths \cite{FH2020}. This is because, in order to establish uniqueness and regularity of the flow leveraging on the regularity of $b$, one needs to close estimates for quantities of the form
\begin{align*}
    \bigg\llbracket \int_0^\cdot b_t(y^1_t)\dd t- \int_0^\cdot b_t(y^2_t)\dd t \bigg\rrbracket_E
\end{align*}
for a suitable choice of the seminorm and function space $E$.
Contrary to H\"older functions, spaces of finite $\fkp$-variation are very convenient for this task, as the $1$-variation of Lebesgue integrals can be controlled by the $L^1$-norm of its integrand.
Let us also point out that Theorem \ref{thm:wellposedness-RDE} holds for any rough path $\bZ$, not necessarily geometric; in particular, existence and regularity of the inverse flow $\Phi^{-1}$ requires the use of time reversal arguments. However a general statement about time-reversed RDEs, for non geometric $\mathfrak{p}$-variation rough paths, seems to be missing in the literarure; we fill here this gap, cf. Definition \ref{defn:time-reversed-path} and Lemma \ref{lem:time-reversal-RDE} from Section \ref{subsec:RDE-flow}.

Next, we pass to consider the linear rough partial differential equations (RPDEs) on $\R^d$ associated to the RDE \eqref{eq:intro_RDE}, namely the rough continuity equation
\begin{equation} \label{eq:intro CE}
\dd \rho_t + \nabla \cdot (b_t \rho_t)\dd t  +  \sum_{k=1}^m \nabla \cdot (\xi_k  \rho_t) \dd \bZ_t^k = 0
\end{equation}
and rough transport equation
\begin{equation} \label{eq:intro TE}
\dd f_t +  b_t \cdot \nabla  f_t\, \dd t  + \sum_{k=1}^m  \xi_k \cdot \nabla f_t\,  \dd \bZ_t^k = 0.
\end{equation}

Equations \eqref{eq:intro CE}-\eqref{eq:intro TE} again can be made meaningful by a rough path formalism; in particular, here we adopt the unbounded rough drivers framework from \cite{BaiGub2017}.
To this end, from now on we will always assume $\bZ$ to be a geometric rough path of finite $\fkp$-variation, which we abbreviate by $\bZ\in \mathcal{C}^\fkp_g$.

In view of applications to nonlinear PDEs, where the drift depends nontrivially on the solution itself, it is often unreasonable to impose too strong regularity assumptions on $b$;
moreover, as long as the vector fields $\xi_k$ are taken smooth enough, it is reasonable to expect the equations \eqref{eq:intro CE}-\eqref{eq:intro TE} to be well-posed under the same conditions on $b$ allowed by the deterministic literature.
We prove that this is indeed the case, and in particular we successfully develop a wellposedness theory for $L^p$-valued solutions, as long as $b$ satisfies the same regularity assumptions as in the celebrated DiPerna--Lions theory \cite{diperna1989ordinary}, filling a noticeable gap in the existing RPDE literature.
%We show well-posedness in $L^p$-spaces of these equations under similar regularity assumptions on $b$ as in \cite{diperna1989ordinary}.
%
The main analytical challenge with such low-regularity solutions is that the RPDE is only satisfied in the sense of distributions (in space); thus, special care is needed whenever using energy estimates, duality or doubling of variables arguments. On top of that, one must handle low time regularity coming from $\bZ$.
Nevertheless, a careful combination of the techniques from \cite{diperna1989ordinary} and \cite{BaiGub2017} allows to successfully solve \eqref{eq:intro CE}-\eqref{eq:intro TE}, under rather minimal requirements on $(b,\xi)$. That is, the joint regularity condition we impose on $(b,\xi)$ coincides with the known ones for each term taken separately, coming respectively from deterministic and rough PDEs.

Compared the standard deterministic theory from \cite{diperna1989ordinary}, our strategy for solving linear RPDEs adds a little twist by mostly focusing on \eqref{eq:intro CE} rather than \eqref{eq:intro TE}. Uniqueness is achieved by establishing a \emph{product formula}, see below, which gives access to duality arguments and weak-strong stability results. \emph{Renormalizability} of transport equations then becomes a mere corollary of the latter, and does not play anymore a pivotal role as it did in \cite{diperna1989ordinary}.
 
Our main result for linear RPDEs is summarized below; we refer the reader to Section \ref{subsec:function.spaces} for the relevant function spaces and notations.

\begin{theorem}\label{thm:wellposedness-linear-RPDE-intro}
    Let $\fkp\in [2,3)$ and assume that $Z$ lifts to a geometric rough path $\bZ\in \clC^{\fkp}_g$; let
    \begin{equation*}%\label{eq:intro_linear_RPDE_assumption}
        \frac{b}{1+|x|}\in L^1_t L^1_x + L^1_t L^\infty_x,\quad b\in L^1_t W^{1,1}_\loc, \quad \nabla\cdot b \in L^1_t L^\infty_x,\quad \xi_k \in C^3_b, \quad \nabla\cdot \xi_k =0.
    \end{equation*}
    Then for any $\rho_0\in L^1_x\cap L^\infty_x$ (resp. $f_0\in L^1_x\cap L^\infty_x$), there exists a unique solution to \eqref{eq:intro CE} (resp. \eqref{eq:intro TE}) belonging to $\cB_b([0,T]; L^1_x\cap L^\infty_x)$; furthermore, $\rho,f\in C([0,T];L^p_x)$ for all $p\in [1,\infty)$.

    Moreover, the following hold:
    \begin{enumerate}
        \item {\em Product formula:} the product $\rho f$ is again a solution to \eqref{eq:intro CE};
        \item {\em Duality:} for every $t\in [0,T]$, $\langle \rho_t,f_t\rangle=\langle \rho_0,f_0\rangle$;
        \item {\em Renormalizability:} for every $\beta\in C^1_b$, $\beta(f)$ is again a solution to \eqref{eq:intro TE}.
    \end{enumerate}
\end{theorem}

Theorem \ref{thm:wellposedness-linear-RPDE-intro} will be proved in Section \ref{subsec:linear-stability}.
Let us stress that Section \ref{sec:linear-RPDEs} overall contains many other useful results, like for instance conservation of mass (Lemma \ref{lem:conservation-mass}), stability in both weak and strong topologies (Corollaries \ref{cor:stability-1-rte} and \ref{cor:linear-RPDE-stability-2}) and criteria for strong compactness (Corollary \ref{cor:linear-compactness}).
The result remain true whenever e.g. $\rho_0\in L^p_x$ for some $p\in [1,\infty]$, as long as $b$ possesses the correct ``complementary integrability'', cf. Theorem \ref{thm:uniqueness-linear-RPDE}.

When $b$ satisfies the Osgood regularity assumption, we can further show that solutions to the continuity and transport equations \eqref{eq:intro CE}-\eqref{eq:intro TE} admit classical flow representations; see Theorem \ref{thm:linear-RPDE-wellposed-lagrangian} in Section \ref{subsec:linear-stability} for a more precise version of the next statement.
Here, given a measure $\mu$ on $\R^d$, $(\Phi_t)_\sharp \mu$ denotes its pushforward under the continuous map $\Phi_t$.

\begin{theorem}\label{thm:flow_repr_intro}%[Flow representation]
    Under the assumptions of Theorem \ref{thm:wellposedness-linear-RPDE-intro}, suppose in addition that $b$ is an Osgood vector field.
    Then the flow $\Phi$ associated to the RDE \eqref{eq:intro_RDE} from Theorem \ref{thm:Osgood intro} is {\em quasi-incompressible}, in the sense that for any $t\geq 0$ and any Borel set $A\subset \R^d$ it holds
    \begin{equation*}
        \exp\Big(-\int_0^t \| \nabla\cdot b_s\|_{L^\infty} \dd s\Big) \mathscr{L}^d(A) 
        \leq \mathscr{L}^d(\Phi_t (A)) 
        \leq  \exp\Big(+\int_0^t \| \nabla\cdot b_s\|_{L^\infty} \dd s\Big) \mathscr{L}^d(A)
    \end{equation*}
    where $\mathscr{L}^d$ denotes the Lebesgue measure on $\R^d$; a similar statement holds with $\Phi_t$ replaced by $\Phi^{-1}_t$.

    Moreover in this case the unique solution $\rho$ to \eqref{eq:intro CE} is given by
    $$
        \rho_t = (\Phi_t)_{\sharp}\rho_0
    $$
    and the unique solution $f$ to \eqref{eq:intro TE} is given by
    $$
        f_t = f_0\circ \Phi^{-1}_t.
    $$
\end{theorem}

Armed with the above results, we are then able to deduce well-posedness in $L^1_x\cap L^{\infty}_x$ and flow-representation for a class of non-linear rough PDEs of the form
\begin{equation}\label{eq:intro_nonlinear_RCE}
    \dd \rho_t + \nabla \cdot ((K \ast \rho_t) \rho_t) \dd t  +  \sum_{k=1}^m \nabla \cdot (\xi_k  \rho_t) \dd \bZ_t^k = 0
\end{equation}
for suitable convolutional kernels $K$. 

\begin{theorem}\label{thm:intro_wellposed_nonlinear}
Let $\xi \in C_b^3$ with $\nabla \cdot \xi_k = 0$ and assume that $Z$ lifts to a geometric rough path $\bZ \in \clC^{\mathfrak{p}}_g$ for some $\mathfrak{p} \in [2,3)$.
Further assume that the convolution kernel $K$ satisfies the following:
\begin{enumerate}
    \item[i.] $K\in L^1_x + L^\infty_x$, $\nabla\cdot K\in L^\infty_x.$
    \item[ii.] $\nabla K$ is a Fourier multiplier of $0$-homegeneity: denoting by $\widehat{\nabla K}$ the Fourier transform of $\nabla K$, it holds $\widehat{\nabla K}\in C^\infty(\R^d\setminus \{0\})$ and there exists $C>0$ such that
    \begin{align*}
        |D^{(\alpha)} \widehat{\nabla K}(\eta)|\leq C |\eta|^{-\alpha} \quad\forall\, \eta\in \R^d\setminus\{0\}, \quad \alpha\in \bigg\{0,\ldots, 2\Big\lfloor 1+\frac{d}{2}\Big\rfloor\bigg\}.
    \end{align*}
\end{enumerate}
Then for every initial condition $\rho_0 \in L^1_x \cap L^{\infty}_x$ there exists a unique solution $\rho$ to  \eqref{eq:intro_nonlinear_RCE} in the class $\cB_b([0,T]; L^1_x \cap L^{\infty}_x)$.
Moreover, $\rho$ is of the form $\rho_t = (\Phi_t)_{\sharp} \rho_0$,
    %$$
    %\rho_t = (\Phi_t)_{\sharp} \rho_0
    %$$
    where $\Phi$ is the flow associated to the RDE \eqref{eq:intro_RDE} with $b_t= K \ast \rho_t$.
    %$$
    %    \dot{y}_t = u_t(y_t)  + \sum_{k=1}^m \xi_k(y_t) \dot{Z}^k_t
    %$$
    %where $u_t = K \ast \rho_t$. 
\end{theorem}

In fact, the results from  Section \ref{sec:nonlinear-RPDEs} allow a more general class of (time-dependent) kernels $K$, see the (slightly more technical) Assumptions \ref{ass:abstract-kernel-1}-\ref{ass:abstract-kernel-2}-\ref{ass:abstract-kernel-3} therein; in particular, Theorem \ref{thm:intro_wellposed_nonlinear} is a consequence of Proposition \ref{prop:nonlinear-uniqueness} and Theorem \ref{thm:nonlinear-wellposedness2} from Section \ref{subsec:nonlinear-uniqueness}.
To the best of our knowledge, even in the deterministic case $\xi\equiv 0 $,  at this level of generality our result on well-posedness for \eqref{eq:intro_nonlinear_RCE} appears to be new; the only work sufficiently close to it we are aware of (imposing stronger assumptions on the kernel $K$) is \cite{inversi2023lagrangian}.

The most important example of nonlinear rough continuity equations of the form \eqref{eq:intro_nonlinear_RCE} we can cover are the 2D rough Euler equations in vorticity form:
\begin{equation}\label{eq:intro Euler}
    \begin{cases}
    \dd \omega_t + u_t  \cdot  \nabla \omega_t\, \dd t  +  \sum_{k=1}^m  \xi_k \cdot \nabla   \omega_t\, \dd \bZ_t^k = 0,\\
    \nabla\cdot u_t=0,\quad \nabla^\perp\cdot u_t = \omega_t.
\end{cases}
\end{equation}
In this case, the velocity field $u$ can be reconstructed from $\omega$ by the Biot--Savart law:
\begin{align*}
    u_t = K \ast \omega_t, \quad \text{where}\quad K(z) = \frac{z^{\perp}}{2 \pi |z|^2}.
\end{align*}
Theorem \ref{thm:intro_wellposed_nonlinear} then recovers the celebrated  wellposedness result due to Yudovich \cite{yudocivh1963,yudovich1995} and extends it to the rough setting.
In particular, the stability estimates we obtain in this case are strong enough to construct a \emph{random dynamical system} underlying the dynamics induces by random $\bZ$. In the next statement, we allow for an infinite time horizon $t\in \R_{\geq 0}=[0,+\infty)$.

\begin{theorem}\label{thm:intro_Euler}
    Let $\fkp \in [2,3)$, $\bZ\in\mathcal{C}^\fkp_g$, $\xi\in C^3_b$ with $\nabla\cdot\xi_k=0$.
    Then for any $\omega_0\in L^1_x\cap L^\infty_x$, there exists a unique global solution $\omega \in \cB_b(\R_{\geq 0};L^1_x\cap L^\infty_x)$ to \eqref{eq:rough-euler}, which moreover belongs to $C(\R_{\geq 0}; L^p_x)\cap C_{w-\ast}(\R_{\geq 0};L^\infty_x)$ for any $p\in [1,\infty)$. The solution is renormalized and of the form
    \begin{equation*}
        \omega_t(x) = \omega_0(\Phi_t^{-1}(x))
    \end{equation*}
    where $\Phi$ is the flow generated by the RDE \eqref{eq:intro_RDE} for $b_t=K \ast \omega_t$, $K$ being the Biot-Savart kernel.
    %$$
    %    \rmd y_t = u_t(y_t) \rmd t + \xi(y_t) \rmd \bZ_t
    %$$
    %for $u_t = K \ast \omega_t$, $K$ being the Biot-Savart kernel.
    Moreover, we have
    \begin{equation*}
        \|\omega_t \|_{L^p_x} = \|\omega_0 \|_{L^p_x}, \qquad \forall\, t\geq 0, \ \forall\, p \in [1,\infty].
    \end{equation*}
    Let $\{\omega^n_0\}_n$ be a bounded sequence in $L^1_x\cap L^\infty_x$, resp. $\omega_0\in L^1_x\cap L^\infty_x$, and denote by $\omega^n$, resp. $\omega$, the associated solutions to \eqref{eq:rough-euler}. Let $p\in (1,\infty)$, then:
    \begin{itemize}
        \item[i)] if $\omega_0^n \rightharpoonup \omega_0$ weakly in $L^p_x$, then $\omega^n$ converge to $\omega$ in $C_w([0,T];L^q_x)\cap C_{w-\ast}([0,T];L^\infty_x)$ for all $q\in (1,\infty)$ and $T\in (0,+\infty)$;
        \item[ii)] if $\omega^n_0\to \omega_0$ strongly in $L^p_x$, then $\omega^n\to \omega$ in $C([0,T]; L^q_x)$ for all $q\in (1,\infty)$ and $T\in (0,+\infty)$.
    \end{itemize}
    Let $R\in (0,+\infty)$ and define
    \begin{align*}
        \mathcal{X}_R := \left\{ \omega_0 \in L^1_x \cap L_x^{\infty} : \|\omega_0\|_{L^1_x \cap L^{\infty}_x} \leq R \right\}.
    \end{align*}
    Suppose now that $\bZ$ is a random geometric $\mathfrak{p}$-rough path cocycle.
    Then the associated RPDE \eqref{eq:rough-euler} generates a continuous random dynamical system on $\cX_R$, when it is endowed with either the strong topology $\tau^{strong}$ induced on $\mathcal{X}_R$ by the $L^p_x$-norm, or the weak topology $\tau^{weak}$ induced by weak convergence in $L^p_x$.
\end{theorem}

Theorem \ref{thm:intro_Euler} comes from a combination of the more general Theorems \ref{thm:yudovich-euler}, \ref{thm:yudovich stability} and \ref{thm:euler_RDS} from Section \ref{subsec:nonlinear-euler-yudovich}.
The stability results in points \emph{i)} and \emph{ii)} above are on par with the deterministic literature, cf. \cite{nguyen2022}.
Let us point out that, with the aforementioned topologies, $(\mathcal{X}_R,\tau^{strong})$ is a separable metric space, while  $(\mathcal{X}_R,\tau^{weak})$ is a compact metric space.
In particular, Theorem \ref{thm:intro_Euler} applies when the driving signal $\bZ$ is (the geometric rough enhancement) of a fractional Brownian motion of Hurst parameter $H\in (1/3,1)$.
In the Brownian case $H=1/2$, the geometric requirement amounts to working with Stratonovich noise, and our result is comparable to previous ones like \cite{brzezniak2016existence}. 

Finally, in the case of unbounded initial vorticity $\omega_0$ and less regular $\xi_k$, we are still obtain a weak existence result for \eqref{eq:intro Euler}, in the style of those from DiPerna--Majda \cite{DiPMaj1987}, Delort \cite{Delort1991} and Schochet \cite{Schochet1995}.
This last statement follows from Propositions \ref{prop:schochet-delort} and \ref{prop:schochet-delort-renormalized} from Section \ref{subsec:nonlinear-euler-delort}.

\begin{theorem}\label{thm:intro-schochet-delort}
    Let $\fkp \in [2,3)$, $\bZ\in\mathcal{C}^\fkp_g$, $\xi\in C^2_b$ with $\nabla\cdot\xi=0$.
    Then for any $p\in [1,\infty)$ and any $\omega_0\in L^1_x\cap L^p_x$, there exists a global weak solution $\omega$ to \eqref{eq:intro Euler} satisfying
    \begin{equation*}
        \sup_{t\geq 0} \| \omega_t\|_{L^q_x} \leq \| \omega_0\|_{L^q_x} \quad \forall\, q\in [1,p].
    \end{equation*}
    If moreover $p\in [2,\infty)$, then $\omega\in C(\R_{\geq 0};L^1_x\cap L^p_x)$ and it is renormalized, in the sense that for any $\beta\in C^1_b$, $v= \beta(\omega)$ is a weak solution to
    \begin{align*}
        \dd v_t + (K\ast \omega_t)\cdot\nabla v_t + \sum_{k=1}^m \xi_k\cdot\nabla v_t\, \dd \bZ^k_t =0. 
    \end{align*}
\end{theorem}

In this paper we always work on the full space $\R^d$.
This implies a few technical difficulties, related to compact embeddings not being naively available anymore, and the need in applications to work with coefficients $b$ only satisfying local regularity and/or growth conditions. Moreover, natural distributional concepts of solutions are only local in nature, based on testing against $C^\infty_c$, which lacks either a Banach or Fréchet structure; therefore at the level of unbounded rough drivers, rather than working with a fixed scale of Banach spaces, we need to consider a family $(\cF_{l,R})_{l,R}$ of them, indexed by the radius $R\in (0,+\infty)$ related to the support of functions; see Section \ref{subsec:unbounded-rough} for the details.

It is clear however that all aforementioned results transfer to the (simpler) case of RPDEs on the torus $\T^d$, with periodic boundary conditions; in this case, both local regularity and growth conditions can be replaced by simpler global requirements.

\subsection{Relations to the existing literature}%\label{subsec:literature}

Stochastic partial differential equations (SPDEs) with Brownian transport-type noise have been advocated in the context of fluid dynamics in many works, see \cite{BrCaFl1991,BrCaFl1992,MikRoz2004,memin2014fluid,Holm2015} and the references therein.
The separation of scales ideas adopted by these authors are older and date back at least to the pioneering work of Hasselmann \cite{Hasselmann1976} in stochastic climate modelling; Hasselmann's paradigm has become a standard in modern applications, see the review \cite{franzke2015stochastic}.
For a recent revisitation of Hasselmann's ideas in the framework of \cite{Holm2015}, see \cite{CrHoKo2023}.
For regular coefficients, a large class of first and second order SPDEs are treated in the monograph \cite{Kunita1997}.
Among the theoretical reasons for transport noise, besides modelling ones like the Wong--Zakai principle, let us mention the fascinating possibility of regularization by noise phenomena, as first noticed in \cite{FGP2010}.
Stochastic fluid dynamics is currently a very active field of research and we do not aim here at describing it in its entirety; we refer to the monographs \cite{Flandoli2015,FlaLuon2023,CriGoo2024} for an overview.

Yudovich's theorem \cite{yudocivh1963,yudovich1995} remains to this day among the sharpest well-posedness results for $2$D Euler, also in view of the recent counterexamples \cite{ABCDLGJK2024,BruCol2023,brue2024flexibility}. It has received numerous revisitations over the years, see for instance \cite{loeper2006uniqueness,crippa2021elementary,nguyen2022}.
In the SPDE case with Stratonovich Brownian noise, with domain $\T^2$, an analogous result was first established in \cite{brzezniak2016existence}. Therein, leveraging on the power of It\^o calculus, the authors only need to require the coefficients $\xi_k$ to be Lipschitz continuous, with suitable summability; they however impose the additional condition $\sum_k \xi_k\otimes \xi_k= c I_2$, where $I_2$ is the $2\times 2$ identity matrix. To the best of our knowledge, even in the Brownian case, our result is the first to treat the full space case $\R^2$; although we need the higher regularity $\xi_k\in C^3_b$, coming from the applicability of rough path theory, we do not need to impose any further condition on the resulting covariance of the noise.

%After the pioneering works of Terry Lyons \cite{Lyons98}, culminating in the seminal paper \cite{Lyons98},
Born with the seminal work of Terry Lyons \cite{Lyons98},
rough paths have now grown into a mature theory; we refer to the monographs \cite{LyQi2002,FV2010,FH2020} for a comprehensive overview.

RDEs \eqref{eq:intro_RDE} are well-known to be solvable when $\xi_k$ are regular enough (typically $\xi_k\in C^\gamma_b$ with $\gamma\sim \fkp$) and the drift $b$ is bounded and globally Lipschitz, see \cite[Ch. 12.1]{FV2010} and \cite[Thm. 8]{DOR2015}; this result can now also be seen as a particular subcase of the one concerning rough stochastic differential equations from \cite{friz2021rough}. Possibly unbounded drifts satisfying monotonicity conditions have been treated in \cite{RieSch2017,BONNEFOI202258}, while measure-valued drifts corresponding to a reflection measure are considered in \cite{DGHTa}. Our work seems to be the first one treating Osgood drifts instead.

Rough transport PDEs remain somewhat less studied, in comparison to the vast literature of RDEs.
Among the first results in this direction, let us mention \cite{CarFri2009,CaFrOb2011,CDFO2013,DOR2015}, mostly based on pathwise arguments like stability properties of the underlying RDEs, random flow transformations, random rough paths and Feynman--Kac representations.
This class of results however did not allow to provide an intrinsic meaning to the RPDE, whose solutions were merely recovered as the unique limit of mollified problems.
This issue was successfully addressed by different methods in \cite{DFS} (later refined in \cite{FNS}) and, more important to our approach, the framework of \emph{unbounded rough drivers} from \cite{BaiGub2017}.
The latter provides a purely Eulerian framework, without relying on representations based on the Lagrangian flow of the underlying RDE; it takes inspiration from the approach to RDEs developed by Davie \cite{davie2008}. In some sense, the present paper merges the ideas of \cite{DFS} and \cite{BaiGub2017} by using the Eulerian formulation of \cite{BaiGub2017} along with a product formula to show well-posedness, as done in \cite{DFS}.
The theory of unbounded rough drivers has been subsequently developed in \cite{DGHT2019, HN2021} and these techniques have been applied to the rough Navier-Stokes equations in \cite{HLN2021,HLN1, FHLN} and to rough Euler equations in \cite{CrisanHolmLeahyNilssen22b}.
Among other approaches to transport equations driven by (higher order) rough paths, let us also mention the recent \cite{BDFT2021}.

While preparing this paper, the preprint \cite{roveri2024wellposednessrough2deuler} appeared on arXiv. Therein, the authors show well-posedness of Euler equations in the Yudovich class on the 2D torus $\T^2$, as well as well-posedness of \eqref{eq:intro_RDE} for log-Lipschitz $b$ (which is a particular case of Osgood drift). Their work shares similar intuitions as ours, based on the use of spaces of finite $\fkp$-variation and unbounded rough drivers.
However their results are as not as general in scope, as they do not include any treatment of linear RPDEs with DiPerna--Lions drifts, nor general nonlinear rough continuity equations or weak existence results of Schochet-Delort type; moreover, differently from ours, the stability results obtained therein are not robust enough to construct the underlying continuous random dynamical system.

\subsection{Future perspectives}
The results obtained in this paper open up several directions to be developed in the future.

First of all, in order to solve the linear RPDEs \eqref{eq:intro CE}-\eqref{eq:intro TE}, here we pursued an Eulerian approach à la DiPerna--Lions \cite{diperna1989ordinary} (and its readaptation in the rough framework from \cite{BaiGub2017}), based on a priori $L^p$-estimates and commutators; it would be interesting instead to focus directly on the construction of a \emph{(rough) Regular Lagrangian Flow} for the RDE \eqref{eq:intro_RDE}, readapting the techniques from Crippa--De Lellis \cite{CriDeL2008}.

In the context of nonlinear fluid dynamics equations, in this paper we only focused on (more challenging) inviscid RPDEs like rough $2$D Euler \eqref{eq:intro Euler}, but it is clear that similar arguments apply to viscous ones like rough $2$D Navier--Stokes:
\begin{align*}
        \dd \omega^\nu_t + (K\ast \omega^\nu_t)\cdot\nabla \omega^\nu_t + \xi\cdot\nabla \omega^\nu_t\, \dd \bZ_t = \nu \Delta\omega^\nu_t. 
\end{align*}
Similarly to the deterministic case studied in \cite{NSW2021,CiCrSp2021,CDE2022}, we expect the techniques from the proof of Theorem \ref{thm:intro_wellposed_nonlinear} to provide \emph{quantitative convergence rates} of rough $2$D Navier--Stokes to rough $2$D Euler in the \emph{vanishing viscosity limit} $\nu\to 0^+$.
Notice however that the use of a Feyman-Kac type flow representation for the rough Navier--Stokes requires to consider an underlying rough stochastic differential equation, admitting both a rough path and a Brownian motion as underlying drivers; therefore, additional techniques from either \cite{DOR2015} or more recently \cite{friz2021rough} might be needed for this task.

The stability properties of the solution map $2$D Euler obtained in Theorem \ref{thm:intro_Euler} immediately imply Wong--Zakai type results, see \cite[Section 9.2]{FH2020} and \cite{roveri2024wellposednessrough2deuler} for similar discussions, as well as support theorems, cf. \cite[Section 9.3]{FH2020}.
In the case of Gaussian rough paths, one can then employ Schilder's theorem and the contraction principle to derive \emph{large deviation results} in the \emph{vanishing noise} limit, by considering \eqref{eq:intro Euler} with $Z$ replaced by $\sqrt{\eps} Z$ and sending $\eps\to 0^+$. It would be interesting to couple this result to the one described in the previous paragraph, and actually consider large deviations in the joint limit $(\eps,\nu)\to (0,0)$ of vanishing noise and vanishing diffusion.

Although our main motivation for studying rough continuity equations \eqref{eq:intro_nonlinear_RCE} comes from fluid dynamics, there is another fundamental PDE sharing a similar structure, which is given by the \emph{Vlasov-Poisson equations}, as noted in \cite{loeper2006uniqueness}; in fact, in the deterministic setting Euler can be recovered by Vlasov through a quasi-neutral limit, see the classical work \cite{Brenier2000}. Loeper's techniques for Vlasov-Poisson have recently received many revisitations, see \cite{Iacobelli2022,CISS2024,IacJun2024}.
A specific type of transport Brownian noise was recently proposed in the Vlasov-Poisson system in \cite{bedrossian2022vlasov}, where local well-posedness in Sobolev spaces was shown; it would be interesting to pursue instead the Lagrangian approach to prove global wellposedness and propagation of regularity, for general rough transport noise sharing the same structure as in \cite{bedrossian2022vlasov}.

Throughout the paper we only considered $2$D Euler in full space $\R^2$; as mentioned, one can similarly treat the torus $\T^2$. It would be interesting to extend the result by allowing general smooth domains $\Omega\subset\R^2$, endowed with the slip condition $u\cdot \overrightarrow{n}=0$ on $\partial\Omega$, where $\overrightarrow{n}$ denotes the outward normal at the boundary. For deterministic Euler, Yudovich theorem still holds in this framework, see \cite[Sec. 2.3, Thm. 3.1]{marchioro2012mathematical}; however, in the SPDE and RPDE literature, we have found almost no references treating this problem (differently from Dirichlet/no slip b.c., which in different contexts is analyzed e.g. in \cite{funaki1979,NevOli2021,noboriguchi2024}).
The closest result we are aware of in this direction is \cite[Thm. 1.17]{Goodair2023}, providing a weak existence result for stochastic $2$D Euler equations in convex domains, under slip b.c. $u\cdot \overrightarrow{n}=0$.

\subsection*{Structure of the paper}

In Section \ref{sec:preliminaries} we collect the notation and preliminary results needed in the main body of the paper; they concern function spaces, space of $\fkp$-variation, rough paths, moduli of continuity and finally some novel concepts of weak convergence for Banach-valued paths.
%
%Section \ref{sec:RDEs} shows well-posedness of the RDE \eqref{eq:intro ODE}, i.e. the Lagrangian perspective, when $b$ is assumed to be bounded and Osgood regular. Moreover, we establish flow-properties and show quasi-incompressibility of the flow in this setting.
Section \ref{sec:RDEs} shows well-posedness of the RDE \eqref{eq:intro_RDE} under Osgood regularity conditions on $b$, as well as regularity and quasi-incompressibility of the induced flow.
Section \ref{sec:linear-RPDEs} is devoted to linear rough continuity and transport equations \eqref{eq:intro CE}-\eqref{eq:intro TE}.
By employing the unbounded rough drivers framework, we provide criteria for existence, uniqueness, as well as renormalizability and continuous dependence on the data (in strong and weak topologies), when the drift $b$ satisfies a DiPerna--Lions regularity condition. %Moreover, we show that the classical Lagrangian representation formulae holds when $b$ in addition is Osgood-regular.
In Section \ref{sec:nonlinear-RPDEs}, we consider general nonlinear rough continuity equations on $\R^d$ of the form \eqref{eq:intro_nonlinear_RCE}; we obtain existence, uniqueness and continuous dependence results. From them, we infer a Yudovich-type theorem in Section \ref{subsec:nonlinear-euler-yudovich}, where we also construct the underlying random dynamical system; we then prove a Schochet--Delort-type result in Section \ref{subsec:nonlinear-euler-delort}.
We collect several technical results in the Appendices \ref{app:tech-lem}, \ref{app:compactness} and \ref{app:smoothing}, concerning respectively rough path lemmas, compactness criteria in weak topologies, and tensorization and smoothing operator results from the theory of unbounded rough drivers. 

\section{Notation and preliminaries}\label{sec:preliminaries}

This section is mostly a collection of standard terminologies and recaps, which can be skipped at a first reading by the experienced reader.

Sections \ref{subsec:notation}-\ref{subsec:function.spaces} consist of notations and function spaces used throughout the whole paper; Sections \ref{subsec:p-variation}-\ref{subsec:osgood} deal respectively with spaces of finite $\mathfrak{p}$-variation and Osgood moduli of continuity, which will be mostly needed in the study of RDEs from Section \ref{sec:RDEs}. Finally, Section \ref{subsec:UCW} is a bit more original and introduces a concept of uniform convergence in weak topologies, which will be crucial when applying compactness arguments in Sections \ref{sec:linear-RPDEs}-\ref{sec:nonlinear-RPDEs}.

\subsection{Notation}\label{subsec:notation}
We write $a \lesssim b$ if there exists a constant $c>0$ such that $a \leq cb$, and $a \lesssim_\lambda b$ if we want to stress that the constant $c$ depends on a parameter $\lambda$.

For $x,y \in \R^d$, we write $x \cdot y$ for their scalar product and $|x|$ for the Euclidean norm.
We write $B_R(x)$ for the $d$-dimensional open ball of radius $R$ centered in $x$; when $x=0$, we use the shorthand notation $B_R := B_R(0)$. We set $\R_{\geq 0}=[0,+\infty)$.
We denote by $\mathscr{L}^d$ the Lebesgue measure on $\R^d$.

For a smooth function $f:\R^d \rightarrow \R^n$, we denote by $\partial_j f_i$ its partial derivatives and by $D_x f=(\partial_j f_i)_{i,j}$ its Fréchet differential, seen as a map from $\R^d$ to $\R^{n\times d}$; when there is no room for confusion, we simply write $Df$.
When $n=1$, we write $\nabla f=(\partial_i f)_{i=1}^d$ for the gradient, and when $n=d$ we write $\nabla \cdot f=\sum_{i=1}^d \partial_i f_i$ for the divergence.

We denote by $C^{\infty}_c(\R^d;\R^n)$, or simply $C^\infty_c$ when there is no risk of confusion, the set of infinitely differentiable, compactly supported functions, with the usual topology of test functions. Given $U\subset \R^d$, we write $C^\infty_c(U)$ for the subset of test functions such that $\supp \varphi\subset U$, where $\supp \varphi$ denotes the support.
We denote by $\cD'(\R^d;\R^m)$, or simply $\cD'$, the dual space of (possibly vector-valued) distributions. 
We keep using the notations $\partial_j f_i$, $D f$ to denote weak derivatives/differentials whenever $f\in\cD'$.
We let $f\ast g$ denote the convolution between (sufficiently integrable) functions, as well as between test functions and (fastly decaying) distributions.

Given Banach spaces $E_1$, $E_2$, $\mathcal{L}(E_1,E_2)$ denotes the Banach space of bounded linear operators from $E_1$ to $E_2$, with operator norm $\| \cdot\|_{\mathcal{L}(E_1,E_2)}$.
Whenever $E_i\subset F$ for $i=1,2$, where $F$ is an ambient Banach space, the intersection $E_1 \cap E_2$ is again a Banach space with the canonical norm $\| \cdot \|_{E_1 \cap E_2} = \| \cdot \|_{E_1} + \| \cdot \|_{E_2}$; we denote by $E_1+E_2$ the space of all elements $f\in F$ which may be written as $f=f_1+f_2$ with $f_i\in E_i$.
%(although we don't need it, it is a Banach space when equipped with the norm $\|f\|_{E_1+E_2} := \inf \{\|x\|_{E_1} + \|y\|_{E_2}: f = x+y\}$)
We denote by $E^\ast$ the topological dual of a Banach space $E$.

When there is no risk of confusion, we will use the notation $\otimes$ to denote different kinds of tensor products. Specifically:
given $x,y\in \R^d$, $x\otimes y\in \R^{d\times d}$ is the matrix given by $(x\otimes y)_{ij}=x_i y_j$;
given functions $f,g:\R^d\to R$, $f\otimes g:\R^{2d}\to\R$ is given by $(f\otimes g)(x,y)=f(x)g(y)$, and we extend this definition by duality to distributions, so that if $f,g\in \cD'(\R^d)$, then $f\otimes g\in\cD'(\R^{2d})$;
if $E_1$, $E_2$ are abstract Banach spaces and $v\in E_1$, $w\in E_2$, then $v\otimes w\in \cL(E_2^\ast,E_1)$ is given by $(v\otimes w)(w^\ast):=v \langle w,w^\ast \rangle_{E_2, E^\ast_2}$, for all $w^\ast\in E_2^\ast$;
with a slight abuse, $v\otimes w$ may also be interpreted as a bilinear operator on $E_1^\ast\times E_2^\ast$, by $(v\otimes w)(v^\ast,w^\ast):= \langle v,v^\ast \rangle_{E_1, E^\ast_1}\,\langle w,w^\ast \rangle_{E_2, E^\ast_2}$, or as an element of the tensor space $E_1\otimes E_2$.

\subsection{Function spaces}\label{subsec:function.spaces}

For $m\geq 0$, we denote by $C^m_b(\R^d;\R^n)$ the Banach space of $m$-times differentiable functions from $\R^d$ to $\R^n$ which are bounded and with continuous bounded derivatives up to order $m$, with norm
\begin{equation*}
    \| f\|_{C^m_b} :=\sum_{k=0}^m \| D^{(k)}f\|_{C^0_b} := \sum_{k=0}^m \sup_{x\in \R^d} | D^{(k)}f(x)|,
\end{equation*}
with the convention that $D^{(0)}f=f$.
When the domain and codomain are clear from the context, we simply write $C^m_b$.
We denote by $C_\loc^m=C^m_\loc(\R^d;\R^n)$ the Fréchet space of functions which are locally in $C^m_b$, namely such that $f\,\varphi\in C^m_b$ for all $\varphi\in C^\infty_c$; $f^n\to f$ in $C^m_\loc$ if and only if $\| \varphi (f^n-f)\|_{C^m_b}\to 0$ for all $\varphi\in C^\infty_c$. Note that convergence in $C^0_\loc$ amounts to uniform convergence on compact sets. 

For $p \in [1,\infty]$. we write $L^p(\R^d;\R^n)$ for the usual Lebesgue spaces; when domain and codomain are clear, we simply write $L^p_x$, and denote their norm by $\|f\|_{L^p_x}$. Given $p\in [1,\infty]$, we denote by $p'$ its conjugate H\"older exponent, viz $1/p+1/p'=1$.
Given a Lebesgue measurable subset $A\subset \R^d$, we denote by $\mathbbm{1}_A$ its characteristic function and set $\| f\|_{L^p(A)}:=\| f\mathbbm{1}_A\|_{L^p_x}$.
$L^p_{\loc}=L^p_\loc(\R^d;\R^n)$ stand for local $L^p$ spaces, given by functions $f$ such that $f \varphi\in L^p_x$ for all $\varphi\in C^\infty_c$.

For $k\in\N$, we write $W^{k,p}(\R^d;\R^n)$ for the usual Sobolev spaces; since we shall only consider Sobolev spaces in the space variable $x$, whenever clear we shortly denote them by $W^{k,p}$ and their norm by
\begin{align*}
    \|f\|_{W^{k,p}}:= \sum_{k=0}^m \| D^{(k)}f\|_{L^p_x}.
\end{align*}
We write $W^{k,p}_\loc$ for the Fréchet space of functions $f\in L^p_\loc$ such that $\varphi\, f \in W^{k,p}$ for all $\varphi\in C^\infty_c$.
For $k=0$, as a convention we have $W^{0,p}=L^p_x$, $W^{0,p}_\loc=L^p_\loc$; $f^n\to f$ in $W^{k,p}_\loc$ if $\| \varphi(f^n-f)\|_{W^{k,p}}\to 0$ for all $\varphi\in C^\infty_c$.

Given a locally convex space $V$, we write either $\langle f, g\rangle_{V,V^\ast}$ or $\langle g, f\rangle_{V^\ast,V}$ for the duality pairing between $V$ and its dual space $V^\ast$.
When there is no risk of confusion, we simply write $\langle\cdot,\cdot \rangle$; we use it for instance to denote the inner product in $L^2_x$, but also the duality pairing between $L^p_x$ and $L^{p'}_x$, or the one between $C^\infty_c(\R^d)$ and the dual space of distributions $\mathcal{D}'(\R^d)$.

Given a closed interval $I \subset [0,\infty)$ and a function $f: I \rightarrow E$ for some set $E$, we shall employ the subscript notation $f_t$ for the evaluation of $f$ at $t \in I$.
When $E$ is a Banach space, we denote by $\mathcal{B}_b(I;E)$ the space of bounded, measurable functions $f:I\to E$, where $E$ is endowed with the Borel $\sigma$-algebra coming from $\| \cdot\|_E$.
$\mathcal{B}_b(I;E)$ is a Banach space with norm
\begin{align*}
    \| f\|_{\mathcal{B}_b(I;E)}:=\sup_{t\in I} \| f_t\|_E.
\end{align*} 
When $E$ is a Banach space and the interval $I$ is clear from the context, we denote by $L^p_tE := L^p(I;E)$ the Bochner--Lebesgue space of (equivalence classes of) strongly measurable functions $f:I\to E$ such that
\begin{align*}
    \|f\|_{L_t^q E}:= \bigg( \int_I \| f_t\|_E^q\, \dd t\bigg)^{1/q}<\infty
\end{align*}
with the usual convention that we take the essential supremum norm if $q=\infty$.

For $E\in\{C^m_b, L^p_x. W^{k,p}\}$, we denote by $L^q_t E_\loc=L^q(I;E_\loc)$ the Fréchet space of functions $f$ such that $\varphi\, f\in L^q_t E$ for all $\varphi\in C^\infty_c$; $f^n\to f$ in $L^q_t E_\loc$ if $\| \varphi(f^n-f)\|_{L^q_t E}\to 0$ for all $\varphi\in C^\infty_c$.
%The local variant $L^p_tW^{k,q}_\loc$ is defined as the set of all $f: I \rightarrow W^{k,q}_\loc$ such that $f \in L_t^p W^{k,q}(U)$ for all open bounded $U \subset \R^d$.

Given a measure space $(E,\mathcal{E},\mu)$ and a measurable mapping $f: (E,\mathcal{E},\mu) \rightarrow (F,\cF)$, where $\mathcal{F}$ is a $\sigma$-algebra on  $F$, we denote by $f_{\sharp} \mu$ the \emph{pushforward} of $\mu$ under $f$, namely the measure on $(F, \mathcal{F})$ defined by
$$
f_{\sharp} \mu (A) := \mu( f^{-1}(A)) \qquad \forall\, A \in \mathcal{F}.
$$

\subsection{Spaces of $\mathfrak{p}$-variation and rough paths}\label{subsec:p-variation}

We recall in this section several basic concepts from rough path theory; for further details, see the monographs \cite{FV2010,FH2020}. We mostly follow the presentations from \cite[Section 2]{DGHT2019}, \cite[Section 2.2]{HLN2021}.

For a closed interval $I$ , we define
\begin{align*}
    \Delta_I := \{(s, t) \in I^2 : s\leq t\}, \quad \Delta^{(2)}_I:=\{ (s,u,t)\in I^3: s\leq u \leq t\}.
\end{align*}
For notational simplicity, for $T\in (0,+\infty)$, we let $\Delta_T := \Delta_{[0,T]}$ and $\Delta^{(2)}_T := \Delta^{(2)}_{[0,T ]}$.
Let $(E,\| \cdot\|_E)$ be a Banach space; given a function $g: \Delta_I\to E$, we denote by $g_{st}$ the evaluation of $g$ at $(s,t)\in \Delta_I$. $g$ is said to have \emph{finite $\mathfrak{p}$-variation} on $I$ for some $\mathfrak{p}\in (0,+\infty)$ if
\begin{equation*}
    \llbracket g\rrbracket_{\mathfrak{p},I;E}:= \sup_{\pi\in \Pi(I)} \bigg( \sum_{(t_i,t_{i+1})\in \pi} \| g_{t_i t_{i+1}} \|_E^\mathfrak{p} \bigg)^{1/\mathfrak{p}} <\infty,
\end{equation*}
where $\Pi(I)$ denotes the set of all possible finite partitions $\pi$ of $I$; we denote by $C^{\mathfrak{p}-\var}_2 (I;E)$ the set of all continuous functions $g:\Delta_I\to E$ of finite $\mathfrak{p}$-variation.
Similarly, we denote by $C^{\mathfrak{p}-\var}(I;E)$ the set of all continuous paths $x:I\to E$ such that $\delta x\in C^{\mathfrak{p}-\var}_2(I;E)$, where $\delta x_{st} := x_t-x_s$ for $s\leq t$.
In the latter case, with a slight abuse, we will identify $\llbracket x\rrbracket_{\mathfrak{p},I;E}$ with $\llbracket \delta x\rrbracket_{\mathfrak{p},I;E}$.
For any $\mathfrak{p}\in [1,\infty)$, the space $C^{\mathfrak{p}-\var}([s,t];E)$ is Banach with norm
\begin{equation*}
    \| x\|_{\mathfrak{p},[s,t];E}:= \| x_s\|_E + \llbracket x\rrbracket_{\mathfrak{p},[s,t];E}.
\end{equation*}
Whenever $E=\R^m$ for some $m\in\N$ and the context is clear, we will shorten the notation and just write $\| x\|_{\mathfrak{p},[s,t]}$, $\llbracket x\rrbracket_{\mathfrak{p},[s,t]}$. When $[s,t]=[0,T]$, we will just write $\| x\|_{\mathfrak{p},T}$, $\llbracket x\rrbracket_{\mathfrak{p},T}$ and denote $C^{\mathfrak{p}-\var}([0,T];\R^m)$ simply by $C^{\mathfrak{p}-\var}$, similarly for $C^{\mathfrak{p}-\var}_2$.

\begin{remark}\label{rem:p-variation-lsc}
    For any $\mathfrak{p}\in [0,+\infty)$, using the definition, it is easy to check that the $C^{\mathfrak{p}-\var}_2$\mbox{-}seminorm is lower semicontinuous (l.s.c.): given a bounded, pointwise converging sequence $\{g^n\}_n$ in $C^{\mathfrak{p}-\var}_2(I;E)$, namely such that $\sup_n \llbracket g^n\rrbracket_{\mathfrak{p},I;E}<\infty$ and $g^n_{st}\to g_{st}$ in $E$ for all $(s,t)\in\Delta_I$ and some $g\in C(\Delta_I;E)$, it holds that $g\in C^{\mathfrak{p}-\var}_2(I;E)$ and
    \begin{equation*}
        \llbracket g\rrbracket_{\mathfrak{p},I;E} \leq \liminf_{n\to\infty} \llbracket g^n\rrbracket_{\mathfrak{p},I;E}.
    \end{equation*}
    Similarly for bounded sequences $\{x^n\}_n\subset C^{\fkp-\var}(I;E)$ converging pointwise to $x\in C(I;E)$.
\end{remark}

We say that a mapping $w:\Delta_I\to [0,+\infty)$ is a \textit{control} on $I$ if it is continuous, $w(s,s)=0$ for all $s\in I$ and it is superadditive, namely
\begin{equation*}
    w(s,u)+w(u,t)\leq w(s,t)\quad \forall (s,u,t)\in \Delta^{(2)}_I.
\end{equation*}

\begin{remark}\label{rem:relation-control-variation}
    Given $g\in C^{\mathfrak{p}-\var}_2(I;E)$, it can be shown (cf. \cite[Proposition 5.8]{FV2010}) that
\begin{align*}
    w_g(s,t):= \llbracket g\rrbracket_{\mathfrak{p},[s,t];E}^\mathfrak{p}
\end{align*}
defines a control. This control is ``optimal'', in the sense that $\| g_{st}\|_E \leq w_g(s,t)^{1/p}$ and for any other control $\tilde w$ satisfying this property it holds $w_g(s,t)\leq \tilde w(s,t)$, cf. \cite[Proposition 5.10]{FV2010}.
\end{remark}

Given any $g:\Delta_I\to E$, we define the second order increment operator $\delta g:\Delta_I^{(2)}\to E$ by
\begin{align*}
    \delta g_{sut} := g_{st}-g_{su}-g_{ut} \quad \forall\, (s,u,t)\in \Delta_I^{(2)}.
\end{align*}
With these preparations, we can now recall some fundamental concepts from rough path theory.

\begin{definition}%\label{defn:rough-path}
    Let $m\in \N$, $T\in (0,+\infty)$ and $\mathfrak{p}\in [2,3)$. A pair $\mathbf{Z}=(Z,\Z):[0,T]\to \R^m\times \R^{m\times m}$ is a {\rm continuous $\mathfrak{p}$-rough path}, $\mathbf{Z}\in\clC^\mathfrak{p}=\clC^\mathfrak{p}([0,T];\R^m)$ for short, if $Z\in C^{\mathfrak{p}-\var}$, $\Z\in C^{\mathfrak{p}/2-\var}_2$ and they satisfy {\rm Chen's relation}
    \begin{equation}\label{eq:chen-relation}
        \delta\Z_{sut} = \delta Z_{su}\otimes \delta Z_{ut}\quad \forall \, (s,u,t)\in \Delta_T^{(2)}.
    \end{equation}
\end{definition}

%Throghout the paper, we will always assume to be working with $\mathbf{Z}$ such that $Z_0=0$; this is natural since only $\delta Z$, which is insensitive to addition of constants in $Z$, appears in \eqref{eq:chen-relation}.
For any $\mathbf{Z}\in \clC^\mathfrak{p}$ and $s\leq t$, we adopt the incremental notation $\mathbf{Z}_{st}:=(\delta Z_{st}, \Z_{st})$. $\clC^\mathfrak{p}$ lacks a linear structure, but is a complete metric space endowed with the metric
\begin{equation}\label{eq:metric-p-rough}
    d_{\mathfrak{p},T}(\mathbf{Z},\tilde{\mathbf{Z}}):= \| Z-\tilde Z\|_{\mathfrak{p},T} + \llbracket \Z - \tilde\Z\rrbracket_{\mathfrak{p}/2,T}.
\end{equation}
With a slight abuse, we will still use the norm notation to denote
\begin{align*}
    \| \mathbf{Z}\|_{\mathfrak{p},T}:= d_{\mathfrak{p},T}(\mathbf{Z},0):= \| Z\|_{\mathfrak{p},T} + \llbracket \Z \rrbracket_{\mathfrak{p}/2,T}
\end{align*}

\begin{remark}\label{rem:control-rough-path}
    Analogously to Remark \ref{rem:relation-control-variation}, the joint condition $(Z,\Z)\in C^{\mathfrak{p}-\var}\times C^{\mathfrak{p}/2-\var}_2$ is equivalent to the existence of a control $w_{\mathbf{Z}}$ such that
    \begin{equation}\label{eq:control-p-rough}
        |\delta Z_{st}|^\mathfrak{p} + |\Z_{st}|^{\mathfrak{p}/2} \leq w_{\mathbf{Z}}(s,t) \quad\forall\, (s,t)\in \Delta_T,
    \end{equation}
    with an explicit choice given by $w_{\mathbf{Z}}(s,t):= \llbracket Z\rrbracket_{\mathfrak{p},[s,t]}^\mathfrak{p} + \llbracket \Z\rrbracket_{\mathfrak{p}/2,[s,t]}^{\mathfrak{p}/2}$. Again this is ``almost optimal'', in the sense that if $\tilde w$ is another control satisfying \eqref{eq:control-p-rough}, then it holds
    \begin{equation*}
         \llbracket Z\rrbracket_{\mathfrak{p},[s,t]}^\mathfrak{p} \leq \tilde w(s,t), \quad \llbracket \Z\rrbracket_{\mathfrak{p}/2,[s,t]}^{\mathfrak{p}/2} \leq \tilde w(s,t), \quad
         w_{\mathbf{Z}}(s,t) \leq 2 \tilde w(s,t).
    \end{equation*}
\end{remark}
Any smooth path $Z:[0,T]\to \R^m$ admits a so called \textit{canonical lift} to a rough path, by setting
\begin{equation}\label{eq:canonical-lift}
    \Z_{st}:= \int_s^t \delta Z_{sr} \otimes \dot Z_r\, \dd r;
\end{equation}
it is easy to check that such $\Z$ satisfies \eqref{eq:chen-relation} and the resulting $\mathbf{Z}$ is called a \textit{smooth rough path}.

\begin{definition}\label{defn:geom-rough-path}
    A continuous $\mathfrak{p}$-rough path $\mathbf{Z}$ is {\rm geometric} if it belongs to the closure of the set of smooth rough paths, under the metric $d_{\mathfrak{p},T}$.
    In other terms, $\mathbf{Z}$ is geometric if there exists a sequence of smooth rough paths $\mathbf{Z}^n$ as defined by \eqref{eq:canonical-lift} such that $d_{\mathfrak{p},T}(\mathbf{Z}, \mathbf{Z}^n)\to 0$.
    The space of geometric $\mathfrak{p}$-rough paths is denoted by $\clC^\mathfrak{p}_g=\clC^\mathfrak{p}_g([0,T];\R^m)$.
\end{definition}

To exemplify the concept, let us mention that in the case of $Z$ being sampled as a Brownian motion, its {\em Stratonovich lift} $(Z,\bbZ^{Strat})$ is a geometric $\mathfrak{p}$-rough path for any $\mathfrak{p}>2$, while the {\em It\^o lift} $(Z,\bbZ^{It\hat{o}})$ is not, cf. \cite[Exercise 13.11]{FV2010}.

Let us recall some properties of controls which will come handy in the sequel.

\begin{remark}\label{rem:properties-controls}
If $w_1$, $w_2$ are controls, then it is immediate to check that so is $a\, w_1 + b\, w_2$ for all $a,\, b\geq 0$, as well as $w_1 w_2$.
More generally, if $w$ is a control and $\gamma:\Delta\to [0,+\infty)$ is increasing, in the sense that $\gamma(s,t)\leq \gamma(s',t')$ whenever $[s,t]\subset [s',t']$, then $\gamma\, w$ is a control.
Moreover for any controls $w_1$, $w_2$ and any $a,\,b>0$, there exists another control $w_3$ such that
\begin{equation*}
w_1(s,t)^a w_2(s,t)^b = w_3(s,t)^{a+b} \quad \forall\, [s,t]\subset [0,T]
\end{equation*}
which is precisely given by $w_3=w_1^{a/(a+b)} w_2^{b/(a+b)}$; if $a>1$ and $w$ is a control, then so is $w^a$. For the proofs of the last two statements, see \cite[Exercise 1.9]{FV2010}.    
\end{remark}

A key tool in rough path theory is the so called \emph{sewing lemma}, first introduced in \cite{gubinelli2004,feyel2006}. The statement here is a convenient rephrasing of \cite[Lemma 2.2]{DGHT2019}, to which we refer the reader for a proof.

\begin{lemma}[Sewing Lemma]\label{lem:sewing}
Let $E$ be a Banach space, $I$ as interval, $g:\Delta_I\to E$. Suppose $g\in C^{1/\eta-\var}_2(I;E)$ for some $\eta>1$ and that there exists a control $w$ such that
\begin{equation*}
    \| \delta g_{sut}\|_E \leq w(s,t)^{\eta} \quad \forall\, (s,u,t)\in \Delta^{(2)}_I;
\end{equation*}
then it holds
\begin{equation*}
    \| g_{st}\|_E\leq C\, w(s,t)^{\eta} \quad \forall\, (s,t)\in \Delta_I ,
\end{equation*}
and therefore as a consequence
\begin{equation*}
\llbracket g\rrbracket_{1/\eta,[s,t];E} \leq C w(s,t)^{\eta} \quad \forall\, (s,t)\in \Delta_I.
\end{equation*}
\end{lemma}

\subsection{Osgood moduli of continuity}\label{subsec:osgood}

We recall here some basic facts concerning moduli of continuity, which will be used later.

\begin{definition}%\label{defn:osgood}
We say that $h:\R_{\geq 0}\to\R_{\geq 0}$ is an {\rm Osgood modulus of continuity} if $h$ is continuous, strictly increasing, subadditive (i.e. $h(r_1+r_2)\leq h(r_1)+h(r_2)$) and satisfying
\begin{equation}\label{eq:defn.osgood}
h(0)=0, \quad \int_{0+} \frac{1}{h(r)} \dd r=+\infty.
\end{equation}
\end{definition}

\begin{remark}\label{rem:modulus-1}
Given a subadditive modulus of continuity $h$, \cite[Lem. B]{medvedev2001} guarantees the existence of another concave modulus of continuity $\tilde h$ satisfying the two-sided bound $h \leq \tilde h \leq 2 h$; thus if needed, we may always assume the Osgood modulus $h$ we are working with to be concave without loss of generality. Together with $h(0)=0$, this implies \emph{pseudoconcavity} of $h$, namely that
\begin{equation*}%\label{eq:pseudoconcavity}
r\mapsto \frac{h(r)}{r} \text{ is decreasing}.
\end{equation*}
In particular, $h$ grows at most linearly at infinity, since $h(r)\leq h(1) r$ for $r\geq 1$, while it satisfies $h(r)\geq h(1) r$ close to the origin.

In particular, if $h$ is concave and Osgood, then so is $r\mapsto h(r)+\alpha r$, for any $\alpha\geq 0$: indeed for any $\eps\in (0,1)$ we have
\begin{align*}
    \int_{0}^\eps \frac{1}{h(r)+\alpha r} \dd r
    \geq \Big( 1 +\frac{\alpha}{h(1)}\Big)^{-1} \int_{0}^\eps \frac{1}{h(r)} \dd r = +\infty
\end{align*}
which implies the validity of \eqref{eq:defn.osgood} for $\tilde h(r)=h(r)+\alpha r$.
% by relabelling $h$, we may henceforth assume that it satisfies $h(r)\geq h(1) r/2$ for all $r\geq 0$.
\end{remark}
Correspondingly to a Osgood modulus $h$, for any $r_0\in (0,+\infty)$, we set
\begin{equation*}
    G(r) := \int_{r_0}^r \frac{1}{h(\tilde r)} \dd \tilde r.
\end{equation*}
\begin{remark}%\label{rem:modulus-2}
Combining the Osgood condition and the at most linear growth at infinity from Remark \ref{rem:modulus-1}, we see that $G$ is an increasing, invertible function that can be extended to the whole $[0,+\infty]$ by setting $G(0):=-\infty$, $G(+\infty):=+\infty$. Correspondingly, for any $\beta\geq 0$ we can define
\begin{equation}\label{eq:osgood-lemma-function}
    M^{h,\beta}(\alpha):=M^h(\alpha,\beta):=G^{-1}(G(\alpha)+\beta).
\end{equation}
$M^{h,\beta}$ is an increasing, continuous function on $[0+\infty]$, whose values are independent of $r_0$ appearing in the definition of $G$; moreover it satisfies $M^{h,\beta}(0)=0$ for any $\beta\geq 0$.
\end{remark}

The importance of Osgood moduli of continuity, in connection to solvability of ODEs, comes from the Bihari--Osgood inequalities (see \cite{bihari1956} and \cite[Lemma 2.1]{LR2013} for more general statements and the proof).

\begin{lemma}\label{lem:bihari}
Let $h$ be an Osgood modulus of continuity, $T\in (0,+\infty]$. Let $f$, $g$ be nonnegative, integrable functions on $[0,T)$ such that
    \begin{equation*}
    f_t \leq K + \int_0^t g_s\, h(f_s)\, \dd s\quad \forall\, t\in [0,T)
    \end{equation*}
for some constant $K\geq 0$. Then it holds
    \begin{equation*}%\label{eq:bihari}
    f_t \leq M^h\Big(K, \int_0^t g_s\, \dd s\Big)\quad \forall\, t\in [0,T)
    \end{equation*}
for $M^h$ as defined in \eqref{eq:osgood-lemma-function}. In particular, if $K=0$, then $f\equiv 0$.
\end{lemma}

\subsection{Uniform convergence in weak topologies}\label{subsec:UCW}

If $V$ is a Banach space, we write $\rightharpoonup$ for weak convergence in $V$; namely, $v^n\rightharpoonup v$ in $v$ if $\langle v^n,u\rangle_{V,V^\ast}\to \langle v,u\rangle_{V,V^\ast}$ for all $u\in V^\ast$. Similarly, if $V$ has a dual structure, $V=U^\ast$, then we denote weak\mbox{-}$\ast$ convergence by $\xrightharpoonup{\ast}$, viz $v^n \xrightharpoonup{\ast} v$ if $\langle v^n,u\rangle_{U^\ast,U}\to \langle v,u\rangle_{U^\ast,U}$ for all $u\in U$. For an overview of basic properties of weak and weak\mbox{-}$\ast$ topologies, we refer to \cite{brezis2011functional}.

\begin{definition}\label{defn:uniform_weak_convergence}
    Let $V$ be a Banach space, $T\in (0,+\infty)$. We denote by $C_w([0,T];V)$ the space of weakly continuous functions $f:[0,T]\to V$, namely such that $f_{t_n}\rightharpoonup f_t$ in $V$ whenever $t_n\to t$.

    Given a sequence $\{ f^n\}_n\subset C_w([0,T];V)$, we say that $f^n\to f$ in $C_w([0,T];V)$ if
    \begin{align*}
        \lim_{n\to\infty} \sup_{t\in [0,T]} |\langle f^n_t-f_t,g\rangle_{V,V^\ast}|=0\quad \forall\, g\in V^\ast.
    \end{align*}
    
    If $V$ has a dual structure, $V=U^\ast$ for some Banach $U$, we similarly denote by $C_{w-\ast}([0,T];V)$ the space of weakly-$\ast$ continuous functions $f:[0,T]\to V$, namely such that $f_{t_n}\xrightharpoonup{\ast} f_t$ in $V$ whenever $t_n\to t$. Given $\{ f^n\}_n\subset C_{w-\ast}([0,T];V)$, $f^n\to f$ in $C_{w-\ast}([0,T];V)$ if, %for any $g\in U$, it holds $\sup_{t\in [0,T]} |\langle f^n_t-f_t,g\rangle_{U^\ast,U}|\to 0$ as $n\to\infty$.
    \begin{align*}
        \lim_{n\to\infty} \sup_{t\in [0,T]} |\langle f^n_t-f_t,g\rangle_{U^\ast,U}|=0\quad \forall\, g\in U.
    \end{align*}
\end{definition}

\begin{remark}\label{rem:uniform_weak_convergence}
    Definition \ref{defn:uniform_weak_convergence} encodes the concept of ``uniform convergence on compacts in the weak topology'' (resp. ``uniform convergence on compacts in the weak\mbox{-}$\ast$ topology''). It satisfies the following properties:
    \begin{itemize}
        \item[i)] Notation is consistent: if $f^n\to f$ in $C_w([0,T];V)$, then $f\in C_w([0,T];V)$. Indeed, for any $g\in V^\ast$, $t\mapsto \langle f_t,g\rangle_{V,V^\ast}$ is continuous as it is the uniform limit of continuous functions. Similarly for convergence in $C_{w-\ast}([0,T];V)$.
        \item[ii)] By testing against $g$ and using properties of uniform convergence, it's easy to see that if $t_n\to t$ and $f^n\to f$ in $C_w([0,T];V)$, then $f^n_{t_n}\rightharpoonup f_t$; similarly for $C_{w-\ast}([0,T];V)$.
        \item[iii)] If $f^n\to f$ in either $C_w([0,T];V)$ or $C_{w-\ast}([0,T];V)$, then
        \begin{equation*}
            \sup_n \sup_{t\in [0,T]} \| f^n_t\|_V<\infty.
        \end{equation*}
        Indeed, suppose this was not the case, say for $f^n\to f$ in $C_{w-\ast}([0,T];V)$. Then we could find a sequence $(n_k,t_{n_k})$ such that $\| f^{n_k}_{t_{n_k}}\|_V\to\infty$. Since $[0,T]$ is compact, up to extracting a (not relabelled) subsequence, we may assume $t_{n_k}\to t$. But then $f^{n_k}_{t_{n_k}}\xrightharpoonup{\ast} f_t$, which by the Banach--Steinhaus theorem implies that $\{f^{n_k}_{t_{n_k}}\}_k$ is bounded in $V$, absurd.
        \item[iv)] By the same logic, it's easy to check that if $f\in C_w([0,T];V)$ (resp. $C_{w-\ast}([0,T];V)$, then
        \begin{align*}
            \sup_{t\in [0,T]} \| f_t\|_V<\infty.
        \end{align*}
        Moreover, if either $V$ is separable or $V=U^\ast$ with $U$ separable, then the weak (resp. weak\mbox{-}$\ast$ topology) and the strong topology generate the same Borel $\sigma$-algebra, therefore in this case $C_w([0,T];V)\subset \mathcal{B}_b([0,T];V)$ (resp. $C_{w-\ast}([0,T];V)\subset \mathcal{B}_b([0,T];V)$).
        In particular, if $V=U^\ast$ for some separable $U$, then we have the chain of inclusions
        \begin{equation*}
            C([0,T];V)\subset C_w([0,T];V)\subset C_{w-\ast}([0,T];V)\subset \mathcal{B}_b([0,T];E)\subset L^\infty([0,T];E).
        \end{equation*}
        \item[v)] Suppose now that $V=U^\ast$ for some separable $U$; in this case, by the previous point (and lower semicontinuity of $\| \cdot \|_V$ w.r.t. weak\mbox{-}$\ast$ convergence), there exists $R>0$ such that
        \begin{equation*}
            \| f^n_t\|_V \leq R, \quad \| f_t\|_V \leq R \quad \forall\, n\in\N,\, t\in [0,T].
        \end{equation*}
        Set $B_V(R)=\{v\in V: \| v\|_V\leq R\}$; by \cite[Thm. 3.28]{brezis2011functional}, there exists a metric $d$ which induces the weak\mbox{-}$\ast$ topology on $B_V(R)$ and such that $(B_V(R),d)$ is a complete metric space. It is then easy to check that $f^n,f\in C([0,T]; B_V(R),d)$ and that
        \begin{align*}
            \lim_{n\to\infty} \sup_{t\in [0,T]} d(f^n_t,f_t) =0.
        \end{align*}
        In other words, under the above assumptions, convergence in $C_{w-\ast}([0,T];V)$ is equivalent to uniform convergence in $C([0,T];B_V(R),d)$, for some $R\in (0,\infty)$.
    \end{itemize}
\end{remark}

Convergence in $C_w([0,T];V)$ (resp. $C_{w-\ast}([0,T];V)$) will be extremely useful throughout the paper, especially in Sections \ref{sec:linear-RPDEs}-\ref{sec:nonlinear-RPDEs}, whenever we will construct weak solutions by compactness arguments; we refer to Appendix \ref{app:compactness} for a class of compactness criteria for convergence in $C_w([0,T];L^p_x)$ and $C_{w-\ast}([0,T];L^\infty_x)$.

\section{RDEs with Osgood drifts}\label{sec:RDEs}

The overall goal of this section is to present a satisfying solution theory for rough differential equations (RDEs) on $\R^d$ of the form
\begin{equation*}
    \dd y_t = b_t(y_t) \dd t + \xi(y_t) \dd \mathbf{Z}_t,
\end{equation*}
where $b:\R_{\geq 0}\times\R^d\to \R^d$, $\xi:\R^d\to \R^{d\times m}$ and $\mathbf{Z}$ is the rough path enhancement of an $\R^m$-valued path $Z$; in particular, the drift $b$ is assumed to be Osgood continuous in space, but not Lipschitz.
%\LG{Possibly some basic context missing: equation set on $\R^d$, $\xi:\R^d\to \R^{d\times m}$, $\mathbf{Z}$ rough path enhancement of an $\R^m$-valued path $Z$.}
In this case , we will see that it is still possible to establish existence and uniqueness of solutions, which form a flow of homeomorphisms on $\R^d$; furthermore the flow depends continuously on $(b,\xi,\mathbf{Z})$ and, under suitable assumptions on the divergences of $b$ and $\xi$, leaves the Lebesgue measure on $\R^d$ quasi-invariant.

Throughout the section, we will always assume the rough paths $\mathbf{Z}\in \clC^\mathfrak{p}$ in consideration to be defined for $t\in [0,+\infty)$. For simplicity, we will work on finite intervals $[0,T]$, but these may be taken arbitrarily large.

\subsection{Forced RDEs and a priori estimates}%\label{subsec:RDE-preliminary}

Whenever useful, we will identify $\xi:\R^d\to \R^{d\times m}$ with a collection $\{\xi_k\}_{k=1}^m$ of vector fields on $\R^d$, and viceversa.
Associated to $\xi$, we define the {\em second order map} $\Xi(x):= D\xi(x)\, \xi(x)$, such that for any fixed $x\in \R^d$, $\Xi(x)$ is a map from $\R^{m\times m}$ to $\R^d$, as follows:
\begin{align*}
    \Xi(x) A := \sum_{j,k=1}^m D\xi_k (x) \xi_j(x) A_{jk}.
\end{align*}
With this notation, we can provide a solution concept for RDEs, in the style of Davie~\cite{davie2008}.
Such concept is useful for numerical schemes and has a natural generalization to rough PDEs, see~\cite{BaiGub2017,DGHT2019} as well as the upcoming Section~\ref{sec:linear-RPDEs}.

The main novelty, compare to the majority of the rough path literature, is that here we focus on RDEs with a finite $1$-variation forcing term $\mu$ appearing on the r.h.s. This will be very useful later in Section \ref{subsec:RDE-flow}, when replacing $\mu$ by the actual drift term $\int_0^\cdot b_s(y_s) \dd s$.

\begin{definition}\label{defn:RDE-forcing}
    Let $\mathfrak{p}\in [2,3)$, $\mathbf{Z}\in \clC^{\mathfrak{p}-\var}$, $\mu\in C^{1-\textnormal{var}}$ and $\xi\in C^2_b$.
    We say that $y:\R_{\geq 0}\to \R^d$ is a {\em solution to the RDE}
    \begin{equation}\label{eq:RDE-with-forcing}
        \dd y_t = \dd \mu_t + \xi(y_t) \dd \mathbf{Z}_t
    \end{equation}
    on an interval $[0,T]$ if, for any $[s,t]\subset [0,T]$, it holds that
    \begin{equation}\label{eq:defn-RDE-forcing}
        \delta y_{st} = \delta \mu_{st} + \xi(y_s) \delta Z_{st} + \Xi(y_s) \Z_{st} + y^\natural_{st}
    \end{equation}
    where the two-parameter function $y^\natural$ defined by \eqref{eq:defn-RDE-forcing} is such that $y^\natural \in C^{\mathfrak{p}/3-\var}_2$. We will sometimes refer to $(x,\mu)$ as the {\rm data} associated to the RDE~\eqref{eq:RDE-with-forcing}, where $x:=y\vert_{t=0}\in\R^d$ is the {\rm initial datum}.
    We say that $y$ is a {\rm global solution} to the RDE if it is a solution on $[0,T]$ for all $T\in (0,+\infty)$.
\end{definition}

\begin{remark}\label{rem:defn-RDE-forcing}
By Remark~\ref{rem:relation-control-variation}, condition $y^\natural\in C^{\mathfrak{p}/3-\var}_2$ is equivalent to the existence of a control $w_\natural$ such that $|y^\natural_{s,t}| \leq w_\natural (s,t)^{3/\mathfrak{p}}$ for all $(s,t)\in \Delta_T$.
Moreover, by the assumptions on $\mu$ and $\mathbf{Z}$ coming from Definition \ref{defn:RDE-forcing} and properties of controls, it is easy to check that if $y$ is a solution to \eqref{eq:RDE-with-forcing}, then necessarily $y\in C^{\mathfrak{p}-\var}$.
By repeatedly applying Taylor expansion, one can check that in the case of smooth $\mu$ and smooth rough path $\mathbf{Z}$, Definition~\ref{defn:RDE-forcing} is in agreement with $y$ being a classical solution to the ODE
\begin{equation*}
    \dot y_t = \dot\mu_t + \xi(y_t) \dot Z_t.
\end{equation*}
Definition~\ref{defn:RDE-forcing} is also in agreement with other solution concepts from rough path theory, like Gubinelli's controlled rough paths and Lyons' original definition, see Sections 8.8-8.9 from \cite{FH2020} for a deeper discussion.
\end{remark}

The next lemma provides a priori estimates for the RDE \eqref{eq:RDE-with-forcing}.

\begin{lemma}\label{lem:apriori-RDE-forcing}
Let $\mathfrak{p}\in [2,3)$, $\mathbf{Z}\in \clC^{\mathfrak{p}-\var}$, $\xi\in C^2_b$ and $y$ be a solution to the RDE~\eqref{eq:RDE-with-forcing} on $[0,T]$ associated to $(x,\mu)$. Then there exists a constant $C=C(\mathfrak{p},\|\xi\|_{C^2_b}, \| \mathbf{Z}\|_{\mathfrak{p},T})$, increasing in all its entries, such that 
\begin{align} 
    \sup_{t\in [0,T]} |y_t-x| & \leq C(1+ \| \mu\|_{1,T}), \label{eq:apriori-RDE-forcing-1} \\
    |\delta y_{st}| \leq \llbracket y\rrbracket_{\mathfrak{p},[s,t]} & \leq C\, \big( w_{\mu}(s,t) + w_{\mathbf{Z}} (s,t)^{1/\mathfrak{p}}\big) \quad \forall\, (s,t) \in \Delta_T, \label{eq:apriori-RDE-forcing-2} \\ 
    |y^{\natural}_{st}| \leq \llbracket y^{\natural}\rrbracket_{\mathfrak{p}/3,[s,t]} & \leq C \big( w_\mu (s,t)\, w_{\mathbf{Z}}(s,t)^{1/\mathfrak{p}}  + w_{\mathbf{Z}}(s,t)^{3/\mathfrak{p}} \big) \quad  \forall \, (s,t) \in \Delta_T  \label{eq:apriori-RDE-forcing-3} .
\end{align}
The controls $w_\mu$ and $w_\mathbf{Z}$ %and $w_\natural$ \TN{I removed this since it does not appear}.
appearing in the above estimates come respectively from Remark~\ref{rem:relation-control-variation} (applied for $g=\mu$ and $\mathfrak{p}=1$) and Remark~\ref{rem:control-rough-path}. %and Remark~\ref{rem:defn-RDE-forcing}. \LG{removed based on the above removal}
\end{lemma}

\begin{proof}
    By Remark~\ref{rem:defn-RDE-forcing}, $y\in C^{\mathfrak{p}-\var}$; denote by $w_y$ the associated control, coming from Remark~\ref{rem:relation-control-variation}.
    By the structure of \eqref{eq:defn-RDE-forcing}, it holds
    \begin{equation}\label{eq:apriori-RDE-forcing-eq1}\begin{split}
    w_y(s,t)^{1/\mathfrak{p}} & \lesssim w_{\mu}(s,t) + w_{\bZ}(s,t)^{1/\mathfrak{p}}  + w_{\natural}(s,t)^{3/\mathfrak{p}},\\
    |y^\natural_{st}| & \lesssim w_{\mu}(s,t) + w_{\bZ}(s,t)^{1/\mathfrak{p}} + w_y(s,t)^{1/\mathfrak{p}},
    \end{split}\end{equation}
    where all hidden constants are allowed to depend on the aforementioned parameters, so that e.g. we may omit $\| \xi\|_{C^2_b}$ and we may write $w_{\bZ}(s,t)^{2/\mathfrak{p}} \lesssim w_{\bZ}(s,t)^{1/\mathfrak{p}}$.
    
    The proof is split in two main steps: we first show that \eqref{eq:apriori-RDE-forcing-3} holds for all $(s,t)$ such that $w_{\bZ}(s,t) \leq h$, for some appropriately chosen $h>0$, aided by the sewing lemma (Lemma \ref{lem:sewing}); we then upgrade this to global bounds, by applying several times Lemma \ref{lem:local-to-global} from Appendix \ref{app:tech-lem}.
    
    In order to apply Lemma \ref{lem:sewing} to $y^\natural$, we start by computing $\delta y^\natural_{sut}$;
    by \eqref{eq:defn-RDE-forcing} and basic algebraic computations, it holds
    \begin{equation}\label{eq:apriori-RDE-forcing-eq3}\begin{split}
        \delta y^\natural_{sut}
        & = \delta \xi(y)_{us}\, \delta Z_{ut} + \Xi(y_s) (\mathbb{Z}_{su} + \mathbb{Z}_{ut} -\mathbb{Z}_{st}) + \delta\Xi(y)_{su} \mathbb{Z}_{ut}\\
        & = \big(\delta \xi(y)_{su} - \Xi(y_s) \delta Z_{su} \big) \delta Z_{ut} + \delta\Xi(y)_{su} \mathbb{Z}_{ut},
    \end{split}\end{equation}
    where in the second step we applied Chen's relation \eqref{eq:chen-relation}.
    Next, let us introduce the notation
    \begin{equation}\label{eq:apriori-RDE-forcing-eq2}
        y^\sharp_{st} := \delta y_{st} - \xi(y_s) \delta Z_{st}= \delta \mu_{st} + \Xi(y_s) \mathbb{Z}_{st} + y^\natural_{st},
    \end{equation}
    where the second identity comes from an application of \eqref{eq:defn-RDE-forcing}, as well as
    \begin{equation*}
        \quad \xi(y)^\sharp_{st}
        := \delta \xi(y)_{su} - \Xi(y_s) \delta Z_{su}
        = \delta \xi(y)_{st} - D \xi(y_s) \xi(y_s) \delta Z_{st}.
    \end{equation*}
    By applying (componentwise) Lemma~\ref{lem:basic-taylor} from Appendix~\ref{app:tech-lem} for $f=\xi\in C^2_b$ and $z=\delta Z_{st}$, it holds
    \begin{equation}\label{eq:apriori-RDE-forcing-eq4}\begin{split}
        |\xi(y)^\sharp_{st}|
        & \lesssim |y^\sharp_{st}|+|\delta y_{st}| |\delta Z_{st}|\\
        & \lesssim w_\mu(s,t) + w_{\bZ}(s,t)^{2/\mathfrak{p}} + w_\natural(s,t)^{3/\mathfrak{p}} + w_y(s,t)^{1/\mathfrak{p}} w_{\bZ}(s,t)^{1/\mathfrak{p}}\\
        & \lesssim w_\mu(s,t) + w_{\bZ}(s,t)^{2/\mathfrak{p}} + w_\natural(s,t)^{3/\mathfrak{p}}
    \end{split}\end{equation}
    where in the last step we applied \eqref{eq:apriori-RDE-forcing-eq1}.
    Observing that by our assumptions $\Xi$ is globally Lipschitz, combining \eqref{eq:apriori-RDE-forcing-eq3} and \eqref{eq:apriori-RDE-forcing-eq4} yields
    \begin{align*}
        |\delta y^\natural_{sut}|
        & \lesssim |\xi(y)^\sharp_{su}| |\delta Z_{ut}| + |\delta y_{su}| |\mathbb{Z}_{ut}|\\
        & \lesssim w_{\bZ}(s,t)^{1/\mathfrak{p}} w_\mu(s,t) + + w_{\bZ}(s,t)^{3/\mathfrak{p}} + w_{\bZ}(s,t)^{1/\mathfrak{p}} w_\natural(s,t)^{3/\mathfrak{p}} + w_{\bZ}(s,t)^{2/\mathfrak{p}} w_y(s,t)^{1/\mathfrak{p}}\\
        & \lesssim w_{\bZ}(s,t)^{1/\mathfrak{p}} w_\mu(s,t) + w_{\bZ}(s,t)^{1/\mathfrak{p}} w_\natural(s,t)^{3/\mathfrak{p}}  + w_{\bZ}(s,t)^{3/\mathfrak{p}}
    \end{align*}
    where in the last passage we used \eqref{eq:apriori-RDE-forcing-eq1}.
    Applying Lemma \ref{lem:sewing} (for $\eta=3/\mathfrak{p}$) we conclude that there exists a constant $K=K(\mathfrak{p},\| \xi\|_{C^2_b}, \| \bZ\|_{\mathfrak{p},T})$ such that, for all $(s,t)\in \Delta_T$, it holds
    \begin{equation}\label{eq:apriori-RDE-forcing-eq5}
        w_\natural(s,t)^{3/\mathfrak{p}} \leq K \big[w_{\bZ}(s,t)^{1/\mathfrak{p}} w_\natural(s,t)^{3/\mathfrak{p}} + w_{\bZ}(s,t)^{1/\mathfrak{p}} w_\mu(s,t) + w_{\bZ}(s,t)^{3/\mathfrak{p}}\big];
    \end{equation}
    in turn, \eqref{eq:apriori-RDE-forcing-eq5} implies that for all $s<t$ such that $w_{\bZ}(s,t)< h:=(2K)^{-\mathfrak{p}}$, it holds
    \begin{equation}\label{eq:apriori-RDE-forcing-eq6}
        |y^\natural_{st}| \leq w_\natural(s,t)^{3/\mathfrak{p}} \leq 2K \big[w_{\bZ}(s,t)^{1/\mathfrak{p}} w_\mu(s,t) + w_{\bZ}(s,t)^{3/\mathfrak{p}}\big].
    \end{equation}

    Having established the desired local bound \eqref{eq:apriori-RDE-forcing-eq6}, we now pass to the global one.
    Reinserting \eqref{eq:apriori-RDE-forcing-eq6} in the first equation in \eqref{eq:apriori-RDE-forcing-eq1}, one finds
    \begin{equation*}
        |\delta y_{st}|
        \lesssim w_\mu(s,t) + w_{\bZ}(s,t)^{1/\mathfrak{p}}
        \lesssim \big( w_\mu(s,t)^\mathfrak{p} + w_{\bZ}(s,t) \big)^{1/\mathfrak{p}}
    \end{equation*}
    whenever $w_{\bZ}(s,t)\leq h$. Observe that by Remark \ref{rem:properties-controls} $\tilde w:= w_\mu^\mathfrak{p} + w_{\bZ}$ is still a control; therefore the hypothesis of Lemma \ref{lem:local-to-global} from Appendix \ref{app:tech-lem} are satisfied for $g=\delta y$, whose application readily yields the desired bound \eqref{eq:apriori-RDE-forcing-2} (modulo reabsorbing several terms coming from $h=(2K)^{-\fkp}$, $K$, etc. in the hidden constant $C$ depending on the aforementioned parameters).
    Estimate \eqref{eq:apriori-RDE-forcing-1} follows from \eqref{eq:apriori-RDE-forcing-2} and the basic bound $|y_t-x|\leq \llbracket y\rrbracket_{\mathfrak{p},[0,t]}$, coming from the definition of $\fkp$-variation.

    Finally, estimate \eqref{eq:apriori-RDE-forcing-eq6} implies that
    \begin{align*}
        |y^\natural_{st}|\lesssim \big[ w_{\bZ}(s,t)^{1/3} w_\mu(s,t)^{\mathfrak{p}/3} + w_{\bZ}(s,t) \big]^{3/\mathfrak{p}} \quad \text{whenever } w_{\bZ}(s,t)\leq h,
    \end{align*}
    where $\tilde w:= w_{\bZ}^{1/3} w_\mu^{\mathfrak{p}/3} + w_{\bZ}$ is again a control by Remark \ref{rem:properties-controls}; applying again Lemma \ref{lem:local-to-global}, this time for $g=y^\natural$ and $\tilde{\mathfrak{p}}=\mathfrak{p}/3$, yields
    \begin{align*}
        \llbracket y^{\natural} \rrbracket_{\mathfrak{p}/3;[s,t]}
        & \lesssim w_{\bZ}(s,t)^{1/\mathfrak{p}}  w_\mu(s,t) + w_{\bZ}(s,t)^{3/\mathfrak{p}} + w_{\bZ}(s,t)^{3/\mathfrak{p}} \sup_{s\leq u\leq v\leq t} |y^{\natural}_{uv}|\\
        & \lesssim  w_{\bZ}(s,t)^{1/\mathfrak{p}}  w_\mu(s,t) + w_{\bZ}(s,t)^{3/\mathfrak{p}} + w_{\bZ}(s,t)^{3/\mathfrak{p}} w_{y}(s,t)^{1/\mathfrak{p}}
    \end{align*}
    where in the last passage we applied the second estimate in \eqref{eq:apriori-RDE-forcing-eq1}. Inserting the bound \eqref{eq:apriori-RDE-forcing-2}  in the above estimate readily yields \eqref{eq:apriori-RDE-forcing-3}.
\end{proof}

While Lemma~\ref{lem:apriori-RDE-forcing} is typically useful to achieve existence of solutions by compactness arguments, the next result can be used to guarantee their uniqueness and continuous dependence on the data of the problem.

\begin{lemma}\label{lem:contraction-RDE-forcing}
Let $\mathfrak{p}\in [2,3)$, $\mathbf{Z}\in \clC^{\mathfrak{p}-\var}$, $\xi\in C^3_b$.
Let $y^i$, $i=1,\,2$, be global solutions to the RDE~\eqref{eq:RDE-with-forcing} with data $(x^i, \mu^i)\in \R^d\times C^{1-\var}$.
Then for any $T>0$ there exists a constant $C=C(\mathfrak{p},T,\|\xi\|_{C^3_b}, \| \mathbf{Z}\|_{\mathfrak{p},T}, \| \mu^i\|_{1,T})$, increasing in all its entries, such that
\begin{equation}\label{eq:contraction-RDE-forcing}
    \sup_{t\in [0,T]} |y^1_t-y^2_t| \leq C\big(\,|x^1-x^2| +  \| \mu^1-\mu^2\|_{1,T}  \big).
\end{equation} 
\end{lemma}

\begin{proof}
    Let $y^i$ be any such solutions and define
    \begin{equation*}
        v := y^1 - y^2, \quad v^{\sharp} := y^{1,\sharp} - y^{2,\sharp}, \quad v^{\natural} := y^{1,\natural} - y^{2,\natural}, \quad \mu = \mu^1 - \mu^2
    \end{equation*}
    where the terms $y^{i,\sharp}$ are defined as in \eqref{eq:apriori-RDE-forcing-eq2}. Set $w_{*} = w_{\mu^1} + w_{\mu^2} + w_{\bZ}$, so that by the a priori estimates from Lemma \ref{lem:apriori-RDE-forcing} it holds
    \begin{equation}\label{eq:contraction-proof-eq1}
        |\delta y^i_{st}|^\mathfrak{p} + |y_{st}^{i,\sharp}|^{\mathfrak{p}/2}+ |y_{st}^{i,\natural}|^{\mathfrak{p}/3} \lesssim w_{*}(s,t) \quad \forall\, i=1,2,\, (s,t)\in \Delta_T.
    \end{equation}
    Additionally, by the definition of the RDE \eqref{eq:defn-RDE-forcing} and the Lipschitz continuity of $\xi$, $\Xi$, we see that
    \begin{equation}\label{eq:contraction-proof-eq2}\begin{split}
        |\delta v_{st}| & \lesssim w_\mu(s,t) + |v_s| w_{\bZ}(s,t)^{1/\mathfrak{p}} + w_{v^{\natural}}(s,t)^{3/\mathfrak{p}},\\
        |v_{st}^{\sharp}| & \lesssim  w_\mu(s,t) + |v_s| w_{\bZ}(s,t)^{2/\mathfrak{p}} + w_{v^{\natural}}(s,t)^{3/\mathfrak{p}}.
    \end{split}\end{equation}

    As before, we start by looking for local estimates; , we will apply several elementary bounds based on Taylor expansions which are collected in Appendix \ref{app:tech-lem} (cf. Lemmas \ref{lem:basic-taylor}-\ref{lem:basic-taylor-2}), as well as Lemma \ref{lem:sewing}.
    By the same computation as in \eqref{eq:apriori-RDE-forcing-eq3} and the definition of $\xi(y^i)^\sharp$, it holds that
    \begin{align*}
        \delta v^\natural_{sut} = \big[ \xi(y^1)^\sharp_{su} - \xi(y^2)^\sharp_{su} \big] \delta Z_{ut} + \big[ \delta\Xi(y^1)_{su} - \delta\Xi(y^2)_{su} \big] \mathbb{Z}_{ut}.
        %=: I^1_{sut} + I^2_{sut}.
    \end{align*}
    Next, we apply Lemma \ref{lem:basic-taylor-2} for the choice $f=\xi\in C^3_b$, $z=\delta Z_{su}$ to find
    \begin{align*}
        |\xi(y^1)^\sharp_{su} - \xi(y^2)^\sharp_{su}|
        & \lesssim |v^\sharp_{su}| + |\delta v_{su}| \big(|y^{1,\sharp}_{su}| 
        + |\delta Z_{su}| \big) + |v_s| \big(|y^{1,\sharp}_{su}| + |\delta y^1_{su}| |\delta Z_{su}|\big)\\
        & \lesssim |v^\sharp_{su}| + |\delta v_{su}| w_\ast(s,t)^{1/\mathfrak{p}} + |v_s| w_\ast(s,t)^{2/\mathfrak{p}}\\
        & \lesssim w_\mu(s,t) + |v_s| w_\ast(s,t)^{2/\mathfrak{p}} + w_{z^{\natural}}(s,t)^{3/\mathfrak{p}}
    \end{align*}
    where we used several times the estimates \eqref{eq:contraction-proof-eq1}-\eqref{eq:contraction-proof-eq2}.
    On the other hand, by assumption $\Xi\in C^2_b$ and so by Lemma \ref{lem:four-point-estim} it holds
    \begin{align*}
        |\delta\Xi(y^1)_{su} - \delta\Xi(y^2)_{su}|
        \lesssim |\delta v_{su}| + |v_s| |\delta y^1_{su}|
        \lesssim w_\mu(s,t) + |v_s| w_\ast(s,t)^{1/\mathfrak{p}} + w_{v^\natural}(s,t)^{3/\mathfrak{p}}, 
    \end{align*}
    where in the second passage we applied again \eqref{eq:contraction-proof-eq2}. Combining everything, we arrive at
    \begin{equation}\label{eq:contraction-proof-eq3}
        |\delta v^\natural_{sut}|
        \lesssim w_*(s,t)^{1/\mathfrak{p}}\, w_\mu(s,t) + \|v\|_{\cB_b([s,t])}\, w_{*}(s,t)^{3/\mathfrak{p}} +  w_\ast(s,t)^{1/\mathfrak{p}}\, w_{z^{\natural}}(s,t)^{3/\mathfrak{p}},
    \end{equation}
    where we employ the short-hand notation $\|v\|_{\cB_b([s,t])}:=\sup_{r\in [s,t]} |v_r|$.
    Noticing that the mapping $(s,t) \mapsto \|v\|_{\cB_b([s,t])} w_*(s,t)^{3/\mathfrak{p}}$ is a control by Remark \ref{rem:properties-controls}, by \eqref{eq:contraction-proof-eq3} and Lemma \ref{lem:sewing} there exists $K>0$ such that
    \begin{align*}
        w_{v^{\natural}}(s,t)^{3/\mathfrak{p}} \leq K \big( w_*(s,t)^{1/\mathfrak{p}} w_\mu(s,t) +  \|v\|_{\cB_b([s,t])}\,  w_{*}(s,t)^{3/\mathfrak{p}} +  w_\ast(s,t)^{1/\mathfrak{p}} w_{v^{\natural}}(s,t)^{3/\mathfrak{p}} \big).
    \end{align*}
    Choosing now any $(s,t)\in \Delta_T$ such that $w_\ast(s,t) \leq (2K)^{-\mathfrak{p}}$, we find the desired local estimate
    \begin{equation}\label{eq:contraction-proof-eq4}
        | v_{st}^{\natural}|
        \leq w_{v^{\natural}}(s,t)^{3/\mathfrak{p}}
        \leq 2K  \big( \|v\|_{\cB_b([s,t])}\,  w_{*}(s,t)^{3/\mathfrak{p}}  + w_*(s,t)^{1/\mathfrak{p}} w_\mu(s,t) \big).
    \end{equation}

    It remains to upgrade this estimate to a global one, which we achieve by applying the so called rough Gr\"onwall lemma (Lemma \ref{lem:rough-gronwall} in Appendix \ref{app:tech-lem}). We can reinsert \eqref{eq:contraction-proof-eq4} in the first equation of \eqref{eq:contraction-proof-eq2} to deduce that
    \begin{align*}
        |\delta v_{st}| \lesssim w_*(s,t)^{1/\mathfrak{p}} \|v\|_{\cB_b([s,t])} + w_\mu(s,t) \quad \text{whenever } w_\ast(s,t) \leq (2K)^{-\mathfrak{p}}.
    \end{align*}
    We can then apply Lemma \ref{lem:rough-gronwall} for the choice $w=w_\ast$, $\gamma=w_\mu$ and $G_t=\|v\|_{\cB_b([0,t])}$ to deduce that there exists a constant $\tilde K>0$ such that
    \begin{equation*}
        \sup_{t\in [0,T]} |v_t| \leq \tilde K e^{\tilde K w_*(0,T)}\, (|z_0| + w_\mu(0,T))
    \end{equation*}
    which upon relabelling of constants implies the conclusion \eqref{eq:contraction-RDE-forcing}. 
\end{proof}

We conclude this section with a short digression on \textit{time reversed rough paths} and \textit{time reversed RDEs}, which will be relevant in order to construct inverse flows.
Let us mention that while time reversal techniques are quite common in the case of geometric rough paths, see e.g.~\cite[Theorem 3.3.3]{LyQi2002} or~\cite[Proposition 11.11]{FV2010}, the literature seems much more sparse in the general case, with some key definitions and properties given in~\cite[Exercise 2.6 \& Prop. 5.12]{FH2020} in the case of H\"older rough paths.
However a general statement on time reversal of RDEs, for non geometric $\mathfrak{p}$-variation rough paths, seems to be missing; we fill here this gap in the literature.

\begin{definition}\label{defn:time-reversed-path}
    Let $\mathbf{Z} \in \clC^\mathfrak{p}([0,T];\R^m)$ for $\mathfrak{p}\in [2,3)$. The {\rm time-reversed rough path} (at time $T$) $\overleftarrow{\mathbf{Z}}^T=\overleftarrow{\mathbf{Z}}=(\overleftarrow{Z},\overleftarrow{\Z})$ is defined by\footnote{Given the presence of $T-t$, whenever dealing with two-parameter maps $g$, if needed we will use the notation $g_{s,t}$ in place of $g_{st}$, to avoid any confusion.}
    \begin{equation*}
        \overleftarrow{Z}_{\! t}:= Z_{T-t}, \quad \overleftarrow{\Z}_{\! st}
        := -\Z_{T-t,T-s} + \delta Z_{T-t,T-s}\otimes \delta Z_{T-t,T-s}
        = -\Z_{T-t,T-s} + \delta\overleftarrow{Z}_{\! st} \otimes \delta\overleftarrow{Z}_{\! st}.
    \end{equation*}
\end{definition}

\begin{remark}\label{rem:time-reversal}
    Definition~\ref{defn:time-reversed-path} is in agreement with \cite[Exercise 2.6]{FH2020}; as therein, one can check that $\overleftarrow{\mathbf{Z}}$ is a rough path, in the sense that it satisfies Chen's relation~\eqref{eq:chen-relation}, and that $\overleftarrow{\mathbf{Z}}$ is geometric if and only if $\mathbf{Z}$ is so. Moreover the map $\mathbf{Z}\mapsto \overleftarrow{\mathbf{Z}}$ is continuous in the rough path metric, as one can check that for any pair of rough paths $\mathbf{Z}=(Z,\Z)$, $\mathbf{W}=(W,\mathbb{W})\in \clC^\mathfrak{p}$ it holds
    \begin{align*}
        \| \overleftarrow{\mathbf{Z}}\|_{\mathfrak{p},T}
        & \lesssim \| \mathbf{Z}\|_{\mathfrak{p},T}, \\% \label{eq:time-reversed-1}\\
        d_{\mathfrak{p},T}(\overleftarrow{\mathbf{Z}},\overleftarrow{\mathbf{W}})
        & \lesssim d_{\mathfrak{p},T}(\mathbf{Z},\mathbf{W}) + \| Z-W\|_{\mathfrak{p},T}^{1/2}( \| Z\|_{\mathfrak{p},T}^{1/2} + \| W\|_{\mathfrak{p},T}^{1/2}). %\label{eq:time-reversed-2}
    \end{align*}
\end{remark}

Similarly to Definition~\ref{defn:time-reversed-path}, whenever dealing with a path $\mu$, we denote by $\overleftarrow{\mu}^T=\overleftarrow{\mu}$ the path $\overleftarrow{\mu}_{\! t}:= \mu_{T-t}$; similarly for $y$.

\begin{lemma}\label{lem:time-reversal-RDE}
    Let $\mathfrak{p}\in [2,3)$, $\mathbf{Z}\in \clC^{\mathfrak{p}-\var}$, $\mu\in C^{1-\var}$ and $\xi\in C^2_b$ and let $y$ be a solution to \eqref{eq:RDE-with-forcing} on $[0,T]$. Then $\overleftarrow{y}$ is a solution on $[0,T]$ to the same RDE, with $(\mu,\mathbf{Z})$ replaced by $(\overleftarrow{\mu},\overleftarrow{\mathbf{Z}})$ and with new initial condition $\tilde x=y_T$.
\end{lemma}

\begin{proof}
    Throughout the proof, we will adopt the following convention: given any two-parameters maps, $f_{st} \approx g_{st}$ means that their difference belongs to $C^{\mathfrak{p}/3}_2$, so that it can be regarded as negligible remainder in the setting of Definition~\ref{defn:RDE-forcing}.
    By the definition of $\overleftarrow{y}$ and~\eqref{eq:defn-RDE-forcing}, it holds
    \begin{align*}
        \delta \overleftarrow{y}_{\! st}
        & = \delta y_{T-s,T-t} = - \delta y_{T-t,T-s}\\
        & \approx -\delta \mu_{T-t,T-s} - \xi(y_{T-t}) \delta Z_{T-t,T-s} - \Xi(y_{T-t})\Z_{T-t,T-s}\\
        & = \delta \overleftarrow{\mu}_{\! st} + \xi(\overleftarrow{y}_{\! s}) \delta \overleftarrow{Z}_{\! st} + (\xi(y_{T-s})-\xi(y_{T-t})) \delta Z_{T-t,T-s} - \Xi(y_{T-t})\Z_{T-t,T-s}.
    \end{align*}
    Applying Taylor expansion to $\xi(y_{T-s})-\xi(y_{T-t})$ and reinserting~\eqref{eq:defn-RDE-forcing} in it, we find
    \begin{align*}
        (\xi(y_{T-s})-\xi(y_{T-t})) \delta Z_{T-t,T-s}
        & \approx D\xi (y_{T-t}) (\delta y_{T-t,T-s}) \delta Z_{T-t,T-s}\\
        & \approx \Xi(y_{T-t}) \delta Z_{T-t,T-s}\otimes \delta Z_{T-t,T-s}
        = \Xi(y_{T-t}) \delta \overleftarrow{Z}_{\! st}\otimes \delta \overleftarrow{Z}_{\! st}
    \end{align*}
    so that
    \begin{align*}
        \delta \overleftarrow{y}_{\! st}
        & \approx \delta \overleftarrow{\mu}_{\! st} + \xi(\overleftarrow{y}_{\! s}) \delta \overleftarrow{Z}_{\! st} +  \Xi(y_{T-t})(\delta \overleftarrow{Z}_{\! st}\otimes \delta \overleftarrow{Z}_{\! st} - \Z_{T-t,T-s})\\
        & = \delta \overleftarrow{\mu}_{\! st} + \xi(\overleftarrow{y}_{\! s}) \delta \overleftarrow{Z}_{\! st} +  \Xi(y_{T-t})\overleftarrow{\Z}_{\! st}.
    \end{align*}
    By our assumptions, $\Xi$ is Lipschitz, $y\in C^{\mathfrak{p}-\var}$ and $\overleftarrow{\Z}\in C^{\mathfrak{p}/2-\var}$, therefore
    \begin{align*}
        \Xi(y_{T-t})\overleftarrow{\Z}_{\! st} \approx \Xi(y_{T-s})\overleftarrow{\Z}_{\! st} = \Xi(\overleftarrow y_{\! s})\overleftarrow{\Z}_{\! st}
    \end{align*}
    which finally yields the conclusion.
\end{proof}

\subsection{Construction of flows and their properties}\label{subsec:RDE-flow}

We are now ready to pass to the study of actual $\R^d$-valued RDEs with time-dependent drift, of the form
\begin{equation}\label{eq:RDE-drift}
    \dd y_t = b_t(y_t) \dd t + \xi(y_t) \dd \mathbf{Z}_t.
\end{equation}
The concept of solution to~\eqref{eq:RDE-drift} coincides with Definition~\ref{defn:RDE-forcing}, for the choice $\mu_t=\int_0^t b_s(y_s) \dd s$, under the condition that $\int_0^T |b_s(y_s)| \dd s<\infty$ for all $T<\infty$, which makes it a well-defined Lebesgue integral and a bounded variation path.
We will enforce the following condition on the drift.

\begin{assumption}\label{ass:osgood-drift-RDE}
There exist a locally integrable $g:\R_{\geq 0}\to\R_{\geq 0}$ and an Osgood modulus of continuity $h$ such that the measurable map $b:\R_{\geq 0}\times \R^d\to\R^d$ satisfies
\begin{equation}\label{eq:ass-osgood-drift-RDE}
    | b_t(x)| \leq g_t, \quad |b_t(x)-b_t(\tilde x)| \leq g_t\, h(|x-\tilde x|) \quad \forall\, (t,x,\tilde x)\in \R_{\geq 0}\times \R^{2d}.
\end{equation}
\end{assumption}

We are now ready to state and prove a more rigorous version of Theorem \ref{thm:Osgood intro} from the introduction, providing an explicit formula for the function $F$ therein.

\begin{theorem}\label{thm:wellposedness-RDE}  %[Well-posedness of RDEs]
Let $\mathfrak{p}\in [2,3)$, $\mathbf{Z}\in \clC^{\mathfrak{p}-\var}$, $\xi\in C^3_b$ and $b$ satisfying Assumption~\ref{ass:osgood-drift-RDE}.
Then for any initial condition $x\in \R^d$, there exists a unique global solution to the RDE~\eqref{eq:RDE-drift}.
Moreover the solution map $(x,t)\mapsto \Phi_t(x)$ defines a family of homeomorphisms from $\R^d$ to itself and for any $T>0$ it holds
\begin{equation}\label{eq:modulus-continuity-flow}
    \sup_{t\in [0,T]} |\Phi_t(x)-\Phi_t(\tilde x)|
    \leq M^h\Big(C |x-\tilde x|, C \int_0^T g_s \dd s\Big)
\end{equation}
where $M^h$ is defined as in~\eqref{eq:osgood-lemma-function}, for $h$ as in Assumption~\ref{ass:osgood-drift-RDE}, and the constant $C$ is the same appearing in Lemma~\ref{lem:contraction-RDE-forcing}, with $\| \mu^i\|_{1,T}$ replaced by $\int_0^T g_s \dd s$.
Moreover, estimate~\eqref{eq:modulus-continuity-flow} holds with $\Phi_t(y)$ replaced by its inverse $\Phi_t^{-1}(y)$.
\end{theorem}

\begin{proof}
    Let $\rho\in C^\infty_c(B_1)$ be a probability density,  $\{\rho^\eps\}_{\eps>0}$ the associated standard mollifiers and set $b^n=\rho^{1/n}\ast b$. It is easy to check that $b^n$ satisfies~\eqref{eq:ass-osgood-drift-RDE} for the same $g$ and $h$ and moreover, for any $t\in\R_{\geq 0}$ such that $g_t<\infty$, it holds
    \begin{align*}
       \sup_{x\in\R^d} |b_t(x)-b^n_t(x)|
       \leq \sup_{x\in\R^d} \int_{B_1} |b_t(x+ y/n)-b(x)|\,\rho(y)\,\dd y
       \leq g_t\, h(1/n) \to 0\text{ as } n\to\infty.
    \end{align*}
    Upon mollifying in time, we can further assume $b^n$ to be smooth in $t$, while keeping a uniform bound in $n$ of the form~\eqref{eq:ass-osgood-drift-RDE}, up to introducing an additional sequence $g^n$ with the properties that
    \begin{equation}\label{eq:flow-proof-eq2}
        \sup_n \int_0^T g^n_s \dd s \leq \int_0^T g_s \dd s,\quad \lim_{n\to\infty} \int_0^T |g^n_s-g_s| \dd s = 0 \quad \forall \, T>0;
    \end{equation}
    furthermore, these approximations can be constructed so that, similarly to the above, we have
    \begin{equation}\label{eq:flow-proof-eq1}
        \lim_{n\to\infty} \int_0^T \|b^n_t-b_t\|_{C^0_b}\, \dd t = 0 \quad \forall \, T\geq 0.
    \end{equation}
    For each smooth $b^n$, since $\xi\in C^3_b$, standard rough path results guarantee the existence of a flow of diffeomorphisms $\Phi^n$, cf. \cite[Proposition 11.11]{FV2010}.
    
    We claim that, for any $T\geq 0$, the sequence $\{\Phi^n\}_n$ is Cauchy w.r.t. uniform convergence in $[0,T]\times \R^d$.
    Indeed, for any $n$, $m\in\N$ and $x\in \R^d$, we can apply Lemma \ref{lem:contraction-RDE-forcing} to $\mu^1_\cdot=\int_0^\cdot b^n_s(\Phi^n_s) \dd s$, $\mu^2_\cdot=\int_0^\cdot b^m_s(\Phi^m_s) \dd s$ to obtain
    \begin{align*}
        I_t(x):= \sup_{r\in [0,t]} |\Phi^n_r(x)-\Phi^m_r(x)|
        \lesssim \bigg\| \int_0^\cdot [b^n_r(\Phi^n_r(x)) -b^m_r(\Phi^m_r(x))] \,\dd r\bigg\|_{1,t}
    \end{align*}
    where the hidden constant does not depend on the given $x$ nor $n$, since by construction it holds
    \begin{align*}
        \max_i \| \mu^i\|_{1,T} 
        & \leq \int_0^T \big(|b^n_t(\Phi^n_t(x))| + |b^m_t(\Phi^m_t(x))|\big) \dd t\\
        & \leq \int_0^T (\| b^n_t\|_{C^0_b} + \| b^m_t\|_{C^0_b}) \dd t 
        \leq 2 \int_0^T g_t\, \dd t.
    \end{align*}
    Thus we obtain
    \begin{align*}
        I_t(x)
        & \lesssim \int_0^t |b^n_r(\Phi^n_r(x)) -b^m_r(\Phi^m_r(x))|\, \dd r\\
        & \lesssim \int_0^t |b^n_r(\Phi^n_r(x)) -b^n_r(\Phi^m_r(x))|\, \dd r + \int_0^T  \| b^n_r - b^m_r\|_{C^0_b} \dd r\\
        & \lesssim \int_0^t g^n_r \, h\big( |\Phi^n_r(x) - \Phi^m_r(x)|\big)\, \dd r + \int_0^t  \| b^n_r - b^m_r\|_{C^0_b} \dd r\\
        & \lesssim \int_0^t g^n_r \, h( I_r(x) )\, \dd r + \int_0^t  \| b^n_r - b^m_r\|_{C^0_b}\, \dd r
    \end{align*}
    where again the hidden constant does not depend on $x$.
    Applying Lemma \ref{lem:bihari} and property \eqref{eq:flow-proof-eq2}, we can find $C>0$ such that
    \begin{equation*}
        I_T(x)
        \leq M^h\Big(C \int_0^T  \| b^n_r - b^m_r\|_{C^0_b}\, \dd r, C \int_0^T g^n_r \dd r\Big)
        \leq M^h \Big(C \int_0^T  \| b^n_r - b^m_r\|_{C^0_b}\, \dd r, C \int_0^T g_r \dd r \Big).
    \end{equation*}
    As the estimate is uniform in $x$ and $M^h$ is continuous and monotone, we get
    \begin{equation*}
       \sup_{n,m\geq N} \sup_{(t,x)\in [0,T]\times \R^d} |\Phi^n_t(x)-\Phi^n_t(x)| 
       \leq M^h\Big(C \sup_{n,m\geq N}  \int_0^T  \| b^n_r - b^m_r\|_{C^0_b}\, \dd r, C \int_0^T g_r \dd r \Big)
    \end{equation*}
    which by virtue of \eqref{eq:flow-proof-eq1} and $M^h(0, \int_0^T g_r \dd r)=0$ shows the claim.

    We deduce that there exist a continuous map $\Phi:\R_{\geq 0}\times \R^d\to \R^d$ such that $\Phi^n\to \Phi$ uniformly in $[0,T]\times\R^d$, for all finite $T$.
    We claim that for each fixed $x$, $y_t:= \Phi_t(x)$ is a solution to the RDE \eqref{eq:RDE-drift} starting at $x$. To see this, define correspondingly $y^n_t:= \Phi^n_t(x)$; by the construction of our approximation sequence and the regularity of $\xi$, $\Xi$,
    %and $\| y^n-y\|_{\infty;[0,T]}$, \LG{I don't know that the regularity of this guy was supposed to mean, I think it was a typo}
    for any $s<t$ it holds
    \begin{align*}
        \lim_{n\to\infty} y^{n,\natural}_{st}
        & = \lim_{n\to\infty} \delta y^n_{st} - \int_s^t b^n_r(y^n_r) \dd r - \xi(y^n_s) \delta Z_{st} - \Xi(y^n_s) \mathbb{Z}_{st}\\
        & = \delta y_{st} - \int_s^t b_r(y_r) \dd r - \xi(y_s) \delta Z_{st} - \Xi(y_s) \mathbb{Z}_{st} =: y^{\natural}_{st};
    \end{align*}
    on the other hand, by the a priori estimates of Lemma \ref{lem:apriori-RDE-forcing} and \eqref{eq:flow-proof-eq2}, it holds $\sup_n \llbracket y^{n,\natural} \rrbracket_{\mathfrak{p}/3,T} <\infty$ for any $T>0$; combined with Remark \ref{rem:p-variation-lsc}, this implies that $y^\natural\in C^{\mathfrak{p}/3-\var}_2$ and so that $y$ is a solution.

    Next we show that any two solutions $y^x_t$, $y^{\tilde x}_t$ to \eqref{eq:RDE-drift}, with initial data $x$, $\tilde x\in\R^d$, satisfy estimate \eqref{eq:modulus-continuity-flow}; this checks both uniqueness (viz $y^x_t=\Phi_t(x)$) and the associated estimate for $\Phi_t(x)$.
    The proof is very similar to the previous one for $\{\Phi^n_t\}_n$, so we mostly sketch it. Setting $I_t = \sup_{s\in [0,t]} |X^x_s - X^{\tilde x}_s|$, by Lemma \ref{lem:contraction-RDE-forcing} and Assumption \ref{ass:osgood-drift-RDE} it holds
    \begin{align*}
        I_t
        \leq C |x-\tilde x| + C \int_0^t |b_s(X^x_s) - b_s(X^{\tilde x}_s)| \dd s
        \leq C |x-\tilde x| + C \int_0^t g_s\, h(I_s) \dd s
    \end{align*}
    where as before the constant $C$ is uniform in $x$, $\tilde x$; Lemma \ref{lem:bihari} then readily implies \eqref{eq:modulus-continuity-flow}.

    It remains to show that for each $t$, the map $\Phi_t$ is invertible, with inverse satisfying the same regularity estimate \eqref{eq:modulus-continuity-flow}, which we do by a time-reversal argument.
    For any fixed $t$, consider now the $t$-time reversal of the RDE \eqref{eq:RDE-drift}, which by Lemma \ref{lem:time-reversal-RDE} is the one associated to the rough path $\overleftarrow \bZ$, same $\xi$ and time-reversed drift $\overleftarrow  b_{\! s}(x):= -b_{t-s}(x)$.
    It is easy to check (by invoking Remark \ref{rem:time-reversal}) that $(\overleftarrow \bZ, \xi, \overleftarrow  b)$ satisfies the same properties as $(\bZ, \xi, b)$, so that going through the same procedure as above we can construct the associated solution map $(s,x) \mapsto \Psi_s^t(x)$, which again satisfies \eqref{eq:modulus-continuity-flow} (up to possibly relabelling $C$).
    On the other hand, again by Lemma \ref{lem:time-reversal-RDE}, a solution to the $t$-time reversed RDE with initial datum $\Phi_t(x)$ is given by $s\mapsto \Phi_{t-s}(x)$, which implies that $x=\Phi_0(x)=\Psi_t^t (\Phi_t(x))$ and thus shows that $\Psi_t^t$ is the inverse of $\Phi_t$.
    To show that \eqref{eq:modulus-continuity-flow} also holds for the inverse flow, notice that \eqref{eq:modulus-continuity-flow} applied to the time $t$ reversed solution $\Psi_{\cdot}^t$ gives
    \begin{equation*}
        \sup_{s \in [0,t]} \left| \Psi_s^t(x) - \Psi_s^t(\tilde{x}) \right| \leq M^h\Big(C|x - \tilde{x}|,\int_0^t g_s \rmd s\Big) \leq M^h\Big(C|x - \tilde{x}|,\int_0^T g_s \rmd s\Big).
    \end{equation*}
    Choosing $s = t$ on the left hand side we find
    \begin{equation*}
        \left| \Phi_t^{-1}(x) - \Phi_t^{-1}(\tilde{x}) \right|
        %\leq M^h(C|x - \tilde{x}|,\int_0^t g_s \rmd s)
        \leq M^h\Big(C|x - \tilde{x}|,\int_0^T g_s \rmd s\Big);
    \end{equation*}
    taking supremum over $t$ on the left hand side gives the conclusion.     
\end{proof}

\begin{remark}%\label{rem:two-parameter-flow-RDE}
    In order not to make the notation too burdensome, in the statement and proof of Theorem \ref{thm:wellposedness-RDE} we considered the solution map $(t,x)\mapsto \Phi_t(x)$ as depending only on the terminal time $t$.
    Up to small modifications, the proof can be readapted to construct the unique \emph{flow of homeomorphisms} $(s,t,x)\mapsto \Phi_{s\to t}(x)$, which stands for the terminal position at time $t$ of a solution to the RDE starting at position $x$ at time $s$.
    In particular, one can still show that $\Phi_{s\to t}$ is a homemorphism; denoting by $\Phi_{s\leftarrow t}$ its continuous inverse, one then has naturally the \emph{group properties}
    \begin{equation*}
        \Phi_{s\to s}= {\rm Id}, \quad
        \Phi_{s\to t}\circ \Phi_{s\leftarrow t}= {\rm Id}, \quad
        \Phi_{u\to t}\circ \Phi_{s\to u} = \Phi_{s\to t} \quad \forall\, s\leq u\leq t.
    \end{equation*}
    A similar extensions applies to the upcoming Corollary \ref{cor:stability-RDE-flow}, concerning convergence of $\Phi^n_{s\to t}$ to $\Phi_{s\to t}$ uniformly on compact sets.
\end{remark}

\begin{corollary}\label{cor:stability-RDE-flow}
    Let $\mathfrak{p}\in [2,3)$ and consider a sequence $\{(\mathbf{Z}^n, \xi^n, b^n)\}_n$ satisfying the assumptions of Theorem~\ref{thm:wellposedness-RDE} uniformly in $n$, namely such that for all $T>0$
    \begin{equation*}
        \sup_{n \in \N} \big\{ \| \xi^n\|_{C^3_b}  +\| \mathbf{Z}^n\|_{\mathfrak{p},T} \big\}<\infty
    \end{equation*}
    and such that $b^n$ satisfy Assumption~\ref{ass:osgood-drift-RDE} for the same $g$ and $h$.
    Further assume that there exist $(b,\xi,\mathbf{Z})$ such that $(b^n,\xi^n)\to (b,\xi)$ uniformly on compact sets and $\sup_{t\in [0,T]} |\mathbf{Z}^n_{0,t}-\mathbf{Z}_{0,t}|\to 0$ for any $T>0$.
    Then the flows $\Phi^n$ associated to $(b^n,\xi^n,\mathbf{Z}^n)$ converge uniformly on compact sets to the flow $\Phi$ associated to $(b,\xi,\mathbf{Z})$, namely
    \begin{equation*}
        \lim_{n\to\infty} \sup_{t\in [0,T],\, y\in B_R} |\Phi^n_t(y)-\Phi_t(y)| = 0 \quad \forall\, T,\,R\in (0,+\infty).
    \end{equation*}
    A similar statement holds for $\Phi^n_t$ and $\Phi_t$ replaced by their inverses.
\end{corollary}

\begin{proof}
    Let us fix any $T>0$. First of all observe that, by the uniform convergence $\mathbf{Z}^n_{0,t}\to \mathbf{Z}_{0,t}$ and Chen's relation \eqref{eq:chen-relation}, it holds $\sup_{s\leq t\leq T} | \mathbb{Z}_{s,t}^n-\mathbb{Z}_{s,t}|=0$ as well, cf. \cite[Exercises 2.4-b) \& 2.9]{FH2020}.
    Combined with the uniform  $\mathfrak{p}$-variation bound, arguing as in \cite[Corollary 5.29]{FV2010}, it follows that for any $\tilde{\mathfrak{p}}\in (\mathfrak{p},3)$ we have $\| \mathbf{Z}^n-\mathbf{Z}\|_{\tilde{\mathfrak{p}},[s,t]}\to 0$. As in \cite[Corollary 5.31]{FV2010}, the associated $\tilde{\mathfrak{p}}$-variation seminorms are equicontinuous, so that there exists a common modulus of continuity $\gamma$ such that
    \begin{equation}\label{eq:stability-flow-proof-eq1}
        \sup_n\, \llbracket \mathbf{Z}^n\rrbracket_{\tilde{\mathfrak{p}}, [s,t]} \leq \gamma (|t-s|) \quad \forall\, (s,t)\in\Delta_T.
    \end{equation}
    For notation simplicity, we will henceforth drop the tilde and assume that \eqref{eq:stability-flow-proof-eq1} holds for $\mathfrak{p}$.

    Interpolating between the $C^3_b$ uniform bound and uniform convergence on compact sets, $\xi^n\to \xi$ also in $C^2_\loc$, therefore $\Xi^n$ converge to $\Xi$ in $C^0_\loc$ as well.

    Next observe that, by combining the uniform assumptions on $(b^n,\xi^n,\bZ^n)$, the uniform estimate \eqref{eq:stability-flow-proof-eq1} and the estimates from Lemma \ref{lem:apriori-RDE-forcing} and Theorem \ref{thm:wellposedness-RDE}, it holds
    \begin{equation*}%\label{eq:stability-flow-proof-eq2}
       \sup_n | \Phi^n_t(x) - \Phi^n_s(\tilde x)|
       \lesssim \gamma(|t-s|) + M^h\Big(C|x-\tilde x|, C\int_0^T g_s \dd s\Big)
    \end{equation*}
    which shows equicontinuity of $\{\Phi^n\}_n$; in particular, by Ascoli--Arzelà we can extract a (not relabelled) subsequence such that $\Phi^n\to F$ uniformly on compact sets.
    If we show that the unique candidate limit is $F=\Phi$, as the argument holds for any subsequence we can extract, conclusion follows.

    Fix any $x\in\R^d$. Arguing as in the proof of Theorem \ref{thm:wellposedness-RDE}, using the fact that $(b^n,\xi^n,\Xi^n)\to (b,\xi,\Xi)$ uniformly on compacts and $\bZ^n\to\bZ$ in $\clC^{\mathfrak{p}-\var}$, while by Lemma \ref{lem:apriori-RDE-forcing} it holds $\sup_n \llbracket y^{n,\natural}_t\rrbracket_{\mathfrak{p}/3-\var}<\infty$ (for $y^n_t:=\Phi^n_t(x)$), it's easy to check that any limit point of $y^n_t$ must be a solution to the RDE associated to $(b,\xi,\bZ)$. By Theorem \ref{thm:wellposedness-RDE}, solutions to such RDE are unique, thus $F_t(x)=\Phi_t(x)$, implying the conclusion.

    As before, the final claim concerning the convergence of the inverse is established by time reversal. By applying \eqref{eq:modulus-continuity-flow} to the inverse flow, we find that $x \mapsto \Psi_t^{n,t}(x)$ are equicontinuous in $x$, uniformly in $t$.
    To find equicontinuity w.r.t. to $t\in [0,T]$, we fix $x$, choose $0\leq t_1 \leq t_2\leq T$ and write
    \begin{align*}
        \left| \Phi^n_{0 \leftarrow t_2}(x) - \Phi^n_{0\leftarrow t_1}(x)\right|
        = \left|\Phi^n_{0 \leftarrow t_2}(x) - \Phi^n_{0\leftarrow t_2}(\Phi^n_{t_1 \rightarrow t_2}(x)) \right| 
        \leq M^h\Big(|x - \Phi^n_{t_1 \rightarrow t_2}(x)|, \int_0^{t_1}g_s \rmd s\Big) .
    \end{align*}
    By applying \eqref{eq:apriori-RDE-forcing-1} to $\mu^n_t = \int_0^t b_r^n(\Phi_r^n(x)) \rmd r$ and using the hypothesis on $b^n$, we get
    \begin{equation*}
        |\Phi^n_{t_1 \rightarrow t_2}(x) - x| \lesssim \int_{t_1}^{t_2} g_r \dd r;
    \end{equation*}
    the latter quantity can be made arbitrarily small by choosing $|t_2-t_1|$ small, since $g\in L^1([0,T])$. Overall this proves equicontinuity in space-time of $\{\Psi^n\}_n$; the rest of the argument is identical to above.
\end{proof}

In the next lemma, we specialize to flows associated to {\em geometric rough paths}; in this case, under suitable assumptions on $(b,\xi)$, we can prove that the flow leaves the Lebesgue measure quasi-invariant, which will be a crucial property in the study of the corresponding rough PDEs.

\begin{corollary}\label{cor:flow-incompressibility}
    Let $\mathfrak{p}\in [2,3)$, $\xi\in C^3_b$, $\mathbf{Z}\in \clC^{\mathfrak{p}-\var}_g$ and $b$ satisfying Assumption~\ref{ass:osgood-drift-RDE}. Suppose further that $\nabla\cdot \xi_k=0$ for all $k=1,\ldots, m$ and
    \begin{equation*}
        \int_0^T \| \nabla\cdot b_s\|_{L^\infty} \dd s <\infty \quad \forall\, T>0.
    \end{equation*}
    Then the flow $\Phi$ associated to the RDE~\eqref{eq:RDE-drift} is quasi-incompressible, in the sense that for any $t>0$ and any Borel set $A\subset \R^d$ it holds
    \begin{equation}\label{eq:flow-incompressibility}
        \exp\Big(-\int_0^t \| \nabla\cdot b_s\|_{L^\infty} \dd s\Big) \mathscr{L}^d(A) 
        \leq \mathscr{L}^d(\Phi_t (A)) 
        \leq  \exp\Big(+\int_0^t \| \nabla\cdot b_s\|_{L^\infty} \dd s\Big) \mathscr{L}^d(A)
    \end{equation}
    where $\mathscr{L}^d$ denotes the Lebesgue measure on $\R^d$.
    A similar statement holds with $\Phi_t$ replaced by $\Phi^{-1}_t$.    
\end{corollary}

\begin{proof}
Set $b^n=\rho^{1/n}\ast b$ as in the proof of Theorem \ref{thm:wellposedness-RDE} and let $\bZ^n=(Z^n, \mathbf{Z}^n)$ be a sequence of smooth rough paths such that $d_{\mathfrak{p},T}(\mathbf{Z}, \mathbb{Z}^n)\to 0$.
Observe that, by properties of convolutions, under our assumptions it holds
\begin{equation}\label{eq:incompressibility-proof-eq1}
    \sup_n \int_0^T \| \nabla\cdot b^n_s\|_{L^\infty} \dd s \leq \int_0^T \| \nabla\cdot b_s\|_{L^\infty} \dd s\quad \forall\, T>0.
\end{equation}
For each $n$, due to the regularity of $(b^n,\xi,\bZ^n)$, there exists a flow of diffeomorphisms $\Phi^n$ associated to the corresponding ODE (thus also RDE); denote by $\bJ^n_t(x):=D_x\Phi^n_t(x)$ its Jacobian and by $\Phi^{n;-1}_t$ its inverse.
The validity of the standard chain rule in this regular setting then yields the classical formula
%
\iffalse
$\bbR^d\rightarrow \clL(\bbR^d,\bbR^d)\cong M_{d\times d}(\bbR)$ satisfies the ODE given by
$$
\frac{d}{dt}\bJ_t^n(x) = D b_t^n(\Phi_t^n(x)) \bJ^n_t(x) \dt + D\xi(\Phi^n_t(x)) \bJ^n_t(x) \dZt^n\,,\quad \bJ_0^n(x)=\bI_{d\times d}\,.
$$
A standard application of the chain rule then shows that the determinant $\operatorname{det} \bJ^n$ satisfies 
\begin{align*}
\frac{d}{dt}\operatorname{det} \bJ_t^n(x) &= \nabla \cdot b_t^n(\Phi_t^n(x)) \operatorname{det} \bJ^n_t(x) \dt + \nabla \cdot \xi(\Phi^n_t(x)) \operatorname{det} \bJ^n_t(x) \dZt^n\\
&= \nabla \cdot b_t^n(\Phi_t^n(x)) \operatorname{det} \bJ^n_t(x) \dt\, ,\qquad \operatorname{det} \bJ_0^n(x)=1\,,
\end{align*}
which can be explicitly solved:
\fi
%
\begin{equation}\label{eq:incompressibility-proof-eq2}\begin{split}
    \operatorname{det} \bJ^n_t(x)
    & =  \exp\bigg(\int_0^t \Big[ \nabla \cdot b_s^n(\Phi_s^n(x)) + \sum_k \nabla\cdot \xi_k(\Phi_s^n(x))\dot Z^{k;n}_s \Big] \ds \bigg)\\
    & =  \exp\bigg(\int_0^t \nabla \cdot b_s^n(\Phi_s^n(x)) \ds \bigg)\,.
\end{split}\end{equation}
where the second passage is due to the assumption $\nabla\cdot\xi_k=0$.
Combined with \eqref{eq:incompressibility-proof-eq1}, this yields the uniform-in-$n$ two-sided bound
\begin{align*}
    \exp\Big(-\int_0^t \|\nabla \cdot b_s\|_{L^\infty}  \ds  \Big) \le \operatorname{det} \bJ^n_t(x)   \le \exp\Big(+\int_0^t \|\nabla \cdot b_s\|_{L^\infty}  \ds  \Big).
\end{align*}
Correspondingly, for any non negative $\varphi\in C^\infty_c$, the change of variables formula yields
\begin{equation}\label{eq:incompressibility-proof-eq3}\begin{split}
    \exp\Big(-\int_0^t \|\nabla \cdot b_s\|_{L^\infty}  \ds  \Big) \int_{\bbR^d}\varphi(x) \dd x
    & \le \int_{\bbR^d} \varphi(\Phi^{n;-1}_t(x)) \dd x \\
    & \le \exp\Big(+\int_0^t \|\nabla \cdot b_s\|_{L^\infty}  \ds  \Big) \int_{\bbR^d}\varphi(x) \dd x.
\end{split}\end{equation}
Since $\{(b^n, \xi, \bZ^n)\}$ satisfy the hypothesis of Corollary \ref{cor:stability-RDE-flow},  
$\varphi\circ \Phi^{n;-1}_t \rightarrow \varphi\circ \Phi^{-1}_t $
uniformly in $x$, and thus passing to the limit, we find the bound  \eqref{eq:incompressibility-proof-eq3} with $\Phi^{-1}_t$ in place of $\Phi^{n;-1}_t$.
By the monotone class theorem, we deduce that the same estimate holds for $\varphi$ replaced by the indicator function $\mathbf{1}_{A}$ of any Borel set $A$ with $\mathscr{L}^d(A)<\infty$, yielding \eqref{eq:flow-incompressibility}.

The analogous statement for $\Phi^{-1}_t$ follows as before by time reversal.
\end{proof}

\section{Linear rough continuity and transport equations}\label{sec:linear-RPDEs}

The main aim of this section is to develop a meaningful theory of weak solutions to rough continuity equations of the form 
\begin{equation}\label{eq:formal-rough-continuity}
    \dd \rho + \nabla\cdot (b\rho) \dd t + \sum_{k=1}^m \nabla\cdot( \xi_k \rho) \dd \bZ^k = 0
\end{equation}
under minimal regularity requirements on $b$.
The results included here range from existence, uniqueness and renormalizability criteria, as well as flow representation, culminating in Theorem \ref{thm:linear-RPDE-wellposed-lagrangian}.
They constitute an %(incomplete)
extension of the DiPerna--Lions theory to the rough case and will be the basic building block to approach nonlinear PDEs in the upcoming Section \ref{sec:nonlinear-RPDEs}.

We will treat the aforementioned rough PDEs in the \emph{unbounded rough drivers} framework, which is recalled in Section \ref{subsec:unbounded-rough}.
We then pass to examine criteria for existence of solutions and a priori estimates in Section \ref{subsec:linear-existence}; Section \ref{subsec:linear-uniqueness} is devoted to the proof of key product and duality formulas, resulting in the uniqueness, renormalizability and stability results from Section \ref{subsec:linear-stability}.

From here on we will work exclusively with geometric rough paths, recall Definition \ref{defn:geom-rough-path}; in other terms, we will always assume $\bZ\in \clC^\fkp_g$ for some $\fkp\in [2,3)$.
Whenever not specified, we will be implicitly working on a compact time interval $[0,T]$, finite but possibly arbitrarily large; all Bochner-Lebesgue spaces $L^q_T E$ must be interpreted as $L^q([0,T];E)$.
Whenever working with $t\in \R_{\geq 0}$, we will state it explicitly, see for instance the upcoming Definition \ref{defn:solution-rough-continuity} and Proposition \ref{prop:existence-linear-RPDE}; in such cases, we might consider elements in $L^q_\loc E:=\bigcap_{T>0} L^q([0,T];E)$.

\subsection{A primer on unbounded rough drivers}\label{subsec:unbounded-rough}

We recall here some basic facts about \emph{unbounded rough drivers}, a framework first developed in \cite{BaiGub2017} in order to give meaning to a general abstract class of rough PDEs (RPDEs).
The advantage of this theory, originally designed for transport equations, is that it allows to easily derive a priori estimates (cf. the key Lemma \ref{lem:apriori-unbounded} below) which are at the heart of existence results, while also allowing nonlinear operations such as \emph{tensorization} (cf. Proposition \ref{prop:tensorization} in Section \ref{subsec:linear-uniqueness}) paving the way for uniqueness statements.

Our exposition mainly follows \cite{DGHT2019}; see also \cite{hocquet2018energy, HLN2021, HHN2025} for other accounts and applications to RPDEs.
Let us stress that, although the abstract theory is designed in an axiomatic fashion, where the fundamental objects involved (scales of spaces and smoothing, see below) are given, in practical applications part of the problem is also identifying what are the correct choices for this setup.

\begin{definition}%\label{defn:p-scale-spaces}
    A tuple $(E_l, \| \cdot\|_l)_{0\leq l\leq 3}$ is a \emph{scale of Banach spaces} if $E_{l+1}$ continuously embeds into $E_l$ for each $l=0,1,2$.
    We denote by $E_{-l}$ the topological dual of $E_l$, so that $E_{-l}$ continuously embeds into $E_{-l-1}$ as well.
\end{definition}

\begin{definition}\label{defn:smoothing}
    A \emph{smoothing} on a scale $(E_l)_{0\leq l\leq 3}$ is a family of operators $(J^\eta)_{\eta\in (0,1]}$ acting on $E_l$ such that, for all $\eta\in (0,1]$, it holds
    \begin{align}
        \| J^\eta - I\|_{\clL(E_l,E_j)}
        & \leq C \eta^{l-j} \quad
        & \text{for } (j,l)\in \{(0,1), (0,2), (1,2)\}\label{eq:defn-smoothing-prop1},\\
        \| J^\eta\|_{\clL(E_j,E_l)}
        & \leq C \eta^{-(l-j)}
        &\text{for } (j,l)\in \{(1,1), (1,2), (2,2), (1,3), (2,3)\}. \label{eq:defn-smoothing-prop2}
    \end{align}
    for some constant $C>0$. In this case, we denote the optimal choice of $C$ by $\interleave J \interleave$.
\end{definition}

\begin{remark}%\label{rem:link-smoothing-interpolation}
    Up to redefining the norms $\| \cdot\|_{E_l}$, we can and will assume in the sequel that $\| \varphi\|_{E_l}\leq \| \varphi\|_{E_{l+1}}$ for $l=0,1,2$.
    Correspondingly, by duality $\| \psi\|_{E_{-l-1}}\leq \| \psi\|_{E_{-l}}$ for $l=0,1,2$.

    Although more elastic in its scope, the concept of smoothing is closely related to the idea of performing interpolation estimates on the dual spaces $(E_{-l})_{0\leq l\leq 3}$, and it does indeed imply their validity.
    For instance, for any $1\leq j<l\leq 3$ and any $\psi\in E_{-0}$, we claim that
    \begin{equation}\label{eq:interpolation-smoothing}
        \| \psi\|_{E_{-j}} \leq 2 \interleave\! J\!\interleave \| \psi\|_{E_{-l}}^{j/l} \| \psi\|_{E_{-0}}^{1-j/l}.
    \end{equation}
    To show \eqref{eq:interpolation-smoothing}, by homogeneity we may assume $\| \psi\|_{E_{-0}}=1$.
    In this case, for any $\varphi\in E_j$ with $\| \varphi\|_{E_j}=1$, it holds
    \begin{align*}
        |\langle \psi, \varphi \rangle|
        \leq |\langle \psi, J^\eta  \varphi \rangle| + |\langle \psi, (I-J^\eta) \varphi \rangle|
        \leq \interleave J\! \interleave \big( \eta^{-l+j} \| \psi\|_{E_{-l}} + \eta^j \big)
    \end{align*}
    Taking first supremum over $\varphi$ and then choosing $\eta= \|\psi\|_{E_{-l}}^{1/l}$ (which is allowed since $\|\psi\|_{E_{-l}} \leq \|\psi\|_{E_{-0}}=1$) yields the desired \eqref{eq:interpolation-smoothing}.
    For further discussion, see also \cite[Section 2.1]{HHN2025}.
\end{remark}

The next lemma provides some practical examples of spaces admitting a smoothing, which will be relevant in the sequel; the proof is postponed to Appendix \ref{app:smoothing}.

\begin{lemma}\label{lem:smoothing-examples}
The following hold.
\begin{itemize}
    \item[a)] For any $p\in [1,\infty]$, the scale of spaces $E_l=W^{l,p}=W^{l,p}(\R^d)$ admits a smoothing.
    \item[b)] For a fixed $R\in [1,\infty)$, consider the scale of spaces $\clF_{l,R}$ given by
    \begin{equation}\label{eq:F_k,R}
        \mathcal{F}_{l,R}=\cF_{l,R}(\R^d):=\{\varphi\in W^{l,\infty}(\R^d): {\rm supp }\, \varphi\subset B_R\};
    \end{equation}
    then $\mathcal{F}_{l,R}$ admits a smoothing $(J^\eta)_{\eta\in (0,1]}$, which moreover can be constructed so that $\interleave J \interleave$ does not depend on $R$.
    \item[c)] For a fixed $R\in [1,\infty)$, define
    \begin{align*}
        x_{\pm} := \frac{x \pm y}{2}, \qquad \zeta_R(x,y) := \frac{|x_+|^2}{R^2} + |x_-|^2
    \end{align*}
    and consider the scale of spaces $\clE_{l,R}$ given by
    \begin{equation}\label{eq:E_k,R}
        \mathcal{E}_{l,R}=\mathcal{E}_{l,R}(\R^{2d}) := \left\{ \Phi \in W^{l,\infty}(\R^{2d}) \, : \zeta_R(x,y) \geq 1 \,\,\Rightarrow \,\,\Phi(x,y) = 0 \right\};
    \end{equation}
    then $\clE_{l,R}$ admits a smoothing $(J^\eta)_{\eta\in (0,1]}$, with $\interleave J \interleave$ independent of $R$ as in point b).
\end{itemize}
\end{lemma}

\begin{remark}\label{rem:smoothing-examples}
    By definition, $\cF_{l,R}(\R^d)$ is a closed subspace of $W^{l,\infty}(\R^d)$, thus a Banach space when endowed with the norm $\| \cdot\|_{W^{l,\infty}(\R^d)}$;
    similarly, $\cE_{l,R}(\R^{2d})$ is a closed subspace of $W^{l,\infty}(\R^{2d})$, thus Banach with norm $\| \cdot\|_{W^{l,\infty}(\R^{2d})}$.
    By definition of $x_\pm$, for any $R\geq 1$ it holds that
    \begin{align*}
        2(|x|^2 + |y|^2) =|x_+|^2+|x_-|^2
        \leq R^2 \bigg( \frac{|x_+|^2}{R^2} + |x_-|^2\bigg) = R^2\, \zeta_R(x,y), 
    \end{align*}
    so that $\cE_{l,R}(\R^{2d})\subset \cF_{l,\sqrt{2} R}(\R^{2d})$ for every $l\in \{0,\ldots,3\}$.
\end{remark}

\begin{definition}\label{defn:unbounded-rough-driver}
    Let $(E_l)_{0\leq l\leq 3}$ be a scale of spaces. We say that a pair $\bA=(A,\mathbb A)$ of $2$-index maps is a \emph{continuous unbounded $\fkp$-rough driver} w.r.t. $(E_l)_{0\leq l\leq 3}$ if the following hold:
    \begin{enumerate}
        \item $A_{st}\in\clL(E_{-l},E_{-l-1})$ for $l\in \{0,2\}$ and $\mathbb{A}_{st}\in\clL(E_{-l},E_{-l-2})$ for $l\in \{0,1\}$;
        \item there exists a control $w_{\bA}$ on $[0,T]$ such that, for $l$ as above, it holds
        \begin{equation*}
            \| A_{st}\|_{\clL (E_{-l},E_{-l-1})} \leq w_{\bA}(s,t)^{\frac{1}{\fkp}}, \quad
            \| \mathbb A_{st}\|_{\clL (E_{-l},E_{-l-2})} \leq w_{\bA}(s,t)^{\frac{2}{\fkp}} \quad\forall\, (s,t)\in\Delta_T;
        \end{equation*}
        \item finally, \emph{Chen's relation} holds, in the sense that
        \begin{equation}\label{eq:chen-unbounded-drivers}
            \delta A_{sut} =0, \quad \delta \mathbb{A}_{sut} = A_{ut} A_{su} \quad \forall\, (s,u,t)\in\Delta^2_T.
        \end{equation}
    \end{enumerate}
\end{definition}

In analogy to Section \ref{sec:RDEs}, we start by defining solutions to RPDEs in the presence of an additional forcing $\mu$, which is of bounded variation in suitable topologies.

\begin{definition}\label{defn:solution-rough-PDE}
    Let $\fkp\in [2,3)$, $(E_l)_{0\leq l\leq 3}$ be a scale of spaces and $\bA$ be a unbounded $\fkp$-rough driver w.r.t. $(E_l)_{0\leq l\leq 3}$; let $\mu\in C^{1-\var} E_{-3}$.
    A bounded Borel path $\rho\in \cB_b([0,T]; E_{-0})$ is a \emph{solution to the rough PDE}
    \begin{equation}\label{eq:rough-abstract-PDE}
        \dd \rho_t + \mu(\dd t) + \bA(\dd t) \rho_t = 0
    \end{equation}
    if there exists $\rho^\natural\in C^{\fkp/3-\var}_2 E_{-3}$ such that
    \begin{equation}\label{eq:defn-solution-rough-PDE}
        \delta \rho_{st} + \delta \mu_{st} + A_{st} \rho_s = \mathbb{A}_{st} \rho_s+ \rho^{\natural}_{st} \quad \forall\, (s,t)\in\Delta_T.
    \end{equation}
\end{definition}

\begin{remark}\label{rem:weak_cts_solutions}
    By the assumptions and \eqref{eq:defn-solution-rough-PDE}, it holds
    \begin{align*}
        \| \delta\rho_{st}\|_{E_{-3}} \lesssim w_\mu(s,t) + w_{\bA}(s,t)^{\frac{1}{\fkp}}+w_{\natural}(s,t)^{\frac{3}{\fkp}} \quad \forall\, (s,t)\in \Delta_T
    \end{align*}
    where $w_\mu$ and $w_\natural$ are the controls associated respectively to $\mu$, $\rho^\natural$. In particular, $\rho\in C([0,T];E_{-3})$; if additionally $E_3$ densely embeds in $E_0$, then a standard duality argument combined with the uniform boundedness of $\rho$ in $E_{-0}$ implies that $\rho\in C_{w-\ast}([0,T];E_{-0})$.
\end{remark}

The next fundamental result provides a link between unbounded rough drivers and smoothing operators on $(E_l)_{0\leq l\leq 3}$, in the form of conditional a priori bounds for solutions to \eqref{eq:rough-abstract-PDE}; in particular, Lemma \ref{lem:apriori-unbounded} informs us that it suffices to control $\| \rho\|_{\cB_b([0,T];E_{-0})}=\sup_{t\in [0,T]} \| \rho\|_{E_{-0}}$ in order to obtain estimates for all the higher order terms coming from the Davie-type expansion \eqref{eq:defn-solution-rough-PDE}.
This type of result was first established in \cite[Theorem 4.4]{BaiGub2017}; the version below can be seen as an extension of \cite[Proposition 3.1]{hocquet2018energy}.

\begin{lemma}\label{lem:apriori-unbounded}
    Let $\fkp$, $(E_l)_{0\leq l\leq 3}$, $\bA$, $\mu$ be as in Definition \ref{defn:solution-rough-PDE} and $\rho$ be a rough solution to \eqref{eq:rough-abstract-PDE}; suppose that $(E_l)_{0\leq l \leq 3}$ admits a smoothing, in the sense of Definition \ref{defn:smoothing}, and that $\mu\in C^{1-var} E_{-2}$.
    Consider the controls defined by
    \begin{align*}
        w_\mu(s,t):=\llbracket \mu \rrbracket_{1,[s,t];E_{-2}}, \quad
        w_\ast(s,t) := \| \rho\|_{\cB_b([s,t];E_{-0})}^{\frac{\fkp}{3}} w_{\bA}(s,t) + w_\mu(s,t)^{\frac{\fkp}{3}} w_{\bA}(s,t)^{1-\frac{\fkp}{3}}
    \end{align*}
    and set $\rho^\sharp_{st}:= \delta \rho_{st} + A_{st} \rho_s$.
    Then for any $T>0$ there exist a constant $C$, depending on $\fkp$, $\interleave J \interleave$ and $w_\bA(0,T)$, increasing in the last two variables, such that for all $(s,t)\in\Delta_T$ it holds
    \begin{align}
        \llbracket \rho^\natural \rrbracket_{\fkp/3,[s,t];E_{-3}}
        & \leq C w_\ast(s,t)^{\frac{3}{\fkp}}, \label{eq:apriori-unbounded-eq1} \\
        \llbracket \rho^\sharp \rrbracket_{\fkp/2,[s,t];E_{-2}}
        & \leq C \big(1 + \| \rho\|_{\cB_b([s,t];E_{-0})}\big) \Big(w_\bA(s,t)^{\frac{2}{\fkp}} + w_\ast(s,t)^{\frac{2}{\fkp}}\Big) + C w_\mu(s,t) \label{eq:apriori-unbounded-eq2} \\
        \llbracket \rho \rrbracket_{\fkp,[s,t];E_{-1}}
        & \leq C \big( 1 + \| \rho\|_{\cB_b([s,t];E_{-0})}^{3/2}\big) \Big(  w_{\bA}(s,t)^{\frac{1}{\fkp}} + w_\ast(s,t)^{\frac{1}{\fkp}} + w_\mu(s,t)^{\frac{1}{2}}\Big).\label{eq:apriori-unbounded-eq3}
    \end{align}
\end{lemma}

\begin{proof}
    Our assumptions imply that all the conditions from \cite[Corollary 2.11]{DGHT2019} are satisfied; we deduce the existence of a parameter $L=L(\fkp, \interleave J \interleave)>0$ such that, for all $s<t$ such that $w_{\bA}(s,t)\leq L$, it holds
    \begin{align*}
        \| \rho^\natural_{st}\|_{E_{-3}}
        \lesssim_\fkp \| \rho\|_{\cB_b([s,t];E_{-0})} w_{\bA}(s,t)^{\frac{3}{\fkp}} + w_\mu(s,t) w_{\bA}(s,t)^{\frac{3}{\fkp}-1}
        \lesssim_{\fkp} w_\ast (s,t)^{\frac{3}{\fkp}}.
    \end{align*}
    Applying the first part of Lemma \ref{lem:local-to-global}, for $g=\rho^\natural$ and $\tilde \fkp=\fkp/3$, it then holds
    \begin{equation}\label{eq:apriori-unbounded-proof1}
        \llbracket \rho^\natural\rrbracket_{\fkp/3,[s,t];E_{-3}}
        \lesssim_{\fkp,L} w_\ast(s,t)^{\frac{3}{\fkp}} + w_{\bA}(s,t)^{\frac{3}{\fkp}} \| \rho^\natural\|_{C(\Delta_{[s,t]};E_{-3})}.
    \end{equation}
    On the other hand, since $\rho$ satisfies \eqref{eq:defn-solution-rough-PDE}, we have
    \begin{align*}
        \| \rho^\natural_{st}\|_{E_{-3}}
        & \leq \| \delta\rho_{s,t} \|_{E_{-3}} + \| \delta\mu_{s,t} \|_{E_{-3}}
        + \| A_{st} \rho_s \|_{E_{-3}} + \| \mathbb{A}_{st} \rho_s \|_{E_{-3}}\\
        & \lesssim_{w_\bA} \| \rho\|_{\cB_b([s,t];E_{-0})} + w_\mu(s,t);
    \end{align*}
    combined with \eqref{eq:apriori-unbounded-proof1} and the definition of $w_\ast$, overall this yields the desired bound \eqref{eq:apriori-unbounded-eq1}.

    Next observe that by \eqref{eq:defn-solution-rough-PDE}, it holds
    \begin{equation}\label{eq:apriori-unbounded-proof3}
        \rho^\sharp_{st}=\delta \rho_{st} + A_{st} \rho_s = -\delta\mu_{st} + \mathbb A_{st} \rho_s + \rho^\natural_{st};
    \end{equation}
    testing against any $\varphi\in E_2$ with $\| \varphi\|_{E_2}=1$, by relation \eqref{eq:apriori-unbounded-proof3} and properties of smoothing operators, we find
    \begin{align*}
        |\langle \rho^\sharp_{st}, \varphi \rangle|
        & = |\langle \delta \rho_{st} + A_{st} \rho_s, (I-J^\eta)\varphi \rangle| + |\langle -\delta\mu_{st} + \mathbb A_{st} \rho_s + \rho^\natural_{st}, J^\eta \varphi \rangle|\\
        & \leq \| \delta\rho_{st}\|_{E_{-0}} \|(I-J^\eta)\varphi\|_{E_0} + \| A_{st}\|_{\cL(E_{-0},E_{-1})} \| \rho_s\|_{E_{-0}} + \| \delta\mu_{st}\|_{E_{-2}} \| J^\eta\varphi\|_{E_2}\\
        & \quad + \| \bA_{st}\|_{\cL(E_{-0},E_{-2})} \| \rho_s\|_{E_{-0}} \| J^\eta\varphi\|_{E_2} + \| \rho^\natural_{st}\|_{E_{-3}} \| J^\eta\varphi\|_{E_3}\\
        & \lesssim%_{\interleave J\interleave}
        \eta^2 \| \rho\|_{\cB_b([s,t];E_{-0})} + \eta \| \rho\|_{\cB_b([s,t];E_{-0})}\, w_{\bA}(s,t)^{\frac{1}{\fkp}} + w_\mu(s,t) \\
        & \quad + \| \rho\|_{\cB_b([s,t];E_{-0})}\, w_{\bA}(s,t)^{\frac{2}{\fkp}} + \eta^{-1} w_\ast(s,t)^{\frac{3}{\fkp}}.
    \end{align*}
    Taking supremum over $\varphi\in E_2$ with $\| \varphi\|_{E_2}=1$ and applying the basic estimate $2xy\leq x^2 + y^2$, we arrive at
    \begin{align*}
        \| \rho^\sharp_{st}\|_{E_{-2}}
        \lesssim \eta^2 \| \rho\|_{\cB_b([s,t];E_{-0})} + w_\mu(s,t) + \| \rho\|_{\cB_b([s,t];E_{-0})} w_{\bA}(s,t)^{\frac{2}{\fkp}} + \eta^{-1} w_\ast(s,t)^{\frac{3}{\fkp}}.
    \end{align*}
    Whenever $w_\ast(s,t)<1$, choosing $\eta=w_\ast(s,t)^{1/\fkp}$ then yields
    \begin{equation*}
        \| \rho^\sharp_{st}\|_{E_{-2}}
        \lesssim \big(1 + \| \rho\|_{\cB_b([s,t];E_{-0})}\big) \Big(w_\bA(s,t)^{\frac{2}{\fkp}} + w_\ast(s,t)^{\frac{2}{\fkp}}\Big) + w_\mu(s,t);
    \end{equation*}
    applying the first part of Lemma \ref{lem:local-to-global}, for $g=\rho^\sharp$ and $\tilde \fkp=\fkp/2$, we arrive at the global estimate
    \begin{equation}\label{eq:apriori-unbounded-proof2}
        \|\rho^\sharp_{st}\|_{E_{-2}}
        \lesssim \big(1 + \| \rho\|_{\cB_b([s,t];E_{-0})}\big) \Big(w_\bA(s,t)^{\frac{2}{\fkp}} + w_\ast(s,t)^{\frac{2}{\fkp}}\Big) + w_\mu(s,t) + \| \rho^\sharp\|_{C(\Delta_{[s,t]};E_{-2})} w_\ast(s,t)^{\frac{2}{\fkp}}.
    \end{equation}
    By the definition of $\rho_{st}^\sharp$, it holds
    \begin{align*}
        \| \rho^\sharp_{st} \|_{E_{-2}} \lesssim \| \rho_{st}\|_{E_{-1}} + \| A_{st} \rho_s\|_{E_{-1}}
        \lesssim_{w_{\bA}} \| \rho\|_{\cB_b([s,t];E_{-0})}
    \end{align*}
    which combined with \eqref{eq:apriori-unbounded-proof2} readily yields \eqref{eq:apriori-unbounded-eq2}.

    The proof of \eqref{eq:apriori-unbounded-eq3} is very similar, so let us only sketch the key passages. Using the relation $\delta\rho_{st}= -A_{st}\rho_s +\rho^\sharp_{st}$ and smoothing operators, one finds
    \begin{align*}
        \| \delta \rho_{st}\|_{E_{-1}}
        \lesssim \eta \| \rho\|_{\cB_b([s,t];E_{-0})} + w_\bA(s,t)^{\frac{1}{\fkp}} + \eta^{-1} w_{\sharp}(s,t)^{\frac{2}{\fkp}},
        \quad w_{\sharp}(s,t):=\llbracket \rho^\sharp\rrbracket_{\fkp/2,[s,t];E_{-2}}^{\frac{\fkp}{2}};
    \end{align*}
    whenever $w_\sharp(s,t)\leq 1$, we can then choose $\eta= w_\sharp(s,t)^{\frac{1}{\fkp}}$ to get an estimate where all controls appearing have powers $1/\fkp$ or higher.
    From there one applies Lemma \ref{lem:local-to-global} and the available estimate for $w_\sharp$ coming from \eqref{eq:apriori-unbounded-eq2} to finally arrive at \eqref{eq:apriori-unbounded-eq3}.
\end{proof}

\begin{remark}%\label{rem:apriori-unbounded}
    It is important to stress that the structure of the spaces $(E_l)_l$ (e.g. their geometry) doesn't play a role in the estimates of Lemma \ref{lem:apriori-unbounded}, only $\interleave J\interleave$ does; this will allows to consider a sequence $\{\rho^n\}_n$ of solutions possibly defined on different scales $(E^n_l)_l$, and pass to the limit in $n$ by exploiting the available uniform bounds. 
\end{remark}

\subsection{Existence of weak solutions}\label{subsec:linear-existence}

From now on, we will work exclusively with divergence free vector fields $\xi_k$, namely $\nabla\cdot\xi_k=0$ for all $k=1,\ldots, m$; to express it, we will just write $\nabla\cdot \xi=0$.
In view of this, whenever convenient, we may (at least formally) rewrite \eqref{eq:formal-rough-continuity} as
\begin{equation}\label{eq:rough-continuity}
    \dd \rho_t + \nabla\cdot (b_t \rho_t) \dd t + \sum_{k=1}^m \xi_k \cdot \nabla \rho_t \dd \bZ_t = 0.
\end{equation}
For simplicity, we just write $\sum_k$ in place of $\sum_{k=1}^m$. The next statement allows us to reinterpret \eqref{eq:formal-rough-continuity}/\eqref{eq:rough-continuity} in an unbounded rough driver fashion.

\begin{lemma}\label{lem:URD_continuity}
    Let $\xi\in C^2_b$, $\nabla\cdot\xi_k=0$, $\bZ\in \clC^{\fkp}_g$ for some $\fkp\in [2,3)$; define $\bA=(A,\bbA)$ by
    \begin{equation}\label{eq:defn-unbounded-continuity}
        A_{st}\varphi := \sum_k \xi_k\cdot\nabla \varphi\, \delta Z^k_{st}, \quad 
        \bbA_{st}\varphi := \sum_{j,k} \xi_k\cdot\nabla(\xi_j\cdot\nabla \varphi) \bbZ^{jk}_{st}.
    \end{equation}
    Then $\bA$ is an unbounded rough driver, in the sense of Definition \ref{defn:unbounded-rough-driver}, on the scales of spaces $E_l= W^{l,\infty}$ and $E_l=\cF_{l,R}$ defined by \eqref{eq:F_k,R}.
    Moreover it holds
    \begin{equation*}%\label{eq:unbounded-continuity-basic}
        w_{\bA}(s,t)\lesssim \| \xi\|_{C^2_b}^\fkp\, w_{\bZ}(s,t)
    \end{equation*}
    where the hidden constant is independent of $R$.
    Finally, $\bA$ is \emph{conservative}, in the sense that
    \begin{equation}\label{eq:conservative-driver}
        A_{st}^\ast = - A_{st}, \quad \bbA_{st}+\bbA_{st}^\ast = - A_{st}^\ast A_{st}.
    \end{equation}
\end{lemma}

\begin{proof}
    These are all classical facts from the unbounded rough drivers framework \cite{BaiGub2017,DGHT2019}, so let us motivate them shortly.
    Chen's relation \eqref{eq:chen-unbounded-drivers} follows from the definition of $\bA$ and the corresponding Chen's relation \eqref{eq:chen-relation} for $\bZ$.
    Note that $A_{st}$ and $\bbA_{st}$ are defined as sums and compositions of the differential operators $V_j= \xi_j\cdot\nabla$, which by the assumptions $\nabla\cdot\xi_j=0$, $\xi_j\in C^2_b$ satisfy
    \begin{align*}
        V_j^\ast=-V_j, \quad
        \| V_j \|_{\clL(W^{l+1,\infty},W^{l,\infty})} \lesssim \| \xi_j\|_{C^2_b} \quad \text{for } l=0,1,2;
    \end{align*}
    being local operators, they also respect the support of function, thus the scales $\cF_{l,R}$. Conditions i)-ii) from Definition \ref{defn:unbounded-rough-driver} then readily follow by duality.
    Finally, $\bA$ being conservative follows from $\nabla\cdot \xi_j=0$ and $\bZ$ being a geometric rough path, see the discussion right after \cite[Definition 5.1]{BaiGub2017}.
\end{proof}

\begin{definition}\label{defn:solution-rough-continuity}
    Let $\xi\in C^2_b$, $\nabla\cdot\xi=0$, $\bZ\in \clC^{\fkp}_g$ for some $\fkp\in [2,3)$ and $b\in L^1_t L^1_\loc$.
    We say that a map $\rho:[0,T]\to L^1_\loc$ is a \emph{weak solution to the rough continuity equation} \eqref{eq:rough-continuity} on $[0,T]$ if, for any $R\in [1,\infty)$, it is a solution to the rough PDE \eqref{eq:rough-abstract-PDE} on $\clF_{l,R}$, in the sense of Definition \ref{defn:solution-rough-PDE}, for the choice $\dot\mu= \nabla\cdot (b\rho)$ and $\bA$ as given in \eqref{eq:defn-unbounded-continuity}.

    In other words, $\rho$ solves \eqref{eq:rough-continuity} if %$\nabla\cdot (b\rho)$ is a well defined distribution and 
    for all $R\in [1,\infty)$, the following hold: $\rho\in \cB_b([0,T];\cF_{-0,R})$, $\nabla\cdot (b\rho)\in L^1_t \clF_{-2,R}$, and the two-parameter map $\rho^\natural$ defined by
    \begin{equation}\label{eq:solution-rough-continuity}
        \delta \rho_{st} + \int_s^t \nabla\cdot (b_u\rho_u) \dd u + \sum_k \xi_k\cdot\nabla \rho_s \delta Z_{st}^k - \sum_{j,k} \xi_k\cdot\nabla(\xi_j\cdot\nabla \rho_s) \bbZ^{jk}_{st} = \rho^\natural_{st}\quad \forall\, (s,t)\in \Delta_T
    \end{equation}
    satisfies $\rho^\natural\in C^{\fkp/3-\var}_2 \clF_{-3,R}$.% for all $R\in [1,\infty)$.

    If $b\in L^1_\loc L^1_x$ and $\bZ\in \mathcal{C}^\mathfrak{p}_g([0,T])$ for every $T\in (0,+\infty)$, we say that $\rho$ is a \em{global weak solution} to \eqref{eq:rough-continuity} it is a solution on $[0,T]$, for every $T\in (0,+\infty)$.
\end{definition}

\begin{remark}\label{rem:defn-solution-rough-continuity}
    Definition \ref{defn:solution-rough-continuity} combines the standard concept of weak solution, based on testing against $\varphi\in C^\infty_c$, with the more quantitative setup from unbounded rough drivers, which requires to work with (a family of) scales of Banach spaces, rather than a locally convex topological vector space like $\cD'$.
    This is the reason why we enforce identity \eqref{eq:solution-rough-continuity} to hold on all scales $\cF_{l,R}$, with arbitrarily large but finite $R$.

    By Remark \ref{rem:weak_cts_solutions}, condition \eqref{eq:solution-rough-continuity} implies the continuity of $t\mapsto \langle \varphi, \rho_t\rangle$ for all $\varphi\in C^\infty_c$. In particular, the map $t\mapsto \rho_t$ is continuous in $\cD'$, which allows to give meaning to an initial (resp. terminal) condition $\rho\vert_{t=0}=\rho_0$ (resp. $\rho\vert_{t=T}=\rho_T$) coupled with the rough PDE.
    
    A sufficient condition for $\nabla\cdot(b \rho)\in L^1_t \cF_{-2,R}$ is to verify that $b\rho\in L^1_t L^1_{\loc}$; indeed, by duality
    \begin{align*}
        |\langle \nabla\cdot(b_t \rho_t), \varphi \rangle|
        = |\langle b_t \rho_t, \nabla\varphi \rangle|
        \leq \| \nabla\varphi\|_{L^\infty_x} \| b_t \rho_t\|_{L^1(B_R)}
        \leq \| \nabla\varphi\|_{\cF_{2,R}} \| b_t \rho_t\|_{L^1(B_R)} \qquad\forall\, \varphi\in \cF_{2,R},
    \end{align*}
    where in the intermediate passages we used the support property of $\varphi$, so that we actually get the stronger outcome that $\nabla\cdot(b \rho)\in L^1_t \cF_{-1,R}$:
    \begin{equation}\label{eq:existence-linear-preliminary}
        \int_0^T \| \nabla\cdot(b_t \rho_t)\|_{\cF_{-1,R}} \dd t
        \leq \int_0^T \| b_t \rho_t\|_{L^1(B_R)} \dd t \quad \forall\, R\geq 0.
    \end{equation}

    Let us finally mention that, if $b$ and $\xi$ are smooth, $\bZ$ is a smooth rough path and $\rho$ is a classical smooth solution to the PDE
    \begin{equation*}%\label{eq:smooth-rough-continuity}
        \partial_t \rho + \nabla \cdot (b \rho) + \sum_k \nabla\cdot (\xi_k \rho)\, \dot Z^k = 0
    \end{equation*}
    then a Taylor expansion readily shows that it is also a solution to the associated rough PDE (similarly to the RDE case from Remark \ref{rem:defn-RDE-forcing}).
\end{remark}

\begin{proposition}\label{prop:existence-linear-RPDE}
    Let $\fkp\in [2,3)$, $\bZ\in \clC^{\fkp}_g$, $\xi\in C^2_b$ with $\nabla\cdot \xi =0$ and $p\in [1,+\infty]$. Let
    \begin{equation}\label{eq:existence-linear-assumption}
        \rho_0\in L^p_x, \quad b\in L^1_t L^{p'}_\loc, \quad \nabla\cdot b\in L^1_t L^\infty_x.
    \end{equation}
    Then there exists a weak solution $\rho\in \cB_b([0,T];L^p_x)$ to the rough continuity equation \eqref{eq:rough-continuity} on $[0,T]$, in the sense of Definition \ref{defn:solution-rough-continuity}, with initial condition $\rho\vert_{t=0}=\rho_0$, which moreover satisfies
    \begin{equation}\label{eq:lq-bound}
        \| \rho_t\|_{L^p_x}
        \leq \exp\bigg( \Big(1-\frac{1}{p} \Big) \int_0^t \| \nabla\cdot b_u\|_{L^\infty_x}\, \dd u\bigg) \| \rho_0\|_{L^p_x} \quad \forall\, t \in [0,T].
    \end{equation}
    If moreover $b$, $\bZ$ are defined for $t\in \R_{\geq 0}$ and the above conditions are satisfied on every compact interval $[0,T]$, then there exists a global weak solution $\rho$, satisfying \eqref{eq:lq-bound} for every $t\geq 0$.
\end{proposition}

\begin{proof}
    For simplicity, we only present the proof on a finite interval $[0,T]$; the global existence statement follows similarly, up to an additional Cantor diagonal argument.
    Let us start by treating the case $p\in (1,+\infty)$.
    
    Given $(b,\bZ,\xi)$ as in the assumptions, we can find smooth approximations $(b^n,\bZ^n,\rho^n_0)$ such that
    \begin{equation}\label{eq:existence-linear-approximations}
        b^n\to b\ {\rm in } \ L^1_t L^{p'}_\loc,\quad
        \int_0^T \| b^n_u\|_{L^{p'}(B_R)} \dd u \leq \int_0^T \| b_u\|_{L^{p'}(B_{R+1})} \dd u , \quad
        \| \nabla\cdot b^n\|_{L^1_t L^\infty_x} \leq \| \nabla\cdot b\|_{L^1_t L^\infty_x},
    \end{equation}
    as well as $d_{p,T}(\bZ^n,\bZ)\to 0$ and $\rho^n_0\to\rho_0$ in $L^p_x$.
    We may further assume that $\| \rho^n_0\|_{L^p_x} \leq \| \rho_0\|_{L^p_x}$ for all $n$ and, arguing as in Corollary \ref{cor:stability-RDE-flow}, we can take the associated controls $w_{\bZ^n}$ to be equicontinuous, namely such that $\sup_n w_{\bZ^n}(s,t) \leq \gamma(|t-s|)$ for modulus of continuity $\gamma$.
    For each $n$, the associated PDE
    \begin{equation*}
        \partial_t \rho^n + \nabla\cdot( b^n\, \rho^n) + \sum_k \nabla\cdot( \xi_k \rho^n) \dot Z^{k,n} = 0
    \end{equation*}
    now classically admits a unique solution, which we denote by $\rho^n_t$.
    Since $\nabla\cdot\xi=0$, it has an explicit solution formula given by
    \begin{equation}\label{eq:existence-linear-proof-1}
        \rho_t^n(x) = \rho^n_0(\Phi^n_{0\leftarrow t}(x))\, \exp\Big( -\int_0^t (\nabla\cdot b^n_u)(\Phi_{u\leftarrow t}^n(x)) \dd u\Big),
    \end{equation}
    where $\Phi^n$ is the flow associated to $(b^n,\xi,\bZ^n)$; together with \eqref{eq:incompressibility-proof-eq2} and \eqref{eq:existence-linear-approximations}, this yields
    \begin{equation}\label{eq:existence-linear-proof-2}\begin{split}
        \| \rho^n_t\|_{L^p}^p
        & = \int_{\R^d} |\rho^n_0(x)|^p \exp\Big( (1-p)\int_0^t (\nabla\cdot b^n_u)(\Phi_{0\to u}^n(x)) \dd u\Big) \dd x\\
        & \leq \| \rho^n_0\|_{L^p}^p \exp\Big( (p-1)\int_0^t \| \nabla\cdot b_u\|_{L^\infty_x}\, \dd u \Big),
    \end{split}\end{equation}
    namely the equivalent of \eqref{eq:lq-bound} for the solutions $\rho^n$ to the mollified equations.
    By Remark \ref{rem:defn-solution-rough-continuity}, $\rho^n$ is also a solution to the rough PDE \eqref{eq:rough-abstract-PDE}, with $\bA^n$ defined in terms of $(\xi,\bZ^n)$ by \eqref{eq:defn-unbounded-continuity} and $\mu^n_t=\int_0^t \nabla\cdot (b^n_u \rho^n_u) \dd u$.
    Arguing as in \eqref{eq:existence-linear-preliminary} and using \eqref{eq:existence-linear-proof-2}, for any $R>0$ it holds
    \begin{align*}
        \| \delta \mu^n_{st} \|_{\cF_{-1,R}}
        \leq \int_s^t \| b^n_u \rho^n_u\|_{L^1(B_R)} \dd u
        \lesssim \|\rho^n\|_{L^\infty_t L^p_x} \int_s^t \| b^n_u\|_{L^{p'}(B_R)} \dd u
        \lesssim \| \rho_0\|_{L^p_x}\, w_{b,R}(s,t)
    \end{align*}
    where we define the control $w_{b,R}:= \int_s^t \| b^n_u\|_{L^{p'}(B_R)} \dd u$.
    We are therefore in the position to apply Lemma \ref{lem:apriori-unbounded}, in combination with the above estimates, to find
    \begin{equation}\label{eq:existence-linear-proof-3}\begin{split}
        \| \delta \rho^n_{st}\|_{\cF_{-1,R}}
        & \lesssim (1+\| \rho_0^n\|_{L^p_x}^2) \Big(w_{\bA^n}(s,t)^\frac{1}{\fkp} + w_{b,R} (s,t)^\frac{1}{\fkp} +  w_{b,R} (s,t)^\frac{1}{2}\Big)\\
        & \lesssim (1+\| \rho_0\|_{L^p_x}^2) \Big(\gamma (|t-s|)^{\frac{1}{\fkp}} + w_{b,R} (s,t)^\frac{1}{\fkp}\Big)
    \end{split}\end{equation}
    which shows equicontinuity of $\rho^n$ in $\cF_{-1,R}$ for all $R\geq 1$.
    Thanks to estimates \eqref{eq:existence-linear-proof-2}-\eqref{eq:existence-linear-proof-3}, the assumptions of Proposition \ref{prop:compactness-lp} in Appendix \ref{app:compactness} are met; we can therefore find a (not relabelled for simplicity) subsequence and a function $\rho\in C_w([0,T];L^p_x)$ such that $\rho^n\to \rho$ in $C_w([0,T];L^p_x)$ (recall Definition \ref{defn:uniform_weak_convergence}).
    In particular, $\rho^n_t \rightharpoonup \rho_t$ in $L^p$ for all $t\in [0,T]$.
    %$\rho^n_t \rightharpoonup \rho_t$ in $L^p$ for all $t$ and $\langle \rho^n_t - \rho_t,\psi\rangle\to 0$ uniformly in $t\in [0,T]$ for all $\psi\in L^{p'}_x$.

    We claim that $\rho$ is the desired weak solution to \eqref{eq:rough-continuity}. Properties of weak convergence and the uniform bound \eqref{eq:existence-linear-proof-2} readily imply \eqref{eq:lq-bound}, while by construction $\rho\vert_{t=0}=\rho_0$ as desired. To see that $\rho$ is a weak solution, let us consider for any $\psi\in \cF_{3,R} $ the expansion
    \begin{equation*}
        \langle \delta \rho^n_{st}, \psi\rangle + \langle \delta \mu^n_{st}, \psi\rangle + \langle \rho^n_s, A^{n,\ast}_{st} \psi\rangle - \langle \rho^n_s, \bbA^{n,\ast}_{st} \psi\rangle = \langle \rho^{n,\natural}_{st}, \psi\rangle
    \end{equation*}
    and study convergence of each term.
    By construction, $A^{n,\ast}_{st}$ and $\bbA^{n,\ast}_{st}$ converge to $A^\ast_{st}$ and $\bbA^\ast_{st}$ respectively, in the appropriate strong operator topologies, which combined with $\rho^n_s \rightharpoonup \rho_s$ in $L^p$ and weak-strong convergence implies that
    \begin{align*}
        \langle \rho^n_s, A^{n,\ast}_{st} \psi\rangle \to \langle \rho_s, A^{\ast}_{st} \psi\rangle, \quad 
        \langle \rho^n_s, \bbA^{n,\ast}_{st} \psi\rangle \to \langle \rho_s, \bbA^{\ast}_{st} \psi\rangle \quad \forall\, (s,t)\in\Delta_T.
    \end{align*}
    Since $b^n\to b$ in $L^1_t L^{p'}_\loc$, $\rho^n_t\rightharpoonup \rho_t$ for all $t\in [0,T]$ and $\psi$ is compactly supported, by weak-strong convergence it holds
    \begin{equation}\label{eq:existence-linear-proof-4}
        \lim_{n\to\infty} \sup_{t\in [0,T]} |\langle \mu^n_t-\mu_t,\psi\rangle|
        \leq \lim_{n\to\infty} \int_0^T |\langle \rho^n_u, b^n_u \cdot\nabla\psi \rangle - \langle \rho_u, b_u \cdot\nabla\psi \rangle| \dd u = 0
    \end{equation}
    and clearly $\langle \delta\rho^n_{st},\psi \rangle \to \langle \delta\rho_{st},\psi\rangle$.
    In order to conclude that $\rho$ is a weak solution, it remains to show that the associated two-parameter $\rho^\natural_{st}$ belongs to $C^{\fkp /3-\var}_2 \cF_{-3,R}$; by the previous estimates, $\langle  \rho^\natural_{st}, \psi\rangle = \lim_{n\to\infty} \langle  \rho^{n,\natural}_{st}, \psi\rangle$ for all $\psi\in \cF_{3,R}$, namely $\rho^\natural_{st}$ is the weak\mbox{-}$\ast$ limit of $\rho^{n,\natural}_{st}$ in $\cF_{-3,R}$.
    On the other hand, applying again Lemma \ref{lem:apriori-unbounded} (more precisely \eqref{eq:apriori-unbounded-eq1}), taking into accounts the previous bounds on $\mu^n$, one finds
    \begin{align*}
        \| \rho^{n,\natural}_{st}\|_{\cF_{-3,R}}^{\fkp/3} \lesssim \| \rho_0^n\|_{L^p_x}^{\fkp/3} (w_{\bZ^n}(s,t) + w_{b,R}(s,t)^{\fkp/3} w_{\bZ^n}(s,t)^{1-\fkp/3} ).
    \end{align*}
    By lower-semicontinuity of norms in weak\mbox{-}$\ast$ topologies, passing to the limit it then holds
    \begin{align*}
        \| \rho^{\natural}_{st}\|_{\cF_{-3,R}}^{\fkp/3}
        \leq \liminf_{n\to\infty} \| \rho^{n,\natural}_{st}\|_{\cF_{-3,R}}^{\fkp/3}
        \lesssim \| \rho_0\|_{L^p_x}^{\fkp/3} (w_{\bZ}(s,t) + w_{b,R}(s,t)^{\fkp/3} w_{\bZ}(s,t)^{1-\fkp/3} )
    \end{align*}
    which finally by Remark \ref{rem:relation-control-variation} implies that $\rho^\natural\in C^{\fkp/3-\var}_2 \cF_{-3,R}$ for all $R\geq 1$.

    Next we consider the case $p=\infty$. The proof is almost identical, as one can still develop uniform estimates for $\| \rho^n\|_{\cB_b([0,T]; L^\infty_x)}$ and $\| \delta \rho^n_{s,t}\|_{\cF_{-1,R}}$ which are robust enough to apply compactness arguments (Proposition \ref{prop:compactness-lp}) and pass to the limit.
    The only differences, due to lack of separability and reflexivity of $L^\infty_x$, is that now $\rho^n_0\to \rho_0$ in $L^p_{\loc}$ for any $p<\infty$ and $\rho^n_0\xrightharpoonup{\ast} \rho_0$ in $L^\infty_x$, while $\rho^n\to \rho$ in $C_{w-\ast}([0,T];L^\infty_x)$; the rest of the proof is identical to before, since we can pass to the limit without problems whenever testing against compactly supported functions.

    Finally we deal with $p=1$, which is a bit more delicate. In this case, $\rho^n_0\to \rho_0$ in $L^1$,
    but we can only allow smooth approximations $b^n$ such that $\| \nabla\cdot b^n_t\|_{L^\infty_x} \leq \| \nabla\cdot b_t\|_{L^\infty_x}$ and
    \begin{equation}\label{eq:existence-linear-approximation-p1}
        \| b^n_t\|_{L^\infty(B_R)} \leq \| b_t\|_{L^\infty(B_R)} \ \forall\, R\geq 1,\quad
        b^n_t(x)\to b_t(x) \text{ for Lebesgue a.e. }  (t,x).
    \end{equation}
    We can still derive uniform bounds on $\| \rho^n\|_{L^\infty_t L^1_x}$ and equicontinuity estimates in $\cF_{-1,R}$ as before; in order to get weak compactness, we aim to apply Corollary \ref{cor:compactness-l1-loc} from Appendix \ref{app:compactness}, which requires to verify \emph{local equi-integrability} of $\rho^n$.
    To this end, fix $\eps>0$; since $\rho^n_0\to \rho_0$ in $L^1_x$, there exists $\delta>0$ such that
    \begin{align*}
        \mathscr{L}^d(A)\leq \delta\quad \Rightarrow \quad \sup_n \int_A |\rho^n_0(x)| \dd x\leq \eps.
    \end{align*}
    Now let us set $\tilde\delta := \delta \exp(-\| \nabla\cdot b\|_{L^1_t L^\infty_x})$; observe that for any Borel set $\tilde A$ with $\mathscr{L}^d(\tilde A)\leq\tilde \delta$, by Corollary \ref{cor:flow-incompressibility} and our choice of approximations, for any $t\in [0,T]$ it holds
    \begin{align*}
        \mathscr{L}^d(\Phi^n_{0\leftarrow t}(\tilde A))
        \leq \exp\Big(\int_0^t \| \nabla\cdot b^n_s \|_{L^\infty_x} \dd s\Big)\, \mathscr{L}^d(\tilde A)  \leq \delta;
    \end{align*}
    therefore by the explicit solution formula \eqref{eq:existence-linear-proof-1}, we obtain
    \begin{align*}
        \mathscr{L}^d(\tilde A)\leq \tilde \delta
        \quad \Rightarrow\quad
        \sup_{t\in [0,T]} \int_{\tilde A} |\rho^n_t(x)| \dd x
        = \sup_{t\in [0,T]} \int_{\Phi^n_{0\leftarrow t} (\tilde A)} |\rho^n_0(x)| \dd x\leq \eps
    \end{align*}
    which proves equi-integrability. 
    We can now apply Corollary \ref{cor:compactness-l1-loc} to find a (not relabelled) subsequence such that $\rho^n \to \rho$ in $C_w([0,T];L^1_\loc)$, which in particular implies that $\rho^n_t\psi\rightharpoonup \rho_t\psi$  weakly in $L^1_x$ for all $t\in [0,T]$ and $\psi \in C_c^{\infty}$;
    combining this fact with the convergence \eqref{eq:existence-linear-approximation-p1} and more refined weak-strong convergence results (for instance Egorov's theorem), one can still show that
    \begin{align*}
        \int_0^\cdot \nabla\cdot( \rho^n_s\, b^n_s)\, \dd s \to \int_0^\cdot \nabla\cdot( \rho_s\, b_s)\,\dd s\ \text{ in }  C_{w-\ast}([0,T];\cF_{-2,R})
    \end{align*}
    as well as in fact the stronger estimate \eqref{eq:existence-linear-proof-4}.
    From here, one can pass to the limit as before to find the conclusion.
\end{proof}

\begin{remark}\label{rem:existence-linear RPDE}
    Let us discuss some variants and extensions of Proposition \ref{prop:existence-linear-RPDE}, under the same assumptions on $\bZ$, $\xi$, $b$.
    If $\rho_0 \in L^{p_1}_x \cap L^{p_2}_x$ for some $p_1\leq p_2$, then going through the same proof one can construct a global solution $\rho$ to \eqref{eq:rough-continuity} satisfying
     \begin{equation*}%\label{eq:lq-bound-2}
        \| \rho_t\|_{L^q_x}
        \leq \exp\bigg( \Big(1-\frac{1}{q} \Big) \int_0^t \| \nabla\cdot b_u\|_{L^\infty_x}\, \dd u\bigg) \| \rho_0\|_{L^q_x} \quad \forall\, t \geq 0, \, q\in [p_1,p_2].
    \end{equation*}
    Similarly, for any $p\in [1,\infty]$ and any $\tilde\rho_T\in L^p_x$, one can construct solutions $\tilde\rho$ with terminal condition $\tilde\rho\vert_{t=T}=\tilde\rho_T$ such that
    \begin{align*}
        \| \tilde\rho_t\|_{L^p_x}\leq \exp\bigg( \Big(1-\frac{1}{p} \Big) \int_t^T\| \nabla\cdot b_u\|_{L^\infty_x}\, \dd u\bigg) \| \rho_T\|_{L^p_x} \quad \forall\, t\in [0,T].
    \end{align*}
    Similar considerations apply to the stochastic transport equation
    \begin{equation}\label{eq:rough_transport}
        \dd f_t + b_t\cdot\nabla f_t\, \dd t + \sum_k \xi_k\cdot\nabla f_t\, \dd \bZ_t = 0;
    \end{equation}
    one can define solutions to \eqref{eq:rough_transport} similarly to Definition \ref{defn:solution-rough-continuity}, up to requiring instead that $b\cdot\nabla f\in L^1_t \cF_{-2,R}$.
    This condition is satisfied whenever $b f\in L^1_t L^1_\loc$, $f\in \cB_b([0,T];L^1_\loc)$ and $\nabla\cdot b\in L^1_t L^\infty_x$: indeed by duality, similarly to \eqref{eq:existence-linear-preliminary}, one has
    \begin{equation}\label{eq:transport-preliminary}
        \int_0^T \| b_t\cdot\nabla f_t\|_{\cF_{-1,R}} \dd t
        \leq \int_0^T \| b_t f_t\|_{L^1(B_R)} \dd t + \sup_{t\in [0,T]} \| f_t\|_{L^1(B_R)} \int_0^T \| \nabla\cdot b\|_{L^\infty_x} \dd t \quad \forall\, R\geq 0.
    \end{equation}
    In the case of \eqref{eq:rough_transport}, for any $p\in [1,\infty]$ and any initial $f_0\in L^p_x$ (respectively terminal $\tilde f_T\in L^p_x$), under the same assumptions as in Proposition \ref{prop:existence-linear-RPDE} one can similarly construct a weak solution $f$ (resp. $\tilde f$) satisfying
    \begin{equation*}%\label{eq:lp-bound-transport}
        \| f_t\|_{L^p_x} \leq \exp\bigg( \frac{1}{p} \int_0^t\| \nabla\cdot b_u\|_{L^\infty_x}\, \dd u\bigg) \| f_0\|_{L^p_x},\quad
        \| \tilde f_t\|_{L^p_x} \leq \exp\bigg( \frac{1}{p} \int_t^T\| \nabla\cdot b_u\|_{L^\infty_x}\, \dd u\bigg) \| \tilde f_T\|_{L^p_x}.
    \end{equation*}    
\end{remark}

\begin{remark} \label{rmk:weak compactness}
    For future convenience, let us collect here the necessary ingredients from the proof of Proposition \ref{prop:existence-linear-RPDE} guaranteeing compactness of $\{\rho^n\}_n$ in $C_w([0,T];L^1_\loc)$ (cf. Proposition \ref{prop:compactness-l1})  when $p=1$.
    To achieve the necessary a priori estimates, all one needs is: 
    \begin{itemize}
        \item $\sup_n \| \nabla \cdot b^n\|_{L^1_tL^{\infty}_x} < \infty$;
        \item $\sup_n \|  b^n\|_{L^1_tL^{\infty}_x(B_R)} < \infty$ for every $R \geq 1$;
        \item equi-integrability on $\rho_0^n$.
    \end{itemize}
    In particular, this part of the proof would already work under the assumption that $\rho^n_0\rightharpoonup \rho_0$ weakly in $L^1_x$. 
    Furthermore, to show convergence of the drifts $\mu^n$, in addition to the above conditions, one only needs to require $b_t^n(x) \rightarrow b_t(x)$ for Lebesgue a.e. $(t,x)$ (see \eqref{eq:existence-linear-approximation-p1}).

    Similar arguments apply to the transport RPDE \eqref{eq:rough_transport} as well; the only major difference is that in this case, in order to show convergence (in the sense of distributions) of the drifts $\tilde \mu^n=\int_0^\cdot b^n_s\cdot\nabla f^n_s\, \dd s$, one needs to additionally require $\nabla\cdot b_t^n(x) \rightarrow \nabla\cdot b_t(x)$ for Lebesgue a.e. $(t,x)$.
\end{remark}

Proposition \ref{prop:existence-linear-RPDE} only requires local integrability conditions on $b$; under an additional growth assumption, see \eqref{eq:growth-condition-pde} below, one can additionally prove uniform $p$\mbox{-}integrability of solutions.
Condition \eqref{eq:growth-condition-pde} first appeared in the DiPerna--Lions theory, where is usually exploited to set up a Gr\"onwall type argument at the level of uniqueness, cf. \cite[Theorem II.1]{diperna1989ordinary}.
Here instead we employ it at the level of a priori estimates, which are valid also in situations where uniqueness is not known, with an approach reminiscent of \cite[Section 3]{CriDeL2008}.

\begin{lemma}\label{lem:equi-p-integrability}
    Let $b,\,\xi,\, \Z$ and $\rho_0$ be as in Proposition \ref{prop:existence-linear-RPDE}, for some $p\in [1,\infty)$.
    Additionally assume that
    \begin{equation}\label{eq:growth-condition-pde}
        \frac{b(x)}{1+|x|} \in L^1_t L^1_x+ L^1_t L^\infty_x.
    \end{equation}
    Then the solution $\rho$ from Proposition \ref{prop:existence-linear-RPDE} can be further constructed so that it is \emph{uniformly $p$\mbox{-}integrable} on any finite interval $[0,T]$.
    In particular it holds
    \begin{equation}\label{eq:equi-p-integrability}
        \lim_{R\to\infty} \sup_{t\in [0,T]} \int_{|x|>R} |\rho_t(x)|^p \dd x = 0.
    \end{equation}
\end{lemma}

\begin{proof}
    We give the proof for $p=1$, the other cases being similar.
    For simplicity, we assume everything smooth and derive estimates which are independent of the smoothness of $(b,\xi,\rho_0)$; the conclusion then follows by passing to the limit in the smooth approximations, as done in Proposition \ref{prop:existence-linear-RPDE}, and using properties of weak convergence in $L^1$, guaranteeing that the bound \eqref{eq:equi-p-integrability} still holds after the limit.
    
    By formula \eqref{eq:existence-linear-proof-1}, it holds
    \begin{equation}\label{eq:equi-p-integrability-proof0}
        \sup_{t\in [0,T]} \int_{|x|>R} |\rho_t(x)| \dd x = \sup_{t\in [0,T]} \int_{|\Phi_t(x)|>R} |\rho_0(x)| \dd x.
    \end{equation}
    We need to derive some bounds on $|\Phi_t(x)|$, for $x\in\R^d$.
    By assumption \eqref{eq:growth-condition-pde}, we can write $|b_t(x)|\leq (1+|x|)(h_t + g_t(x))$ for some $g\in L^1_t L^1_x$ and $h \in L^1_t$.
    Since $\Phi_t(x)$ solves an RDE, we can apply estimate \eqref{eq:apriori-RDE-forcing-1} from Lemma \ref{lem:apriori-RDE-forcing} for $\mu_t= \int_0^t b_s(\Phi_s(x)) \dd s$;
    in particular, we can find a constant $\kappa=\kappa(\|\xi\|_{C^2_b})>0$ such that for any $t\in [0,T]$ it holds
    \begin{align*}
    1 + |\Phi_t(x)|
    & \leq 1 + |x| + \kappa + \kappa \int_0^t |b_s(\Phi_s(x))| \dd s\\
    & \leq  1 + |x| + \kappa + \kappa \int_0^t (h_s + g_s(\Phi_s(x))) (1 + |\Phi_s(x)|) \dd s
    \end{align*}
    An application of Gr\"onwall's lemma then yields
    \begin{equation}\label{eq:equi-p-integrability-proof}\begin{split}
        \log\Big(1 + \sup_{t\in [0,T]} |\Phi_t(x)|\Big)
        & \leq \log(1 + \kappa + |x| ) + \kappa \| h\|_{L^1_t}  + \int_0^T \kappa g_s(\Phi_s(x)) \dd s\\
        & =: \log(1 +\kappa + |x|) +  \kappa \| h\|_{L^1_t}  + G(x).
    \end{split}\end{equation}
    Moreover, by Corollary \ref{cor:flow-incompressibility}, it holds
    \begin{equation}\label{eq:equi-p-integrability-proof2}
        \int_{\R^d} G(x) \dd x = \kappa \int_0^T \int_{\R^d} g_s(\Phi_s(x)) \dd x \dd s
        \leq \kappa \exp(\| \nabla\cdot b\|_{L^1_t L^\infty_x}) \| g\|_{L^1_t L^\infty_x}.
    \end{equation}
    We now go back to the original quantity we want to estimate; let us conveniently choose $R=\exp(1+3\tilde R)$ for $\tilde R$ large enough so that $\tilde R\geq \kappa \| h\|_{L^1_t}$.
    Estimate \eqref{eq:equi-p-integrability-proof} informs us that, whenever $\Phi_t(x)>R$, at least one between $\log(1+\kappa + |x|)$ and $G(x)$ must be larger than $\tilde R$. Therefore, for any fixed $K>0$, we find
    \begin{align*}
        \int_{|\Phi_t(x)|>R} |\rho_0(x)| \dd x
        & \leq  \int_{|\rho_0(x)| > K} |\rho_0(x)| \dd x + \int_{|\Phi_t(x)|>R,\, |\rho_0(x)| \leq K} |\rho_0(x)| \dd x\\
        & \leq \int_{|\rho_0(x)| > K} |\rho_0(x)| \dd x + \int_{\log(1+\kappa + |x|) > \tilde R} |\rho_0(x)| \dd x 
        + K \int_{|G(x)|>R} 1 \dd x\\
        & \leq \int_{|\rho_0(x)| > K} |\rho_0(x)| \dd x + \int_{\log(1+\kappa+|x|) > \tilde R} |\rho_0(x)| \dd x + \frac{K}{\tilde R} \| G\|_{L^1_x};
    \end{align*}
    combined with \eqref{eq:equi-p-integrability-proof0} and \eqref{eq:equi-p-integrability-proof2}, this yields the quantitative estimate
    \begin{equation}\begin{split}\label{eq:equi-p-integrability-proof3}
        & \sup_{t\in [0,T]} \int_{|x|>e^{1+3\tilde R}} |\rho_0(x)| \dd x\\
        & \quad \leq \int_{|\rho_0(x)| > K} |\rho_0(x)| \dd x + \int_{|x|> e^{\tilde R}-1-\kappa} |\rho_0(x)| \dd x + \frac{\kappa K}{\tilde R} e^{\| \nabla\cdot b\|_{L^1_t L^\infty_x}} \| g\|_{L^1_t L^\infty_x}, \quad \forall\, \tilde R\geq \kappa \|h\|_{L^1_x}.
        % &=: I_1 + I_2 + I_3
    \end{split}\end{equation}
    This is a quantitative bound, which is stable w.r.t. weak convergence; therefore even though we performed all computations in the smooth case, after passing to the limit in the approximations, estimate \eqref{eq:equi-p-integrability-proof3} remains true for the weak solution $\rho$ constructed in Proposition \ref{prop:existence-linear-RPDE}.
    
    Since $\rho_0\in L^1_x$, we can first take $\lim\sup_{\tilde R\to\infty}$ on both sides to see that the second and third term on the RHS of \eqref{eq:equi-p-integrability-proof3} vanish, and subsequently take $K\to\infty$ (employing again integrability of $\rho_0$) to conclude that \eqref{eq:equi-p-integrability} holds (for $p=1$).
\end{proof}

Notice that, by Lemma \ref{lem:equi-p-integrability}, the solutions $\rho^n$ associated to smooth approximations constructed in Proposition \ref{prop:existence-linear-RPDE} are always uniformly-$p$-integrable, in the sense that $\{|\rho^n_t|^p; t\in [0,T], n\in\N \}$ is a uniformly integrable family (i.e. equi-integrable and tight) whenever $\{|\rho^n_0|^p; n\in\N \}$ is so.

Under suitable integrability requirements, we can additionally show that $L^1_x$-valued solutions to the rough continuity equation preserve mass, as formally expected by integration by parts. Combined with the upcoming product formula from Section \ref{subsec:linear-uniqueness}, Lemma \ref{lem:conservation-mass} below is the key to establish uniqueness results by a duality argument, see the proof of forthcoming Theorem \ref{thm:uniqueness-linear-RPDE}.
By no coincidence, the proof of Lemma \ref{lem:conservation-mass} is similar to that of \cite[Theorem II.1]{diperna1989ordinary}, although here no Sobolev regularity is required.
Note that in the next statement, very few conditions are imposed on $b$, in particular we do not enforce neither \eqref{eq:existence-linear-assumption} nor \eqref{eq:growth-condition-pde}.

\begin{lemma} \label{lem:conservation-mass}
    Let $\fkp\in [2,3)$, $Z\in \clC^{\fkp}_g$, $\xi\in C^2_b$, $\nabla\cdot \xi=0$; let $\rho$ be a solution to \eqref{eq:rough-continuity}, in the sense of Definition \ref{defn:solution-rough-continuity}, additionally satisfying
    \begin{equation}\label{eq:hypothesis-conservation-mass}
        \rho\in \cB_b([0,T]; L^1_x), \quad \frac{b\,\rho}{1+|x|} \in L^1_t L^1_x.
    \end{equation}
    Then it holds $\langle \rho_t, 1\rangle = \langle \rho_0,1\rangle$ for all $t\in [0,T]$.
\end{lemma}

\begin{proof} 
    Consider a function $\varphi\in C^\infty_c$ such that $\varphi(x)=1$ for $|x|\leq 1/2$ and $\varphi(x)=0$ for $|x|\geq 1$ and set $\varphi^R(x):= \varphi(x/R)$, so that $\varphi^R\in \cF_{3,R}$ for all $R$.
    Our aim is to derive uniform-in-$R$ estimates for $\langle \delta \rho_{st},\varphi^R\rangle$, so to pass to the limit as $R\to\infty$.

    Let us define a control by $w_\rho(s,t):= \int_s^t \| (b\rho)_r /(1+|x|) \|_{L^1_x} \dd r$; in the rest of the proof, for notational simplicity, we will allow hidden constants to depend on $w_{\bA}$, $w_\rho$ and $\|\rho\|_{L^\infty_t L^1_x}$ whenever needed.
    By assumption \eqref{eq:hypothesis-conservation-mass} and the definition of $\cF_{l,R}$, for any $R\geq 1$ it holds
    \begin{align*}
        \| \delta \mu_{st}\|_{\cF_{-1,R}} 
        \leq \int_s^t \| \nabla\cdot (b_r \rho_r)\|_{\cF_{-1,R}} \dd r
        \leq \int_s^t \int_{B_R} |(b_r \rho_r)(x)| \dd x\, \dd r
        \lesssim R w_\rho(s,t);
    \end{align*}
    moreover $\| \rho_t\|_{\cF_{-0,R}} \leq \| \rho_t\|_{L^1_x}$ uniformly in $R$, therefore we are in the position to apply Lemma \ref{lem:apriori-unbounded}.
    In particular, upon defining a new control $\tilde w := w_{\bA} + w_{\rho}^{\fkp/3} w_{\bA}^{1-\fkp/3} + w_\rho$, for all $R$ large enough estimates \eqref{eq:apriori-unbounded-eq2}-\eqref{eq:apriori-unbounded-eq3} correspond to
    \begin{equation}\label{eq:mass-conservation-eq1}
        \| \rho^\sharp_{st}\|_{\cF_{-2,R}} \lesssim R \tilde w^{2/\fkp}, \quad
        \| \delta \rho_{st}\|_{\cF_{-1,R}} \lesssim R^{1/2} \tilde w^{1/\fkp}.
    \end{equation}
    In order to estimate $\langle \varphi^R, \rho^\natural_{st}\rangle$, we are going to employ Lemma \ref{lem:sewing}. To this end, first observe that by applying Definition \ref{defn:solution-rough-PDE} and Chen's relation \eqref{eq:chen-unbounded-drivers}, it holds
    \begin{equation}\label{eq:mass-conservation-eq2}
        \delta\rho^\natural_{sut} = A_{ut} \rho^\sharp_{su} + \bbA_{ut} \delta \rho_{su}
    \end{equation}
    for $\rho^\sharp_{su}=\delta\rho_{st}+A_{st}\rho_s$ as in Lemma \ref{lem:apriori-unbounded}.
    Next observe that, by definition \eqref{eq:defn-unbounded-continuity}, the operators $A_{st}^\ast$ and $\bbA_{st}^\ast$ only involve first and second order derivative, thus by scaling it holds
    \begin{equation}\label{eq:mass-conservation-eq3}
        \| A^\ast_{st} \varphi^R\|_{W^{2,\infty}} \lesssim R^{-1} w_{\bA}^{1/\fkp}(s,t), \quad
        \| \bbA^\ast_{st} \varphi^R\|_{W^{1,\infty}} \lesssim R^{-1} w_{\bA}^{2/\fkp}(s,t).
    \end{equation}
    Testing identity \eqref{eq:mass-conservation-eq2} against $\varphi^R$ and using the estimates \eqref{eq:mass-conservation-eq1}-\eqref{eq:mass-conservation-eq3}, we find
    \begin{align*}
        |\langle \delta \rho^\natural_{sut},\varphi^R\rangle|
        \lesssim \| A^\ast_{ut} \varphi^R\|_{W^{2,\infty}} \| \rho^\sharp_{su}\|_{\cF_{-2,R}} + \| \bbA_{ut}^\ast \varphi^R\|_{W^{1,\infty}} \| \delta\rho_{su}\|_{\cF_{-1,R}}
        \lesssim \tilde w(s,t)^{3/\fkp}.
    \end{align*}
    By Lemma \ref{lem:sewing}, we can conclude that for all $R$ large enough it holds
    \begin{equation}\label{eq:mass-conservation-eq4}
        |\langle \rho^\natural_{st}, \varphi^R\rangle| \lesssim \tilde w(s,t)^{3/\fkp} \quad \forall (s,t)\in \Delta_T.
    \end{equation}
    Next we claim that, for $\mu=\nabla\cdot( b\rho)$, under assumption \eqref{eq:hypothesis-conservation-mass} we have
    \begin{equation}\label{eq:mass-conservation-eq5}
        \lim_{R\to\infty} \langle \delta\mu_{st}, \varphi^R\rangle = 0\quad  \forall (s,t)\in \Delta_T.
    \end{equation}
    Indeed, for any fixed $r$ by the definition of $\varphi^R$ it holds
    \begin{align*}
        |\langle b_r\, \rho_r, \nabla\varphi^R\rangle|
        \leq \int_{\R^d} \frac{|(b_r \rho_r)(x)|}{R} \Big|\nabla\varphi\Big(\frac{x}{R}\Big)\Big| \dd x
        \lesssim \| \nabla \varphi\|_{L^\infty} \int_{R/2\leq |x|\leq R} \frac{|(b_r \rho_r)(x)|}{1 + |x|} \dd x;
    \end{align*}
    by assumption, for a.e. $r\in [0,T]$, $b_r\rho_r/(1+|x|)\in L^1_x$ and it is integrated on the domain $B_R\setminus B_{R/2}$ which escapes at $\infty$ as $R\to\infty$; by dominated convergence, it follows that
    \begin{align*}
        \lim_{R\to\infty} \int_0^T |\langle b_r\, \rho_r, \nabla\varphi^R\rangle| \dd r = 0
    \end{align*}
    thus proving the claim \eqref{eq:mass-conservation-eq5}.
    We now have all the ingredients to conclude. Indeed, by Definition \ref{defn:solution-rough-continuity}, testing $\delta\rho$ against $\varphi^R$, we have
    \begin{align*}
        |\langle \delta \rho_{s,t},\varphi^R\rangle|
        \leq |\langle \delta \mu_{s,t},\varphi^R\rangle| + |\langle \rho_s, A^\ast_{st} \varphi^R\rangle| + |\langle \rho_s,\bbA_{st}^\ast \varphi^R\rangle|  + |\langle \rho^\natural_{st}, \varphi^R\rangle|;
    \end{align*}
    applying the assumption \eqref{eq:hypothesis-conservation-mass} and the estimates \eqref{eq:mass-conservation-eq3}, \eqref{eq:mass-conservation-eq4}, \eqref{eq:mass-conservation-eq5}, we find
    \begin{align*}
        |\langle \delta \rho_{s,t},1\rangle|
        = \lim_{R\to\infty} |\langle \delta \rho_{s,t},\varphi^R\rangle|
        \lesssim \tilde w(s,t)^{3/\fkp}.
    \end{align*}
    Namely, $t\mapsto \langle \rho_t,1\rangle$ is of finite $\fkp/3$-variation, with $\fkp/3<1$, thus necessarily constant.
    \end{proof}

\begin{remark}
    In the simplest scenario, condition \eqref{eq:hypothesis-conservation-mass} is satisfied when $\rho\in \cB_b([0,T];L^1_x \cap L^\infty_x)$ and $b$ satisfies the growth assumption \eqref{eq:growth-condition-pde}.
\end{remark}

\subsection{Product formula for Sobolev drifts}\label{subsec:linear-uniqueness}
As in Section \ref{subsec:linear-existence}, in the following we always assume the vector fields $\xi_k$ to be divergence free.
The goal of this section is to show a {\em product formula} between the rough continuity equation
\begin{equation}\label{eq:rc_in_unique}\tag{RCE}
\rmd \rho_t + \nabla \cdot (b_t \rho_t)\, \rmd t  + \sum_k \nabla\cdot (\xi_k \rho_t)\, \rmd \bZ_t^k = 0
\end{equation}
and the rough transport equation
\begin{equation}\label{eq:lt_in_unique}\tag{RTE}
\rmd  f_t + b_t \cdot \nabla  f_t  \rmd t +  \sum_k \xi_k \cdot \nabla f_t\,  \rmd \bZ_t^k = 0.
\end{equation}
To explain what we mean, assume everything to be smooth for the moment; then we have 
\begin{align*}
\rmd   (\rho_t f_t)  & =  f_t \rmd  \rho_t + \rho_t \rmd  f_t \\
&= -  f_t \Big(\nabla \cdot (b_t \rho_t)\rmd t   + \sum_k \nabla\cdot (\xi_k \rho_t) \rmd \bZ_t^k \Big)  - \rho_t \Big(  b_t \nabla \cdot f_t \rmd t   +  \sum_k \xi_k \nabla \cdot f_t  \rmd \bZ_t^k \Big) \\
 & = - \nabla \cdot( \rho_t f_t  b_t ) \rmd t - \sum_k \nabla\cdot (\xi_k \rho_t  f_t   ) \rmd \bZ_t^k.
\end{align*}
In other words,
%if $\rho$ solves \eqref{eq:rc_in_unique} and $f$ solves \eqref{eq:lt_in_unique}, then
the product $\rho f$ still solves \eqref{eq:rc_in_unique}.
Integrating the RPDE for $\rho f$ in space, using Lemma \ref{lem:conservation-mass}, we expect to find %the {\em duality relation}
$$
\rmd \bigg( \int_{\R^d} \rho_t(x)f_t(x)\rmd x\bigg) =0 \quad \Rightarrow \quad \langle \rho_t, f_t \rangle = \langle \rho_0, f_0 \rangle \quad \forall\, t\geq 0.
$$

The main goal of this section is to formalize the above heuristics in the rough setting, under the suitable regularity assumptions.

\begin{theorem}[Product formula and duality]\label{thm:product-rule}
    Let $\fkp\in [2,3)$, $\bZ\in \clC^{\fkp}_g$, $\xi\in C^3_b$ with $\nabla\cdot \xi =0$.
    Let $p,q,r\in [1,\infty]$ be parameters such that $1/p+1/q+1/r=1$ and assume that
    \begin{equation}\label{eq:product_formula_assumption}
        b\in L^1_t W^{1,r}_\loc, \quad
        \rho\in \cB_b([0,T];L^p_\loc), \quad f\in \cB_b([0,T];L^q_\loc),
    \end{equation}
    where $\rho$ and $f$ are respectively solutions to \eqref{eq:rc_in_unique} and \eqref{eq:lt_in_unique};
    %then the product $\rho f\in \cB_b([0,T];L^{r'}_x)$ is a solution to \eqref{eq:rc_in_unique}.
    then the product $\rho f\in \cB_b([0,T];L^{r'}_\loc)$ is a solution to \eqref{eq:rc_in_unique}.
    If additionally
    \begin{equation}\label{eq:duality_formula_assumption}
        \rho f\in \cB_b([0,T]; L^1_x), \quad \frac{b \rho f}{1+|x|}\in L^1_t L^1_x,
    \end{equation}
    then we have the \emph{duality formula}
    \begin{equation}\label{eq:duality formula}
        \langle \rho_t, f_t\rangle=\langle \rho_0,f_0\rangle \qquad \forall\, t\in [0.T].
    \end{equation}
\end{theorem}

As in \cite{BaiGub2017}, in order to prove Theorem \ref{thm:product-rule}, we employ a doubling of variables procedure; we start by deriving the RPDE satisfied by the tensor product $(\rho \otimes f)(x,y) : = \rho(x) f(y)$.
More generally, given $\mu,\nu\in \mathcal{D}'(\R^d)$, we denote by $\mu\otimes \nu$ their distributional tensor, which is again a distribution in doubled variables, namely it belongs to $\cD'(\R^{2d})$.
We will adopt the same tensor-notation for operators: if $B_i\in \cL(E_i,F_i)$ for $i=1,2$, then we denote by $B_1 \otimes B_2$ the element of $\cL(E_1\otimes E_2;F_1\otimes F_2)$ which is the unique linear extension of the mapping 
\begin{align*}
    (B_1 \otimes B_2) (e_1 \otimes e_2) := (B_1 e_1) \otimes (B_2 e_2)\quad \forall\, e_1\in E_1,\, e_2\in E_2.
\end{align*}
Given any Banach space, we denote by $I$ the identity operator on it.

With this preparation, we can present a tensorization statement, which is the analogue of \cite[Proposition 5.1]{hocquet2018energy}.
Here, $\mathcal{E}_{l,R}=\cE_{l,R}(\R^{2d})$ denote the spaces introduced in Lemma \ref{lem:smoothing-examples}; by Remark \ref{rem:smoothing-examples}, they are subspaces of $\cF_{l,\sqrt 2 R}(\R^{2d})$.

\begin{proposition}\label{prop:tensorization}
Let $\rho$, $f$ be respectively solutions to \eqref{eq:rc_in_unique} and \eqref{eq:lt_in_unique}.
Let $p,q,r\in [1,\infty]$ be parameters such that $1/p+1/q+1/r=1$ and assume that
\begin{align*}
    \rho \in \cB_b([0,T];L^p_\loc), \quad
    f \in \cB_b([0,T]; L^q_\loc), \quad
    b \in L^1_t L^r_\loc, \quad \nabla\cdot b \in L^1_t L^r_\loc.
\end{align*}
Then $\rho\otimes f\in \cB_b([0,T];L^1_\loc(\R^{2d}))$ is an unbounded rough driver solution of 
\begin{equation} \label{eq:tensored equation}
    \rmd (\rho_t \otimes f_t) + \big( \nabla \cdot ( b_t \rho_t) \otimes f_t + \rho_t \otimes (b_t \cdot \nabla f_t) \big) \rmd t + \big( (\xi_k \cdot \nabla  \rho_t) \otimes f_t + \rho_t \otimes (\xi_k \cdot \nabla f_t) \big) \rmd \bZ^k_t = 0
\end{equation}
in the scale of spaces $\clE_{l,R}$ defined by \eqref{eq:E_k,R}, for every $R\geq 1$. That is to say, we have
\begin{equation*}%\label{eq:tensored equation differential form}
    \rmd (\rho \otimes f)_t  + M(\rmd t) + \mathbf{X}(\rmd t) (\rho \otimes f)_t = 0
\end{equation*}
for the unbounded rough driver $\bX$ given by the \emph{second quantization} of $\bA$, namely
\begin{equation} \label{eq:second quantization}
    \bX_{st} = \bX(\bA)_{st} = (X_{st},\mathbb{X}_{st})
    := (A_{st} \otimes I + I \otimes A_{st},  \mathbb{A}_{st} \otimes I + I \otimes \mathbb{A}_{st} + A_{st} \otimes A_{st} ),
\end{equation}
and the forcing $M$ given by
\begin{equation} \label{eq:second quantization drift}
    M_t := \int_0^t \big( [\nabla \cdot ( b_u \rho_u)] \otimes f_u + \rho_u \otimes (b_u \cdot \nabla f_u) \big) \rmd u.
\end{equation}
Moreover $M$ belongs to $C^{1-\textrm{var}} \clE_{-1,R}$ for every $R \geq 1$.
\end{proposition}

Results in the style of Proposition \ref{prop:tensorization} appear in several works across the unbounded rough drivers literature,
%see e.g. \cite[Proposition 5.1]{hocquet2018energy},
albeit on different scales of spaces from the one we are using in the present paper. Thus, we include a self-contained proof which relies on the following lemma, whose proof is postponed to Appendix \ref{app:smoothing}.

\begin{lemma}\label{lem:tensor-product}
    Let $R\geq 1$, $l$, $j\in\N$. Then the distributional tensor product $(f,g)\mapsto f\otimes g$ is a bounded bilinear map from $\cF_{-l,R+1}\times \cF_{-j,R+1}$ to $\clE_{-l-j,R}$.
\end{lemma}

\begin{proof}[Proof of Proposition \ref{prop:tensorization}]
    Let $R\geq 1$ fixed. As in Lemma \ref{lem:apriori-unbounded}, we introduce the notations
    \begin{align*}
        \mu_t := \int_0^t \nabla \cdot (\rho_u b_u) \rmd u, \quad
        \rho^\sharp_{st}:= \delta\rho_{st}+A_{st}\rho_s,\quad
        \nu_t := \int_0^t b_u  \cdot \nabla f_u \rmd u, \quad
        f^\sharp_{st}:= \delta f_{st}+A_{st}f_s.
    \end{align*}
    so that 
    \begin{align*}
        \delta \rho_{st} & = - \delta \mu_{st}  - A_{st} \rho_s + \mathbb{A}_{st} \rho_s + \rho_{st}^{\natural}, \qquad
        \rho^\sharp_{st} = - \delta \mu_{st}  + \mathbb{A}_{st} \rho_s + \rho_{st}^{\natural}\\
        \delta f_{st} & = - \delta \nu_{st} - A_{st} f_s + \mathbb{A}_{st} f_s + f_{st}^{\natural}, \qquad
        f^\sharp_{st} = - \delta \mu_{st}  + \mathbb{A}_{st} f_s + f_{st}^{\natural}.
    \end{align*}
    Note that by assumption, $\rho,f\in \cB_b([0,T];\cF_{-0,R+1})$; moreover arguing as in \eqref{eq:existence-linear-preliminary}-\eqref{eq:transport-preliminary}, it holds $\mu,\nu \in C^{1-\var} \cF_{-1,R+1}$.
    Now we write
    \begin{align}
        \delta (\rho \otimes f)_{st}
        & = \delta \rho_{st} \otimes  f_{s}  + \rho_{s} \otimes \delta f_{st}  + \delta \rho_{st} \otimes \delta f_{st}  \notag \\
        & = -\delta \mu_{st} \otimes f_s  - A_{st} \rho_s \otimes f_s + \mathbb{A}_{st} \rho_s \otimes f_s + \rho_{st}^{\natural} \otimes f_s \notag \\
        &\quad  -\rho_s \otimes \delta \nu_{st}  - \rho_s \otimes A_{st} f_s  + \rho_s \otimes \mathbb{A}_{st} f_s + \rho_s \otimes f_{st}^{\natural} \notag  \\
        &\quad + \rho_{st}^{\sharp} \otimes \delta f_{st} + A_{st} \rho_s \otimes A_{st} f_s - A_{st} \rho_s \otimes f_{st}^{\sharp} \notag
    \end{align}
    which by the definition of $\bX$ yields
    \begin{equation}\label{eq:tensored equation expansion}
        \delta (\rho \otimes f)_{st}
        = -\int_s^t \left( \dot{\mu}_u \otimes f_u + \rho_u \otimes \dot{\nu}_u \right) \rmd u  - X_{st} (\rho_s \otimes f_s) + \mathbb{X}_{st} (\rho_s \otimes f_s) + (\rho \otimes f)_{st}^{\natural},
    \end{equation}
    where we define the remainder 
    \begin{align*}
        (\rho \otimes f)_{st}^{\natural}
        & := - \int_s^t \dot{\mu}_u \otimes \delta f_{su} \rmd u - \int_s^t \delta \rho_{su} \otimes \dot{\nu}_u \rmd u %\\ &
        + \rho_{st}^{\natural} \otimes f_s + \rho_s \otimes f_{st}^{\natural} + \rho_{st}^{\sharp} \otimes \delta f_{st} - A_{st} \rho_s \otimes f_{st}^{\sharp} \\
        %& \in \big(\cF_{-1,R} \otimes \cF_{-1,R} \big) \oplus \big(\cF_{-1,R} \otimes \cF_{-1,R} \big) \\
        %& \oplus \big(\cF_{-3,R} \otimes \cF_{-0,R} \big) \oplus \big( \cF_{-0,R} \otimes \cF_{-3,R} \big) \oplus \big( \cF_{-2,R} \otimes \cF_{-1,R} \big) \oplus \big( \cF_{-1,R} \otimes \cF_{-2,R} \big)
        &\, \in \big(\cF_{-1,R} \otimes \cF_{-1,R} \big) + \big(\cF_{-1,R} \otimes \cF_{-1,R} \big) 
        + \big(\cF_{-3,R} \otimes \cF_{-0,R} \big)\\
        & \qquad + \big( \cF_{-0,R} \otimes \cF_{-3,R} \big) + \big( \cF_{-2,R} \otimes \cF_{-1,R} \big) + \big( \cF_{-1,R} \otimes \cF_{-2,R} \big)
    \end{align*}
    Using the assumptions and Lemmas \ref{lem:apriori-unbounded} and \ref{lem:tensor-product}, we get that each term above has finite $\fkp/3$\mbox{-}variation in $\mathcal{E}_{-3,R}$.
    
    In order to conclude, it remains to show that $M$ defined by \eqref{eq:second quantization drift} belongs to $C^{1-\textrm{var}}\clE_{-1,R}$. Similarly to \eqref{eq:existence-linear-preliminary}-\eqref{eq:transport-preliminary}, we have
    \begin{align*}
        \| \nabla\cdot (b_t \rho_t)\|_{\cF_{-1,R+1}}
        &\leq \| b_t \rho_t\|_{L^1(B_{R+1})}
        \lesssim_R \| b_t\|_{L^r(B_{R+1})} \| \rho_t\|_{L^p(B_{R+1})},\\
        \| b_t\cdot\nabla f_t\|_{\cF_{-1,R+1}}
        & \leq \| b_t f_t\|_{L^1(B_{R+1})} + \| \nabla\cdot b_t f_t\|_{L^1(B_{R+1})}\\
        & \lesssim_R \big(\| b_t\|_{L^r(B_{R+1})} + \| \nabla \cdot b_t\|_{L^r(B_{R+1})}\big) \| f_t\|_{L^q(B_{R+1})}.
    \end{align*}
    Combined with Lemma \ref{lem:tensor-product} and the assumptions, this yields the estimate
    \begin{align*}
        \| M\|_{C^{1-\var} \cE_{-1,R}}
        & = \int_0^T \big\| [\nabla \cdot ( b_t \rho_t)] \otimes f_t + \rho_t \otimes (b_t \cdot \nabla f_t) \big\|_{\cE_{-1,R}} \rmd t\\
        & \lesssim_R \big(\| b\|_{L^1_t L^r(B_{R+1})} + \| \nabla \cdot b\|_{L^1_t L^r(B_{R+1})}\big) \| \rho\|_{\cB_b([0,T];L^p(B_{R+1}))} \| f \|_{\cB_b([0,T];L^q(B_{R+1}))}
    \end{align*}
    and thus the conclusion.
 \end{proof}

The next step is to derive an equation for the product $\rho_t(x) f_t(x)$, which formally amounts to testing \eqref{eq:tensored equation} against the distribution $\delta_{x=y}$.
To achieve this, we follow \cite{BaiGub2017} by employing a so-called \emph{blow-up transformation}: for $\eps\in (0,1)$ and $\Psi\in L^\infty(\R^{2d})$, define 
$$
T_{\eps} \Psi(x,y) : = \eps^{-d}\, \Psi\Big( x_+ + \frac{x_-}{\eps}, x_+ - \frac{x_-}{\eps}\Big),
$$
where as before $x_{\pm} = (x \pm y)/2$.
It's easy to see that $T_{\eps}$ is a linear isomorphism from $W^{l,\infty}(\R^{2d})$ to itself, for every $l\in\N$.

Recall the function $\zeta_R(x,y)=|x_+|^2/R^2+|x_-|^2$ from Lemma \ref{lem:smoothing-examples}; note that
\begin{align*}
    \zeta_R\Big( x_+ + \frac{x_-}{\eps}, x_+ - \frac{x_-}{\eps}\Big)
    = \frac{|x_+|^2}{R^2} + \frac{|x_-|^2}{\eps^2}
    \geq \zeta_R(x,y)\quad \forall\, \eps\in (0,1),
\end{align*}
so that $T_{\eps}$ leaves the spaces $\clE_{l,R}$ invariant.
$\tilde \cE_{l,\eps,R}:=T_\eps(\cE_{l,R})$ in general is strictly smaller that $\cE_{l,R}$, and can be similarly characterized as
\begin{align*}
        \tilde{\cE}_{l,\eps,R}= \left\{ \Phi \in W^{l,\infty}(\R^{2d}) \, : \frac{|x_+|^2}{R^2} + \frac{|x_-|^2}{\eps^2} \geq 1 \,\,\Rightarrow \,\,\Phi(x,y) = 0 \right\};
\end{align*}
given the above formula, it's immediate to check that $\tilde \cE_{l,\eps,R}$ is a closed subspace of $W^{l,\infty}(\R^{2d})$ (thus Banach with induced norm) and that local operators (like differential ones) map $\tilde{\cE}_{l,\eps,R}$ into $\tilde{\cE}_{k,\eps,R}$ for suitable $l$ and $k$, since they respect properties of supports.
Moreover, $T_\eps$ is a bijection from $\clE_{l,R}$ to $\tilde{\cE}_{l,\eps,R}$, thus admits an inverse $T^{-1}_\eps\in \cL(\cE_{l,R},\tilde{\cE}_{l,\eps,R})$.

The reason for introducing $T_\eps$, as we will see more in detail later (see eq. \eqref{eq:approx_dirac_delta}), is that at least formally
\begin{align*}
    \rho(x)f(x)
    = \langle \rho\otimes f, \delta_{x=y}\rangle
    = \lim_{\eps\to 0} \langle \rho\otimes f, T_\eps \Psi \rangle
    =  \lim_{\eps\to 0} \langle T_\eps^\ast (\rho\otimes f), \Psi \rangle
\end{align*}
for some appropriately chosen $\Psi$. In the last passage, we denoted by $T_{\eps}^{\ast}$ the dual operator of $T_\eps$, which belongs to $\cL(\cE_{-l,R},\cE_{-l,R})$ for every $l$ and $R\geq 1$. 
W.r.t. the $L^2(\bbR^{2d})$-pairing, whenever the argument is regular enough, $T_{\eps}^{\ast}$ admits the pointwise representation
$$T_{\eps}^*\Psi(x,y) = \Psi( x_+ + \eps x_-,x_+ -  \eps x_- ).$$
Consequently, by applying $T^*_{\eps}$ to \eqref{eq:tensored equation expansion}, our candidate approximation $T_{\eps}^* (\rho_t \otimes f_t)$ of the product $\rho_t(x) f_t(x)$ satisfies
\begin{equation} \label{eq:approx product}
    \rmd T_{\eps}^* (\rho_t \otimes f_t) + \rmd M_t^{\eps}  +  \mathbf{X}^{\eps}(\rmd t) T_{\eps}^{*}(\rho_t \otimes f_t)  = 0
\end{equation}
on the scales $(\mathcal{E}_{l,R})_l$, where we set
\begin{align}
    \bX^{\eps} & := (T_{\eps}^* X T_{\eps}^{*,-1}, T_{\eps}^* \mathbb{X} T_{\eps}^{*,-1}) := ( (T_\eps^{-1} X T_\eps)^\ast, (T_\eps^{-1} \bbX T_\eps)^\ast)\label{eq:approx of product URD}\\
    M_t^{\eps} & := T^\ast_\eps M_t
    = \int_0^t T_{\eps}^* \left(\nabla \cdot  (\rho_r b_r) \otimes f_r + \rho_r \otimes  (b_r \cdot \nabla f_r) \right) \rmd r \label{eq:approx of product drift}
\end{align}
for $(X,\mathbb{X})$, $M$ as defined in \eqref{eq:second quantization}-\eqref{eq:second quantization drift}.

Since the operators $X_{st}$ (resp. $\bbX_{st}$) consist of sums and compositions of differential operators, it's easy to check that they map $\tilde{\cE}_{l,\eps,R}$ into $\tilde{\cE}_{k,\eps,R}$ for suitable $l$ and $k$; consequently, $T_\eps^{-1} X_{st} T_\eps$ (resp. $T_\eps^{-1} \bbX_{st} T_\eps$) map $\cE_{l,R}$ into $\cE_{k,R}$.
Since $(\cE_{l,R})_l$ is a scale of spaces, arguing as in Lemma \ref{lem:URD_continuity}, we get the following.

\begin{lemma}
    Let $\xi\in C^2_b$, $\nabla\cdot\xi_k=0$, $\bZ\in \clC^{\fkp}_g$ for some $\fkp\in [2,3)$; let $\eps\in (0,1)$ and $\bX^{\eps}$ be defined by \eqref{eq:approx of product URD}.
    Then $\bX^\eps$ is an unbounded rough driver on the scales of spaces $\cE_{l,R}$ defined by \eqref{eq:E_k,R} and it holds
    \begin{equation*}%\label{eq:unbounded-continuity-basic}
        w_{\bX^\eps}(s,t)\lesssim \| \xi\|_{C^2_b}^\fkp\, w_{\bZ}(s,t)
    \end{equation*}
    where the hidden constant is independent of $R$, but may depend on $\eps$.
\end{lemma}

Our next goal, under additional regularity assumptions on $(b,\xi)$ is to obtain uniform-in-$\eps$ estimates for the unbounded rough driver $\mathbf{X}^{\eps}$ and the forcing $M^{\eps}$. 
%The following result guarantees that the unbounded rough drivers $(\bX^{\eps})_{\eps \in (0,1)}$ remain uniformly bounded in $\eps$.

\begin{proposition}[Proposition 3.4 from \cite{DGHT2019}]\label{proposition:renormalizable noise}
Let $\fkp\in [2,3)$, $\bZ\in \clC^{\fkp}_g$, $\xi\in C^3_b$ and $\nabla \cdot \xi = 0$; let $\mathbf{X}^{\eps}$ be defined by \eqref{eq:approx of product URD}. Then we have
\begin{equation*}\begin{split}
     \| X^{\eps}_{st} \|_{\mathcal{L}( \mathcal{E}_{l,R}, \mathcal{E}_{l-1,R})} 
     & \leq C w_{\bZ}(s,t)^{1/\fkp} \qquad \text{for } l \in \{-0,-2\},\\
     \| \mathbb{X}^{\eps}_{st} \|_{\mathcal{L}( \mathcal{E}_{l,R}, \mathcal{E}_{l-2,R})}
     & \leq C w_{\bZ}(s,t)^{2/\fkp} \qquad \text{for } l \in \{-0,-1\},
\end{split}\end{equation*}
where the constant $C$ depends on $\|\xi\|_{C^3_b}$ but is uniform in $\eps$ and $R \geq 1$. 
\end{proposition}

\begin{proposition} \label{proposition:renormalizable drift}
Let $\rho$, $f$ be respectively solutions to \eqref{eq:rc_in_unique} and \eqref{eq:lt_in_unique}.
Let $p,q,r\in [1,\infty]$ be parameters such that $1/p+1/q+1/r=1$; assume that
\begin{align*}
    \rho \in \cB_b([0,T];L^p_\loc), \quad
    f \in \cB_b([0,T]; L^q_\loc), \quad
    b \in L^1_t W^{1,r}_\loc
\end{align*}
and let $M^{\eps}$ be defined by \eqref{eq:approx of product drift}.
Then there exists a constant $C$, independent of $R$ and $\eps$, such that 
$$
\| \delta M_{st}^{\eps} \|_{\mathcal{E}_{-1,R}} \leq C \|\rho\|_{\cB_b([0,T]; L^p(B_{R+1}))} \|f\|_{\cB_b([0,T]; L^q(B_{R+1}))}\int_s^t \|b_u\|_{W^{1,r}(B_{R+1})} \rmd u.
$$
\end{proposition}

\begin{proof}
We begin by noticing that
\begin{align*}
    \nabla_x T_{\eps} & = \frac12 ( T_{\eps} \nabla_x  + T_{\eps} \nabla_y)  + \frac{1}{2 \eps} ( T_{\eps} \nabla_x  - T_{\eps} \nabla_y)  \\ 
    \nabla_y T_{\eps} & = \frac12 ( T_{\eps} \nabla_x  + T_{\eps} \nabla_y)  - \frac{1}{2 \eps} ( T_{\eps} \nabla_x  - T_{\eps} \nabla_y);
\end{align*}
therefore for any $\Psi\in \cE_{1,R}$, it holds that
\begin{align}
    - \langle \delta M^{\eps}_{st}, \Psi \rangle 
    & = - \int_s^t \langle \nabla_x \cdot  ((\rho_u b_u) \otimes f_u) + \nabla_y \cdot  (\rho_u \otimes (b_u f_u)) - \rho_u \otimes ((\nabla \cdot b_u) f_u) , T_{\eps} \Psi \rangle \rmd u \notag \\ %\label{eq:blow-up of drift0}
    & = \int_s^t \langle (\rho_u b_u) \otimes f_u, \nabla_x T_{\eps} \Psi \rangle  + \langle \rho_u  \otimes (b_u f_u), \nabla_y T_{\eps} \Psi \rangle \notag + \langle \rho_u  \otimes ((\nabla \cdot b_u) f_u), T_{\eps} \Psi \rangle \rmd u \\
    & = \int_s^t \frac12 \langle (\rho_u b_u )\otimes f_u + \rho_u  \otimes (b_u f_u), T_{\eps} \nabla_x \Psi + T_{\eps} \nabla_y \Psi \rangle \notag  \rmd u \\
    & \quad + \int_s^t \frac{1}{2 \eps} \langle (\rho_u b_u) \otimes f_u - \rho_u  \otimes (b_u f_u), T_{\eps} \nabla_x \Psi - T_{\eps} \nabla_y \Psi \rangle \notag \rmd u \\
    & \quad + \int_s^t \langle \rho_u  \otimes ((\nabla \cdot b_u) f_u),  T_{\eps}\Psi \rangle  \rmd u \notag \\
    & = \int_s^t\frac12 \langle T_{\eps}^* \big[(\rho_u b_u) \otimes f_u + \rho_u  \otimes (b_u f_u)\big],  \nabla_x \Psi +  \nabla_y \Psi \rangle \rmd u \label{eq:blow-up of drift1}\\
    & \quad + \int_s^t\frac{1}{2 \eps} \langle T_{\eps}^* \big[(\rho_u b_u) \otimes f_u - \rho_u  \otimes (b_u f_u)\big], \nabla_x \Psi -\nabla_y \Psi \rangle \rmd u \label{eq:blow-up of drift2}  \\
    & \quad + \int_s^t \langle T_{\eps}^* \big[\rho_u  \otimes ((\nabla \cdot b_u) f_u)\big],  \Psi \rangle \rmd u  \label{eq:blow-up of drift3}\\
    & =: I^\eps_1 + I^\eps_2 + I^\eps_3. \notag
\end{align}
To bound $I^\eps_1$ and $I^\eps_3$, we use integrability of $b$, $\nabla\cdot b$ as follows.
For $I^\eps_1$, by H\"older's inequality
\begin{align*}
    I_1^\eps
    & \lesssim  \int_s^t \int_{B_{R}} \int_{B_{1}} \big|\rho_u(x_+ + \eps x_-) \{ b_u(x_+ + \eps x_-) + b_u(x_+ - \eps x_-) \} f_u(x_+ - \eps x_-) \\
    & \qquad\quad \big( \nabla_x \Psi +  \nabla_y  \Psi \big)(x_+ + x_-, x_+ - x_-) \big| \rmd x_- \rmd x_+  \rmd u \\
    & \leq  \int_s^t  \int_{B_1} \| \rho_u( \cdot + \eps x_-)\|_{L^p(B_{R})} ( \| b_u( \cdot + \eps x_-)\|_{L^r(B_{R})} + \| b_u( \cdot - \eps x_-)\|_{L^r(B_{R})} )\\
    & \qquad\quad \| f_u( \cdot - \eps x_-)\|_{L^q(B_{R})}  \rmd x_-  \| \nabla_x \Psi + \nabla_y \Psi\|_{L^{\infty}(B_{R})} \rmd u  \\
    & \lesssim \int_s^t \|\rho_u\|_{L^p(B_{R+1})} \|b_u\|_{L^{r}(B_{R+1})} \|f_u\|_{L^q(B_{R+1})} \|\Psi\|_{\mathcal{E}_{1,R}} \rmd u\\
    & \lesssim \|\rho\|_{\cB_b([0,T]; L^p(B_{R+1}))} \|f\|_{\cB_b([0,T]; L^q(B_{R+1}))}\int_s^t \|b_u\|_{L^r(B_{R+1})} \rmd u \, \|\Psi\|_{\mathcal{E}_{1,R}}
\end{align*}
similarly, for $I^\eps_3$ we have
\begin{align*}
    I^\eps_3
    & \lesssim \int_s^t \|\rho_u\|_{L^p(B_{R+1})} \| \nabla \cdot b_u\|_{L^{r}(B_{R+1})} \|f_u\|_{L^q(B_{R+1})} \rmd u\, \|\Psi\|_{\mathcal{E}_{0,R}}\\
    & \lesssim \|\rho\|_{\cB_b([0,T]; L^p(B_{R+1}))} \|f\|_{\cB_b([0,T]; L^q(B_{R+1}))}\int_s^t \|\nabla\cdot b_u\|_{L^r(B_{R+1})} \rmd u \, \|\Psi\|_{\mathcal{E}_{0,R}}.
\end{align*}
To bound $I^\eps_2$, we need Sobolev regularity of $b$ to cancel out the factor $\eps^{-1}$; this term behaves similarly to a DiPerna--Lions commutator.
For Lebesgue a.e. $u$ it holds that $\rho_u \in L^p(B_{R+1}), f_u \in L^q(B_{R+1})$ and $b_u \in W^{1,r}(B_{R+1})$; for such $u$ and a.e. $x_+ \in B_R$, $x_- \in B_1$, we have
\begin{equation}\label{eq:estimates_drift_proof}\begin{split}
    \frac{1}{2\eps}T_{\eps}^* \big[ (\rho_u b_u)& \otimes f_u - \rho_u  \otimes (b_u f_u)\big](x,y)\\
    & = \rho_u(x_+ + \eps x_-) f_u(x_+ - \eps x_-) \frac{ b_u(x_+ + \eps x_-) - b_u(x_+ - \eps x_-)  }{2 \eps} \\
    & = \frac12 \rho_u(x_+ + \eps x_-) f_u(x_+ - \eps x_-) \Big( \int_{-1}^1 Db_u(x_+ + \theta \eps x_-) \rmd \theta\Big) x_- .
\end{split}\end{equation}
Integrating w.r.t. $u$ and using that $\Psi\in \cE_{1,R}$ is compactly supported in the $x_-$-direction, we find 
\begin{equation}\label{eq:epsilon cancellation in the drift2}\begin{split}
    %\int_s^t \Big| \frac{1}{2 \eps} &  \langle T_{\eps}^*   (\rho_u b_u \otimes f_u - \rho_u  \otimes b_u f_u), \nabla_x \Psi -\nabla_y \Psi \rangle \Big| \rmd u\\
    I^\eps_2
    & \lesssim \int_s^t  \int_{B_{R}} \int_{B_{1}} |\rho_u(x_+ + \eps x_-)| |f_u(x_+ - \eps x_-)| \int_{-1}^1 |Db_u(x_+ + \theta \eps x_-)| \dd \theta\, \rmd x_- \rmd x_+\\
    & \qquad \quad \| \nabla_x \Psi -  \nabla_y \Psi \|_{L^{\infty}} \rmd u\\
    % old version of the line above: & \qquad\quad \big( \nabla_x \Psi -  \nabla_y \Psi \big)(x_+ + x_-, x_+ - x_-) \big| \rmd x_- \rmd x_+ 
    % & \leq  \int_s^t  \int_{B_{1}} \| \rho_u( \cdot + \eps x_-)\|_{L^p(B_{R})}   \| f_u( \cdot - \eps x_-)\|_{L^q(B_{R})}  \| \frac12 \int_{-1}^1 Db_u(\cdot + \theta \eps x_-) \rmd \theta \|_{L^r(B_{R})}  |x_-| \rmd x_-  \notag \\
    & \lesssim  \int_s^t  \|\rho_u \|_{L^{p}(B_{R+1})}  \|f_u\|_{L^q(B_{R+1})} \| Db_u\|_{L^r(B_{R+1})} \rmd u\,  \|\Psi\|_{\mathcal{E}_{1,R}}\\
    &  \lesssim \|\rho\|_{\cB_b([0,T]; L^p(B_{R+1}))} \|f\|_{\cB_b([0,T]; L^q(B_{R+1}))}\int_s^t \|D b_u\|_{L^r(B_{R+1})} \rmd u \, \|\Psi\|_{\mathcal{E}_{1,R}}
\end{split}\end{equation}
Overall, this proves the claim.  
\end{proof}

We now have all the ingredients to derive the RPDE for the product $\rho f$.
Towards this end, in the rest of this section, we fix a radially symmetric probability density $\chi \in C^{\infty}_c$, such that $\supp (\chi) \subset B(0,1/2)$. Let $\{\chi^\eps\}_\eps$ denote the associated family of mollifiers, $\chi^\eps(x)=\eps^{-d}\chi(\eps^{-1}x)$.
Then for any $\phi\in \cF_{3,R}$, we may define
\begin{equation}\label{eq:defn_test_doublevariables}
    \Psi(x,y):= \phi(x_+) \chi(x_-) \in \cE_{3,2R};
\end{equation}
by the definition of $T_\eps$ and properties of mollifiers, we have
\begin{equation}\label{eq:approx_dirac_delta}
    T_{\eps} \Psi(x,y) = \phi\Big(\frac{x+y}{2}\Big) \chi^\eps(x-y)\rightarrow \phi(x) \delta_{x=y}
\end{equation}
as $\eps\to 0$, where convergence holds in the sense of distributions.

The goal is now to test \eqref{eq:approx product} against $\Psi$ as defined above, send $\eps \rightarrow 0$, and identify the limits of $(\bX^\eps,M^\eps)$ to find that
\begin{equation*}
\rmd \langle \rho f, \phi \rangle -  \langle \rho b f, \nabla \phi \rangle  \rmd t - \langle \rho \xi_k f, \nabla \phi \rangle \rmd \bZ_t^k = 0
\end{equation*}
holds for any $\phi \in \cF_{3,R}$, which is to say that $\rho f$ satisfies the rough continuity equation \eqref{eq:rc_in_unique}, since $\xi_k$ is divergence free.  

The next result guarantees pointwise convergence of the unbounded rough drivers $\mathbf{X}^\eps$ as $\eps\to 0$.

\begin{proposition}[Proposition 5.13 from \cite{DGHT2019}]\label{proposition:noise converges}
Let $\fkp\in [2,3)$, $\bZ\in \clC^{\fkp}_g$, $\xi\in C^3_b$ and $\nabla \cdot \xi = 0$.
Let $\mathbf{X}^{\eps}$ be defined by \eqref{eq:approx of product URD}, $\rho$ and $f$ be measurable functions such that $\rho$, $f$, $\rho f\in L^1_\loc$, $\chi$ as given above.
Then for every $R\geq 1$ and $\phi \in \cF_{3, R}$, defining $\Psi$ as in \eqref{eq:defn_test_doublevariables}, we have that
\begin{equation*}
    \lim_{\eps \rightarrow 0} \langle  \rho \otimes f, X_{st}^{\eps,\ast} \Psi \rangle =  \langle  \rho f, A_{st}^*  \phi\rangle, \qquad
    \lim_{\eps \rightarrow 0} \langle  \rho \otimes f, \mathbb{X}_{st}^{\eps,\ast} \Psi \rangle =  \langle  \rho f, \mathbb{A}_{st}^*  \phi
    \rangle\qquad \forall\, (s,t)\in \Delta_T.
\end{equation*}
\end{proposition}

Next, we show that the forcing $M^\eps$ converges as well.

\begin{proposition}\label{proposition:drift converges}
Let $\rho$, $f$, $b$ satisfy the assumptions of Proposition \ref{proposition:renormalizable drift} and let $\phi$, $\chi$ and $\Psi$ as in Proposition \ref{proposition:noise converges}.
Then for any $(s,t)\in \Delta_T$ we have
$$
\lim_{\eps \rightarrow 0} \langle \delta M_{st}^{\eps}, \Psi \rangle  = - \int_s^t \langle \rho_u f_u b_u , \nabla \phi \rangle \dd u .
$$
\end{proposition}

\begin{proof}

Note that, since $1/p + 1/q + 1/r = 1$, we have that $\rho f b,\, \rho f Db \in L^1_t L^1_\loc$.
As in Proposition \ref{proposition:renormalizable drift}, developing the expression for $\langle \delta M^\eps_{st},\Psi\rangle$, we obtain the terms $\{I^\eps_i\}_{i=1}^3$ defined by \eqref{eq:blow-up of drift1}-\eqref{eq:blow-up of drift3}.
Recall formula \eqref{eq:approx_dirac_delta} for $T_\eps\Psi$ and further notice that 
\begin{align*}
    \nabla_x \Psi(x,y) & = \frac12 \nabla \phi \left( \frac{x+y}{2} \right) \chi(x-y) +  \phi \left( \frac{x+y}{2} \right) \nabla \chi(x-y), \\
    \nabla_y \Psi(x,y) & = \frac12 \nabla \phi \left( \frac{x+y}{2} \right) \chi(x-y) -  \phi \left( \frac{x+y}{2} \right) \nabla \chi(x-y),\\
    T_{\eps} \Big[ (\nabla_x + \nabla_y) \Psi \Big] (x,y) & = \nabla \phi \left( \frac{x+y}{2} \right) \chi^\eps(x-y).
\end{align*}
Plugging these identities into $I^\eps_1$ defined by \eqref{eq:blow-up of drift1}, we find
\begin{align*}
    I^\eps_1
%& =\int_s^t \frac12 \langle T_{\eps}^* (\rho_u b_u \otimes f_u + \rho_u  \otimes b_u f_u),   \nabla_x \Psi +  \nabla_y \Psi \rangle \rmd u \\
    & = \frac12 \int_s^t \int_{B_{2R}} \int_{B_{2R}} \rho_u(x) [b_u(x) + b_u(y) ] f_u(y) \nabla \phi \left( \frac{x+y}{2} \right) \chi^{\eps}(x-y) \rmd x\rmd y \rmd u \\
    & \to \int_s^t\int_{\R^d} \rho_u(x)b_u(x) f_u(x) \nabla \phi(x)\rmd x \rmd u
\end{align*}
where convergence follows from the integrability assumptions and properties of mollifiers $\{\chi^\eps\}_\eps$.
For $I^\eps_3$, since $\rho_u (\nabla \cdot b_u) f_u \in L^1_t L^1_\loc$, we similarly get 
$$
    \lim_{\eps\to 0} I^\eps_3 =
    %\int_s^t \langle T_{\eps}^* \big[ \rho_u  \otimes ((\nabla \cdot b_u) f_u)\big],  \Psi \rangle \rmd u \rightarrow
    \int_s^t \langle \rho_u (\nabla \cdot b_u) f_u, \phi \rangle \rmd u.
$$
Arguing as in \eqref{eq:estimates_drift_proof}, changing the area of integration into $(x_+, x_-)$-variables and plugging in $\Psi(x,y) = \phi(x_+) \psi(2 x_-)$, we can write $I^\eps_3$ from \eqref{eq:blow-up of drift2} as
\begin{equation}\label{eq:blow-up of drift3 with blow up transformation}\begin{split}
    I^\eps_3 & = 2^{2d} \int_s^t \int_{-1}^1 \int_{B_1} \int_{B_R} \rho_u(x_+ + \eps x_-) f_u(x_+ - \eps x_-)  \phi(x_+)\\
    &\qquad\qquad D b_u(x_+ + \theta \eps x_-)  x_- \cdot \nabla \chi(2 x_-) \rmd x_+ \rmd x_- \rmd \theta \, \dd u.
\end{split}\end{equation}
For any $x_- \in B_1$, $\theta \in [-1,1]$ and almost every $u \in [s,t]$, as $1/p + 1/q + 1/r = 1$, it follows that 
$$
\lim_{\eps \rightarrow 0} \rho_u(\cdot  + \eps x_-) f_u(\cdot  - \eps x_-) Db_u(\cdot  + \theta \eps x_-) = \rho_u f_u Db_u  \textrm{ in } L^1(B_{R+1})  
$$
since translations are continuous operators on $L^1_\loc$. Thus, for Lebesgue a.e. $\theta$ and $x_-$, we find
$$
\lim_{\eps \rightarrow 0}  \int_{B_R} \rho_u(x_+ + \eps x_-) f_u(x_+ - \eps x_-) \phi(x_+)  D b_u(x_+ + \theta \eps x_-)    \rmd x_+  = \int_{B_R} \rho_u(x_+) f_u(x_+)  \phi(x_+)  D b_u(x_+)    \rmd x_+ .
$$
Using the bound \eqref{eq:epsilon cancellation in the drift2} and dominated convergence, by the above and identity \eqref{eq:blow-up of drift3 with blow up transformation}, we find
\begin{align*}
    \lim_{\eps\to 0} I^\eps_2
    & = 2^{2d+1} \int_s^t \int_{B_1} \int_{B_R}  \rho_u(x_+)  f_u(x_+) \phi(x_+)  \big[ D b_u(x_+ )  x_-\big] \cdot \nabla \chi(2 x_-) \rmd x_+ \rmd x_-  \rmd u \\
    & =  \int_s^t \int_{B_R} \rho_u(x_+)  f_u(x_+) \phi(x_+)  D b_u(x_+ ):\bigg( 2^{2d+1}  \int_{B_1} x_-  \otimes \nabla \chi(2 x_-) \rmd x_-\bigg)  \rmd x_+ \rmd u
\end{align*}
where we employed the notation $A:B=\sum_{i,j} A_{ij} B_{ij}$ for the Frobenius product of matrices. By integration by parts, we have 
\begin{align*}
    2^{2d+1} \int_{B_1} x_-  \otimes \nabla \chi(2 x_-) \rmd x_- 
    = 2^{2d}  \int_{B_1} x_-  \otimes \nabla  ( \chi(2 x_-) ) \rmd x_-
    = - I_d\, 2^d \int_{B_1} \chi(2x_-) \rmd x_-
    = - I_d.
\end{align*}
%where $\mathbf{1} := (1,1, \dots, 1) \in \R^d$.
Overall, we deduce that
\begin{align*}
    %\int_s^t \int_{B_R} \rho_u(x_+)  f_u(x_+) &  D b_u(x_+ )   2^{d+1} \int_{B_1} x_-  \nabla \chi(2 x_-) \rmd x_- \phi(x_+) \rmd x_+   \\
    \lim_{\eps\to 0} I^\eps_2
    = - \int_s^t \int_{\R^d} \rho_u(x_+)  f_u(x_+)  \nabla \cdot b_u(x_+ ) \phi(x_+) \rmd x_+  \rmd u
    = - \lim_{\eps\to 0} I^\eps_3
\end{align*}
which proves the claim. 
\end{proof}

We are finally ready to present the

\begin{proof}[Proof of Theorem \ref{thm:product-rule}]
    Thanks to assumption \eqref{eq:product_formula_assumption}, we can apply all the results presented in this section.
    Let $R\geq 1$ be fixed. Let $\phi \in \cF_{3,R}$ and define $\Psi$ as in \eqref{eq:defn_test_doublevariables};
    testing \eqref{eq:approx product} against $\Psi$, we have
    \begin{equation} \label{eq:before taking the limit} 
         \langle \delta(\rho \otimes f)_{st}, T_{\eps} \Psi \rangle  + \langle \delta M_{st}^{\eps}, \Psi \rangle  +  \langle \rho_s \otimes f_s, X^{\eps,\ast}_{st} \Psi \rangle - \langle \rho_s \otimes f_s, \bbX^{\eps,\ast}_{st} \Psi \rangle = \langle T_{\eps}^* (\rho \otimes f)_{st}^{\natural}, \Psi \rangle.
    \end{equation}
    By Propositions \ref{proposition:noise converges} and \ref{proposition:drift converges}, as well as identity \eqref{eq:approx_dirac_delta} and standard properties of mollifiers, we see that the LHS of \eqref{eq:before taking the limit} converges to
    $$
        \langle \delta(\rho f)_{st}, \phi \rangle  - \int_s^t \langle \rho_u f_u b_u, \nabla \phi\rangle  \rmd u  + \langle \rho_s f_s, A_{st}^\ast \phi \rangle -\langle \rho_s f_s, \bbA_{st}^{\ast} \phi \rangle.
    $$
    Thus, also the RHS of \eqref{eq:before taking the limit} is converging, and we denote its limit by $\langle (\rho f)^{\natural}_{st}, \phi \rangle$. It is clear that $(\rho f)^{\natural}_{st} \in \clF_{-2,R}$; in order to verify that $\rho f$ satisfies Definition $\rho f$, it only remains to show that $(\rho f)^{\natural}\in C^{\fkp/3-\var}_2 \cF_{-3,R}$.
    %it is a remainder as a $\clF_{-3,R}$-valued map.

    From Proposition \ref{proposition:renormalizable noise}, we see that the unbounded rough drivers $\bX^\eps$ satisfy uniform-in-$\eps$ bounds in operator norms; moreover it is clear that $\rho\otimes f\in \cB_b([0,T];\cE_{-0,2R})$, and Proposition \ref{proposition:renormalizable drift} ensures uniform-in-$\eps$ bounds for $\| M^\eps\|_{C^{1-\var} \cE_{-1,2R}}$.
    Therefore we can apply Lemma \ref{lem:apriori-unbounded} (for the RPDE satisfied by $\rho\otimes f$ on the scale $(\cE_{l,2R})_l$) to deduce that the remainders $T_{\eps}^* (\rho \otimes f)^{\natural}$ remain uniformly bounded in $\eps$, i.e. there exists a control $w_{\otimes, R, \natural}$ such that 
    \begin{equation*}%\label{eq:uniform bounds on tensor remainder}
        \sup_{\eps>0} \| T_{\eps}^* (\rho \otimes f)^{\natural}_{st}\|_{\mathcal{E}_{-3,2R}} \leq w_{\otimes, R, \natural}(s,t)^{3/\fkp}.
    \end{equation*}
    Since $\|\Psi\|_{\mathcal{E}_{3,2R}} \lesssim \|\phi\|_{\cF_{3,R}}$ it follows that 
    $$
        |\langle (\rho f)^{\natural}_{st}, \phi \rangle | 
        = \lim_{\epsilon \rightarrow 0} | \langle T_{\eps}^* (\rho \otimes f)^{\natural}_{st} , \Psi\rangle | \lesssim w_{\otimes, R, \natural}(s,t)^{3/\fkp} \| \phi \|_{\cF_{3,R}}
    $$
    which overall shows that $(\rho f)^{\natural}\in C^{\fkp/3-\var}_2 \cF_{-3,R}$. As the argument holds for any $R\geq 1$, we conclude that $\rho f$ solves the rough continuity equation \eqref{eq:rc_in_unique}.

    Finally, since $\rho f$ solves \eqref{eq:rc_in_unique}, the duality formula \eqref{eq:duality formula} follows from assumption \eqref{eq:duality_formula_assumption} and Lemma \ref{lem:conservation-mass} applied to $\tilde \rho=\rho f$.
\end{proof}

\begin{remark}
    If the vector fields $\xi_k$ are not divergence free, in light of the heuristics computations presented at the beginning of Section \ref{subsec:linear-uniqueness}, we still expect the product and duality formulas from Theorem \ref{thm:product-rule} to be true.
    However at a technical level, obtaining a priori estimates for $L^p_x$-norms of solutions becomes more challenging, since we do not have a clean expression like \eqref{eq:incompressibility-proof-eq2} anymore; moreover handling commutators estimates in the style of Propositions \ref{proposition:renormalizable noise}-\ref{proposition:noise converges} gets harder, since we cannot rely anymore on cancellations coming from conservativity \eqref{eq:conservative-driver}.
    Since we are mostly interested in the incompressible case, in light of applications to fluid dynamics, we leave this extension for future investigations.
\end{remark}

\subsection{Uniqueness, stability, renormalizability and compactness}\label{subsec:linear-stability}

With Theorem \ref{thm:product-rule} at hand, we are finally ready to prove uniqueness of solutions. We will keep referring to the rough PDEs \eqref{eq:rc_in_unique}-\eqref{eq:lt_in_unique} introduced in the previous section.

\begin{theorem}\label{thm:uniqueness-linear-RPDE}
    Let $\fkp\in [2,3)$, $Z\in \clC^{\fkp}_g$, $\xi\in C^3_b$ with $\nabla\cdot \xi =0$.
    Let $b$ satisfy
    \begin{equation}\label{eq:condition-uniqueness}
        \frac{b}{1+|x|}\in L^1_t L^1_x + L^1_t L^\infty_x, \quad
        \nabla\cdot b\in L^1_t L^\infty_x, \quad
        b\in L^1_t W^{1,1}_\loc.
    \end{equation}
    Then for any $\rho_0\in L^1_x\cap L^\infty_x$ there exists a unique solution $\rho\in   \cB_b([0,T];L^1_x\cap L^\infty_x)$ to \eqref{eq:rc_in_unique}.

    Similarly, for any $p\in [1,\infty]$ and any $\rho\in L^p_x$, there exists a unique solution $\rho\in \cB_b([0,T]; L^p_x)$ to \eqref{eq:rc_in_unique} as soon as
    \begin{equation}\label{eq:condition-uniqueness-variant}
        \frac{b}{1+|x|}\in L^1_t L^{p'}_x, \quad
        \nabla\cdot b\in L^1_t L^\infty_x, \quad
        b\in L^1_t W^{1,p'}_\loc.
    \end{equation}
\end{theorem}

\begin{remark}\label{rem:uniqueness-RTE}
    As in Remark \ref{rem:existence-linear RPDE}, one can similarly prove uniqueness statements for \eqref{eq:lt_in_unique}, as well as both rough PDEs \eqref{eq:rc_in_unique}-\eqref{eq:lt_in_unique} with terminal conditions $\rho\vert_{t=T}=\rho_t$.
\end{remark}

\begin{proof}
We give the proof of the first statement, the second one being similar.
By linearity, it suffices to show that any solution $\rho\in \cB_b([0,T];L^1_x\cap L^\infty_x)$ with $\rho_0=0$ is necessarily $\rho\equiv 0$.
Fix any $\tau \in(0,T]$ and any $\varphi\in C^\infty_c$; let $f\in \cB_b([0,\tau]; L^1_x\cap L^\infty_x)$ be a solution to \eqref{eq:lt_in_unique} with terminal condition $f\vert_{t=\tau}=\varphi$, whose existence follows by Remark \ref{rem:existence-linear RPDE}.
By Theorem \ref{thm:product-rule} (with $p=q=\infty$, $r=1$), $\rho f$ is still a solution to \eqref{eq:rc_in_unique}, which by the above properties satisfies $\rho f \in \cB_b([0,\tau]; L^1_x\cap L^\infty_x)$; combined with the first condition in \eqref{eq:condition-uniqueness}, this implies that
$$\frac{b \rho f}{1+|x|}\in L^1_t L^1_x.$$
Again by Theorem \ref{thm:product-rule} it then holds
\begin{align*}
    \langle \rho_\tau, \varphi\rangle = \langle \rho_\tau, f_\tau\rangle = \langle \rho_0,f_0\rangle = 0.
\end{align*}
As the reasoning holds for all $\varphi\in C^\infty_c$, we conclude that $\rho_\tau=0$; by the arbitrariness of $\tau\in (0,T]$, conclusion follows.
\end{proof}

Having established uniqueness, we pass to study the continuous dependence of solutions on the data of the problem.
From now on, for simplicity of exposition, we mostly focus on $L^p_x$-valued solutions, which are well-posed under condition \eqref{eq:condition-uniqueness-variant}; his however doesn't play any major role in the statements and proofs, which up to minor modifications hold can be rephrased to accommodate the setting of \eqref{eq:condition-uniqueness}.

\begin{corollary}\label{cor:linear-RPDE-stability-1}
    Let $\fkp\in [2,3)$, $p\in [1,\infty]$ and assume that
    \begin{equation}\label{eq:condition-stability}
       \frac{b}{1+|x|}\in L^1_t L^{p'}_x, \quad
        \nabla\cdot b\in L^1_t L^\infty_x, \quad
        b\in L^1_t W^{1,p'}_\loc, \quad
        \xi\in C^3_b, \quad \nabla\cdot\xi =0,
        \quad \bZ\in \clC^\fkp_g.
    \end{equation}
    Given any $\rho^1_0$, $\rho^2_0\in L^p_x$, the associated unique solutions $\rho^i\in \cB_b([0,T]; L^p_x)$ to \eqref{eq:rc_in_unique} satisfy
    \begin{equation}\label{eq:linear-RPDE-stability}
        \| \rho^1-\rho^2\|_{\cB_b([0,T];L^p_x)}
        = \sup_{t\in [0,T]} \| \rho^1_t - \rho^2_t\|_{L^p_x}
        \leq \exp\bigg( \int_0^T \| \nabla\cdot b_u\|_{L^\infty_x} \dd u \bigg) \| \rho^1_0 - \rho^2_0\|_{L^p_x}.
    \end{equation}
    Consider now a sequence $(b^n,\xi^n,\bZ^n)$ satisfying condition \eqref{eq:condition-stability} uniformly in $n$, namely such that $\nabla\cdot \xi^n =0$ for all $n$ and 
    \begin{equation}\label{eq:hypothesis-unifboundedness-linear-RPDE-stability}
        \sup_{n\in\N} \big\{ \| \xi^n\|_{C^3_b} + \| \bZ^n\|_{\fkp,T} + \|  \nabla \cdot b^n\|_{L^1_t L^\infty_x} + \Big\| \frac{b^n}{1+|x|} \Big\|_{L^1_t L^{p'}_x} + \| \nabla b^n\|_{L^1_t L^{p'}(B_R)} \big\}<\infty \quad \forall\, R\geq 1.
    \end{equation}
    Further assume that there exist $(b,\xi,\bZ)$ such that
    \begin{equation}\label{eq:hypothesis-stability-linear-RPDE}
        b^n\to b \ {\rm in} \ L^1_t L^1_\loc, \quad
        \xi^n\to \xi \ {\rm in} \ C^0_\loc, \quad
        \sup_{t\in [0,T]} |\bZ^n_{0,t}-\bZ_{0,t}| \to 0;
    \end{equation}
    then for any sequence $\rho^n_0\rightharpoonup \rho_0$ in  $L^p_x$ for $p\in [1,\infty)$ (resp. $\rho^n_0 \xrightharpoonup{\ast} \rho_0$ for $p=\infty$), denoting by $\rho^n$, resp. $\rho$, the solutions to \eqref{eq:rc_in_unique} associated to $(b^n,\xi^n,\bZ^n)$, resp. $(b,\xi,\bZ)$, it holds $\rho^n\to \rho$ in $C_w([0,T];L^p_x)$ (resp. in $C_{w-\ast}([0,T];L^\infty_x)$ for $p=\infty$).
\end{corollary}

\begin{proof}
    The first claim is a direct consequence of linearity: $\rho^1-\rho^2$ is the unique solution associated to initial condition $\rho^1_0-\rho^2_0$, which by Proposition \ref{prop:existence-linear-RPDE} must satisfy \eqref{eq:lq-bound}, yielding \eqref{eq:linear-RPDE-stability}.

    To show the second part, notice that as in the proof of Corollary \ref{cor:stability-RDE-flow} we get the convergence $\bZ^n \rightarrow \bZ$ in $\clC^{\tilde{p}}$ for any $\tilde{p} \in (p,3)$ and equicontinuity of $(s,t) \mapsto   \sup_n\, \llbracket \mathbf{Z}^n\rrbracket_{\tilde{\mathfrak{p}}, [s,t]}.$
    Arguing as in the proof of Proposition \ref{prop:existence-linear-RPDE}, we can show that the sequence $\{\rho^n\}_n$ is compact in $C_w([0,T];L^p_x)$ ($C_{w-\ast}([0,T];L^\infty_x)$ if $p=\infty$).
    In particular, in the case $p=1$, using the fact that $\{\rho^n_0\}_n$ are uniformly integrable (by the Dunford-Pettis theorem, since $\rho^n_0\rightharpoonup \rho_0$) we can argue as in Lemma \ref{lem:equi-p-integrability} (by combining estimate \eqref{eq:equi-p-integrability-proof3} with assumption \eqref{eq:hypothesis-unifboundedness-linear-RPDE-stability}) to deduce that
    \begin{align*}
        \lim_{R\to\infty} \sup_{n\in\N}\sup_{t\in [0,T]} \int_{|x|>R} |\rho^n_t(x)|\dd x = 0;
    \end{align*}
    in particular, the tightness condition from Proposition \ref{prop:compactness-l1}-iii. is satisfied, hence we get compactness in $C_w([0,T];L^1_x)$ and not just $C_w([0,T];L^1_\loc)$.
    
    Arguing as in Proposition \ref{prop:existence-linear-RPDE}, it follows that $\{\rho^n\}_n$ admits a convergent subsequence in $C_w([0,T];L^p_x)$ ($C_{w-\ast}([0,T];L^\infty_x)$ for $p=\infty$), and that its limit point is a solution to the rough continuity \eqref{eq:rc_in_unique} associated to $(b,\xi,\bZ)$, belonging to $C_w([0,T];L^p_x)$. Since uniqueness holds for \eqref{eq:rc_in_unique} in this class, and the argument holds for any subsequence we can extract, we conclude that the full sequence converges.
\end{proof}

A similar stability result holds for \eqref{eq:lt_in_unique} as well; notice however the difference between condition \eqref{eq:hypothesis-stability-linear-RPDE} and \eqref{eq:hypothesis-stability-variant} below.

\begin{corollary}\label{cor:stability-1-rte}
    Let $\fkp\in [2,3)$, $p\in [1,\infty]$ and $(b,\xi,\bZ)$ satisfying \eqref{eq:condition-stability}. Then for any $f_0^1$, $f^2_0\in L^p_x$, the associated unique solutions $f^i\in \cB_b([0,T]; L^p_x)$ to \eqref{eq:lt_in_unique} satisfy
    \begin{equation*}
        \| f^1-f^2\|_{\cB_b([0,T];L^p_x)}
        = \sup_{t\in [0,T]} \| f^1_t - f^2_t\|_{L^p_x}
        \leq \exp\bigg( \int_0^T \| \nabla\cdot b_u\|_{L^\infty_x} \dd u \bigg) \| f^1_0 - f^2_0\|_{L^p_x}.
    \end{equation*}
    Consider now a sequence $(b^n,\xi^n,\bZ^n)$ satisfying condition \eqref{eq:hypothesis-unifboundedness-linear-RPDE-stability}, $\nabla\cdot \xi^n=0$ for all $n$, and further assume that there exist $(b,\xi,\bZ)$ such that
    \begin{equation}\label{eq:hypothesis-stability-variant}
        b^n\to b \ {\rm in} \ L^1_t L^1_\loc, \quad
        \nabla\cdot b^n\to \nabla\cdot b\ {\rm in} \ L^1_t L^1_\loc, \quad
        \xi^n\to \xi \ {\rm in} \ C^0_\loc, \quad
        \sup_{t\in [0,T]} |\bZ^n_{0,t}-\bZ_{0,t}| \to 0;
    \end{equation}
    then for any sequence $f^n_0\rightharpoonup f_0$ in  $L^p_x$ for $p\in [1,\infty)$ (resp. $f^n_0 \xrightharpoonup{\ast} f_0$ for $p=\infty$), denoting by $f^n$, resp. $f$, the solutions to \eqref{eq:lt_in_unique} associated to $(b^n,\xi^n,\bZ^n)$, resp. $(b,\xi,\bZ)$, it holds $f^n\to f$ in $C_w([0,T];L^p_x)$ (resp. in $C_{w-\ast}([0,T];L^\infty_x)$ for $p=\infty$).
\end{corollary}

\begin{proof}
    Existence follows from Remark \ref{rem:existence-linear RPDE}, uniqueness from Remark \ref{rem:uniqueness-RTE}. The stability estimate then follows again by linearity of the equation and the estimates from Remark \ref{rem:existence-linear RPDE}.
    The rest of the proof is almost identical to that of Corollary \ref{cor:linear-RPDE-stability-1}, so let us only highlight the difference coming from assumption \eqref{eq:hypothesis-stability-variant}.
    In order to pass to the limit as $n\to\infty$ and show that $f$ is a solution to \eqref{eq:lt_in_unique}, one must now show that
    \begin{align*}
        b^n\cdot\nabla f^n = \nabla\cdot (b^n f^n)-(\nabla\cdot b^n) f^n
        \to \nabla\cdot (b f)-(\nabla\cdot b) f = b\cdot\nabla f
    \end{align*}
    in the sense of distributions; this is where, in order to employ weak-strong convergence arguments, we additionally need $\nabla\cdot b^n\to\nabla\cdot b$ in $L^1_t L^1_\loc$ (cf. also Remark \ref{rmk:weak compactness}).
\end{proof}

\begin{proposition}\label{prop:absolute-continuity-Lp}
    Let $\fkp\in [2,3)$, $p\in [1,\infty)$, $(b,\xi,\bZ)$ satisfying \eqref{eq:condition-stability} and $\rho_0\in L^p_x$. Then the associated solution $\rho$ to \eqref{eq:rc_in_unique} satisfies $\rho\in C([0,T];L^p_x)$ and
    \begin{equation}\label{eq:evolution_Lp_norm}
        \frac{\dd }{\dd t} \| \rho_t\|_{L^p_x}^p
        = \int_{\R^d} (1-p) |\rho_t(x)|^p (\nabla\cdot b)(x) \dd x.
    \end{equation}
    As a consequence, for all $(s,t)\in\Delta_T$ we have
    \begin{equation}\label{eq:estimate-lp-continuity}
        \big|\, \| \rho_t\|_{L^p_x} - \| \rho_s\|_{L^p_x} \big|
        \leq \frac{1}{p'}\exp\bigg(\frac{\| \nabla\cdot b\|_{L^1_t L^\infty_x}}{p'} \bigg) \bigg( \int_s^t \| \nabla\cdot b_r\|_{L^\infty_x}\, \dd u\bigg) \| \rho_0\|_{L^p_x}.
    \end{equation}
\end{proposition}

\begin{proof}
    Let us first show how to derive \eqref{eq:estimate-lp-continuity} from \eqref{eq:evolution_Lp_norm}. We have
    \begin{align*}
        \frac{\dd }{\dd t} \| \rho_t\|_{L^p_x}
        = \frac{1}{p}  \| \rho_t\|_{L^p_x}^{1-p}\, \frac{\dd }{\dd t} \| \rho_t\|_{L^p_x}^p
        \leq \bigg(1-\frac{1}{p}\bigg) \| \rho_t\|_{L^p_x} \| \nabla\cdot b_t\|_{L^\infty_x};
    \end{align*}
    applying estimate \eqref{eq:lq-bound}, one then finds
    \begin{align*}
        \frac{\dd }{\dd t} \| \rho_t\|_{L^p_x}
        \leq \frac{1}{p'}\exp\bigg(\frac{\| \nabla\cdot b\|_{L^1_t L^\infty_x}}{p'} \bigg)  \| \nabla\cdot b_t\|_{L^\infty_x} \| \rho_0\|_{L^p_x}
    \end{align*}
    from which \eqref{eq:estimate-lp-continuity} follows upon integrating in time.

    Recall that $\rho\in C_w([0,T];L^p_x)$; it follows from \eqref{eq:estimate-lp-continuity} that, whenever $t_n\to t$, $\rho_{t_n}\rightharpoonup \rho_t$ and $\| \rho_{t_n}\|_{L^p_x}\to \| \rho_t\|_{L^p_x}$; for $p\in (1,\infty)$, by uniform convexity of $L^p_x$, this implies that $\rho_{t_n}\to \rho_t$ strongly in $L^p_x$ and therefore $\rho\in C([0,T];L^p_x)$.
    
    To handle the case $p=1$, we need an extra argument: first assume that $\rho_0\in L^1_x\cap L^\infty_x$; then by the above, $\rho\in C([0,T];L^1(B_R))$ for any $R<\infty$, and moreover $\{\rho_t\}_{t\in [0,T]}$ is tight by \eqref{eq:equi-p-integrability} (for $p=1)$, therefore $\rho\in C([0,T];L^1_x)$. For general $\rho_0\in L^1_x$, consider a sequence $\{\rho_0^n\}_n\subset L^1_x\cap L^\infty_x$ such that $\rho_0^n\to \rho_0$ in $L^1_x$; then by the above the associated solutions $\rho^n$ belong to $C([0,T];L^1_x)$ and $\rho^n\to \rho$ in $\cB_b([0,T];L^1_x)$ by \eqref{eq:linear-RPDE-stability}, therefore $\rho\in C([0,T];L^1_x)$ as well. 

    It only remains to prove formula \eqref{eq:evolution_Lp_norm}; for simplicity, we show it for $\rho_0\in L^1_x\cap L^\infty_x$, as the general case follows from an approximation procedure like the one outlined above, thanks to the stability property \eqref{eq:linear-RPDE-stability}. We split the proof of \eqref{eq:evolution_Lp_norm} in several steps.

    \textit{Step 1: smooth approximations.} Consider a nice sequence of smooth approximations $(\rho^n_0,\xi^n,\bZ^n)$ of $(\xi,\bZ)$, e.g. like the ones from Proposition \ref{prop:existence-linear-RPDE}; we may further construct it in such a way that $\nabla\cdot b^n\to \nabla \cdot b$ in $L^1_t L^q_\loc$ for all $q\in [1,\infty)$ and
    \begin{equation}\label{eq:absolute-continuity-proof0}
        \sup_n \| \rho^n_0\|_{L^q_x} \leq \| \rho_0\|_{L^q_x} \quad \forall\, q\in [1,\infty], \quad
        \sup_n \| \nabla\cdot b^n_t\|_{L^\infty_x} \leq \| \nabla\cdot b_t\|_{L^\infty_x} \quad \text{for Leb. a.e. } t\in [0,T].
    \end{equation}
    Let $\rho^n$ be the solution to \eqref{eq:rc_in_unique} associated to $(\xi^n,\bZ^n)$, with initial condition $\rho^n_0$. Since $\rho^n$ now solves a classical PDE and $\nabla\cdot\xi^n=0$, for any $q\in [1,\infty)$ it holds that
    \begin{equation}\begin{split}\label{eq:absolute-continuity-proof1}
        \frac{\dd }{\dd t} \| \rho^n_t\|_{L^q_x}^q
        & = q \int_{\R^d} |\rho^n_t|^{q-2} \text{sgn} (\rho^n_t) \partial_t \rho^n_t \dd x\\
        & = - q \int_{\R^d} |\rho^n_t|^{q-2} \text{sgn} (\rho^n_t) \nabla\cdot \big(b\rho^n_t + \xi^n \rho^n_t \bZ^n_t\big) \dd x\\
        & = \int_{\R^d} (1-q) |\rho^n_t(x)|^q (\nabla\cdot b^n)(x) \dd x
    \end{split}\end{equation}
    where in the last passage we used integration by parts and the fact that $\nabla\cdot \xi^n=0$. This is exactly \eqref{eq:evolution_Lp_norm} for $(\rho^n,b^n)$ and $q=p$; therefore arguing as above and using \eqref{eq:absolute-continuity-proof0}, we find
    \begin{equation}\label{eq:absolute-continuity-proof2}\begin{split}
        \big|\, \| \rho^n_t\|_{L^q_x} - \| \rho^n_s\|_{L^q_x} \big|
        & \leq \frac{1}{q'}\exp\bigg(\frac{\| \nabla\cdot b^n\|_{L^1_t L^\infty_x}}{q'} \bigg) \bigg( \int_s^t \| \nabla\cdot b^n_r\|_{L^\infty_x}\, \dd u\bigg) \| \rho^n_0\|_{L^q_x}\\
        & \leq \frac{1}{q'}\exp\bigg(\frac{\| \nabla\cdot b\|_{L^1_t L^\infty_x}}{q'} \bigg) \bigg( \int_s^t \| \nabla\cdot b_r\|_{L^\infty_x}\, \dd u\bigg) \| \rho^n_0\|_{L^q_x}
    \end{split}\end{equation}

    \textit{Step 2: reduction to pointwise convergence of $\| \rho^n_t\|_{L^2_x}$.}
    We claim that, in order to conclude, it suffices to show that $\| \rho^n_t\|_{L^2_x}\to \| \rho_t\|_{L^2_x}$ for all $t\in [0,T]$. Indeed, arguing by compactness as in Corollary \ref{cor:linear-RPDE-stability-1}, we know that $\rho^n\to \rho$ in $C_w([0,T];L^q_x)$ for all $q\in [1,\infty)$, which combined with the claim and uniform convexity of $L^2_x$ implies that $\rho^n_t\to \rho_t$ strongly in $L^2_x$. Since the sequence is bounded in $L^1_x\cap L^\infty_x$, convergence in $L^q_x$ for $q\in (1,\infty)$ then follows by interpolation. By \eqref{eq:absolute-continuity-proof2}, the maps $\{t\mapsto \| \rho^n_t\|_{L^q_x}\}$ are equicontinuous; for $q\in (1,\infty)$, together with strong pointwise convergence and convergence in $C_w([0,T];L^q_x)$, this implies that $\rho^n\to \rho$ in $C([0,T];L^q_x)$ by Corollary \ref{cor:compactness-lp}.
    In the case $q=1$, as before we can use convergence in $C([0,T];L^2(B_R))$ and tightness (coming from in Corollary \ref{cor:linear-RPDE-stability-1}) to deduce convergence in $C([0,T];L^1_x)$ as well.
    Overall this shows that $\rho\in C([0,T];L^q_x)$ for all $q\in [1,\infty)$, being a limit of smooth functions in this topology; moreover with strong convergence in $C([0,T];L^q_x)$ at hand, we can pass to the limit in \eqref{eq:absolute-continuity-proof1} (using that $\nabla\cdot b^n\to \nabla \cdot b$ in $L^1_t L^q_\loc$ for all $q\in [1,\infty)$) to deduce the validity of \eqref{eq:evolution_Lp_norm}.

    \textit{Step 3: duality.} Let us fix any $\tau\in (0,T]$  and let $f^n$ (resp. $f$) denote the unique solution to \eqref{eq:lt_in_unique} on $[0,\tau]$ with associated terminal condition $f^n\vert_{t=\tau}=\rho^n_\tau$ (resp. $f\vert_{t=\tau}=\rho_\tau$).
    Since $\rho^n_\tau\rightharpoonup \rho_\tau$ in $L^2_x$ and the $b^n$ satisfy \eqref{eq:absolute-continuity-proof0}, by Corollary \ref{cor:stability-1-rte} (and time reversal) $f^n_0\rightharpoonup f_0$ in $L^2_x$.
    On the other hand, $\rho^n_0\to \rho_0$ in $L^2_x$, so by weak-strong convergence and the duality formula \eqref{eq:duality formula} we find
    \begin{align*}
        \lim_{n\to\infty} \| \rho^n_\tau\|_{L^2_x}^2
        = \lim_{n\to\infty} \langle \rho^n_0, f^n_0 \rangle
        = \langle \rho_0, f_0  \rangle = \| \rho_\tau\|_{L^2_x}^2.
    \end{align*}
    As the argument holds for any $\tau\in (0,T]$, this verifies the claim from Step 2 and concludes the proof.
\end{proof}

\begin{remark}
    Formula \eqref{eq:evolution_Lp_norm} is only possible thanks to the assumption $\nabla\cdot \xi=0$, otherwise some ``rough component'' would appear on the r.h.s. when computing $\dd \| \rho_t\|_{L^p_x}^p$.
\end{remark}

We can now actually strengthen Corollary \ref{cor:linear-RPDE-stability-1} to get stability results in stronger topologies.

\begin{corollary}\label{cor:linear-RPDE-stability-2}
    Let $\fkp\in [2,3)$, $p\in [1,\infty)$; let $(b^n,\xi^n,\bZ^n)$ be a sequence such that $\nabla\cdot \xi^n=0$ for all $n$, \eqref{eq:hypothesis-unifboundedness-linear-RPDE-stability} is satisfied and there exist $(b,\xi,\bZ)$ such that \eqref{eq:hypothesis-stability-variant} holds.
    Let $\rho^n_0\to \rho_0$ strongly in $L^p_x$ and denote by $\rho^n$, $\rho$, the associated solutions to \eqref{eq:rc_in_unique}.
    Then it holds $\rho^n\to \rho$ in $C([0,T];L^p_x)$.
\end{corollary}

\begin{proof}
    Since $\rho^n_0\to \rho_0$ in $L^p_x$, for any $\eps>0$, by mollifications and cutoffs we can construct another sequence $\{\tilde\rho^n_0\}_n\subset L^1_x\cap L^\infty_x$, $\tilde \rho_0\in L^1_x\cap L^\infty_x$ such that
    \begin{align*}
        \tilde\rho^n_0\to \tilde\rho_0 \text{ in } L^q_x \quad
        \forall\, q\in [1,\infty), \quad \sup_n \| \tilde\rho^n_0-\rho^n_0\|_{L^p_x}\leq \eps, \quad \| \tilde\rho_0-\rho_0\|_{L^p_x}\leq \eps.
    \end{align*}
    As a consequence, it suffices to give the proof for $\{\rho^n_0\}_n$ bounded in $L^1_x\cap L^\infty_x$ and converging in $L^1_x\cap L^q_x$ for $q\in [1,\infty)$; the general case then follows by triangular inequality, the stability estimate \eqref{eq:linear-RPDE-stability} and the arbitrariness of $\eps>0$.

    The majority of the proof, up to minor details, is similar to that of Proposition \ref{prop:absolute-continuity-Lp}, so let us mostly sketch it. We can assume $p\in (1,\infty)$, the case $p=1$ then following by tightness. Running the same duality argument as in Step 3 above, one can show that $\rho^n_t\to \rho_t$ in $L^q_x$ for all $t\in [0,T]$ and $q\in [1,\infty)$; by Corollary \ref{cor:compactness-lp}, in order to achieve convergence in $C([0,T];L^p_x)$, it then suffices to show equicontinuity of $\{t\mapsto \| \rho^n_t\|_{L^p_x}\}_n$. By formula \eqref{eq:evolution_Lp_norm}, in order to do so, it suffices to show that
    \begin{align*}
        (1-p) |\rho^n_t(x)|^p (\nabla\cdot b^n)(x) \to (1-p) |\rho_t(x)|^p (\nabla\cdot b)(x) \quad \text{in } L^1_t L^1_x.
    \end{align*}
    This now follows from a combination of assumptions \eqref{eq:hypothesis-unifboundedness-linear-RPDE-stability}, \eqref{eq:hypothesis-stability-variant}, the pointwise-in-time strong convergence $\rho^n_t\to\rho_t$ in $L^q_x$ and Vitali's convergence theorem. 
\end{proof}

A consequence of Corollary \ref{cor:linear-RPDE-stability-2}, in the context of transport equations, is the property of \emph{renormalizability}.
In the spirit of \cite{diperna1989ordinary}, we will say that a solution $f$ to the rough transport equation \eqref{eq:lt_in_unique} is \emph{renormalized} if, for any $\beta\in C^1_b$, $\beta(f)$ also solves \eqref{eq:lt_in_unique}.

\begin{corollary}\label{cor:renormalized}
    Let $\fkp\in [2,3)$, $p\in [1,\infty)$ and $(b,\bZ,\xi)$ satisfy \eqref{eq:condition-stability} for these parameters.
    Then for any $f_0\in L^p_x$, the unique solution $f\in C([0,T];L^p_x)$ to \eqref{eq:lt_in_unique} is renormalized.
\end{corollary}

\begin{proof}
    Consider a family of smooth approximants $(b^n,\bZ^n, \rho^n_0)$ as in Proposition \ref{prop:existence-linear-RPDE}, $f^n$ be the associated solutions to \eqref{eq:lt_in_unique}.
    By Remark \ref{rem:existence-linear RPDE}, these are also solutions to the corresponding classical PDE and are therefore renormalized, namely $\beta(f^n)$ solves the same equation with initial datum $\beta(f^n_0)$.
    On the other hand, by Corollary \ref{cor:linear-RPDE-stability-2} $f^n_t\to f_t$ in $C([0,T]; L^p_x)$ and so $\beta(f^n_t) \xrightharpoonup{\ast} \beta(f_t)$ for all $t\geq 0$;
    by the weak stability from Corollary \ref{cor:linear-RPDE-stability-1}, we deduce that $\beta(f)$ is the unique solution associated to $(b,\bZ, \beta(f_0))$.
\end{proof}

We are now ready to complete the

\begin{proof}[Proof of Theorem \ref{thm:wellposedness-linear-RPDE-intro}]
    Existence and uniqueness of solutions come from Theorem \ref{thm:uniqueness-linear-RPDE} and Remark \ref{rem:uniqueness-RTE}; product formula and duality from Theorem \ref{thm:product-rule}. Renormalizability for the transport equation follows from Corollary \ref{cor:renormalized} and the fact that solutions belong to $C([0,T];L^p_x)$ from Proposition \ref{prop:absolute-continuity-Lp} (for simplicity stated only for \eqref{eq:rc_in_unique}, with almost identical proof for \eqref{eq:lt_in_unique}).  
\end{proof}

Under suitable assumptions on the drift, the above stability results readily imply compactness of solutions in strong topologies. Such results are often achieved by working in the Lagrangian framework, where convergence is stated at the level of flows $\Phi^n_t(x)$, cf. \cite{CriDeL2008}. Here instead we get it at the Eulerian level, up to requiring an additional compactness condition on $\nabla\cdot b^n$; notice that this is automatically satisfied when working with incompressible fluids, since $\nabla\cdot b^n\equiv 0$.
For simplicity, we state the result in the case of autonomous coefficients; generalizations to the time-dependent case can be achieved easily, up to enforcing additional sufficient conditions to guarantee time-compactness as well (like Aubin-Lions-Simon's compactness criteria from \cite{Simon1987}).

\begin{corollary}\label{cor:linear-compactness}
    Let $\fkp\in [2,3)$, $p\in [1,\infty)$, $\bZ\in \clC^\fkp_g$ and $\rho_0\in L^p_x$ be fixed.
    Consider a bounded sequence $\{(b^n,\xi^n)\}$ in $(L^{p'}_x\cap W^{1,p'}_\loc)\times C^3_b$ such that:
    \begin{itemize}
        \item $\nabla\cdot\xi^n=0$ for all $n$;
        \item $\{\nabla\cdot b^n\}_n$ is bounded in $L^\infty_x$;
        \item $\{\nabla\cdot b^n\}_n$ is compact in $L^1_\loc$;
    \end{itemize}
    denote by $\rho^n$ the solution associated to $(b^n,\xi^n,\bZ,\rho_0)$. Then the sequence $\{\rho^n\}_n$ is compact in $C([0,T];L^p_x)$.
\end{corollary}

\begin{proof}
    By Rellich-Kondrakov and Ascoli-Arzelà theorems, $W^{1,p}(B_R)\hookrightarrow L^p(B_R)$ and $C^3(\overline{B_R})\hookrightarrow C^2(\overline{B_R})$ with compact embeddings.
    By a Cantor diagonal argument, we can therefore find a (not relabelled) subsequence such that $b^n\to b$ in $L^p_\loc$, $\xi^n\to \xi$ in $C^2_\loc$; by lower-semicontinuity, the limits $(b,\xi)$ still satisfy $\nabla\cdot b\in L^\infty_x$, $\nabla\cdot\xi=0$.
    Corollary \ref{cor:linear-RPDE-stability-2} then ensures the convergence $\rho^n\to \rho$ in $C([0,T];L^p_x)$, where $\rho$ is the unique solution associated to $(b,\xi,\bZ,\rho_0)$. Overall this proves compactness of $\{\rho^n\}_n$ in $C([0,T];L^p_x)$.
\end{proof}

Finally, we can combine the RDE theory developed in Section \ref{sec:RDEs} with the RPDE one from this section to deduce an explicit Lagrangian representation formula for solutions, whenever $b$ satisfies all the relevant assumptions.

The next result is a more detailed version of Theorem \ref{thm:flow_repr_intro} from the introduction.

\begin{theorem}\label{thm:linear-RPDE-wellposed-lagrangian}
    Let $\fkp\in [2,3)$, $Z\in \clC^{\fkp}_g$, $\xi\in C^3_b$ with $\nabla\cdot \xi =0$.
    Suppose that
    \begin{align*}
        \frac{b}{1+|x|}\in L^1_t L^1_x + L^1_t L^\infty_x, \quad \nabla\cdot b\in L^1_t L^\infty_x, \quad b\in L^1_t W^{1,1}_\loc,
    \end{align*}
    % old condition:
    %$b\in L^1_t L^\infty_x$, $\nabla\cdot b\in L^1_t L^\infty_x$, $\nabla b\in L^1_t L^{p'}_x$
    and additionally assume that $b$ is an Osgood drift, in the sense that it satisfies Assumption \ref{ass:osgood-drift-RDE}.
    Then for any $\rho_0\in L^1_x\cap L^\infty_x$ there exists a unique solution $\rho\in \cB_b([0,T];L^1_x\cap L^\infty_x)$ to \eqref{eq:rc_in_unique}, which is given by the formula $\rho_t = (\Phi_t)_\sharp \rho_0$, namely
    \begin{equation}\label{eq:flow_representation_RCE}
        \langle  \rho_t, \varphi \rangle
        = \int_{\R^d} \varphi(\Phi_t(x)) \rho_0(x) \dd x\quad \forall\, \varphi\in C^\infty_c;
    \end{equation}
    here $\Phi$ is the flow associated to the underlying RDE, given by Theorem \ref{thm:wellposedness-RDE}, which satisfies the quasi-incompressibility condition \eqref{eq:flow-incompressibility} from Corollary \ref{cor:flow-incompressibility}.
    Moreover, $\rho\in C([0,T];L^p_x)$ for all $p\in[1,\infty)$, $\rho\in C_{w-\ast}([0,T];L^\infty_x)$ and
    \begin{equation*}
        \frac{\dd }{\dd t} \| \rho_t\|_{L^p_x}^p
        = \int_{\R^d} (1-p) |\rho_t(x)|^p (\nabla\cdot b)(x) \dd x\quad \forall\, p\in [1,\infty).
    \end{equation*}
    Furthermore, for any $f_0\in L^1_x\cap L^\infty_x$, there exists a unique solution $f\in \cB_b([0,T];L^1_x\cap L^\infty_x)$ to \eqref{eq:lt_in_unique}; for any $t\in [0,T]$, it holds
    \begin{equation}\label{eq:flow_representation_RTE}
        f_t(x)=f_0(\Phi_t^{-1}(x)) \quad \text{ for Lebesgue a.e. }x\in\R^d.
    \end{equation}
\end{theorem}

\begin{proof}
    Existence and uniqueness of $\rho\in \cB_b([0,T];L^1_x \cap L^\infty_x)$ comes from Theorem \ref{thm:uniqueness-linear-RPDE}.
    Similarly to Corollary \ref{cor:linear-RPDE-stability-1}, consider suitable smooth approximations of $(b,\xi,\bZ,\rho)$ and let $\rho^n$ be the associated solutions.
    It classically holds $\rho^n_t = (\Phi^n_t)_\sharp \rho^n_0$, where $\Phi^n$ are the associated flows.
    For any $t\geq 0$, by Corollary \ref{cor:stability-RDE-flow}, $\Phi^n_t \to \Phi_t$ uniformly on compact sets, while by Corollary \ref{cor:linear-RPDE-stability-2}, $\rho^n\to \rho$ in $C_w([0,T];L^p_x)$ for any $p\in [1,\infty)$. Combining these facts with weak-strong convergence and passing to the limit, one readily concludes that \eqref{eq:flow_representation_RCE} holds.
    The formula for $\frac{\dd}{\dd t} \|\rho_t\|_{L^p_x}^p$ can then be proved going through the same arguments as in Proposition \ref{prop:absolute-continuity-Lp}.

    Existence and uniqueness of $f\in \cB_b([0,T];L^1_x \cap L^\infty_x)$ again comes from Theorem \ref{thm:uniqueness-linear-RPDE}. In order to prove \eqref{eq:flow_representation_RTE}, first notice that thanks to the quasi-incompressibility of the flow $\Phi$, the integrability of $\rho_0$ and standard density arguments, formula \eqref{eq:flow_representation_RCE} extends to any $\varphi\in L^p_x$, for any $p\in [1,\infty]$.
    In light of this, by applying the duality formula, for any fixed $t\in [0,T]$ we find
    \begin{align*}
        \int_{\R^d} f_0(x) \rho_0(x) \dd x = \int_{\R^d} f_t(x) \rho_t(x) \dd x = \int_{\R^d} f_t(\Phi_t(x)) \rho_0(x) \dd x.
    \end{align*}
    As the argument works for any choice of $\rho_0\in L^1_x\cap L^\infty$, we deduce that $f_0(x)=f_t(\Phi_t(x))$ for Lebesgue a.e. $x$, from which conclusion follows by applying $\Phi_t^{-1}$ on both sides (by quasi-incompressibility, $\Phi_t^{-1}$ preserves negligible sets).
\end{proof}

\begin{remark}
    In line with the Di Perna--Lions theory \cite{diperna1989ordinary}, it is reasonable to expect the Lagrangian representation formulas to hold true under the sole assumption \eqref{eq:condition-uniqueness}, without the need for Osgood regularity (Assumption \ref{ass:osgood-drift-RDE}).
    This would require to introduce a suitable concept of \emph{(rough) generalized Lagrangian flow}, in analogy to the deterministic theory \cite{diperna1989ordinary,Ambrosio2004,CriDeL2008}; we leave this problem for future investigations.
\end{remark}

\section{Nonlinear rough continuity equations}\label{sec:nonlinear-RPDEs}

Like in the majority of Section \ref{sec:linear-RPDEs}, we will enforce throughout this section that $\nabla\cdot \xi=0$. Here, if not stated otherwise, we work on a fixed finite time interval $[0,T]$.

Having established a robust wellposedness theory for linear rough PDEs, we now pass to the study a class of nonlinear continuity equations on $\R^d$ of the form
\begin{equation}\label{eq:nonlinear-RPDE}
    \rmd \rho_t + \nabla\cdot [(K_t\oast \rho_t) \rho_t]\rmd t  + \sum_{k=1}^m \nabla \cdot (\xi_k \rho_t)\rmd \bZ_t^k=0.
\end{equation}
Here $d\geq 2$ and $K:[0,T]\times\R^d\times\R^d\to \R^d$, $K=K_t(x,y)$, is a given measurable kernel satisfying suitable assumptions (see below); for $f:\R^d\to\R$, we define
\begin{equation*}
(K_t\oast f)(x)=\int_{\R^d} K_t(x,y) f(y) \dd y, \quad (K_t\oast^T f)(x):=\int_{\R^d} K_t(y,x) f(y) \dd y.
\end{equation*}
The notation is meant to stress that, at least formally, $(K\oast)^\ast=K\oast^T$: whenever everything is sufficiently integrable, it holds $\langle K\oast f,g \rangle=\langle f, K\oast^T g\rangle$.
As before, the RPDE \eqref{eq:nonlinear-RPDE} is coupled with an initial condition $\rho\vert_{t=0}=\rho_0$.
We will impose three relevant conditions on the kernel $K$, which play different roles in the proofs and cover an interesting class of examples; combined together, they guarantee wellposedness of \eqref{eq:nonlinear-RPDE} for $\rho_0\in L^1_x\cap L^\infty_x$, see Theorem \ref{thm:nonlinear-wellposedness2} for more details.

\begin{assumption}\label{ass:abstract-kernel-1}
There exist a constant $C_K>0$ and an Osgood modulus $h$ such that
\begin{align}
& \int_{\R^d} |K_t(x,y)| |f(y)| \dd y \leq C_K \| f\|_{L^1_x\cap L^\infty_x}, \label{eq:abstract-kernel-eq1} \\
& \int_{\R^d} |K_t(x,y)-K_t(x',y)| |f(y)| \dd y
\leq C_K\, h(|x-x'|) \| f\|_{L^1_x\cap L^\infty_x}, \label{eq:abstract-kernel-eq2}
\end{align}
for all $t\in [0,T]$ and $f\in L^1_x\cap L^\infty_x$.
Moreover, the same estimates hold with $K_t(x,y)$ replaced by $\tilde K_t(x,y)=K_t(y,x)$.
\end{assumption}

\begin{assumption}\label{ass:abstract-kernel-2}
There exists a constant $C_K>0$ such that
\begin{equation*}
\|\nabla\cdot (K_t\oast f)\|_{L^\infty_x} \leq C_K \| f\|_{L^1_x} \quad \forall\, t\in [0,T],\, f\in L^1_x.
\end{equation*}
\end{assumption}

\begin{assumption}\label{ass:abstract-kernel-3}
For any $R>0$, it holds that
\begin{equation*}
\| \nabla (K_t \oast f)\|_{L^1(B_R)} \leq C_{K,R} \| f\|_{L^1_x\cap L^\infty_x} \quad \forall\, t\in [0,T],\,  f\in L^1_x\cap L^\infty_x. 
\end{equation*}
for a suitable constant $C_{K,R}>0$.
\end{assumption}

\begin{remark}\label{rem:abstract-ass-1}
Assumption \ref{ass:abstract-kernel-1} guarantees that, given any $f\in L^1_x\cap L^\infty_x$ and any $t\in [0,T]$, $K_t\oast f\in C^0_b$ with Osgood modulus of continuity proportional to $h$; moreover the same holds for $K\oast^T f$.
Assumption \ref{ass:abstract-kernel-2} provides a control on the quasi-incompressibility properties of $K\oast f$, in a way which only depends on $\| f\|_{L^1_x}$.
Finally, Assumption \ref{ass:abstract-kernel-3} implies that $K_t\oast f\in W^{1,1}_{\loc}$.
\end{remark}

\subsection{Existence}%\label{susbec:nonlinear-existence}

In agreement with Definition \ref{defn:solution-rough-continuity}, we adopt the following terminology.

\begin{definition}\label{defn:solution_nonlinear_RPDE}
    We say that a map $\rho\in \cB_b([0,T];L^1_\loc)$ is a \emph{weak solution to the nonlinear RPDE \eqref{eq:nonlinear-RPDE}} if, for all $R\in [1,\infty)$, the following hold:
    \begin{itemize}
        \item[i.] $\nabla\cdot ([K\oast \rho] \rho)$ is a well-defined distribution, belonging to $L^1_t \cF_{-2,R}$;
        \item[ii.] there exists $\rho^\natural$ beloging to $C^{\fkp/3-\var}_2 \cF_{-3,R}$ such that
    \begin{equation*}
        \delta \rho_{st} + \int_s^t \nabla\cdot ([K_r\oast \rho_r] \rho_r) \dd r + A_{st} \rho_s = \bbA_{st} \rho_s + \rho^\natural_{st}\quad \forall\, (s,t)\in \Delta_T
    \end{equation*}
    where $\bA=(A,\bbA)$ is the unbounded rough driver associated to $(\xi,\bZ)$ given by \eqref{eq:defn-unbounded-continuity}.
    \end{itemize}   
\end{definition}

\begin{remark}
    Under condition \eqref{eq:abstract-kernel-eq1} on $K$, if $\rho\in L^2_t (L^1_x\cap L^\infty_x)$, then necessarily $\nabla\cdot ([K\oast \rho] \rho) \in L^1_t \cF_{-1,R}\subset L^1_t \cF_{-2,R}$, thanks to the pointwise estimate
    \begin{align*}
        \| \nabla\cdot ([K_r\oast \rho_r] \rho_r)\|_{\cF_{-1,R}}
        \lesssim \| [K_r\oast \rho_r] \rho_r\|_{L^1(B_R)}
        \leq \| K_r\oast \rho_r\|_{L^\infty_x} \| \rho_r\|_{L^1_x} 
        \lesssim C_K \| \rho_r\|_{L^1_x\cap L^\infty_x}^2.
    \end{align*}
    There are however interesting situations where less integrability on $\rho$ is needed and bounds in $\cF_{-2,R}$ become crucial, cf. Section \ref{subsec:nonlinear-euler-delort}.   
\end{remark}

\begin{proposition}\label{prop:nonlinear-existence}
    Let $\fkp\in [2,3)$, $\bZ\in C^\fkp_g$, $\xi\in C^2_b$ with $\nabla\cdot\xi=0$, and let $K$ be a kernel satisfying Assumptions \ref{ass:abstract-kernel-1}-\ref{ass:abstract-kernel-2}.
    Then for any $\rho_0\in L^1_x\cap L^\infty_x$, there exists a solution $\rho\in \cB_b([0,T];L^1_x\cap L^\infty_x)$ to \eqref{eq:nonlinear-RPDE}, which moreover satisfies the estimate
    \begin{equation} \label{eq:nonlinear-existence estimate}
       \sup_{t \in [0,T]} \|\rho_t\|_{L^p_x} \leq \exp\left( \left(1-\frac{1}{p} \right) T C_K\|\rho_0\|_{L^1_x} \right) \|\rho_0\|_{L^p_x} \quad \forall\, p\in [1,\infty].
    \end{equation}
%   for all $p \in [1,\infty]$. 
\end{proposition}

\begin{proof}
    %\red{L: All main estimates outlined, in case to come back to add more details.}
    %
    The proof is based on a priori estimates and compactness arguments, similar to the ones already presented in Proposition \ref{prop:existence-linear-RPDE} and Lemma \ref{lem:equi-p-integrability}.
    For this reason, we will sketch several passages and mostly focus on the ones related to the nonlinear kernel $K$, which require a different treatment.
    
    Let $\{\chi^\eps\}_{\eps>0}$ be a family of standard mollifiers on $\R^d$ associated to a radial $\chi\in C^\infty_c$. Correspondingly, consider the mollified kernels $K^\eps$ given by
    \begin{align*}
        K_t^\eps(x,y):=\int_{\R^d\times \R^d} \chi^\eps(x-\tilde x) K_t(\tilde x,\tilde y) \chi^\eps(\tilde y- y) \dd \tilde x\, \dd \tilde y;
    \end{align*}
    observe that they act on functions $f$ by the relation
    \begin{align*}
        K_t^\eps \oast f = \chi^\eps \ast (K_t\oast (\chi^\eps \ast f)).
    \end{align*}
    Using the above relations, it's easy to verify that the family $\{K^\eps\}_{\eps>0}$ still satisfies Assumptions \ref{ass:abstract-kernel-1}-\ref{ass:abstract-kernel-2} with same modulus $h$ and uniform-in-$\eps$ constants.
    Given $\rho_0\in L^1_x\cap L^\infty_x$, $\bZ\in \clC^{\mathfrak{p}}_g$, consider suitable smooth approximations $\{\rho_0^\eps\}_\eps$ and $\bZ^\eps$, such that
    \begin{align*}
        \bZ^\eps\to \bZ \text{ in } C^\fkp_g, \quad
        \rho^\eps_0\to \rho_0 \text{ in } L^p_x \quad \forall\,p\in [1,\infty), \quad
        \| \rho^\eps_0\|_{L^q}\leq \| \rho_0\|_{L^q} \quad \forall\, q\in [1,\infty]
    \end{align*}
    %(e.g. as in Proposition \ref{prop:existence-linear-RPDE})
    and define the mollified nonlinear PDEs
    \begin{equation}\label{eq:nonlinear-existence-eq1}
        \partial_t \rho_t^\eps + \nabla \cdot [(K_t^\eps\oast \rho_t^\eps) \rho_t^\eps] + \sum_k \nabla\cdot(\xi_k \rho^\eps_t)  \,\dot Z^{k,\eps}_t = 0.
    \end{equation}
    Since for $\eps>0$ the kernel $K^\eps$ is smooth and $\xi\in C^2_b$, standard results (dating back at least to Dobrushin \cite{Dobrushin1979}) ensure existence and uniqueness of solutions to \eqref{eq:nonlinear-existence-eq1}.
    By Remark \ref{rem:defn-solution-rough-continuity}, such solutions solve the nonlinear RPDE associated to $\bZ^\eps$ and $u^\eps:= K^\eps\oast \rho^\eps$; moreover $\rho^\eps_t = (\Phi^\eps_t)_\sharp \rho^\eps_0$ where $\Phi^\eps_t$ is the flow associated to $(u^\eps, \xi, \bZ^\eps)$.
    Observe that, by properties of the push-forward map, it holds
    \begin{equation}\label{eq:nonlinear-uniqueness-eq2}
        |\rho^\eps_t|=(\Phi^\eps_t)_\sharp |\rho^\eps_0|,
    \end{equation}
    where $|\mu|=\mu^+ +\mu^-$ denotes the variation of a signed measure $\mu$;
    in particular, $\| \rho^\eps_t\|_{L^1_x} = \| \rho^\eps_0\|_{L^1_x} \leq \| \rho_0\|_{L^1_x}$.
    By Assumption \ref{ass:abstract-kernel-2}, this implies that
    \begin{equation*}%\label{eq:uniform divergence bound}
        \sup_{\eps>0,\, t\geq 0} \| \nabla\cdot u^\eps_t\|_{L^\infty_x} \leq C_K \| \rho^\eps_t\|_{L^1_x} \leq C_K \|\rho_0\|_{L^1_x}
    \end{equation*}
    which in turn by estimate \eqref{eq:lq-bound} and the choice of approximations yields
   \begin{equation} \label{eq:uniform infty bound}
        \sup_{\eps>0,t\in [0,T]} \| \rho^\eps_t\|_{L^p_x} \leq \exp \bigg( \Big(1-\frac{1}{p}\Big) C_K T \| \rho_0\|_{L^1_x}\bigg) \| \rho_0\|_{L^p_x} \quad \forall\, p\in [1,\infty].
    \end{equation}
    Combining this estimate (for $p\in{1,\infty}$) together with assumptions \eqref{eq:abstract-kernel-eq1}-\eqref{eq:abstract-kernel-eq2}, we get
    \begin{equation}\label{eq:nonlinear-existence-eq3}
        \| u^\eps_t\|_{L^\infty_x} + \sup_{x\neq y} \frac{|u^\eps_t(x)-u^\eps_t(y)|}{h(|x-y|)} \lesssim \| \rho_0\|_{L^1_x\cap L^\infty_x}
    \end{equation}
    uniformly in $t\in [0,T]$ and $\eps>0$.
    Arguing as in Proposition \ref{prop:existence-linear-RPDE}, Lemma \ref{lem:equi-p-integrability} and Corollary \ref{cor:linear-RPDE-stability-1}, one can then check that the assumptions of Propositions \ref{prop:compactness-l1}-\ref{prop:compactness-lp} are satisfied;
    %: i) the maps $\{t\mapsto \rho^\eps_t\}_\eps$ are equicontinuous in $\cF_{-1,R}$ for any $R>1$; ii) the maps $\{\rho^\eps_t\}_{\eps,t}$ are equi-integrable; by Proposition \ref{prop:compactness-l1},
    therefore we can extract a (not relabelled) subsequence such that $\rho^\eps \rightarrow \rho$ in $C_w([0,T];L^1_x)$.
    Since $\rho^{\eps}$ is bounded in $\cB_b([0,T];L^1_x \cap L^\infty_x)$ by \eqref{eq:uniform infty bound}, we also get convergence in $C_w([0,T];L^p_x)$ for $p \in (1,\infty)$ and in $C_{w-\ast}([0,T];L^\infty_x)$.
    
    It remains to show that $\rho_t$ is a weak solution to the nonlinear RPDE, in the sense of Definition \ref{defn:solution_nonlinear_RPDE}; we will only show convergence of the nonlinear term, namely that for all $t\geq 0$ it holds
    \begin{equation}\label{eq:nonlinear-existence-eq4}
        \langle \varphi, \nabla (u^\eps_t \rho^\eps_t)\rangle
        = - \langle \nabla \varphi, (K^\eps_t\oast \rho^\eps_t) \rho_t^\eps\rangle
        \to - \langle \nabla \varphi, (K_t \oast  \rho_t) \rho_t\rangle
        = \langle \varphi, \nabla ([K_t\oast \rho_t] \rho_t)\rangle
    \end{equation}
    since the other terms can be treated as in Proposition \ref{prop:existence-linear-RPDE}.

    Towards this end, recall that $K^\eps_t\oast \rho^\eps_t = \chi^\eps \ast (K_t\oast (\chi^\eps \ast \rho^\eps_t) )$ and $\rho^\eps_t \rightharpoonup \rho_t$ in $L^1_x$; by strong continuity of convolutions, it holds $\chi^\eps\ast \rho^\eps_t \rightharpoonup \rho_t$ in $L^1_x$ as well.
    On the other hand, for any $\varphi\in C^\infty_c$ it holds $\chi^\eps\ast \varphi\to \varphi$ in $L^1_x\cap L^\infty_x$ and by Assumption \ref{ass:abstract-kernel-1}, the dual operator $K\oast^T\in \cL(L^1_x\cap L^\infty_x; C^0_b)$. Combining all these observations, by weak-strong convergence, for any $\varphi\in C^\infty_c$ it holds that
    \begin{align*}
        \langle \varphi, K^\eps_t\oast \rho^\eps_t\rangle
        = \langle K_t\oast^T (\chi^\eps \ast \varphi), \chi^\eps\ast \rho^\eps_t\rangle
        \to \langle K_t\oast^T \varphi, \rho_t\rangle.
    \end{align*}
    For fixed $t$, by \eqref{eq:nonlinear-existence-eq3} and Ascoli-Arzelà, the convergence of $K_t^{\eps} \oast \rho_t^{\eps}$ can be updated to uniform on compact sets, i.e. $K^\eps_t\oast \rho^\eps_t \to K_t\oast \rho_t$ in $C^0_\loc$.
    Again by weak-strong convergence, it then follows that $(K^\eps_t\oast \rho^\eps_t) \rho_t^\eps \rightharpoonup (K_t\oast \rho_t) \rho_t$ in $L^1_\loc$, yielding \eqref{eq:nonlinear-existence-eq4}.

    To see that \eqref{eq:nonlinear-existence estimate} holds, it suffices to pass to the limit in \eqref{eq:uniform infty bound}, using the lower semicontinuity of $\| \cdot\|_{L^p_x}$ w.r.t. weak convergence (weak-$\ast$ convergence for $p=\infty$).
\end{proof}

\subsection{Uniqueness and stability}\label{subsec:nonlinear-uniqueness}

\begin{proposition}\label{prop:nonlinear-uniqueness}
    Let $\fkp\in [2,3)$, $\bZ\in C^\fkp_g$, $\xi\in C^3_b$ with $\nabla\cdot\xi=0$, and let $K$ be a kernel satisfying Assumptions \ref{ass:abstract-kernel-1}-\ref{ass:abstract-kernel-2}-\ref{ass:abstract-kernel-3}.
    Then for any $\rho_0\in L^1_x\cap L^\infty_x$, there can be at most one solution $\rho\in \cB_b([0,T]; L^1_x\cap L^\infty_x)$ to \eqref{eq:nonlinear-RPDE}. Moreover, $\rho\in C([0,T];L^p_x)\cap C_{w-\ast}([0,T];L^\infty_x)$ for all $p\in [1,\infty)$ and
    $\rho$ is given by 
    $$
        \rho_t = (\Phi_t)_{\sharp} \rho_0
    $$
    where $\Phi$ is the flow generated by the RDE
    $$
        \rmd y_t = u_t(y_t) \rmd t + \xi(y_t) \rmd \bZ_t
    $$
    for $u_t = K_t \oast \rho_t$. 
\end{proposition}

\begin{proof}
    Existence of a solution $\rho\in \cB_b([0,T];L^1_x\cap L^\infty_x)$ comes from Proposition \ref{prop:nonlinear-existence}.
    By Remark \ref{rem:abstract-ass-1}, for $u:=K\oast \rho$, we have
    \begin{equation}\label{eq:nonlinear-uniqueness-eq0}
        u\in L^\infty_t L^\infty_x, \quad
        \nabla\cdot u\in L^\infty_t L^\infty_x, \quad
        u\in L^\infty_t W^{1,1}_\loc
    \end{equation}
    and $u$ is Osgood continuous; therefore the flow representation and continuity properties
    of $\rho$ follow from Theorem \ref{thm:linear-RPDE-wellposed-lagrangian}. It only remains to show uniqueness of solutions.
    
    Let $\rho_0\in L^1_x\cap L^\infty_x$ be fixed and assume we are given two distinct solutions $\rho^i\in L^{\infty}_t (L^1_x\cap L^\infty_x)$, $i=1,2$.
    As above, $\rho^i$ are solutions to the linear RPDEs with drifts $u^i_t := K_t\oast \rho^i_t$, which are Osgood continuous and satisfy \eqref{eq:nonlinear-uniqueness-eq0}, therefore $\rho^i= (\Phi^i_t)_\sharp \rho_0$, where $\Phi^i$ are the flows associated to the RDEs
    \begin{equation}\label{eq:nonlinear-uniqueness-eq1}
        \dd \Phi^i_t(x) = u^i_t(\Phi^i_t(x)) \dd t + \xi (\Phi^i_t(x)) \dd \bZ_t.
    \end{equation}
    Moreover arguing as in \eqref{eq:nonlinear-uniqueness-eq2}, it holds $|\rho^i_t|=(\Phi^i_t)_\sharp |\rho_0|$. %Observe that, by properties of the push-forward map, it holds
    Define the quantity
    \begin{align*}
        I_t  := \int_{\R^d} |\Phi^1_t(x)-\Phi^2_t(x)| |\rho_0(x)| \dd x;
    \end{align*}
    in order to conclude, it suffices to show that $I_t\equiv 0$ for all $t\in [0,T]$. Indeed if this is the case, then for any $\varphi\in C^\infty_c$ it holds
    \begin{align*}
        |\langle \varphi, \rho^1_t-\rho^2_t\rangle|
        = \bigg|\int_{\R^d} \big[ \varphi(\Phi^1_t(x)) -  \varphi(\Phi^2_t(x)) \big] \rho_0(x) \dd x\bigg|
        \lesssim \int_{\R^d} |\Phi^1_t(x) - \Phi^2_t(x)| |\rho_0(x)| \dd x = 0
    \end{align*}
    implying that $\rho^1_t = \rho^2_t$ for all $t\in [0,T]$.

    In order to estimate $I_t$, we need to manipulate the RDEs \eqref{eq:nonlinear-uniqueness-eq1}.
    First observe that each of them has associated forcing terms $\mu^i_t(x)=\int_0^t u^i_s(\Phi^i_s(x)) \dd s$ satisfying
    \begin{align*}
        \| \mu^i(x)\|_{1,T} \leq \| u^i\|_{L^1_t L^\infty_x} \lesssim \| \rho^i\|_{\cB_b([0,T];L^1_x\cap L^\infty_x)} \quad\forall\, x\in\R^d.
    \end{align*}
    Therefore we can apply Lemma \ref{lem:contraction-RDE-forcing} to deduce that
    \begin{align*}
        |\Phi^1_t(x)-\Phi^2_t(x)| \lesssim \int_0^t |u^1_s(\Phi^1_s(x))-u^2_s(\Phi^2_s(x))| \dd s 
    \end{align*}
    where the hidden constant is uniform in $t\in [0,T]$ and $x\in \R^d$.
    Integrating in space and applying Fubini, we find
    \begin{align*}
        I_t
        & \lesssim \int_0^t \int_{\R^d} |u^1_s(\Phi^1_s(x))-u^2_s(\Phi^2_s(x))| |\rho_0(x)|\dd x\\
        & \lesssim \int_0^t \int_{\R^d} \big[ |u^1_s(\Phi^1_s(x))-u^1_s(\Phi^2_s(x))| |\rho_0(x)| \big]
        + \big[ |(K_s\oast \rho^1_s)(\Phi^2_s(x))-(K_s\oast \rho^2_s)(\Phi^2_s(x))| |\rho_0(x)|\big] \dd x\, \dd s\\
        & =: \int_0^t (I_s^1 + I_s^2) \dd s.
    \end{align*}
    By Assumption \ref{ass:abstract-kernel-1}, $u^1_t$ is Osgood with modulus proportional to $\| \rho_t\|_{L^1\cap L^\infty} h$; by Remark \ref{rem:modulus-1}, we can assume $h$ to be concave. Then by Jensen's inequality we find
    \begin{align*}
        I^1_s
        \lesssim \| \rho^1_s\|_{L^1_x\cap L^\infty_x} \int_{\R^d} h(|\Phi^1_s(x))-\Phi^2_s(x)|) |\rho_0(x)|\dd x
        \lesssim \| \rho^1_s\|_{L^1_x\cap L^\infty_x}\, h( I_s ) .
    \end{align*}
    On the other hand, by applying the relations $\rho^i_s = (\Phi^i_s)_\sharp \rho_0$, $|\rho^2_s|=(\Phi^2_s)_\sharp |\rho_0|$ and assumption \eqref{eq:abstract-kernel-eq2}  (for $\tilde K$ in place of $K$), it holds
    \begin{align*}
        I^2_s
        & = \int_{\R^d} |(K_s\oast \rho^1_s)(x)-(K_s\oast \rho^2_s)(x))| |\rho^2_s(x)|\dd x\\
        & = \int_{\R^d} \bigg|\int_{\R^d} [K_s(x,\Phi^1_s(y)) - K_s(x,\Phi^2_s(y))] \dd y \bigg|\,|\rho^2_s(x)|\dd x\\
        & \leq \| \rho^2_s\|_{L^\infty_x} \int_{\R^d} \int_{\R^d} |K_s(x,\Phi^1_s(y)) - K_s(x,\Phi^2_s(y))| \dd x\, \dd y\\
        &  \lesssim \| \rho^2_s\|_{L^\infty_x} \int_{\R^d} h( |\Phi^1_s(y)) - \Phi^2_s(y)) |) \dd y
        \lesssim \| \rho^2_s\|_{L^\infty_x}  h(I_s).
    \end{align*}
    Combining everything, we arrive at
    \begin{align*}
        I_t \lesssim \int_0^t (\| \rho^1_s\|_{L^1_x\cap L^\infty_x} + \| \rho^2_s\|_{L^1_x\cap L^\infty_x}) h(I_s) \dd s
    \end{align*}
    which implies the conclusion by Lemma \ref{lem:bihari}.
\end{proof}

We are now ready to present the

\begin{proof}[Proof of Theorem \ref{thm:intro_wellposed_nonlinear}]
    In light of Proposition \ref{prop:nonlinear-uniqueness}, we only need to verify that the convolutional kernel $K$ satisfies Assumptions \ref{ass:abstract-kernel-1}-\ref{ass:abstract-kernel-2}-\ref{ass:abstract-kernel-3}.
    Let $K=K^1+K^2$ with $K^1\in L^1_x$, $K^2\in L^\infty_x$, as granted by assumption \textit{i.} Then by Young's inequalities
    \begin{align*}
        & \big\| |K|\ast |f| \big\|_{C^0_b}
        \leq \big\| |K^1|\ast |f| \big\|_{C^0_b} + \leq \big\| |K^2|\ast |f| \big\|_{C^0_b}
        \leq \| K^1\|_{L^1_x} \| f\|_{L^\infty_x} + \| K^1\|_{L^2_x} \| f\|_{L^1_x},\\
        & \| \nabla\cdot (K\ast f)\|_{L^\infty_x}
        \leq \| \nabla\cdot K\|_{L^\infty_x} \| f\|_{L^1_x},
    \end{align*}
    which verifies \eqref{eq:abstract-kernel-eq1} and Assumption \ref{ass:abstract-kernel-2}.
    Moreover by \textit{ii.} and Mihlin's multiplier theorem \cite[Theorem 6.2.7]{Grafakos.classical}, for any $p\in (1,\infty)$ we have
    \begin{align*}
        \| \nabla (K\ast f)\|_{L^p_x}
        = \| (\nabla K)\ast f\|_{L^p_x}
        \lesssim_p \| f\|_{L^p_x} \leq \| f\|_{L^1_x\cap L^\infty_x}
    \end{align*}
    verifying Assumption \ref{ass:abstract-kernel-3}.
    It remains to check property \eqref{eq:abstract-kernel-eq2}; to this end, using Young's inequality as above, it suffices to show that $K=\tilde K^1+\tilde K^2$, where $\tilde K^1\in C^1_b$ and
    \begin{equation*}%\label{eq:thm_intro_nonlinear_proof1}
        \| \tilde K^2(x+\cdot)-\tilde K^2\|_{L^1_x} \lesssim h(|x|) \quad \forall\, x\in \R^d,
    \end{equation*}
    for some Osgood modulus of continuity $h$.
    We claim that we can achieve the above decomposition with $\tilde K^2\in L^1_x$ and
    \begin{align*}
        h(r) = r(1-\ln r) \mathbbm{1}_{(0,1)}(r) + r \mathbbm{1}_{[1,\infty)}(r);
    \end{align*}
    since $\tilde K^2\in L^1_x$, only the behaviour of $h$ for small enough $r$ matters, so it suffices to show that
    \begin{equation}\label{eq:thm_intro_nonlinear_proof2}
        \| \tilde K^2(x+\cdot)-\tilde K^2\|_{L^1_x} \lesssim -|x| \log |x| \quad \forall\, x\in B_{1/2}.
    \end{equation}
    To achieve the desired decomposition, we work with inhomogeneous Littlewood-Paley blocks $\{\Delta_j\}_{j\geq -1}$, see \cite{BCD}.
    Since $K\in L^1_x+L^\infty_x$, $\Delta_j K\in C^\infty_b$ for all $j\geq -1$, therefore we can set
    \begin{align*}
        \tilde K^1=\Delta_{-1} K + \Delta_0 K, \quad \tilde K^2=\sum_{j\geq 1} \Delta_j K.
    \end{align*}
    Thanks to assumption \textit{ii.} and the action of multipliers on functions whose Fourier transform is supported in an annulus \cite[Lemma 2.2]{BCD}, for every $j\geq 1$ it holds
    \begin{align*}
        \| \Delta_j \tilde K^2\|_{L^1_x}
        \leq \| \Delta_{j-1} K\|_{L^1_x} + \| \Delta_{j} K\|_{L^1_x} + \| \Delta_{j+1} K\|_{L^1_x}
        \lesssim 2^{-j}
    \end{align*}
    so that $\tilde K^2$ belongs to the Besov space $B^1_{1,\infty}$ (thus also in $L^1_x$).
    
    To conclude, it suffices to show that any $f\in B^1_{1,\infty}$ satisfies \eqref{eq:thm_intro_nonlinear_proof2}; by homogeneity, we may assume $\|f\|_{B^1_{1,\infty}}=1$.
    Let $x\in\R^d$, $|x|\leq 1/2$ be fixed; then by Bernstein inequalities and properties of translations in $W^{1,1}$, for any $N\in \N$ it holds that
    \begin{align*}
        \| f(x+\cdot) - f\|_{L^1_x}
        & \leq \sum_{j=-1}^\infty \| (\Delta_j f)(x+\cdot) - \Delta_j f\|_{L^1_x}\\
        & \lesssim \sum_{j<N} |x| \| \nabla \Delta_j f\|_{L^1_x} + \sum_{j \geq N} 2 \| \Delta_j f\|_{L^1_x} \\
        &\lesssim N |x| + \sum_{j \geq N} 2^{-j}
        \sim N |x| + 2^{-N}.
    \end{align*}
    Taking $N$ such that $2^{-N} \sim |x|$ (which is allowed since $|x|\leq 1/2$), conclusion follows.
\end{proof}

Combining Propositions \ref{prop:nonlinear-existence}-\ref{prop:nonlinear-uniqueness} and the stability results from Section \ref{subsec:linear-stability}, we obtain a comprehensive well-posedness statement (in the sense of Hadamard) for \eqref{eq:nonlinear-RPDE} in the class of initial conditions $\rho_0\in L^1_x\cap L^\infty_x$.

\begin{theorem}\label{thm:nonlinear-wellposedness2}
    Let $\fkp \in [2,3)$. Consider a sequence $\{(\mathbf{Z}^n, \xi^n, K^n)\}_n$ satisfying the assumptions of Proposition~\ref{prop:nonlinear-uniqueness} uniformly in $n$, namely such that:
    \begin{itemize}
        \item[i)] $\nabla\cdot \xi^n=0$ and $\bZ^n\in \mathcal{C}^\fkp_g$ for all $n$;
        \item[ii)] $\sup_{n \in \N} \big\{ \| \xi^n\|_{C^3_b}  +\| \mathbf{Z}^n\|_{\mathfrak{p},T} \big\}<\infty$;
        \item[iii)] the kernels $K^n$ satisfy Assumptions \ref{ass:abstract-kernel-1}-\ref{ass:abstract-kernel-2}-\ref{ass:abstract-kernel-3} with the same constants $C_K, C_{K,R}$ independent of $n$.
    \end{itemize} 
    Further assume that there exist $(\xi,\mathbf{Z},K)$ such that:
    \begin{itemize}
        \item[1)] $\xi^n\to \xi$ uniformly on compact sets;
        \item[2)] $\sup_{t\in [0,T]} |\mathbf{Z}^n_{0,t}-\mathbf{Z}_{0,t}|\to 0$;
        \item[3)] for all $\varphi\in C^\infty_c$, $K^n_t\oast^T \varphi\to K_t\oast^T \varphi$ strongly in $L^\infty_x$, for all $t\in [0,T]$.
    \end{itemize}
    Let $\{\rho^n_0\}_n$ be a bounded sequence in $L^1_x\cap L^\infty_x$, resp. $\rho_0\in L^1_x\cap L^\infty_x$, and denote by $\rho^n$, resp. $\rho$, the associated unique solution to \eqref{eq:nonlinear-RPDE} associated to $(\mathbf{Z}^n, \xi^n, K^n)$, resp. $(\mathbf{Z}, \xi, K)$, coming from Proposition \ref{prop:nonlinear-uniqueness}. Then:
    \begin{itemize}
        \item[a)] if $\rho^n_0\rightharpoonup \rho_0$ in $L^1_x$, then $\rho^n\to \rho$ in $C_w([0,T];L^q_x)\cap C_{w-\ast}([0,T];L^\infty_x)$ for all $q\in [1,\infty)$;
        \item[b)] if either $\rho^n_0\xrightharpoonup{\ast} \rho_0$ in $L^{\infty}_x$, or $\rho_0^n \rightharpoonup \rho_0$ in $L^p_x$ for some $p \in (1,\infty)$, then $\rho^n\to \rho$ in $C_w([0,T];L^q_x)\cap C_{w-\ast}([0,T];L^\infty_x)$ for all $q\in (1,\infty)$.
    \end{itemize}
    If in addition to 1)-2)-3) we also have
    \begin{itemize}
        \item[4)] for any $t\in [0,T]$, the linear operators $g\mapsto \nabla\cdot( K^n_t \oast g)$ are uniformly compact from $L^1_x\cap L^\infty_x$ to $L^1_\loc$, in the following sense: if $\{g^n\}_n$ is a bounded sequence in $L^1_x\cap L^\infty_x$ such that $g^n\xrightharpoonup{\ast} g$ in $L^{\infty}_x$, then $\nabla\cdot (K^n_t\oast g^n) \to \nabla\cdot (K_t\oast g)$ in $L^1_\loc$,
    \end{itemize}
    then:
    \begin{itemize}
        \item[c)] if $\rho^n_0\to \rho_0$ in $L^p_x$ for some $p\in [1,\infty)$, then $\rho^n\to \rho$ in $C([0,T]; L^q_x)$ for all $q\in \{p\}\cup (1,\infty)$.
    \end{itemize}
\end{theorem}

\begin{proof}
We start by proving b).
First notice that, since $\{\rho_0^n\}_n$ is bounded in $L^1_x \cap L_x^{\infty}$, weak convergence in $L^p_x$ for some $p\in (1,\infty)$ is equivalent to weak convergence in any other $L^{\tilde{p}}_x$ with $\tilde p\in (1,\infty)$, as well as to weak-$\ast$ convergence in $L^\infty_x$.
Moreover, since $\{\rho_0^n\}_n$ is bounded in $L^1_x \cap L_x^{\infty}$, it follows from \eqref{eq:nonlinear-existence estimate} that $\{\rho^n\}_{n \geq 1}$ is bounded in $\mathcal{B}_b([0,T];L^p_x)$ for any $p \in [1,\infty]$.
Arguing similarly to the proof of Proposition \ref{prop:nonlinear-existence}, we can extract a subsequence $\{\rho^{n_k}\}_{k \geq 1}$ which converges in $C_{w-\ast}([0,T];L^{\infty}_x)$ and $C_w([0,T];L^q_x)$ for all $q\in (1,\infty)$ to some limit $\rho\in \cB_b([0,T];L^1_x\cap L^\infty_x)$.\footnote{Since in b) we are not assuming $\rho^n_0\rightharpoonup \rho_0$ weakly in $L^1_x$, in general we do not have tightness estimates, so that we cannot invoke Lemma \ref{lem:equi-p-integrability} and Proposition \ref{prop:compactness-l1} to deduce compactness in $C_w([0,T];L^1_x)$.}
Once we show that $\rho$ solves the nonlinear RPDE, uniqueness implies convergence of the full sequence $\{\rho^n\}_{n \geq 1}$, thus proving c). 

We focus on the convergence of the drift terms; setting $u^n := K^n \oast \rho^n$, $u := K \oast \rho$, we want to show that $u^{n_k}\to u$ in $L^1_t C^0_\loc$.
Once this is shown, the rest of the argument is almost identical to that of Corollary \ref{cor:linear-RPDE-stability-1}.
%As before, we focus only on convergence of the nonlinear term. For notational convenience, we write $u^n := K^n \oast \rho^n$ and $u := K \oast \rho$ and first show that $u^{n_k}_t\to u_t$ in the sense of distributions.
We start by showing that, for any fixed $t\in [0,T]$, $u^{n_k}_t\to u_t$ in the sense of distributions.
Indeed, for any $\varphi \in C_c^{\infty}$ and $t \in [0,T]$ we have
$$
    \langle u_t^{n_k}, \varphi \rangle  = \langle K^{n_k}_t \oast \rho_t^{n_k}, \varphi \rangle  = \langle  \rho_t^{n_k}, K^{n_k}_t\oast^T\varphi \rangle \rightarrow \langle  \rho_t, K_t\oast^T\varphi \rangle = \langle  K_t \oast \rho_t, \varphi \rangle = \langle u_t, \varphi \rangle
$$
where we used weak-strong convergence and assumption 3). 
Next, notice that thanks to the uniform bounds in Assumptions \ref{ass:abstract-kernel-1}-\ref{ass:abstract-kernel-2}- \ref{ass:abstract-kernel-3} and the uniform boundedness of $\rho^n$ in $\cB_b([0,T];L^1_x\cap L^\infty_x)$, we have
\begin{align*}
    \| u^n_t\|_{L^\infty_x} + \sup_{x\neq y} \frac{|u^n_t(x)-u^n_t(y)|}{h(|x-y|)} <\infty.
\end{align*}
Therefore by Ascoli--Arzelà, for any fixed $t$, $\{u^{n_k}_t\}$ is compact in $C^0_\loc$ and weakly converging to $u_t$; thus $u^{n_k}_t\to u_t$ in $C^0_\loc$ as well.
As the argument holds for any $t$, $u^{n_k}\to u$ in $L^1_t C^0_\loc$ by dominated convergence, which overall concludes the proof of c).

The proof of a) is almost identical, up to the fact that now since $\rho^n_0\rightharpoonup \rho_0$ in $L^1_x$, we can successfully apply Proposition \ref{prop:compactness-l1} to deduce convergence in $C_w([0,T];L^1_x)$ as well.

It remains to prove $c)$, under the additional condition 4). Noting that under c) the assumptions of b) are also satisfied, by the previous proof we know that $\rho^n\to \rho$ in $C_w([0,T];L^q_x)$ for all $q\in (1,\infty)$; if moreover $p=1$, then the assumptions of a) are satisfied as well, so that $\rho^n\to \rho$ in $C_w([0,T];L^1_x)$.
In what follows, let us only consider $p\in (1,\infty)$, the case $p=1$ being similar.
By part b), we know that $\rho^n$ are solutions to the rough continuity equations associated to $(u^n,\xi^n,\bZ^n)$, where $(\xi^n,\bZ^n)$ satisfy \textit{i)}-\textit{ii)} and 1)-2), and for every $R\in (0,+\infty)$ we have
\begin{align*}
    \sup_n \Big\{ \| u^n\|_{L^\infty_t L^\infty_x} + \| \nabla u^n \|_{L^\infty_t L^1(B_R)}\Big\}<\infty, \quad u^n\to u \text{ in } L^1_t C^0_\loc.
\end{align*}
In order to conclude, it remains to show that $\nabla\cdot u^n\to \nabla\cdot u$ in $L^1_t L^1_\loc$, as in that case we can go through the same linear stability arguments as in Corollary \ref{cor:linear-RPDE-stability-2}.
Since $\rho^n\to \rho$ in $C_{w-\ast}([0,T];L^\infty_x)$, we can apply assumption 4) to deduce that
\begin{align*}
    \nabla\cdot u^n_t = \nabla\cdot (K^n_t \oast \rho^n_t) \to \nabla\cdot (K_t \oast \rho_t)
    = \nabla\cdot u_t \text{ in } L^1_\loc,\quad \forall\, t\in [0,T].
\end{align*}
By dominated convergence, $\nabla\cdot u^n\to \nabla\cdot u$ in $L^1_t L^1_\loc$, yielding the conclusion.
\end{proof}

\subsection{Yudovich theory for rough $2$D Euler and random dynamical systems}\label{subsec:nonlinear-euler-yudovich}

We now apply the results from the previous section to the specific case of the rough stochastic Euler equations on $\R^2$, in vorticity form, given by
\begin{equation}\label{eq:rough-euler}
    \dd \omega_t + \nabla\cdot [(K\ast \omega_t)\cdot \omega_t]\, \rmd t + \sum_{k=1}^m \nabla\cdot (\xi_k\, \omega_t)\, \rmd \bZ^k_t = 0
\end{equation}
where $K$ is the Biot-Savart kernel, $K(x)=c x^\perp/|x|^2$ for $x^\perp=(-x_2,x_1)$.

It is well-known that $K$ satisfies the assumptions of Theorem \ref{thm:intro_wellposed_nonlinear}.
Indeed, as a Fourier multiplier $K=\nabla^\perp \Delta^{-1}$, so that $\nabla\cdot K=0$ in the sense of distributions, and
\begin{align*}
    \widehat{\nabla K}(\eta)= -\frac{\eta\otimes \eta}{|\eta|^2}
\end{align*}
which is $0$-homogeneous. Finally, by its explicit formula in real space variables, it's easy to check that $K=K\mathbbm{1}_{|x|\leq 1} + K \mathbbm{1}_{|x|> 1}\in L^1_x+L^\infty_x$.

Therefore by the proof of Theorem \ref{thm:intro_wellposed_nonlinear} we deduce that $K$ satisfies Assumptions \ref{ass:abstract-kernel-1}-\ref{ass:abstract-kernel-2}-\ref{ass:abstract-kernel-3}, with Osgood modulus of continuity $h$ (cf. \eqref{eq:thm_intro_nonlinear_proof2}) given by
$$
    h(r) = r(1-\ln r) \mathbbm{1}_{(0,1)}(r) + r \mathbbm{1}_{[1,\infty)}(r), \quad h(0)=0.
$$

Since $K$ does not depend on time, it makes sense here to directly state results for global solutions, namely indexed over $t\in \R_{\geq 0}$.
In light of forthcoming applications to random dynamical systems, it makes sense to introduce some definitions.
We set
\begin{align*}
    \mathcal{C}^\fkp_g(\R_{\geq 0};\R^m) := \bigcap_{T > 0} \mathcal{C}_g^{\mathfrak{p}}([0,T] ; \R^m)
\end{align*}
which is a complete metric space when endowed with the distance 
\begin{align*}
    d_{\fkp,\infty} (\bZ,\tilde\bZ) := \sum_{n=1}^\infty \min\big\{2^{-n}, d_{\fkp,n}(\bZ,\tilde\bZ)\big\}
\end{align*}
where $d_{\fkp,T}$ was given by \eqref{eq:metric-p-rough}.
In particular, $\bZ^n\to \bZ$ in $\mathcal{C}^\fkp_g(\R_{\geq 0};\R^m)$ if and only if $\bZ^n\to \bZ$ in $\mathcal{C}^\fkp_g([0,T];\R^m)$ for all $T\in (0,+\infty)$.
We regard $\mathcal{C}^\fkp_g(\R_{\geq 0};\R^m)$ as measurable space when endowed with the Borel $\sigma$-algebra induced by the metric $d_{\fkp,\infty}$.

Notice that, for any $u\geq 0$, the {\em translations} $\tau_u$ given by
\begin{align*}
    \tau_u\bZ:= (\tau_u Z,\tau_u \bbZ), \quad
    \tau_u Z_t := Z_{t+u}, \quad
    \tau_u \bbZ_{s,t}=\bbZ_{s+u,t+u}
\end{align*}
form a family of continuous mappings from $\mathcal{C}^\fkp_g(\R_{\geq 0};\R^m)$ to itself, such that $\tau_u\tau_v=\tau_{u+v}$ for all $u,v\geq 0$.

We are now ready to state a rough version of Yudovich's theorem.

\begin{theorem}\label{thm:yudovich-euler}
    Let $\fkp \in [2,3)$, $\bZ\in\mathcal{C}^\fkp_g(\R_{\geq 0};\R^m)$, $\xi\in C^3_b$ with $\nabla\cdot\xi=0$.
    Then for any $\omega_0\in L^1_x\cap L^\infty_x$, there exists a unique global solution $\omega \in \cB_b(\R_{\geq 0};L^1_x\cap L^\infty_x)$ to \eqref{eq:rough-euler}, which moreover belongs to $C(\R_{\geq 0}; L^p_x)\cap C_{w-\ast}(\R_{\geq 0};L^\infty_x)$ for any $p\in [1,\infty)$. The solution is renormalized and it is of the form
    \begin{equation}\label{eq:euler_flow_representation}
        \omega_t(x) = \omega_0(\Phi_t^{-1}(x))
    \end{equation}
    where $\Phi$ is the flow generated by the RDE
    $$
        \rmd y_t = u_t(y_t) \rmd t + \xi(y_t) \rmd \bZ_t
    $$
    for $u_t = K \ast \omega_t$, $K$ being the Biot-Savart kernel. Moreover, we have
    \begin{equation} \label{eq:euler lp bound}
        \|\omega_t \|_{L^p_x} = \|\omega_0 \|_{L^p_x}, \qquad \forall\, t\geq 0, \ \forall\, p \in [1,\infty].
    \end{equation}
\end{theorem}

\begin{proof}
    Existence and uniqueness on fixed intervals $[0,T]$ follow from Propositions \ref{prop:nonlinear-existence} and \ref{prop:nonlinear-uniqueness}.
    By uniqueness, all maps defined in this way are consistent with each another; since we can take $T$ arbitrarily large, this defines a unique global solution, given by formula $\omega_t = (\Phi_t)_\sharp \omega_0$.
    Note that $\omega$ is also a solution of the linear equation with drift $u$, and $\nabla \cdot u_t = 0$ since $\nabla \cdot K = 0$ in the sense of distributions;
    therefore $\omega$ solves both a continuity and a transport equation at the same time, and the representation \eqref{eq:euler_flow_representation} (as well as renormalizability) follow from Theorem \ref{thm:linear-RPDE-wellposed-lagrangian}.
    Equality \eqref{eq:euler lp bound} now follows from \eqref{eq:euler_flow_representation} and incompressibility of $\Phi_t^{-1}$.
\end{proof}

As before, also in the Euler case we obtain stability results, w.r.t. both weak and strong topologies. For simplicity, here we keep the Biot-Savart kernel $K$ fixed, but we could approximate it by smooth kernels as well (e.g. by mollifications, or multiplications by a cutoff close to 0 $0$).

\begin{theorem}\label{thm:yudovich stability}
    Let $\fkp \in [2,3)$. Consider a sequence $\{(\mathbf{Z}^n, \xi^n)\}_n$ satisfying the assumptions of Theorem~\ref{thm:yudovich-euler} uniformly in $n\in\N$ and $T\in (0,+\infty)$, namely such that:
    \begin{itemize}
        \item[i)] $\nabla\cdot \xi^n=0$ and $\mathcal{C}^\fkp_g(\R_{\geq 0};\R^m)$ for all $n$;
        \item[ii)] $\sup_{n \in \N} \big\{ \| \xi^n\|_{C^3_b}  +\| \mathbf{Z}^n\|_{\mathfrak{p},T} \big\}<\infty$ for all $T>0$.
    \end{itemize} 
    Further assume that there exist $(\xi,\mathbf{Z})$ such that
    \begin{itemize}
        \item[1)] $\xi^n\to \xi$ uniformly on compact sets;
        \item[2)] $\sup_{t\in [0,T]} |\mathbf{Z}^n_{0,t}-\mathbf{Z}_{0,t}|\to 0$, for all $T>0$.
    \end{itemize}
    Let $\{\omega^n_0\}_n$ be a bounded sequence in $L^1_x\cap L^\infty_x$, resp. $\omega_0\in L^1_x\cap L^\infty_x$, and denote by $\omega^n$, resp. $\omega$, the associated unique solution to \eqref{eq:rough-euler} associated to $(\mathbf{Z}^n, \xi^n)$, resp. $(\mathbf{Z}, \xi)$.
    Then, for all $T\in (0,+\infty)$, it holds that:
    \begin{itemize}
        \item[a)] if $\omega^n_0\rightharpoonup \omega_0$ in $L^1_x$, then $\omega^n\to \omega$ in $C_w([0,T];L^q_x)\cap C_{w-\ast}([0,T];L^\infty_x)$ for all $q\in [1,\infty)$;
        \item[b)] if either $\omega^n_0\xrightharpoonup{\ast} \omega_0$ in $L^{\infty}_x$, or $\omega_0^n \rightharpoonup \omega_0$ in $L^p_x$ for some $p \in (1,\infty)$, then $\omega^n\to \omega$ in $C_w([0,T];L^q_x)\cap C_{w-\ast}([0,T];L^\infty_x)$ for all $q\in (1,\infty)$.
        \item[c)] if $\omega^n_0\to \omega_0$ in $L^p_x$ for some $p\in [1,\infty)$, then $\omega^n\to \omega$ in $C([0,T]; L^q_x)$ for all $q\in \{p\}\cup (1,\infty)$.
    \end{itemize}
\end{theorem}

\begin{proof}
    The statement is a direct consequence of Theorem \ref{thm:nonlinear-wellposedness2}, since conditions i)-ii)-iii) and 1)-2)-3) are satisfied for any finite interval $[0,T]$. Notice moreover that 4) is automatically satisfies, since $K^n=K$ for all $n$ and $\nabla\cdot (K\ast g)=(\nabla\cdot K)\ast g=0$ for all functions $g\in L^1_x$.
\end{proof}

Next, we show how the above results can be used to generate a {\em random dynamical system} for the rough $2$D Euler system \eqref{eq:rough-euler}, when $\bZ$ is a {\em random driver}.
To this end, we need to introduce some notation and definitions; we follow \cite{BRS}.
Given a probability space $\Omega$, we will denote its elements by $\varpi$, so to avoid notational conflict with the vorticity $\omega$.

\begin{definition}\label{defn:meas_metric_DS}
    Let $(\Omega, \mathcal{F}, \bbP)$ be a probability space and let $\theta = ( \theta_t)_{t \geq 0}$ be a family of measurable mappings from $(\Omega,\cF)$ to itself. The say that the tuple $(\Omega,\cF,\bbP,\theta)$ is a \emph{measurable metric dynamical system}\footnote{Contrary to what the name suggests, there is no metric space structure involved in Definition \ref{defn:meas_metric_DS}.} provided that:
    \begin{itemize}
        \item $(t,\varpi) \mapsto \theta_t \varpi$ is $\mathcal{B}(\R_{\geq 0}) \otimes \mathcal{F}/\mathcal{F}$ measurable;
        \item $\theta_0 = Id_\Omega$;
        \item $\theta_{s+t} = \theta_t \circ \theta_s$ for all $s,t\geq 0$;
        \item $\bbP \circ \theta_t^{-1} = \bbP$ for all $t\geq 0$.
\end{itemize}
\end{definition}

\begin{definition}
    Let $(\Omega,\cF,\bbP,\theta)$ be a measurable metric dynamical system.
    A \emph{geometric $\mathfrak{p}$-rough path cocycle} is a random variable $\bZ:\Omega \to \mathcal{C}^\fkp_g(\R_{\geq 0};\R^m)$ such that, for every $\varpi \in \Omega$, we have the cocycle property
    \begin{equation} \label{eq:rough cocycle property}
        \delta Z_{s+h,t+h}(\varpi) = \delta Z_{st}( \theta_h \varpi), \qquad \bbZ_{s+h,t+h}(\varpi) = \bbZ_{st}( \theta_h \varpi).
    \end{equation}
\end{definition}

Note that, given a geometric $\mathfrak{p}$-rough path cocycle $\bZ$, we can always realize it as a canonical process, namely take $\Omega=\mathcal{C}^\fkp_g(\R_{\geq 0};\R^m)$, $\mathbb{P}$ the law of $\bZ$ and $\theta_h=\tau_h$.
The cocycle property \eqref{eq:rough cocycle property} then amounts to the stochastic process having stationary rough path increments.

It is shown in \cite[Section 2]{BRS} that a large class of Gaussian stochastic processes, including fractional Brownian motion for $H \in (1/3, 1)$, can be lifted to such a geometric $\mathfrak{p}$-rough path cocycle with $\fkp\in [2,3)$.

\begin{definition}
    Let $(\Omega,\cF,\bbP,\theta)$ be a measurable metric dynamical system and let $(\mathcal{X},\tau)$ be a topological space. A mapping
    $$
    \Phi : [0,\infty) \times \Omega \times \mathcal{X} \rightarrow \mathcal{X}
    $$
    a \emph{continuous random dynamical system} (continuous RDE for short) on $\mathcal{X}$ provided that:
    \begin{itemize}
        \item $\Phi(t,\varpi, \cdot) :\mathcal{X} \rightarrow \mathcal{X}$ is continuous for every $(t,\varpi) \in \R_{\geq 0} \times \Omega$;
        \item $\Phi(0,\varpi,\cdot) = \textrm{Id}_\mathcal{X}$ for every $\varpi \in \Omega$;
        \item $\Phi(t+s,\varpi,\cdot) = \Phi(t,\theta_s \varpi,\cdot) \circ \Phi(s,\varpi,\cdot)$ for all $(s,t,\varpi) \in \R_{\geq 0}^2 \times \Omega$. 
\end{itemize}
\end{definition}

In light of Theorem \ref{thm:yudovich-euler} for any $R\in (0,+\infty)$, the set
\begin{align*}
    \mathcal{X}_R := \left\{ \omega_0 \in L^1_x \cap L_x^{\infty} : \|\omega_0\|_{L^1_x \cap L^{\infty}_x} \leq R \right\}
\end{align*}
is invariant under the dynamics \eqref{eq:rough-euler}.
Notice that $\cX_R$ is a closed, convex, bounded subset of $L^p_x$ (by Fatou's lemma) and therefore it is also closed w.r.t. to the weak topology of $L^p_x$, for any $p\in [1,\infty)$; similarly, it is closed under the weak-$\ast$ topology of $L^\infty_x$.
Being closed and bounded, by the Banach-Alaoglu theorem, $\cX_R$ is compact in the weak-$\ast$ topology, with topology induced by a metric. Moreover, since $\cX_R$ is bounded in $L^p_x$ for every $p$, it is also weakly compact in $L^p_x$ for every $p\in (1,\infty)$; in fact, for such $p$, the weak-$\ast$ topology of $L^\infty_x$ and weak topology of $L^p_x$ coincide on $\cX_R$, as they are induced by the same metric.
Similarly, it is clear that given a sequence $\{f^n\}_n\subset \cX_R$, $f^n\to f$ in $L^p_x$ for some $p\in (1,\infty)$ if and only if $f^n\to f$ in $L^q_x$ for all $q\in (1,\infty)$, and that different $L^p_x$ norms induce the same metric.

To sum up, for any $R\in (0,+\infty)$, we have two main candidate topologies we can endow $\cX_R$ with:
\begin{itemize}
    \item The topology induced by the weak-$\ast$ topology of $L^\infty_x$, equivalently induced by the weak topology of $L^p_x$ for any $p\in (1,\infty)$, which we denote by $\tau^{weak}$. $(\cX_R,\tau^{weak})$ is then a compact metric space.
    \item The topology induced by $\| \cdot\|_{L^p_x}$ for any $p\in (1,\infty)$, which we denote by $\tau^{strong}$. $(\cX_R,\tau^{strong})$ is then a separable metric space.
\end{itemize}

\begin{theorem}\label{thm:euler_RDS}
    Let $\fkp\in [2,3)$, $R\in (0,+\infty)$, $\xi \in C^3_b$ with $\nabla \cdot \xi = 0$ and $\bZ$ be a geometric $\mathfrak{p}$-rough path cocycle.
    Then \eqref{eq:rough-euler} generates a continuous random dynamical system on $(\cX_R,\tau^{weak})$ and on $(\cX_R,\tau^{strong})$.
\end{theorem}

\begin{proof}
    Denote by $S(s,t,\varpi,\omega_0)$ the solution of \eqref{eq:rough-euler} at time $t$ when the equation is started in $\omega_0 \in \mathcal{X}_R$ at time $s$ and the driving rough path is $\bZ(\omega)$.
    From \eqref{eq:euler lp bound}, $S(s,t,\varpi,\cdot)$ is a well-defined map from $\mathcal{X}_R$ into itself, which is continuous w.r.t. both $\tau^{weak}$ and $\tau^{strong}$ thanks to points b) and c) of Theorem \ref{thm:yudovich stability}.
    
    It is straightforward to check that by uniqueness we get the flow property
    $$
S(s,t,\varpi,\omega_0) = S(u,t,\varpi, \cdot ) \circ S(s,u,\varpi,\omega_0), \qquad \forall \, (s,u,t,\varpi,\omega_0) \in \Delta_T^{(2)} \times \Omega \times \mathcal{X}<-r.
    $$
    Moreover, using uniqueness and the cocycle property of $\bZ$ we find that 
    $$
S(h,t+h,\varpi,\omega_0) = S(0,t,\theta_h\varpi,\omega_0)  \qquad \forall \, (t,h,\varpi,\omega_0) \in [0,\infty)^2 \times \Omega \times \mathcal{X}_R.
    $$
    Combining all these observations, we conclude that the map $\Phi: [0,\infty) \times \Omega \times \mathcal{X}_R \rightarrow \mathcal{X}_R$ defined by
    $$
        \Phi(t,\varpi,\omega_0) = S(0,t,\varpi,\omega_0)
    $$
    is a continuous random dynamical system.
\end{proof}

\begin{remark}
    $(\cX_R,\tau^{weak})$ and $(\cX_R,\tau^{strong})$ are natural choices to discuss continuous RDS induced by the Euler dynamics, in light of their compactness/separability properties, but they are not the only possible ones.
    For instance, we could have endowed $\cX_R$ with the convergence induced by $\| \cdot\|_{L^1_x}$, which is slight stronger than convergence in $L^p_x$ for some $p>1$.
    Alternatively, we could have considered the whole $\cX=L^1_x\cap L^\infty_x$, with norm $\| \cdot\|_{L^1_x\cap L^\infty_x}$; notice however that this space is not separable, and convergence in $L^\infty_x$ is a very strong assumption. Moreover, if $f^n\to f$ in $L^1_x\cap L^\infty_x$, then we can find $R<\infty$ such that $\{f^n\}_n\subset \cX_R$ 
    and $f^n\to f$ in $(\cX_R,\tau^{strong})$.
    Finally, one could be tempted to consider $\cX_R$ endowed with weak $L^1_x$-topology; however in this case we cannot prove Theorem \ref{thm:euler_RDS}, as this topology is not metrizable, therefore the sequential continuity provided by part a) of Theorem \ref{thm:yudovich stability} is not enough to guarantee continuity of the solution map.
\end{remark}

\subsection{More refined existence results for $2$D Euler}\label{subsec:nonlinear-euler-delort}

As the Biot--Savart kernel $K$ presents a specific structure, we can obtain more sophisticated results than in the general case of \eqref{eq:nonlinear-RPDE}. 
In particular, we can derive a weak existence result, in the style of those from DiPerna-Majda \cite{DiPMaj1987}, Delort \cite{Delort1991} and Schochet \cite{Schochet1995}, concerning $L^1_x\cap L^p_x$-valued solutions, for any $p\geq 1$.
To this end, we need a preliminary lemma.

\begin{lemma}\label{lem:biot-savart}
    Let $K$ be the Biot-Savart kernel on $\R^2$; then the map $f\mapsto \nabla\cdot [(K\ast f) f]$, previously defined from $L^1_x\cap L^\infty_x$ to $(W^{1,\infty})^\ast$, extends uniquely to a map from $L^1_x$ to $(W^{2,\infty})^\ast$, satifying
    \begin{align*}
        \| \nabla\cdot [(K\ast f) f]\|_{(W^{2,\infty})^\ast} \lesssim \| f\|_{L^1_x}^2.
    \end{align*}
    Moreover if $f^n \rightharpoonup f$ in $L^1_x$, then $\nabla\cdot [(K\ast f^n) f^n]$ converge to $\nabla\cdot [(K\ast f) f]$ in the sense of distributions.
\end{lemma}

\begin{proof}
    Since $K(z)\sim z^\perp/|z|^2$ is an odd kernel, upon symmetrization one finds
    \begin{equation}\label{eq:biot-savart-symmetry}
        \langle \nabla\cdot [(K\ast f) f], \varphi\rangle
        \sim \int_{\R^2\times\R^2} \frac{(x-y)^\perp \cdot (\nabla\varphi(x)-\nabla\varphi(y))}{|x-y|^2} f(x) f(y) \dd x \dd y
    \end{equation}
    which implies the estimate
    \begin{align*}
        |\langle \nabla\cdot [(K\ast f) f], \varphi\rangle| \leq \| f\|_{L^1_x}^2 \| \varphi\|_{W^{2,\infty}}.
    \end{align*}
    The second claim follows from formula \eqref{eq:biot-savart-symmetry} and the fact that, if $\rho^n\rightharpoonup \rho$ in $L^1(\R^2)$, then $\rho^n\otimes \rho^n\rightharpoonup \rho \otimes\rho$ in $L^1(\R^2\times \R^2)$.
\end{proof}

\begin{remark}
    Lemma \ref{lem:biot-savart} readily allows to give meaning to weak solutions $\omega\in \cB_b([0,T];L^1_x)$ to \eqref{eq:rough-euler}, in the style of Definition \ref{defn:solution-rough-continuity}.
    Indeed, in this case the forcing term $\mu$ with $\dot\mu=\nabla\cdot[(K\ast\omega)\omega]$ is well defined and
    \begin{equation}\label{eq:estim_nonlinearity}
        \| \mu\|_{C^{1-\var}([s,t];\cF_{-2,R})}
        \leq \int_s^t \| \nabla\cdot( [K\ast \omega_s] \omega_s)\|_{\cF_{-2,R}} \dd s
        \lesssim |t-s| \| \omega\|_{\cB_b([s,t];L^1_x)}^2
    \end{equation}
    uniformly in $s\leq t$ and $R\geq 1$.
\end{remark}

\begin{proposition}\label{prop:schochet-delort}
    Let $\fkp \in [2,3)$, $\bZ\in\mathcal{C}^\fkp_g(\R_{\geq 0};\R^m)$, $\xi\in C^2_b$ with $\nabla\cdot\xi=0$.
    Then for any $p\in [1,\infty)$ and any $\omega_0\in L^1_x\cap L^p_x$, there exists a global weak solution $\omega\in \cB_b(\R_{\geq 0}; L^1_x\cap L^p_x)$ to \eqref{eq:rough-euler} satisfying
    \begin{equation}\label{eq:euler.existence.bound}
        \sup_{t\geq 0} \| \omega_t\|_{L^q_x} \leq \| \omega_0\|_{L^q_x} \quad \forall\, q\in [1,p].
    \end{equation}
\end{proposition}

\begin{proof}
    We give the proof for $p=1$, the other cases being similar.
    For fixed $\bZ$, consider a sequence of approximations $(\xi^n,\omega_0^n)$ of $(\xi,\omega_0)$ obtained by mollifications, so that $\xi_n\in C^3_b$, $\omega^n_0\in L^1_x\cap L^\infty_x$ and
    \begin{align*}
        \sup_n \|\xi^n\|_{C^2_b}\leq \| \xi\|_{C^2_b},\quad \nabla\cdot \xi^n=0, \quad \xi^n\to\xi \text{ in }C^2_\loc, \quad
        \sup_n \| \omega^n_0\|_{L^1_x} \leq \| \omega_0\|_{L^1_x}, \quad \omega^n_0\to \omega_0 \text{ in } L^1_x.
    \end{align*}
    For each $\omega^n_0$, let $\omega^n$ be the unique solution in $\cB_b(\R_{\geq 0};L^1_x\cap L^\infty_x)$ to
    \begin{align*}
        \dd \omega^n_t + \nabla\cdot [(K\ast \omega^n_t)\cdot \omega^n_t]\, \rmd t + \sum_{k=1}^m \nabla\cdot (\xi^n_k\, \omega_t)\, \rmd \bZ^k_t = 0
    \end{align*}
    whose existence is guaranteed by Theorem \ref{thm:yudovich-euler}.
    By \eqref{eq:euler_flow_representation}, such solutions are of the form
    \begin{equation}\label{eq:schochet_proof1}
        \omega^n_t(x) = \omega^n_0(\Phi^{-1;n}_t(x))
    \end{equation}
    where $\Phi^n_t$, resp. $\Phi^{-1;n}_t$, are incompressible flows. By construction, we have
    \begin{equation}\label{eq:schochet_proof2}
        \sup_{t\geq 0} \| \omega^n_t\|_{L^1_x}=\| \omega^n_0\|_{L^1_x} \leq \| \omega_0\|_{L^1_x} \quad \forall\, n\in\N.
    \end{equation}
    Set $K^1:= K\mathbbm{1}_{|x|\leq 1}$, $K^2:= K\mathbbm{1}_{|x|> 1}$, so that the Biot-Savart kernel decomposes
    \begin{equation}\label{eq:biot-savart-decomposition}
        K=K^1+K^2, \quad K^1\in L^q_x\quad \forall\, q\in [1,2),\quad K^2\in L^{\tilde q}_x\quad \forall\, \tilde q\in (2,\infty];
    \end{equation}
    then the velocity fields $u^n_t:=K\ast\omega^n_t$ associated to the solutions similarly decompose as $u^n_t=u^{n,1}_t+u^{n,2}_t$, where by \eqref{eq:schochet_proof2} and Young convolutional inequalities we have
    \begin{equation}\label{eq:schochet_proof3}\begin{split}
        & \sup_{n\in\N} \sup_{t\in [0,T]} \| u^{n,1}_t\|_{L^1_x} \leq \| K^1\|_{L^1_x} \| \omega^n_t\|_{L^1_x} \lesssim \| \omega_0\|_{L^1_x},\\
        & \sup_{n\in\N} \sup_{t\in [0,T]} \| u^{n,2}_t\|_{L^\infty_x} \leq \| K^2\|_{L^\infty_x} \| \omega^n_t\|_{L^1_x} \lesssim \| \omega_0\|_{L^1_x}.
    \end{split}\end{equation}
    By the uniform estimate \eqref{eq:schochet_proof3}, arguing as in Lemma \ref{lem:equi-p-integrability} (in particular, bound \eqref{eq:equi-p-integrability-proof3}) and Corollary \ref{cor:linear-RPDE-stability-1}, we deduce that
    \begin{equation}\label{eq:schochet_proof4}
        \lim_{R\to\infty} \sup_{n\in\N}\sup_{t\in [0,T]} \int_{|x|>R} |\omega^n_t(x)|\dd x = 0 \qquad \forall\, T\in (0,\infty).
    \end{equation}
    Estimate \eqref{eq:schochet_proof2} provided equiboundedness for $\{\omega^n\}_n$, \eqref{eq:schochet_proof4} tightness; using the flow representation \eqref{eq:schochet_proof1}, the fact that $\omega^n_0\to\omega_0$ and arguing as in Proposition \ref{prop:existence-linear-RPDE}, local equi-integrability of $\{\omega^n\}_n$ follows as well.
    In order to invoke Proposition \ref{prop:compactness-l1}, it remains to verify weak equicontinuity.

    Note that $\omega^n$ solves a rough continuity equation with forcing $\mu^n_t:=\int_0^t \nabla\cdot [(K_s\ast\omega^n_s)\omega^n_s] \dd s$ and we have the uniform bound \eqref{eq:schochet_proof2}. We can thus invoke the a priori estimates from Lemma \ref{lem:apriori-unbounded} (in particular, apply bound \eqref{eq:apriori-unbounded-eq3}, together with \eqref{eq:schochet_proof2} and the formula for $w_\ast$) to find
     \begin{align*}
        \| \omega^n_t-\omega^n_s\|_{\cF_{-1,R}}
        \leq \llbracket \omega^n \rrbracket_{\fkp,[s,t];E_{-1}}
        \lesssim w_{\bA^n}(s,t)^{\frac{1}{\fkp}} + w_{\bA^n}(s,t)^{\frac{1}{\fkp}-\frac{1}{3}} w_{\mu^n}(s,t)^{\frac{1}{3}} + w_{\mu^n}(s,t)^{\frac{1}{2}}
    \end{align*}
    where the operators $\bA^n$ are defined in function of $(\xi^n,\bZ)$ as in Lemma \ref{lem:URD_continuity}. Noting that $\bZ$ (which is independent of $n$) is responsible for the time regularity of $\bA^n$, and that $\{\xi^n\}$ are bounded in $C^2_b$, for any $T\in (0,\infty)$ we can find a control $\gamma_T$ such that
    \begin{align*}
        \sup_{n\in\N} w_{\bA^n}(s,t) \leq \gamma_T(s,t) \quad \forall\, (s,t)\in\Delta_T.
    \end{align*}
    Moreover by estimate \eqref{eq:estim_nonlinearity}, we have
    \begin{align*}
        w_{\mu^n}(s,t)
        =\llbracket \mu^n \rrbracket_{C^{1-\var}([s,t];\cF_{-2,R})}
        \lesssim |t-s| \| \omega^n\|_{\cB_b([0,T];L^1_x)}^2
        \leq |t-s| \|\omega_0\|_{L^1_x}
    \end{align*}
    and so overall we find
    \begin{align*}
        \sup_{n\in\N} \| \omega^n_t-\omega^n_s\|_{\cF_{-1,R}} \lesssim_T \gamma_T(s,t)^{\frac{1}{\fkp}} + \gamma_T(s,t)^{\frac{1}{\fkp}-\frac{1}{3}} |t-s|^{\frac{1}{3}} + |t-s|^{\frac{1}{2}} \quad \forall\, (s,t)\in \Delta_T, \quad T\in (0,+\infty).
    \end{align*}
    We can therefore apply Proposition \ref{prop:compactness-l1} to extract a (not relabelled) subsequence such that $\omega^n\to \omega$ in $C_w([0,T];L^1_x)$, for fixed finite $T$; up to a Cantor diagonal argument, we may further assume that the above convergence holds for all $T\in (0,+\infty)$.
    
    By \eqref{eq:schochet_proof2} and properties of weak convergence, $\omega$ satisfies \eqref{eq:euler.existence.bound}, so it only remains to show that solves \eqref{eq:rough-euler}.
    By Lemma \ref{lem:biot-savart} and dominated convergence, it's easy to check that
    \begin{align*}
        \mu^n_t =\int_0^t \nabla\cdot [(K_s\ast\omega^n_s)\omega^n_s] \dd s \to \int_0^t \nabla\cdot [(K_s\ast\omega_s)\omega_s] \dd s=: \mu_t \quad\forall t\geq 0
    \end{align*}
    in the sense of distributions; the rest of the argument needed in order to pass to the limit in the rough part is then identical to that of Proposition \ref{prop:existence-linear-RPDE}, which overall shows that $\omega$ solves \eqref{eq:rough-euler} as desired.
\end{proof}

Under additional regularity of $\xi$ and higher integrability, the weak solutions from Proposition \ref{prop:schochet-delort} satisfy additional properties.

\begin{proposition}\label{prop:schochet-delort-renormalized}
    Let $\fkp \in [2,3)$, $\bZ\in\mathcal{C}^\fkp_g(\R_{\geq 0};\R^m)$, $\xi\in C^3_b$ with $\nabla\cdot\xi=0$.
    Let $p\in [2,\infty)$, $T\in (0,+\infty)$ and suppose that $\omega\in\cB_b([0,T];L^1_x\cap L^p_x)$ is a weak solution to \eqref{eq:rough-euler}.
    Then $\omega\in C([0,T];L^1_x\cap L^p_x)$ and is renormalized: for any $\beta\in C^1_b$, $v= \beta(\omega)$ is a weak solution on $[0,T]$ to the RPDE
    \begin{align*}
        \dd v_t + (K\ast \omega_t)\cdot\nabla v_t + \xi\cdot\nabla v_t\, \dd \bZ_t =0. 
    \end{align*}
\end{proposition}

\begin{proof}
    Since $\omega$ also solves a linear transport (equiv. continuity) RPDE with drift $u=K\ast \omega$, in order to conclude we want to invoke Theorem \ref{thm:uniqueness-linear-RPDE}, Proposition \ref{prop:absolute-continuity-Lp} and Corollary \ref{cor:renormalized}.
    To this end, since $\omega\in \cB_b([0,T];L^2_x)$ and $\nabla\cdot u=0$, it suffices to check that
    \begin{equation}\label{eq:renormalized_goal}
        \frac{u}{1+|x|} \in \cB_b([0,T];L^2_x), \quad \nabla u\in \cB_b([0,T];L^2_x).
    \end{equation}
    The second requirement immediately follows from properties of the Biot-Savart kernel, so we focus on the first one. Employing the decomposition \eqref{eq:biot-savart-decomposition} of $K$ and setting corresponding
    \begin{align*}
        u=u^1+u^2=: K^1\ast\omega+K^2\ast\omega,
    \end{align*}
    by \eqref{eq:biot-savart-decomposition} and Young convolution inequalites it holds that
    \begin{align*}
       & \| u^1_t\|_{L^2_x} \leq \| K^1\|_{L^1_x} \| \omega_t\|_{L^2_x} \lesssim \|  \omega\|_{\cB_b([0,T];L^2_x)},\\
       & \bigg\| \frac{u^2_t}{1+|x|} \bigg\|_{L^2_x} \leq \| K^2\ast\omega_t\|_{L^4_x}\, \bigg\| \frac{1}{1+|x|} \bigg\|_{L^4_x} \lesssim \| K^2\|_{L^4_x} \| \omega_t\|_{L^1_x} \lesssim \|  \omega\|_{\cB_b([0,T];L^1_x)}
    \end{align*}
    which overall proves \eqref{eq:renormalized_goal}.
\end{proof}

\begin{proof}[Proof of Theorem \ref{thm:intro-schochet-delort}]
    This is just a combination of Propositions \ref{prop:schochet-delort} and \ref{prop:schochet-delort-renormalized}, the latter applied on $[0,T]$ for any finite $T$.
\end{proof}

\begin{remark}
    Our results are partially incomplete compared to the deterministic case \cite{DiPMaj1987,Delort1991,Schochet1995}, where the possibility of measure-valued initial conditions $\omega_0$ ``with a preference sign'' is allowed (think of $\omega_0=\delta_0+\tilde \omega_0$, where $\tilde\omega_0\in L^1_x$.
    Moreover in the series of works \cite{BoBoCr2016,CiCrSp2020,CiCrSp2021}, it has been shown that in the full range $p\in [1,\infty)$, $(L^1_x\cap L^p_x)$-valued solutions can be constructed in such a way that they are renormalized and admit a Lagrangian representation. We leave such more delicate extension to future works.
\end{remark}

\appendix

\section{Some useful lemmas}\label{app:tech-lem}

We collect in this appendix some basic results that have been used throughout the paper. We start with some well-known Taylor-type expansion which are frequently needed in rough path theory.

To this end, let us consider a path $y:I\to \R^d$ and a regular function $f: \R^d \rightarrow \R$; one can then immediately extend everything to $\R^m$-valued functions by arguing componentwise.
For $k\in \{1,2\}$, define
\begin{equation*}
    \lp f \rp^{k,y}_{st} = \int_0^1 (1- \lambda)^{k-1} f(y_s + \lambda \delta y_{st}) \dd \lambda;
\end{equation*}
then by Taylor's formula we have
\begin{equation*}
    \delta f(y)_{st} := f(y_t) - f(y_s) = \lp Df \rp^{1,y}_{st} \delta y_{st}, \quad
    \lp f \rp^{1,y}_{st} - f(y_s) = \lp Df \rp^{2,y} \delta y_{st} .
\end{equation*}

\begin{lemma}\label{lem:basic-taylor}
Let $f\in C^2_b(\R^d;\R)$, $\xi\in C^0_b(\R^m;\R^d)$; then there exists a constant $C=C(\| \xi\|_{C^0_b}, \| f\|_{C^2_b})$ such that for any $y:I\to \R^d$ and any $z\in \R^m$ it holds 
\begin{equation*}%\label{eq:appendix-taylor}
    |\delta f(y)_{st} - Df(y_s) \xi(y_s) z|
    \lesssim |\delta y_{st}-\xi(y_s) z| + |\delta y| |z| \quad \forall \, (s,t)\in\Delta_I,\, z\in \R^m.
\end{equation*}
\end{lemma}

\begin{proof}
By Taylor expansion and algebraic manipulations, it holds
\begin{align*}
    \delta f(y)_{st} - Df(y_s) \xi(y_s) z
    & = \lp Df \rp^{1,y}_{st} \delta y_{st} - Df (y_s) \xi(y_s) z\\
    & = \lp Df \rp^{1,y}_{st} [\delta y_{st}-\xi(y_s) z] + \lp D^2 f\rp^{2,y}_{st} \delta y_{st}\, \xi(y_s) z;
\end{align*}
the regularity assumptions on $f$, $\xi$ then readily imply the conclusion.  
\end{proof}

\begin{lemma}\label{lem:four-point-estim}
Let $f\in C^2_b(\R^d;\R)$, then there exists $C=C(\| f\|_{C^2_b})$ such that
\begin{equation*}
    |f(x_1)-f(x_2)-f(x_3)+f(x_4)|\lesssim |x_1-x_2||x_2-x_4| + |x_1-x_2-x_3+x_4| \quad \forall\, x_i\in\R^d.
\end{equation*}
\end{lemma}

\begin{proof}
    Again applying Taylor expansions and algebraic manipulations, we find
    \begin{align*}
        f& (x_1) -f(x_2)-f(x_3)+f(x_4)\\
        & = \bigg(\int_0^1 \big[ Df(x_2+\lambda (x_1-x_2)) - Df(x_4+\lambda (x_1-x_2))\big] \dd \lambda\bigg) (x_1-x_2) + \big[ f(x_4+x_1-x_2)-f(x_3)];
    \end{align*}
    since by assumption $f$ and $Df$ are globally Lipschitz, conclusion follows.
\end{proof}

In order to state the next lemma, it is convenient to introduce some basic notation. Given $f\in C^2_b$, $\xi\in C^0_b$ and $y$, $z$ as in Lemma \ref{lem:basic-taylor}, let us set
\begin{equation*}
    f(y)^\sharp_{st}:= \delta f(y)_{st} - Df(y_s) \xi(y_s)z, \quad y^\sharp_{st}:=\delta y_{st} - \xi(y_s)z;
\end{equation*}
moreover for fixed $\xi$ and $z$, given two distinct paths $y^1$, $y^2$, consider
\begin{equation*}
    v_t:= y^1_t-y^2_t, \quad
    v^\sharp_{st} := y^{1,\sharp}_{st} - y^{2,\sharp}_{st} = \delta v_{st} - (\xi(y^1_s)-\xi(y^2_s) ) z.
\end{equation*}

\begin{lemma}\label{lem:basic-taylor-2}
    Let $f\in C^3_b$, $\xi\in C^1_b$; then there exists a constant $C=C(\| \xi\|_{C^1_b}, \| f\|_{C^3_b})$ such that for any $y^1,y^2:I\to \R^d$ and any $z\in \R^m$, for $v$ as defined above, it holds 
\begin{equation*}%\label{eq:appendix-taylor-2}
    |f(y^1)^\sharp_{st} - f(y^2)^\sharp_{st}|
    \lesssim |v^\sharp_{st}| + |\delta v_{st}| \big(|y^{1,\sharp}_{st}| 
    + |z| \big) + |v_s| \big(|y^{1,\sharp}_{st}| + |\delta y^1_{st}| |z|\big) \quad \forall \, (s,t)\in\Delta_I.
\end{equation*}
\end{lemma}

\begin{proof}
    Similar computations to those from Lemma \ref{lem:basic-taylor} give us
    \begin{align*}
        f(y^1& )^\sharp_{st} - f(y^2)^\sharp_{st}
        = \lp D f\rp_{st}^{1,y^1} \, y^{1,\sharp}_{st} - \lp D f\rp_{st}^{1,y^2} \, y^{2,\sharp}_{st} + \\
        & + \Big(\int_0^1 \big[ Df(y^1_s + \lambda \delta y^1_{st}) - Df(y^1_s) \big] \dd \lambda\Big)\, \xi(y^1_s) z - \Big(\int_0^1 \big[ Df(y^2_s + \lambda \delta y^2_{st}) - Df(y^2_s) \big] \dd \lambda\Big)\, \xi(y^2_s) z\\
        & = I_1+I_2+I_3+I_4
    \end{align*}
    for
    \begin{align*}
        I_1 & := \lp D f\rp_{st}^{1,y^2} (y^{1,\sharp}_{st} - y^{2,\sharp}_{st}) = \lp D f\rp_{st}^{1,y^2} v^\sharp_{st},\\
        I_2 & := \Big( \int_0^1 \big[ Df (y^1_s+\lambda \delta y^1_{st}) - Df (y^2_s+\lambda \delta y^2_{st})\big] \dd \lambda\Big)\, y^{1,\sharp}_{st},\\
        I_3 & := \Big(\int_0^1 \big[ Df(y^1_s + \lambda \delta y^1_{st}) - Df(y^1_s) -  Df(y^2_s + \lambda \delta y^2_{st}) + Df(y^2_s)\big] \dd \lambda\Big)\, \xi(y^2_s) z,\\
        I_4 &:= \lp D^2 f\rp^{2,y^2}_{st} \delta y^1_{st}\, \big( \xi(y^1_s)-\xi(y^2_s)\big) z.
    \end{align*}
    We can estimate terms $I_1$, $I_2$ and $I_4$ in a standard manner, yielding
    \begin{align*}
        |I_1| \lesssim  |v^\sharp_{st}|,
        \quad |I_2|\lesssim (|v_s| +|\delta v_{st}|) |y^{1,\sharp}_{st}|,
        \quad |I_4| \lesssim |\delta y^1_{st}| |v_s| |z|.
    \end{align*}
    For term $I_3$ we can instead apply Lemma \ref{lem:four-point-estim}, which gives
    \begin{align*}
        |I_3| \lesssim \big(|\delta y^1_{st}| |v_s| + |\delta v_{st}| \big) |z|.
    \end{align*}
    Combining all the estimates gives the conclusion.
\end{proof}

The next basic lemma allows to pass from local to global estimates, which is often useful to obtain a priori estimates when combined with sewing techniques (Lemma \ref{lem:sewing}).

\begin{lemma}\label{lem:local-to-global}
    Let $I$ be an interval, $E$ a Banach space, $g\in C(\Delta_I;E)$; suppose there exist constants $C$, $h>0$, an exponent $\mathfrak{p}\in (0,+\infty)$ and some controls $w$, $\tilde w$ on $I$ such that
    \begin{equation}\label{eq:local-to-global-assumption}
        \| g_{st}\|_E \leq C\, \tilde w(s,t)^{\frac{1}{\mathfrak{p}}} \quad \text{ for all } (s,t)\in \Delta_I \text{ such that } w(s,t) \leq h.
    \end{equation}
    Then $g\in C^{\mathfrak{p}-\var}_2(I;E)$ and it holds
    \begin{equation}\label{eq:local-to-global-conclusion}
        \llbracket g\rrbracket_{\mathfrak{p},[s,t];E}
        \lesssim_\mathfrak{p} C\, \tilde w(s,t)^{\frac{1}{\mathfrak{p}}} + h^{-\frac{1}{\mathfrak{p}}}\, w(s,t)^{\frac{1}{\mathfrak{p}}} \| g\|_{C(\Delta_{[s,t]};E)} \quad \forall\, (s,t)\in\Delta_I
    \end{equation}
    where we set $\| g\|_{C(\Delta_{[s,t]};E)} :=\sup_{(u,v)\in \Delta_{[s,t]}} \| g_{uv}\|_E$.
    
    If additionally $\mathfrak{p}\in [1,\infty)$ and $g$ is the increment of a path, i.e. $g=\delta x$ for some $x\in C(I;E)$, then it holds
    \begin{equation}\label{eq:local-to-global-conclusion-2}
        \llbracket x\rrbracket_{\mathfrak{p},[s,t];E}
        \leq 2C\, \tilde w(s,t)^{\frac{1}{\mathfrak{p}}} \Big(1 + h^{\frac{1}{\mathfrak{p}}-1} w(s,t)^{1-\frac{1}{\mathfrak{p}}}\Big)
        \quad \forall\, (s,t)\in\Delta_I
    \end{equation}
    where the estimate now does not depend on $\mathfrak{p}$.
\end{lemma}

\begin{proof}
    Up to relabelling the controls $w$, $\tilde w$, we can assume $C=h=1$ in both cases.
    
    Fix $(s,t)\in\Delta_I$ and consider any finite partition $\pi=\{t_i\}_{i=0}^N$ of $[s,t]$. Correspondingly, define $\mathcal{I}:=\{i\in \{0,\ldots, N-1\} : w(t_i,t_{i+1})\leq 1\}$; by superadditivity, $\sum_i w(t_i,t_{i+1}) \leq w(s,t)$, therefore the complement $\mathcal{I}^c=\{0,\ldots, N-1\}\setminus \mathcal{I}$ can have at most cardinality $w(s,t)$.
    In order to estimate the $\mathfrak{p}$-variation of $g$ along the partition $\pi$, we split the sum and treat differently $i\in\mathcal{I}$ and $i\in\mathcal{I}^c$:
    \begin{align*}
        \sum_{i} \| g_{t_i t_{i+1}} \|_E^\mathfrak{p}
        & \leq \sum_{i\in \mathcal{I}} \| g_{t_i t_{i+1}} \|_E^\mathfrak{p} + \sum_{i\in \mathcal{I}^c} \| g_{t_i t_{i+1}} \|_E^\mathfrak{p}\\
        & \leq \sum_{i\in \mathcal{I}} \tilde w(t_i,t_{i+1})  + {\rm Card} (\mathcal{I}^c)\, \| g\|_{C(\Delta_{[s,t]};E)}^\mathfrak{p}\\
        & \leq \tilde w(s,t) + w(s,t) \| g\|_{C(\Delta_{[s,t]};E)}^\mathfrak{p}
    \end{align*}
    where we first used the assumption \eqref{eq:local-to-global-assumption}, then superadditivity of $\tilde w$ and finally the above observation on ${\rm Card} (\mathcal{I}^c)$. Taking supremum over all possible partitions and elevating both sides of the estimate to $1/\fkp$ readily yields \eqref{eq:local-to-global-conclusion}.

    Concerning the claim \eqref{eq:local-to-global-conclusion-2}, we proceed as in \cite[Proposition 5.10(ii)]{FV2010}.
    Fix any $[s,t]\subset I$; we only need to consider the case $w(s,t)>1$, otherwise there is nothing to prove (since \eqref{eq:local-to-global-conclusion-2} is implied by \eqref{eq:local-to-global-assumption}).
    Define $t_0=s$ and $t_{i+1}=\inf \{u>t_i\,:  w(t_i,u)=1\}\wedge t$, then by superadditivity it holds $t_n=t$ for $n\geq w(s,t)$ and the resulting partition contains at most $N=1+w(s,t)\leq 2 w(s,t)$ subintervals. It holds
    \begin{align*}
        \| \delta x_{st}\|_E
        & \leq \sum_{i=0}^{N-1} \| \delta x_{t_i,t_{i+1}}\|_E
        \leq \sum_{i=0}^{N-1} \tilde w(t_i,t_{i+1})^{\frac{1}{\mathfrak{p}}}\\
        & \leq N^{1-\frac{1}{\mathfrak{p}}} \Big( \sum_{i=0}^{N-1} \tilde w(t_i,t_{i+1}) \Big)^{\frac{1}{\mathfrak{p}}}
        \leq 2\, w(s,t)^{1-\frac{1}{\mathfrak{p}}}\, \tilde w(s,t)^{\frac{1}{\mathfrak{p}}}
    \end{align*}
    where we used the previous observation on $N$, superadditivity of $\tilde w$ and H\"older's inequality.
    Combined with the assumption \eqref{eq:local-to-global-assumption}, this overall shows that for any $(s,t)\in\Delta_I$ it holds
    \begin{align*}
        \| \delta x_{st}\|_E \leq 2 \tilde w(s,t)^{\frac{1}{\mathfrak{p}}} \big(1 + w(s,t)^{\frac{1}{\mathfrak{p}}} \big)
    \end{align*}
    which readily yields the conclusion in view of Remark \ref{rem:relation-control-variation}.
\end{proof}

Another basic tool in a priori estimates is the so called rough Gr\"onwall lemma.
%; the version given below is taken from \cite[Lemma A.2]{HLN2021}.  
Recall that, as in Remark \ref{rem:properties-controls}, a function $\gamma:\Delta_T \to [0,+\infty)$ is increasing if $\gamma(s,t)\leq \gamma(s',t')$ whenever $[s,t]\subset [s',t']$.

\begin{lemma}[Lemma A.2 from \cite{HLN2021}]\label{lem:rough-gronwall}
Let $G:[0,T]\to \R_{\geq 0}$ be a path such that for some constants $C$, $L>0$, $\mathfrak{p}\in [1,\infty)$, some control $w$ and some increasing function $\gamma$ on $\Delta_T$, one has
\begin{equation*}%\label{eq:rough-gronwall-hypothesis}
G_t-G_s \leq C \, \sup_{r\leq t}\, G_r \, w(s,t)^{\frac{1}{\mathfrak{p}}} + \gamma(s,t)
\end{equation*}
for every $s<t$ satisfying $w(s,t)\leq L$. Then it holds
\begin{equation*}%\label{eq:rough-gronwall-conclusion}
\sup_{t\in [0,T]} G_t \leq 2 \exp\bigg( \frac{w(0,T)}{\lambda L}\bigg) \Big( G_0 + \gamma(0,T) \Big)
\end{equation*}
for $\lambda := 1 \wedge [L (2C e^2)^\mathfrak{p}]^{-1}$.
\end{lemma}

\section{Compactness criteria}\label{app:compactness}

We collect here some useful compactness Ascoli--Arzelà type results, taylored to uniform convergence in weak topologies (cf. Definition \ref{defn:uniform_weak_convergence}).

To this end, recall that $\langle \cdot,\cdot\rangle$ denotes both the duality pairing between test functions $C^\infty_c$ and distributions $\cD'$, and between $L^p_x$ and its dual $L^{p'}_x$.
For any $p\in [1,\infty)$, we employ the notation $C_w([0,T];L^p_x)$ in agreement with Definition \ref{defn:uniform_weak_convergence}, similarly for $C_{w-\ast}([0,T];L^\infty_x)$.

The next statement is taken from \cite{crippa2021elementary}, although we conveniently rephrase it in terms of Definition \ref{defn:uniform_weak_convergence}, so that it can be interpreted as a compactness criterion for convergence in $C_w([0,T];L^1)$.

\begin{proposition}[Theorem A.1 from \cite{crippa2021elementary}]\label{prop:compactness-l1}
    Let $T\in (0,+\infty)$, $\{f^n\}_n$ be a sequence in $\mathcal{B}_b([0,T];L^1_x)$ which is bounded and equi-integrable in space, uniformly in time, namely such that:
    \begin{itemize}
        \item[i.] {\rm (uniform boundedness)}  $\sup_{n\in\N, t\in [0,T]} \| f^n_t\|_{L^1}<\infty$.
        \item[ii.] {\rm (local equi-integrability)} For any $\eps>0$ there exists $\delta_\eps>0$ such that
        \begin{align*}
            \mathscr{L}^d(A)<\delta_\eps\quad \Rightarrow \quad \sup_{n\in\N, t\in [0,T]} \int_A |f^n_t(x)| \dd x <\eps;
        \end{align*}  
        \item[iii.] {\rm (tightness)} For any $\eps>0$ there exists $\Omega_\eps\subset\R^d$ such that
        \begin{align*}
            \mathscr{L}^d(\Omega_\eps)<\infty \quad \text{and} \quad \sup_{n\in\N, t\in [0,T]} \int_{\R^d\setminus \Omega_\eps} |f^n_t(x)| \dd x <\eps.
        \end{align*}
    \end{itemize}
    Further assume the following:
    \begin{itemize}
        \item[iv.] {\rm (weak equicontinuity)} For every $\varphi\in C^\infty_c$, the functions $\{t\mapsto \langle f^n_t, \varphi \rangle\}_n$ are uniformly equicontinuous on $[0,T]$; namely, for any $\eps>0$ there exists $\delta=\delta(\varphi,\eps)>0$ such that
        \begin{align*}
            s, t\in [0,T],\, |t-s|<\delta \quad \Rightarrow\quad \sup_{n\in\N} |\langle f^n_t-f^n_s,\varphi\rangle|<\eps.
        \end{align*}
    \end{itemize}
    Then there exists a subsequence $\{f^{n_k}\}_k$ and a function $f$ such that $f^{n_k}\to f$ in $C_w([0,T];L^1_x)$.
\end{proposition}

Under weaker assumptions, we can still deduce ``uniform weak convergence in $L^1_\loc$''.

\begin{definition}
    We say that $f^n\to f$ in $C_w([0,T];L^1_\loc)$ if, for every $\varphi\in C^\infty_c$, $f^n \varphi\to f\varphi$ in $C_w([0,T];L^1)$.
\end{definition}

\begin{corollary}\label{cor:compactness-l1-loc}
    Let $T\in (0,+\infty)$, $\{f^n\}_n$ be a sequence in $\mathcal{B}_b([0,T];L^1_x)$ which is bounded, locally equi-integrable and weakly equicontinuous, namely satisfying properties \textit{i.}, \textit{ii.} and \textit{iv.} from Proposition \ref{prop:compactness-l1}. Then there exists a subsequence $\{f^{n_k}\}_k$ and a function $f\in \mathcal{B}_b([0,T];L^1_x)$ such that $f^{n_k}\to f$ in $C_w([0,T];L^1_\loc)$. Moreover it holds
    \begin{equation}\label{eq:lsc_weak_l1loc}
        \| f\|_{\mathcal{B}_b([0,T];L^1_x)} = \sup_{t\in [0,T]}\| f_t\|_{L^1}
        \leq \liminf_{n\to\infty} \sup_{t\in [0,T]} \| f^n_t\|_{L^1} = \liminf_{n\to\infty} \| f^n\|_{\mathcal{B}_b([0,T];L^1_x)}.
    \end{equation}
\end{corollary}

\begin{proof}
    First notice that, by definition of convergence in $C_w([0,T];L^1_x)$, if $g^n\to g$ in $C_w([0,T];L^1_x)$, then also $g^n \varphi\to g \varphi$ in $C_w([0,T];L^1_x)$, for any $\varphi\in L^\infty_x$.
    
    Consider a sequence $\{\psi^R\}_{R\in\N}\subset C^\infty_c$ of nonnegative functions such that $\psi^R\equiv 1$ on $B_R$, $\psi^R\equiv 0$ on $B_{R+1}^c$ and $\psi^R \psi^{R+1}=\psi^R$.
    By assumption, for any fixed $R$, $\{f^n \psi^R\}_n$ satisfies the assumptions of Proposition \ref{prop:compactness-l1}, therefore we can extract a (not relabelled) subsequence so that $f^n \psi^R\to g^R$ for some $g^R\in C_w([0,T];L^1_x)$. By the initial observations, these limits are consistent, in the sense that $g^{R+1} \psi^R= g^R$. By a Cantor diagonal argument, we can deduce the existence of a (n.r.) subsequence and a function $f:[0,T]\to L^1_\loc$ such that $f^n \psi^R\to f \psi^R$ in $C_w([0,T];L^1_x)$ for all $R\in \N$.
    Again by the initial remark, since for any given $\varphi\in C^\infty_c$ we can find $R$ large enough such that $\psi^R \varphi = \varphi$, we deduce that $f^n \varphi\to f \varphi$ for all $\varphi\in C^\infty_c$.

    It remains to show that $f\in \mathcal{B}_b([0,T];L^1_x)$ and estimate \eqref{eq:lsc_weak_l1loc}. To this end, recall that $C^\infty_c$ is a norming subspace for $L^1_x$, in the sense that for any $g\in L^1_\loc$ it holds that
    \begin{equation}\label{eq:norming_l1}
        \| g\|_{L^1_x} = \sup\Big\{|\langle g, \varphi\rangle|: \varphi\in C^\infty_c \text{ such that } \| \varphi\|_{L^\infty_x} \leq 1\Big\}.
    \end{equation}
    It follows that the Borel $\sigma$-algebra of $L^1_x$ is generated by the functions $g\mapsto \langle g,\varphi\rangle$ ranging over $\varphi\in C^\infty_c$. Since by construction $f$ is such that $t\mapsto \langle f_t,\varphi\rangle$ is continuous for any $\varphi\in C^\infty_c$, we deduce that $f$ is Borel measurable.
    Concerning estimate \eqref{eq:lsc_weak_l1loc}, by virtue of \eqref{eq:norming_l1}, for any fixed $t\in [0,T]$ and any $\varphi\in C^\infty_c$ with $\| \varphi\|_{L^\infty_x} \leq 1$ it holds
    \begin{align*}
        |\langle f_t, \varphi \rangle|
        = \lim_{n\to\infty} |\langle f^n_t, \varphi \rangle|
        \leq \liminf_{n\to\infty}  \|f^n_t\|_{L^1_x}
    \end{align*}
    Taking the supremum over such $\varphi$ thus yields
    \begin{align*}
        \| f_t\|_{L^1_x}
        \leq \liminf_{n\to\infty}  \|f^n_t\|_{L^1_x}
        \leq \liminf_{n\to\infty}  \sup_{t\in [0,T]} \|f^n_t\|_{L^1_x}
    \end{align*}
    from which the conclusion follows upon taking another supremum over $t\in [0,T]$.
\end{proof}

We haven't found a reference in the literature for the next result, which is an analogue of Proposition \ref{prop:compactness-l1} for $L^p_x$-spaces; therefore we give a short proof.

\begin{proposition}\label{prop:compactness-lp}
    Let $p\in (1,\infty)$, $T\in (0,+\infty)$, $\{f^n\}_n$ be a bounded sequence in $\mathcal{B}_b([0,T];L^p_x)$; further assume that, for each $\varphi\in C^\infty_c$, the functions $\{t\mapsto \langle f^n_t, \varphi\rangle\}_n$ are equicontinuous on $[0,T]$.
    Then $\{f^n\}_n\subset C_w([0,T];L^p_x)$ and there exists a subsequence $\{f^{n_k}\}_k$ such that $f^{n_k}\to f$ in $C_w([0,T];L^p_x)$.
    
    In the case $p=\infty$, under the same assumptions, all the conclusion of the statement still apply, with the only differences that $\{f^n\}_n\subset C_{w-\ast}([0,T];L^\infty_x)$ and $f^{n_k}\to f$ in $C_{w-\ast}([0,T];L^\infty_x)$.
\end{proposition}

\begin{proof}
    We only treat $p\in (1,\infty)$, the case $p=\infty$ being similar.
    Fix $R\geq 0$ large enough such that $\| f^n_t\|_{L^p_x} \leq R$ for all $n\in\N$ and $t\in [0,T]$; consider a countable collection $\{\psi^m\}_{m\in\N}\subset C^\infty_c$ which is dense in $L^{p'}_x$.
    By the assumptions, for any fixed $m$, $\{t\mapsto \langle f^n_t, \varphi^m\rangle\}_n$ are equibounded, equicontinuous maps, therefore by Ascoli-Arzelà they are precompact in $C([0,T])$.
    Performing a Canton diagonal extraction argument, we can then find a subsequence such that $\langle f^{n_k}_t, \psi^m\rangle$ converge as $k\to\infty$ to some limit $g^m$ in $C([0,T])$, for all $m\in\N$.
    From now on we will work only with this subsequence and drop the subindex $n_k$ for simplicity.

    Fix any $t\in [0,T]$, then the sequence $\{f^n_t\}_n$ is bounded in $L^p_x$, thus precompact in the weak topology; at the same time, since $\{\psi^m\}_m$ are dense in $L^{p'}_x$ and $\langle f^n_t, \psi^m\rangle$ converge for all $m$, there can exist at most one limit point $f_t$ for the whole sequence.
    We can conclude that $f^n_t\rightharpoonup f_t$ in $L^p_x$, for all $p\in [0,T]$, and that $g^m_t=\langle f_t,\psi^m\rangle$.
    Since $\| f^n_t\|_{L^p}\leq R$, lower semicontinuity of the norm in the weak topology implies that $\| f_t\|_{L^p}\leq R$ as well, for all $t\in [0,T]$, so that $f\in \mathcal{B}_b([0,T];L^p_x)$.
    
    It remains to show convergence in $C_w([0,T];L^p_x)$. 
    Given $\psi\in L^{p'}_x$, fix $\eps>0$ and choose $\psi^m$ such that $\| \psi-\psi^m\|_{L^{p'}_x}<\eps$; then
    \begin{align*}
        \limsup_{n\to\infty} \sup_{t\in [0,T]} |\langle f^n_t-f_t,\psi \rangle|
        & \leq \limsup_{n\to\infty} \sup_{t\in [0,T]} |\langle f^n_t-f_t,\psi-\psi^m \rangle| + \limsup_{k\to\infty} \sup_{t\in [0,T]} |\langle f^n_t-f_t,\psi^m \rangle|\\
        & \leq \limsup_{k\to\infty} \sup_{t\in [0,T]} \| f^n_t-f_t\|_{L^p_x} \| \psi-\psi^m\|_{L^{p'}_x}
        \leq 2 R \eps.
    \end{align*}
    By the arbitrariness of $\eps>0$, the conclusion follows.    
\end{proof}

Finally, we provide a criterion to upgrade uniform convergence in the weak topology to uniform convergence in the strong one.

\begin{corollary}\label{cor:compactness-lp}
    Let $p\in (1,\infty)$, $T\in (0,+\infty)$ and $\{f^n\}_n$ be a sequence satisfying the assumptions of Proposition \ref{prop:compactness-lp} and such that $f^n\to f$ in $C_w([0,T];L^p_x)$.
    Then the following are equivalent:
    \begin{itemize}
        \item[i.] The functions $\{t\mapsto \| f^n_t\|_{L^p_x} \}_n$ are equicontinuous on $[0,T]$ and
        \begin{align*}
            \limsup_{n\to\infty} \| f^n_t\|_{L^p_x} \leq \| f_t\|_{L^p_x} \quad \forall\, t\in [0,T].
        \end{align*}
        \item[ii.] $f^n,f\in C([0,T];L^p_x)$ and we have the strong convergence
        \begin{equation}\label{eq:cor-compactness-lp}
            \lim_{n\to\infty} \sup_{t\in [0,T]} \| f^n_t-f_t\|_{L^p_x} = 0.
        \end{equation}
    \end{itemize}
\end{corollary}

\begin{proof}
    The implication $ii.\Rightarrow i.$ is trivial, so let us show $i.\Rightarrow ii.$
    
    Recall that, for any sequence $g^n$ in $L^p_x$ such that $g^n\rightharpoonup g$ and $\| g^n\|_{L^p_x} \to \| g\|_{L^p_x}$, it holds $\|g^n- g\|_{L^p_x}\to 0$; this is a consequence of uniform convexity of $L^p$ spaces and the associated properties, see \cite[Sections 3.7 \& 4.3]{brezis2011functional}.
    In view of this fact and the assumptions, we immediately deduce that $f^n\in C([0,T];L^p_x)$ for all $n$; moreover by $f^n_t\rightharpoonup f_t$, lower semicontinuity of $\| \cdot\|_{L^p_x}$ in weak topologies and the assumptions, it holds
    \begin{align*}
        \lim_{n\to\infty} \| f^n_t\|_{L^p_x} = \| f_t\|_{L^p_x} \quad \forall\, t\in [0,T].
    \end{align*}
    As the family $\{t\mapsto \| f^n_t\|_{L^p_x} \}_n$ is uniformly continuous, thus compact, we deduce that the above convergence is not only pointwise but rather uniform, namely
    \begin{equation}\label{eq:cor-compactness-proof1}
        \lim_{n\to\infty}\sup_{t\in [0,T]} \big| \| f^n_t\| - \| f_t\| \big| = 0;
    \end{equation}
    but then $t\mapsto \| f_t\|_{L^p_x}\in C([0,T])$ and arguing as above $f\in C([0,T];L^p_x)$.
    Next recall that, since $f^n\to f$ in $C_w([0,T];L^p_x)$ by assumption, by Remark \ref{rem:uniform_weak_convergence} it holds that
    \begin{equation}\label{eq:cor-compactness-proof2}
        f^n_{t_n}\rightharpoonup f_t \text{ in } L^p_x \text{ as } n\to\infty \text{ for any sequence $\{t_n\}_n\subset [0,T]$ such that } t_n\to t.
    \end{equation}
    Suppose now by contradiction that the conclusion \eqref{eq:cor-compactness-lp} is not true; then there exist $\delta>0$ and a sequence $\{t_n\}_n$ such that $\| f^n_{t_n} - f_{t_n}\|_{L^p_x} >\delta >0$ for all $n$ sufficiently large.
    Since $[0,T]$ is compact, we can further assume such that $t_n\to t$ for some $t\in [0,T]$.
    But then by \eqref{eq:cor-compactness-proof1} and \eqref{eq:cor-compactness-proof2} it holds $f^n_{t_n}\rightharpoonup f_t$ and $\| f^n_{t_n}\|_{L^p_x}\to \| f_t\|_{L^p_x}$, which by uniform convexity implies $f^n_{t_n}\to f_t$ strongly in $L^p_x$. Since $f\in C([0,T];L^p_x)$, $f_{t_n}\to f_t$ strongly in $L^p_x$ as well; but then by triangular inequality we find    
    $0= \lim_{n\to\infty} \| f^n_{t_n} - f_{t_n}\|_{L^p_x}>\delta>0$, which is absurd.
\end{proof}

\begin{remark}
    It is clear from the proof that a statement in the style of Corollary \ref{cor:compactness-lp} holds for $L^p_x$ replaced by any separable, uniformly convex Banach space $V$.
    On the other hand, it does not hold for non-uniformly convex spaces like $L^1_x$: already without dependence on $t$, a standard counterexample on $L^1((0,2\pi))$ is given by $f^n(x)=1+\sin(nx)$, satisfying $f^n\rightharpoonup 1$, $\| f^n\|_{L^1(0,2\pi)}=2\pi=\|1\|_{L^1(0,2\pi)}$ for all $n$, but $\| f^n-1\|_{L^1(0,2\pi)}=4\neq 0$ for all $n$.
\end{remark}

\section{Smoothing operators and tensor products}\label{app:smoothing}

The first aim to this section is to present the proof of Lemma \ref{lem:smoothing-examples}. To this end, we start by treating specifically Part b) of the statement in the proposition below; 
recall the definition of $\clF_{l,R}$ by \eqref{eq:F_k,R}.

\begin{proposition}\label{prop:smoothing-F_kR}
    For all $R\in [3,\infty)$, the scale $(\clF_{l,R})_{0\leq l \leq 3}$ admits a smoothing $(J^\eta)_{\eta\in (0,1]}$, in the sense of Definition \ref{defn:smoothing}; moreover it can be chosen in such a way that $\interleave J\interleave$ does not depend on $R$.
\end{proposition}

\begin{proof}
    As the proof is very similar to other existing in the literature, cf. \cite[Section 5.3]{DGHT2019} and \cite[Proposition A.3]{hocquet2018energy}, we will mostly sketch it, making however particular attention on explaining why the resulting estimates do not degenerate as $R$ becomes larger and larger.
    For practical convenience, we will restrict ourselves to $R\geq 3$; it is clear however that, up to tuning the parameters correctly, any $R\geq 1$ could be considered (or more generally $R\geq R_0>0$, for any fixed threshold $R_0$, in which case the resulting estimates only depend on $R_0$).

    For $\eta\in (0,1)$, consider a family $\{\Theta^{R,\eta}\}_{\eta\in (0,1)}\subset C^\infty_c$ such that $\Theta^{R,\theta}(x)\in [0,1]$ for all $x$ and
    \begin{equation}\label{eq:cutoff-requirements}
        \Theta^{R,\eta}\equiv 1 \text{ on } B_{R-3\eta},
        \quad \Theta^{R,\eta}\equiv 0 \text{ on } \R^d\setminus B_{R-2\eta},
        \quad | D^k \Theta^{R,\eta}| \lesssim \eta^{-k}\quad \forall\, k\geq 0.
    \end{equation}
    Let us illustrate a standard way to construct $\Theta$. Consider a smooth, radially symmetric probability density $h$ supported on $B_1$; define the associated mollifiers $h^\eta=\eta^{-d} h(\eta^{-1} \cdot)$ and $M^\eta \varphi = h^\eta\ast \varphi$.
    Then we can take $\psi^R(x):= \mathbbm{1}_{B_{R-5\eta/2}}$ and $\Theta_\eta := M^{\eta/2} \psi^R$.
    The verification of the first two requirements from \eqref{eq:cutoff-requirements} are immediate, in light of the supports of $\psi^R$ and $h^{\eta/2}$ and properties of convolution; for the last one, by Young's convolutional inequality and scaling it holds
    \begin{align*}
        \| D^l \Theta^{R,\eta}\|_{L^\infty_x}
        \leq \| \psi^R\|_{L^\infty_x} \| D^l h^{\eta/2}\|_{L^1_x}
        \lesssim \eta^{-l} \| D^l h\|_{L^1_x}\quad \forall\, l\geq 0
    \end{align*}
    uniformly in $R$.
    Arguing as in \cite[Lemma A.3]{hocquet2018energy}, one can then show that
    \begin{equation}\label{eq:cutoff-prop1}
        \| \Theta^{R,\eta} \varphi\|_{W^{l,\infty}} \leq C_h \| \varphi\|_{W^{l,\infty}}
        \quad \text{for all } l\in \{0,1,2\} \text{ and } \varphi\in \clF_{l,R}
    \end{equation}
    as well as
    \begin{equation}\label{eq:cutoff-prop2}
        \| (1-\Theta^{R,\eta}) \varphi\|_{W^{j,\infty}} \leq C_h \eta^{l-j} \| \varphi\|_{W^{l,\infty}}
        \quad \text{for all } l,j \in \{0,1,2,3\} \text{ with } j\leq l \text
        { and } \varphi\in \clF_{l,R},
    \end{equation}
    where the costant only depends on the choice of the smooth function $h$.
    Let us shortly illustrate the main difficulty in deriving the bounds \eqref{eq:cutoff-prop1}-\eqref{eq:cutoff-prop2} and how to overcome it.
    Whenever estimating $W^{l,\infty}$-norms, the most problematic terms comes from the highest order derivatives $\partial^\alpha$, $|\alpha|=l$, falling on $\Theta^{R,\eta}$ as a consequence of looking at $\partial^\alpha (\Theta^{R,\eta} \varphi)$ and Leibniz rule; the main task is to estimate $\| (\partial^\alpha \Theta^{R,\eta}) \varphi\|_{L^\infty_x}$, where $\| \partial^\alpha_x \Theta^{R,\eta}\|_{L^\infty_x}$ grows like $\eta^{-|\alpha|}$.
    However, the only $x$ where $D^2_x\Theta_\eta(x)$ is non zero by construction belong to $\clA_{R,\eta}:= B_{R-2\eta}\setminus B_{R-\eta}$; by hypothesis $\varphi$ is supported on $B_R$, therefore all its derivative vanish at the boundary; for any $x\in \mathcal{A}_{R,\eta}$, by Taylor expanding $\varphi$ around $\bar x = R x/|x|$, we deduce that
    \begin{align*}
        |\varphi(x)|
        \lesssim \| \varphi\|_{W^{l,\infty}}\, |x-\bar x|^l \lesssim \| \varphi\|_{W^{l,\infty}}\, \eta^l
        \quad \forall\, x\in\clA_{R,\eta}
    \end{align*}
    Such estimates compensate the growth of derivatives of $\Theta^{R,\eta}$ for $x\in\clA_{R,\eta}$, allowing in the end to prove to prove the desired bounds \eqref{eq:cutoff-prop1}-\eqref{eq:cutoff-prop2}.

    Next, let us recall that the smoothing operators ${M^\eta}$ as defined above form a smoothing on the scales $(W^{l,\infty})_{0\leq l\leq 3}$, cf. \cite[Lemma A.2]{hocquet2018energy}.

    Finally, we can define a family of operators $(J^{R,\eta})_\eta$ on the scale $\clF_{l,R}$ by
    \begin{equation*}
        J^{R,\eta} \varphi:= M^\eta(\Theta^{R,\eta} \varphi).
    \end{equation*}
    The definition is meaningful: by construction, ${\rm supp}\, (\Theta^{R,\eta} \varphi) \subset B_{R-2\eta}$, so that by properties of convolutions ${\rm supp}\, (J^{R,\eta} \varphi) \subset B_{R-\eta}$, showing that $J^{R,\eta} \varphi\in \clF_{l,R}$.
    
    The proof that all conditions from Definition \ref{defn:smoothing} are satisfied by $(J^{R,\eta})_\eta$ follows from combining the bounds \eqref{eq:cutoff-prop1}-\eqref{eq:cutoff-prop2} (which were uniform in $R$), the fact that $(M^\eta)_\eta$ is a smoothing on $W^{l,\infty}$ and the triangular inequality.
    For instance, for any $(j,l)\in\{(0,1),(0,2),(1,2)\}$ and $\varphi\in \clF_{l,R}$, it holds
    \begin{align*}
        \| \varphi - J^{R,\eta} \varphi\|_{W^{j,\infty}}
        & \leq \| \varphi - M^\eta \varphi\|_{W^{j,\infty}} + \| M^\eta(\varphi - \Theta^{R,\eta} \varphi) \|_{W^{j,\infty}}\\
        & \lesssim \eta^{l-j} \| \varphi\|_{W^{l,\infty}} + \|\varphi - \Theta^{R,\eta} \varphi \|_{W^{j,\infty}}
        \lesssim \eta^{l-j} \| \varphi\|_{W^{l,\infty}}
    \end{align*}
    where we repeatedly used the smoothing properties of $(M^\eta)_\eta$ and the bounds \eqref{eq:cutoff-prop2}.
    The verification of the other properties is similar; overall, this yields the desired uniform-in-$R$ bounds for $\interleave J^R\interleave$.
\end{proof}

We are now ready to complete the

\begin{proof}[Proof of Lemma \ref{lem:smoothing-examples}]
    Part b) of the statement is covered by Proposition \ref{prop:smoothing-F_kR} above, while part c) is given by \cite[Corollary 5.4]{DGHT2019}.
    The statement from a) is much more classical and comes from interpolation theory, but let us briefly give an explicit example of the smoothing $(J^\eta)_\eta$.
    Consider the inhomogeneous Littlewood-Paley blocks $(\Delta_j)_{j\geq -1}$ associated to a partition of unity and the corresponding low-frequency cut-off operator $S_j$ given by
    \begin{equation*}
        S_j \varphi = \sum_{i\leq j-1}\Delta_i \varphi,
    \end{equation*}
    see \cite{BCD} for more details. Set $J^\eta:= S_{j(\eta)}$ for $j(\eta)$ defined as the integer part of $-\log_2 \eta$; then $(J^\eta)_{\eta}$ satisfies the requirements \eqref{eq:defn-smoothing-prop1}-\eqref{eq:defn-smoothing-prop2}.
    Indeed, the boundedness of $S_j$ as element of $\mathcal{L}(W^{l,p},W^{l,p})$ comes from \cite[Remark 2.11]{BCD}, while the bounds for $\| J^\eta\|_{\cL(W^{l,p},W^{l',p})}$ follow from a standard application of Bernstein estimates, cf. \cite[Lemma 2.1]{BCD}.    
\end{proof}

Next we present the proof of Lemma \ref{lem:tensor-product}; recall again the definitions of $\cF_{l,R}$ and $\clE_{l,R}$ from \eqref{eq:F_k,R} and \eqref{eq:E_k,R} respectively.

\begin{proof}[Proof of Lemma \ref{lem:tensor-product}]
    In order to prove the claim we need to show that, for any given $f\in \cF_{-l,R+1}$, $g\in \cF_{-j,R+1}$ and $\Phi\in\clE_{l+j,R}$, it holds
    \begin{align*}
        |\langle f\otimes g, \Phi\rangle| \lesssim  \| f\|_{\cF_{-l,R+1}} \| g\|_{\cF_{-j,R+1}} \| \Phi\|_{\clE_{l+j},R};
    \end{align*}
    recall that the distributional tensor product is defined (or equivalently characterized) by duality by
    \begin{align*}
        \langle f\otimes g, \psi\rangle := \int_{\R^{2d}} f(x) g(y) \psi(x,y)\, \dd x \dd y \qquad \forall\, \psi\in C^\infty_c(\R^{2d})
    \end{align*}

    Fix $\Phi\in \clE_{l+j,R}$; observe that by definition its support is included in $\{(x,y): |x_+|\leq R, |x_{-}|\leq 1\}$ and thus in $B_{R+1}\times B_{R+1} = \{(x,y): |x|\vee |y|\leq R+1\}$.
    For fixed $g$, introduce the auxiliary function
    \begin{align*}
        \varphi^g(x) := \langle  g, \Phi(x,\cdot) \rangle
        =\int_{\R^d} g(y) \Phi(x,y) \dd y  
    \end{align*}
    so that $\langle f\otimes g, \psi\rangle = \langle f, \varphi^g\rangle$.
    In order to conclude, it then suffices to show that
    \begin{align*}
         \| \varphi^g\|_{\cF_{l,R+1}}
         \lesssim \| g\|_{\cF_{-j,R+1}} \| \Phi\|_{\clE_{l+j,R}} 
    \end{align*}
    since in that case
    $|\langle f, \varphi^g\rangle|
    \leq\|f\|_{\cF_{-l,R+1}} \| \varphi^g\|_{\cF_{l,R+1}}
    \lesssim \|f\|_{\cF_{-l,R+1}} \| g\|_{\cF_{-j,R+1}} \| \Phi\|_{\clE_{l+j,R}}$.

    By the previous observation on the support of $\Phi$, it's clear that $\Phi(x,\cdot)\in \cF_{j,R+1}$ for all $x$, so that $\varphi^g$ is well defined, and that moreover $\varphi^g(x)=0$ for all $x\in B_{R+1}^c$. Therefore we only need to check that $\varphi^g\in W^{l,\infty}$.
    If $l=0$ this is immediate, since by definition
    \begin{align*}
        \sup_{x\in\R^d} |\varphi^g(x)|
        \leq \| g\|_{\cF_{-j,R+1}} \sup_{x\in\R^d} \| \Phi(x,\cdot)\|_{\cF_{j,R+1}}
        \lesssim \| g\|_{\cF_{-j,R+1}}\| \Phi\|_{\clE_{j,R}}.
    \end{align*}
    Assume now $l>0$; we first claim that for all derivatives $\alpha$ up to order $l-1$, we have the relation
    $\partial^\alpha_x \varphi^g(x) = \langle g, \partial^\alpha_x \Phi(x,\cdot)\rangle$.
    It is enough to prove this for $|\alpha|=1$ and $l\geq 2$, since the general case follows by iteration.
    By using difference quotients, we have
    \begin{align*}
        \partial_{x_i} \varphi^g(x)
        = \lim_{h\to 0} \Big\langle g, \frac{\Phi(x+h e_i,\cdot)-\Phi(x,\cdot)}{h} \Big\rangle
        = \langle g, \partial_{x_i} \Phi(x,\cdot)\rangle
    \end{align*}
    where the last passage comes from the easily verifiable fact that, if $\Phi\in W^{l+j,\infty}$ with $l\geq 2$, then $(\Phi(x+he_i,\cdot)-\Phi(x,\cdot))/h$ converges to $\partial_{x_i} \Phi(x,\cdot)$ in $W^{j,\infty}$, for any fixed $x\in\R^d$.

    Having established the claim, in order to conclude that $\varphi^g\in W^{l,\infty}$, it suffices to show that all its derivatives of order $\alpha$ with $|\alpha|= l-1$, are bounded and globally Lipschitz (recall that $f\in W^{1,\infty}$ if and only if $f$ is bounded and Lipschitz, in which case its optimal Lipschitz constant is exactly $\| Df\|_{L^\infty_x}$).
    We omit boundedness, which is simpler, and pass to verify the Lipschitz property.
    Observing that for $\Phi\in \clE_{l+j,R}$, it holds $\partial^\alpha_x \Phi\in \clE_{j+1,R}$, we find
    \begin{align*}
        |\partial^\alpha_x \varphi^g(x) - \partial^\alpha_x \varphi^g(z)|
        & = |\langle g, \partial^\alpha_x \Phi(x,\cdot) - \partial^\alpha_x \Phi(z,\cdot)  \rangle|\\
        & \leq \| \partial^\alpha_x \Phi(x,\cdot) - \partial^\alpha_x \Phi(z,\cdot)\|_{\clF_{j,R+1}} \| g\|_{\clF_{-j,R+1}}\\
        & \lesssim \| \partial^\alpha_x \Phi\|_{\clE_{j+1,R}} \| g\|_{\clF_{-j,R+1}} |x-z| 
        \lesssim \| \Phi\|_{\clE_{l+j,R}} \| g\|_{\clF_{-j,R+1}} |x-z|,
    \end{align*}
    which overall proves that $\| \varphi^g\|_{W^{l,\infty}} \lesssim \| \Phi\|_{\clE_{l+j,R}} \| g\|_{\clF_{-j,R+1}}$ and yields the conclusion. 
\end{proof}

\section*{Acknowledgements and funding information}
We are grateful to Daniel Goodair for discussions about $2$D Euler equations on bounded domains and for pointing out the results from \cite{Goodair2023}.

LG is supported by the SNSF Grant 182565 and by the Swiss State Secretariat for Education, Research and Innovation (SERI) under contract number MB22.00034, and by the Istituto Nazionale di Alta Matematica (INdAM) through the project GNAMPA 2025 “Modelli stocastici in Fluidodinamica e Turbolenza”.

JML acknowledges the support received from the US AFOSR Grant FA8655-21-1-7034.

\bibliographystyle{alpha}
\bibliography{bibliography}
\end{document}